\documentclass[a4paper,twoside,11pt]{amsart} 
\usepackage{amsmath,amssymb,amsfonts} 
\usepackage{amsthm}
\usepackage{graphics} 
\usepackage{color} 
\usepackage{hyperref} 
\usepackage{graphicx,eucal}
\usepackage{a4wide} 
\usepackage{enumerate} 
\usepackage[all]{xy} 
\usepackage{changebar}
\cbcolor{red}
\nochangebars
\usepackage{tikz}
\usepackage{tikz-cd}
\usepackage{tcolorbox}

\newcommand{\U}{\mathcal{U}}

\newcommand{\image}{\text{Im}}
\newcommand{\curvecomplex}{\mathcal{C}}
\newcommand{\betas}{\mathfrak{b}}

\newcommand{\mrs}{\mathcal{S}}
\newcommand{\mct}{\overline{\mathcal{S}}^c}

\newcommand{\mnc}{\overline{\mathcal{S}}}
\newcommand{\nc}{\overline{\mathcal{M}}}

\newcommand{\const}{\mathrm{const}}

\newcommand{\prel} {p_{rel}}

\newcommand{\rank}{\operatorname{rank}}
\newcommand{\per}{\operatorname{Per}}

\newcommand{\percompact}{\mathcal{P}er}

\newcommand{\vol}{\operatorname{vol}}

\newcommand{\curvesystem}{\overline{c}}

\newcommand{\spl}{\text{simple}}
\newcommand{\Forget}{\text{For}}

\newtheorem{theorem}{Theorem}[section]
\newtheorem{proposition}[theorem]{Proposition}
\newtheorem{corollary}[theorem]{Corollary}
\newtheorem{lemma}[theorem]{Lemma}

\theoremstyle{definition}
\newtheorem{definition}[theorem]{Definition}
\newtheorem{remark}[theorem]{Remark}
\newtheorem{example}[theorem]{Example}

\title{A transfer principle: from periods to isoperiodic foliations}
\author{Gabriel Calsamiglia}
\address{Departamento de Matem\' atica Aplicada, Universidade Federal Fluminense,Rua Professor Marcos Waldemar de Freitas Reis, s/n, Campus do Gragoat\' a, Bloco G, Niter\' oi, 24210-201 Brazil} \email{gcalsamiglia@id.uff.br}
\author{Bertrand Deroin}
\address{CNRS \& Universit\'e Cergy-Pontoise, UMR 8088, 2 av. Adolphe Chauvin, 95302 Cergy-Pontoise Cedex, France}
\email{bertrand.deroin@u-cergy.fr}
\author{Stefano Francaviglia}
\address{Dipartimento di Matematica, Universit\`a di Bologna, P.zza
Porta S. Donato 5, 40126 Bologna, Italy }
\email{stefano.francaviglia@unibo.it}

\begin{document}
\maketitle
\begin{abstract}
 We classify the possible closures of leaves of the isoperiodic foliation 
 defined on the Hodge bundle over the moduli space of genus $g\geq 2$ curves and prove that the foliation is ergodic on those sets. The results derive from the connectedness properties of the fibers  of the period map defined on the Torelli cover of the moduli space. Some consequences on the topology of Hurwitz spaces of primitive branched coverings over elliptic curves are also obtained. To prove the results we develop the theory of augmented Torelli space, the branched Torelli cover of the Deligne-Mumford compactification of the moduli space of curves.
\end{abstract}

\section{Introduction}

\subsection{Overview}
Let $\Omega \mathcal M_g$ be the moduli space of abelian differentials on compact genus $g\geq 2$ smooth curves. The period of an element $(C,\omega) \in \Omega \mathcal M_g$ is the element of $H^1 (C, \mathbb C)\simeq \text{Hom} ( H_1(C,\mathbb Z), \mathbb C)$ that is defined by
\begin{equation}\label{eq: period}  \text{Per} (C,\omega)  : \gamma \in H_1 (C,\mathbb Z) \mapsto \int _\gamma \omega  \in \mathbb C. \end{equation}
The periods of an abelian differential do not allow to recover the abelian differential itself, even infinitesimally.  Actually, it is always possible to find non trivial isoperiodic deformations of a given abelian differential, namely an immersed complex submanifold $L\subset \Omega \mathcal M_g$ such that the period of a form $(C,\omega)\in L$ is a locally constant function, when we use the local identifications of the $H^* (C,\mathbb C)$'s given by the Gauss-Manin connection. For instance,  if $f_t:C_t\rightarrow C$ is a continuous family of degree $d$ branched coverings over $C$ depending on the parameter $t$, the period map of the family $\omega_t=f_t^*\omega$ is locally constant.

The case $g=2$ is instructive: every genus two curve is a double cover of $\mathbb P^1 $ ramified over six distinct points, say $0, 1, \infty, x_1, x_2, x_3$. An abelian differential on such a curve can be written as the hyperelliptic integrand $$ \frac{(ax+b ) dx} {\sqrt{x(x-1)(x-x_1)(x-x_2)(x-x_3)}}.$$ Picard-Fuchs theory (see \cite[p.60]{Clemens} for details) tells us that isoperiodic deformations on $\Omega \mathcal M_2$ are integral curves of the  following vector field

\begin{small}
\begin{equation}\label{eq: genus two}  \sum _j  \frac{x_j (1-x_j)}{ax_j+b} \frac{\partial }{\partial x_j}    -\frac{1}{2}  \frac{\partial }{\partial a} -\frac{1}{2} \Big(1+  \sum_j \frac{b(x_j-1)}{ax_j +b} \Big)  \frac{\partial }{\partial b}. \end{equation}
\end{small}
Apart from the invariant closed subsets characterized by topological properties of the set of periods, there is an interesting family of closed invariant sets for the isoperiodic foliation in genus two called the Hilbert modular invariant manifolds, introduced by Calta in \cite{Calta} and McMullen in \cite{McMullen3, McMullen4}. The set of curves of genus two whose Jacobian has real multiplication by a quadratic order of discriminant $D$ is a Hilbert modular surface $X_D$, and a remarkable fact is that the subset of the Hodge bundle  \(\Omega X_D\)  over $X_D$ consisting of eigenforms for the complex multiplication is invariant by the isoperiodic foliation. 

In any genus, the collection of all maximal isoperiodic deformations defines a holomorphic foliation $\mathcal F_g$ of $\Omega^* \mathcal M_g$, called the isoperiodic foliation\footnote{In the literature, this foliation is also called the kernel foliation, or the absolute period foliation. } (see \cite{Movasati} for further examples). It has dimension $2g-3$, and it is also algebraic: its leaves are solutions of a system of algebraic equations analogous to \eqref{eq: genus two} with respect to the Deligne-Mumford algebraic structure on moduli space. It admits a real analytic global first integral defined by the volume of the flat metric induced by the abelian differential. Complex multiplication on the forms induces isomorphisms between the restrictions of the foliation to the volume levels.

  It turns out that the restriction of $\mathcal{F}_g$ to the volume levels carries a transverse structure modelled on the homogeneous space 
\begin{equation} \label{eq: homogeneous space} \text{Sp} (2g,\mathbb R)/ U, \text{ with } U= I_ 2 \times \text{Sp} (2g-2,\mathbb R) \end{equation}
with changes of coordinates acting by right multiplications by the elements of the lattice \(\text{Sp} (2g, \mathbb Z) \).
There are some well known examples of closed invariant sets for the given restriction and an integer $d\geq 2$: the moduli spaces \(H_{g,d} \subset \Omega \mathcal M_g\) of forms that are pull-back of elliptic differentials by primitive branched coverings of degree \(d\). The set $H_{g,d}$ is saturated by the isoperiodic foliation and each leaf there is algebraic. 

The main goal of this paper is to investigate the topological properties of the isoperiodic foliation, and derive some dynamical consequences: we prove that any dynamical property satisfied by the action of the lattice \(\text{Sp} (2g, \mathbb Z)\) on the homogeneous space \eqref{eq: homogeneous space} can be transferred to a property satisfied by the isoperiodic foliation. This fact is what we call the transfer principle. Using Ratner's theory, it allows to describe explicitely the closure of \emph{each} leaf  and prove that it is a real analytic subset. We also obtain the ergodicity of the isoperiodic foliation on these sets, a fact that has been proven independently by Hamenst\" adt in \cite{Hamenstadt}.

The proof of the transfer principle  is achieved by studying the topological properties of the fibers of the period map defined on the moduli space of homologically marked genus \(g\) curves, with values in the affine space \(\mathbb C^{2g}\): we prove that all the fibers that do not correspond to leaves in $H_{g,2}$ with $g\geq 4$ are connected. This result provides a solution to a problem posed by McMullen  in \cite[p. 2282]{McMullen} for genus $g\geq 4$. The main idea of the proof of the connectedness of the fibers  is to show that any form can be deformed continuously and isoperiodically to a stable form on a nodal curve, and then connect all possible stable forms obtained in this way using an inductive argument on the genus. As for fibers formed by branched double coverings over an elliptic differential, we generalize an invariant described by Arnold \cite{Arnold} for branched double coverings over the Riemann sphere  (i.e. hyperelliptic loci) to prove they are disconnected when the genus is at least five.

 \subsection{Statement of results}
 
 A homologically marked genus \(g\) curve is a couple \((C,m)\), where \(C\) is a genus \(g\) curve and \(m : \mathbb Z^{2g} \rightarrow H_1 (C,\mathbb Z)\) is a symplectic identification. The moduli space of marked genus \(g\) curves is the so called Torelli covering \(\mrs_g \rightarrow\mathcal M_g\). Its covering group  $\mathcal{I}_g\subset\pi_{1}(\mathcal{M}_g)$-- called the Torelli group-- is identified with the set of classes in the mapping class group of an orientable closed surface of genus $g$ that act trivially on the first homology group of the surface.   We denote by \(\Omega\mrs_g\) the pull-back of the Hodge bundle to \( \mrs_g\). We analyze the topological properties of the period map
$$
\text{Per}_g: \Omega\mrs_g\rightarrow \mathcal H_g\cup 0 \text{ defined by }\per_g(C,m,\omega)=\per(C,\omega)\circ m .$$
where \(\mathcal H_g \subset \text{Hom} (\mathbb Z^{2g} , \mathbb C)  \) denotes the open subset formed by periods of non zero abelian differentials on a homologically marked genus \(g\) curve. The isoperiodic foliation lifts to a foliation of $\Omega\mrs_g$ whose leaves are the connected components of the level sets of the period map. 
 \begin{definition}[Primitive degree]\label{def:primitive degree}
 Given $p\in \text{Hom} (\mathbb{Z}^{2g}, \mathbb C) $ we define 
 \begin{itemize}
 \item its volume $\vol(p)=\Re (p)\cdot \Im(p)$, the  symplectic product on $\text{Hom} (\mathbb Z^{2g} , \mathbb R) $ 
 \item its primitive degree, denoted $\deg(p)$ as $\infty$ if $\Lambda= p (\mathbb Z^{2g} ) \subset \mathbb{C}$ is non-discrete and as $$\deg (p)=\frac{\vol(p)}{\vol(\mathbb{C}/\Lambda)}\text{ if } \Lambda\subset \mathbb{C} \text{ is discrete. }$$ 
 \end{itemize}

\end{definition}

When $p\in \text{Hom} (\mathbb Z^{2g} , \mathbb{C} ) $ is the period of some non-zero abelian differential \(\omega\) on a homologically marked smooth curve \((C,m)\), $\vol(p)$ corresponds to its volume \(\frac{i}{2} \int \omega\wedge \overline{\omega} \) too, and is therefore positive. When the periods of the form define a lattice $\Lambda\subset\mathbb{C}$, the number $\deg(p)$  corresponds to the topological degree of the primitive branched covering $C\rightarrow \mathbb{C}/\Lambda$ defined by integration of the form (i.e. it induces a surjection at the first homology group level). The degree is a positive integer and it is one only if $g=1$. The conditions $\vol(p)>0$ and either $g=1$ or $\deg(p)>1$  will be referred to as \text{Haupt conditions}. 

 
 Our main result is 
 
\begin{theorem} 
\label{t:connectedness}
For $g\geq 2$, the fibers of the period map $\per_g$ over points of primitive degree at least three are connected.
\end{theorem}

Its proof, whose structure is detailed in section \ref{ss: strategy}, will take up most of the paper. 

Regarding the fibers of $\per_g$ over points of primitive degree two, it was already known that they are connected for $g=2,3$ (see \cite{McMullen}). We suspect that the equivalent statement is still true in genus four, but we were not able to prove it. Nevertheless, in higher genera we prove

\begin{theorem}\label{t:degree 2} If $g\geq 5$, 
the fibers of $\per_g$ over points of primitive degree two are disconnected. 
\end{theorem}

 Theorem \ref{t:degree 2} is proved using an adaptation of a method due to Arnold \cite{Arnold} based on an invariant associated to branched double coverings. It enabled him to prove that the symplectic representation \(\pi_1 (H_{g,2}(\mathbb{P}^1)) \rightarrow \text{Sp} (2g, \mathbb Z) \), induced by the natural inclusion $H_{g,2}(\mathbb{P}^1)\rightarrow \mathcal{M}_g$ of the Hurwitz space $H_{g,2}(\mathbb{P}^1)$ of genus $g$ branched double coverings over the Riemann sphere, is not onto for \(g\geq 3\). In our case the invariant will be extended to branched double coverings over an elliptic curve.

 The connectedness of the fibers of the lift of  $\per_g$ to the universal cover of $\Omega\mathcal{M}_g$ fails in general. For example, in genus $g=3$ the fibers of $\per_3$ over points of primitive degree two are biholomorphic to Siegel space $\mathfrak{S}_2$, hence simply connected. Therefore there are infinitely many components of the lift of such a fiber to the universal cover of $\Omega\mathcal{M}_3$. In genus $g=2$, the generating family of $\pi_1(\mathcal{M}_2)$ given by  Mess in \cite{Mess} allows to prove that the lift of the generic fiber is disconnected (see Corollary \ref{c: fibers on T_2 are disconnected}). In both cases the proof relies on the fact that the projection of the isoperiodic sets to Siegel space via the Torelli map, do not accumulate on some of the components of the boundary of Torelli space.



A first application of these results concerns the topology of Hurwitz spaces of branched coverings over elliptic curves. The connectedness of the moduli space of genus $g>1$ primitive branched coverings of degree \(d\) over an elliptic curve was proven by Berstein and Edmonds (see \cite{BE}). Our method allows to retrieve this result, and to get some new information on the fundamental group of these moduli spaces whenever $d>2$:

\begin{corollary}
Let $H_{g,d}(E)$ be the Hurwitz space of degree $d$ and genus $g\geq 2$ primitive branched coverings over the elliptic curve $E$. Let $p: H_1(\Sigma_g)\twoheadrightarrow H_1(E) $ be the homology map of any of its elements. For $d\geq 3$ the homomorphism \begin{equation}\label{eq: homomorphism} \pi_1(H_{g,d}(\mathbb{C}/\Lambda))\rightarrow\text{Stab}(p)\subset\text{Sp}(2g,\mathbb Z)\text{ is onto.}\end{equation} 

\end{corollary}

This result is in fact more general: the fundamental group of the leaves of the isoperiodic foliation that correspond to a period $p\in\mathcal{H}_g$ when $\deg(p)\neq 2$ surjects onto \(\text{Stab} (p)\subset \text{Sp}(2g,\mathbb Z)\). It is an immediate consequence of Theorem \ref{t:connectedness} (see Remark \ref{rem:fundamental group of leaves}).

Let us mention that, by analogy with the fact that the level sets of the period map might be disconnected in the universal cover of \(\Omega\mathcal M_g\),  the morphism from the fundamental group of a leaf to $\pi_1(\Omega\mathcal{M}_g)$ -- which is isomorphic to the mapping class group-- has an image that might be strictly contained in the stabilizer of its corresponding period. This is precisely what happens in genus two (see Corollary \ref{c: fibers on T_2 are disconnected}).

A second application concerns the transfer principle: the map $\per_g$ is equivariant with respect to the action of the covering group   $\text{Sp}(2g,\mathbb{Z})$ of $\pi:\Omega\mrs_g\rightarrow\Omega\mathcal{M}_g$ if we consider its action on $\text{Hom}(\mathbb{Z}^{2g},\mathbb{C})$ by precomposition. This implies that to any $\mathcal{F}_g$-saturated subset $B\subset\Omega\mathcal{M}_g$ we can associate the $\text{Sp}(2g,\mathbb{Z})$-invariant subset $ A=\per_g(\pi^{-1}(B))\subset\mathcal{H}_g$.
Reciprocally, for each $\text{Sp}(2g,\mathbb{Z})$-invariant subset $A\subset\mathcal{H}_g$ we can associate an $\mathcal{F}_g$ saturated subset $B=\pi(\per_g^{-1}(p))$.  These correspondences are inverse one of the other if and only if the fibers of $\per_g$ over $A$ are connected. This correspondence is what allows to transfer properties of the $\text{Sp}(2g,\mathbb{Z})$ action on the homogeneous space \eqref{eq: homogeneous space} to properties of the isoperiodic foliation, and justifies the title of the paper.
 It view of Theorem \ref{t:connectedness},  the transfer principle can be applied on the complement of the primitive degree two  points  of $\mathcal{H}_g$.  
Using Ratner's theory, we deduce the following
\begin{theorem} [Dynamics of isoperiodic foliations]
\label{t:dynamics}
Let $g>2$ and $(C,\omega) \in \Omega^* \mathcal M_g$ of volume $V= \frac{i}{2} \int \omega \wedge \overline{\omega}$ and $\Lambda$ the closure of the image of its periods. Then the closure of the leaf $L(C,\omega)$ passing through $(C,\omega)$ is, up to the action of $\text{GL}(2,\mathbb R)$
\begin{itemize}

\item ($\Lambda$ is discrete) the component of Hurwitz space consisting of genus $g$ primitive branched coverings over $(\mathbb C / \Lambda, dz)$ of volume $V$.
\item ($\Lambda$ is $\mathbb R + i \mathbb Z$) the set of abelian differentials with periods contained in $\Lambda$, with primitive imaginary part, and with volume $V$,
\item ($\Lambda = \mathbb C$) the subset of $\Omega \mathcal M_g$ consisting of abelian differentials of volume $V$,
\end{itemize}
If $g=2$ the same statement holds, with an extra possibility occurring when $\omega$ is an eigenform for real multiplication by a real quadratic order $\mathfrak{o}_D$ of discriminant $D>0$. In this case the closure is the Hilbert modular invariant submanifold $\Omega X_D$.

Moreover, the restriction of $\mathcal F_g$ to any of these real analytic subsets of $\Omega \mathcal M_g$ is ergodic with respect to the Lebesgue class.
\end{theorem} 

 Note in particular that this result classifies the algebraic leaves of the isoperiodic foliation: the only closed leaves correspond to Hurwitz spaces and they are known to be algebraic. 

\subsection{Strategy of the proof of Theorem \ref{t:connectedness}} \label{ss: strategy}

McMullen proved that each fiber of the period map- referred to as isoperiodic moduli space-  embeds, via the Torelli map, in Siegel space  $\mathfrak{S}_g$. The image is a Zariski dense subset in a slice of the Schottky locus by a Siegel space of lower genus $\mathfrak{S}_{g-1}$. In particular, in genus 2 or 3, he deduces that it is connected. In higher genera, the complicated geometry of the Schottky locus makes it difficult to implement this analytical approach.

Our strategy consists in analyzing the isoperiodic degenerations of abelian differentials towards abelian differentials on nodal stable curves. We point out that we are merely interested in degenerating the underlying curve, not the abelian differential (i.e.  the form does not degenerate to a zero form on any component). The first step of the proof is to prove that it is always possible to find such degenerations. Adding the obtained nodal forms allows to define a bordification of the isoperiodic moduli space, naturally stratified by the number of nodes of the underlying curves. We prove that this stratification is locally homeomorphic to the pull back by an infinite abelian covering of the stratification associated to a collection of normal crossing divisors on a smooth manifold, ramifying along some components of the divisors. In particular, the connectedness of the fiber is equivalent that of the bordification. For genus at least four, we prove the connectedness of the boundary of a fiber that is not formed by primitive double coverings of an elliptic differential by induction on the genus. This step is essentially achieved by algebraic arguments that encode the combinatorics of the boundary stratification.

As in other instances of this degeneration strategy -- see \cite{5A18,5A19} for contemporary examples with strata of differentials-- the definition of the bordification requires a comprehensive theory. In the present case the convenient ambient space is the augmented Torelli space, quotient of the augmented Teichm\" uller space by the Torelli group. For lack of a comprehensive adapted reference we include the theory in section \ref{s:augmented Torelli and period maps} and a guideline for the proof in section \ref{s:guideline}. 

\subsection{Notes and references} \label{ss: notes and references}
    
Our strategy for the proof of the connectedness of the boundary of a fiber of $\per$ works for $g\geq 4$ but leads to difficult algebraic problems due to Haupt's conditions when $g=2,3$. The description of the fibers as slices of Schottky space in those cases, 
is crucial for the inductive argument to work. 

The same type of difficulties were already present in the characterization of periods of holomorphic forms on Riemann surfaces by Otto Haupt in \cite{Haupt}, namely that those are the periods satisfying Haupt's conditions.  
Most of his work boiled down to treat the arithmetic involved in the genus two case, caused by forms belonging to the Hilbert modular invariant manifolds $X_D$ described above. From the point of view of slices of Schottky loci, the proof of Haupt consists in finding, by elementary algebraic methods, a point in the intersection of the slice with the boundary strata of Torelli space and then apply a surgery.  We actually provide a simpler proof of the result of Otto Haupt, which only differs by the remark that the genus two case can be treated using the Torelli map. Torelli's result (in \cite{Torelli}) was published two years before his death, in 1913, shortly before the outbrake of World War One. We wonder whether  Haupt was aware of Torelli's results. 

An alternative proof of the result of Otto Haupt with techniques that were completely out of reach at his time was given by Misha Kapovich in \cite{Kapovich}.  Rather surprisingly, it relies on Ergodic and Ratner's theory. 


The main observation of Kapovich is that the set of periods of holomorphic one forms on homologically marked Riemann surfaces is invariant by the action of the integral symplectic group. Its action on the period domain is homogeneous and can be studied through Ratner theory. Our approach pushes the analysis of the period map a bit further to allow to transfer properties of the action to properties of the foliation induced by the map. 

At around the time our transfer principle was announced, Hamenst\" adt proved the ergodicity of the isoperiodic foliation, with a different method. in \cite{Hamenstadt}. She announced analogous results for the intersection of the foliation with connected components of strata of abelian differentials having a simple zero. Weiss and Chaika recently announced, at the Geometry and Dynamics (online) seminar at TAU,  a proof of the ergodicity of the isoperiodic foliation on any connected component of any stratum by yet another method. They use recent ongoing research of Eskin, Brown, Filip, Rodriguez-Hertz, etc.  on generalizations of the  'Magic Wand Theorem'  of Eskin-Mirzakhani (see \cite{EskinMirzakhani}).


The problem of connectedness of isoperiodic sets has been considered by several authors. Martin Schmoll established in \cite{Schmoll}, among other things, connectedness of moduli spaces of degree \(d\) coverings of a given elliptic curve.

Remark that the same statement fails at the level of the Torelli covering for genus at least five and degree two,  as stated in Theorem \ref{t:connectedness in genus 2 and 3}. We realized this fact after having pre-published a first version of our work on the arXiv, stating erronously that those fibers were also connected. Their disconnectedness can be established by generalizing to the context of moduli spaces of double coverings over an elliptic curve a famous work of Arnold, who observed that the inclusion of the hyperelliptic locus in the moduli space of genus $g\geq 3$ curves is not surjective (to the integral symplectic group) at the homological level. See \cite{Arnold}.

Fortunately, this exception did not fraud the whole (inductive!) argument, and constitute the single exception of disconnected isoperiodic moduli spaces of abelian differentials.


A problem that arises naturally is the description of the connected components of the intersection of the fibers of $\per$ with the other strata and more generally the affine invariant manifolds that have been brought to light in the work of Eskin and Mirzakhani \cite{EskinMirzakhani}. Kontsevich and Zorich gave a description of the connected components of strata of abelian differentials without any condition on the periods in \cite{KC}; they found cases with up to three components. A direct application of our theorem shows that for generic period, the isoperiodic sets in the stratum will have several components too. Another extreme example of disconnected isoperiodic set: the intersection of a leaf of the isoperiodic foliation corresponding to non-discrete periods with the minimal stratum  forms an infinite discrete set. Such phenomena were studied by McMullen in the genus two case \cite{McMullen}.

The analogous approach using a transfer principle for studying the dynamical properties of the isoperiodic foliation on strata can \textit{a priori} be considered. There have been recent advances in this direction that point towards understanding the dynamical properties that could eventually be transferred:  
firstly that of determining the image of each stratum by the period map, and secondly that of determining the monodromy representation, i.e. the image of the symplectic representation of a connected component of a stratum. Concerning the first, a clear obstruction in the case of elements of $H_{g,d}$ is that the order at each branch point cannot be larger than $d-1$. In recent preprints \cite{BJJP}, \cite{Le Fils} the authors show that,
together with the positive volume condition, these are the only obstructions to realizing the
periods in the stratum. For the second, the monodromy in some connected components of strata has been computed in recent works of Hamenst\" adt \cite{Ham18,Ham20} and Calderon and Salter \cite{CS1,CS2}. Recent advances in the description and properties of the closure of strata in spaces of stable forms were achieved by Bainbridge, Chen, Gendron, Grushevsky and M\" oller in \cite{5A18, 5A19}.  Nevertheless, we point out that the techniques developed in this article do not seem to be enough to compute the connected components of the fibers of the period map on a stratum. For the time being it is unclear whether the transfer of properties to the isoperiodic foliation on a stratum can be carried.



Some interesting recent works have established similar dynamical properties for the isoperiodic foliations on strata and respectively on affine manifolds. The first contribution is by Hooper-Weiss in \cite{HW}; they show that the leaf of the isoperiodic foliation of the Arnoux-Yoccoz surface contained in the stratum \(\mathcal H (g-1,g-1) \) is dense. We have been aware recently that K. Winsor has proven that in case $g=3$ such a leaf has inifinite genus. Also, Ygouf has given an interesting criterion enabling to decide if leaves of isoperiodic foliation are dense in certain affine invariant manifolds of rank one. The rank has been defined in the context of affine manifolds by Wright as half of the codimension of the isoperiodic foliation \cite{Wright}.


We get a nice description of the isoperiodic foliation in the Hilbert  invariant submanifolds $\Omega X_D$ of $\mathcal{F}_2$ after taking projectivisation. The foliation corresponds to the horizontal foliation of the uniformization of $X_D$ by the product $\mathbb{H}\times\mathbb{H}$ of two copies of the upper-half plane in $\mathbb{C}$ (see \cite{McMullen4}). In that paper McMullen describes precisely the $\text{GL}^+(2,\mathbb{R})$-action after projectivization, and finds some very interesting real analytic foliations by Riemann surfaces on $X_D$, that are not transversely holomorphic. 

It follows from works of Calta (in \cite{Calta}) and McMullen (in \cite{McMullen4}) that the union of leaves of $\mathcal{F}_2$ intersecting a closed  $GL^+(2,R)$ -orbit in the minimal stratum  \(\mathcal H(2)\)  has interestingly the structure of a closed analytic subset in the generic stratum $\mathcal H(1,1)$. In the problem paper \cite{HMSZ}, Problem 12, the authors ask to what extent this phenomenon is general. Our theorem shows that in the generic stratum it holds only for the closed leaves. The question remains completely open in the other strata. 
 \subsection{Organization of the paper} In section \ref{s:results without compactification} we give the proof of the results that do not need the bordification of spaces, and Theorem  \ref{t:dynamics} as a consequence of Theorem \ref{t:connectedness}. In sections \ref{s:augmented Torelli and period maps} to \ref{s:connectedness of the boundary} we prove Theorem \ref{t:connectedness}. Along the way, in section \ref{s: isoperiodic sets on curves with one node} we analyze isoperiodic sets with one node and give a proof of Haupt`s theorem. In the Appendix I (Section \ref{s:appendix2} we collect the relevant proofs of the dynamical properties in \cite{Kapovich} and extend it to the case of genus two. In Appendix II and Appendix III we give the proofs of technical results needed along the paper that are presumably well known, but for which we have not found specific reference. 
\subsection{Acknowledgements} 
We warmly thank the referees of this article for the careful reading and the many improvements that their comments made possible. We also hank U. Hamenst\" adt, P. Hubert, M. Kapovich, E. Lanneau, F. Loray, D. Margalit, M. M\" oller, G. Mondello, H. Movasati, A. Putman and A. Wright for useful conversations. This paper was partially supported by the France-Brazil agreement in Mathematics. G. Calsamiglia was partially supported by Faperj/CNPq/CAPES/Mathamsud/ Cofecub, B. Deroin by ANR project LAMBDA ANR-13-BS01-0002, and S. Francaviglia by GNSAGA group of INdAM, and by PRIN 2017JZ2SW5. It was mainly developed at Universidade Federal Fluminense, IMPA, ENS/Paris, U. Cergy-Pontoise, UPMC, and Universit\`a de Bologna, to whom we thank the nice working conditions provided.
\section{Detailed guideline for the proof of Theorem \ref{t:connectedness}}\label{s:guideline}
The  Deligne-Mumford-Knudsen orbifold (compact) bordification $\nc_{g,n}$ of the moduli space $\mathcal{M}_{g,n}$ of smooth genus $g$  
curves with $n$ marked points by adding the moduli spaces of stable genus $g$ curves with $n$ marked points. It is a compact orbifold in which the boundary forms a normal crossing divisor. Each stratum of complex codimension $k\geq 0$ of the divisor corresponds to a stratum of curves with $k$ nodes. An intermediate (non-compact) bordification $\mathcal{M}_{g,n}\subset\nc_{g,n}^c\subset\nc_{g,n}$ is obtained if we add the moduli spaces of stable curves of compact type, i.e. only with separating nodes. Since the boundary has complex codimension one, it cannot separate the total space. 

The Hodge bundle can be naturally extended to the bundle $\Omega\nc_{g,n}$ of holomorphic stable forms. A stable form can be thought as a meromorphic form on each component of the normalization of a stable curve, called a part, having at worst simple poles at the marked points where they are glued and each pair of points that are glued have opposite residue. The local isoperiodic equivalence relation can be naturally extended to $\Omega\nc_{g,n}$. As in the smooth case, any stable form admits a non-trivial local isoperiodic deformation. However, the local isoperiodic deformation space is not always a complex manifold, but rather a (possibly singular ) analytic set. For  generic forms with zero components or residues at the nodes they are strictly contained in the boundary. In each stratum the local isoperiodic deformation is parametrized by a product of local isoperiodic deformations of the parts.
\cbstart
In Section \ref{s:augmented Torelli and period maps} we will prove that the local space of isoperiodic deformations at a point of the subset $\Omega_0^*\nc_{g,n}\subset\Omega\nc_{g,n}$ of stable forms with zero residues at the nodes and no zero components is -- at the level of the orbifold chart-- a smooth complex manifold transverse to each boundary component passing through the point. This implies that the set of points in the local isoperiodic complex manifold that lie on some boundary stratum of the ambient space form a normal crossing divisor. Furthermore the local divisor has precisely one codimension one component for each boundary component of $\Omega\nc_{g,n}$ passing through the point. In particular the set of points lying on boundary strata do not separate the isoperiodic local space in several components and are stratified by the number of nodes in the underlying curves.

To bordify the fibers of the period map we need to consider the pull back of the bundle of stable forms by the \textit{ramified} cover $\mnc_{g,n}\rightarrow\nc_{g,n}$  associated to  the Torelli subgroup $\mathcal{I}_{g,n}$. This space can be realized as the quotient of Augmented Teichm\" uller space, and   we call it Augmented Torelli space. The structure and relevant properties of these spaces are  detailed in Section \ref{s:augmented Torelli and period maps}. We loose the manifold and holomorphic structure around the ramification locus, but we still have a topological stratified space by complex manifolds that lifts the stratification of the boundary divisor of $\mathcal{M}_{g,n}$. \cbend The local projection has two properties that we will exploit: over each local connected component of a stratum of the boundary divisor $\nc_{g,n}$ it is an abelian cover (hence a connected complex manifold of the same dimension). In particular, the open stratum (the complement of the boundary strata) is locally connected around each point on some boundary stratum, Moreover, as happens with the stratification of a normal crossing divisor, every local stratum can be identified by the closures of the codimension one strata where it belongs to or not. This last property is what allows to associate a natural simplicial complex $\mathcal{C}(\mnc_{g,n})$ to the boundary of  $\mathcal{S}_{g,n}$ called the dual boundary complex. A vertex is considered for each connected component of codimension one of the stratification, and among $k$ of them we attach a $k-1$ simplex for each connected component of the intersection of their closures. These simplices codify all possible connected components of the boundary stratification. 

The Hodge bundle $\Omega\mnc_{g,n}$ inherits the same boundary stratification by the number of nodes, with each stratum having the same codimension as in $\mnc_{g,n}$. A word of warning: this stratification is not to be confused with the famous substratification given by the zero and polar sets of the forms. The maximal isoperiodic deformation space on  $\Omega\mrs_{g,n}$ is bordified by considering its closure in $\Omega^*_0\mnc_{g,n}$. It coincides with the maximal isoperiodic deformation there. The transversality of the local isoperiodic deformation space with the boundary components show that we have a  stratification of its boundary points by the number of nodes that has similar local properties as $\Omega\mnc_{g,n}$. In particular the set of boundary points does not locally separate the isoperiodic space we can associate a similar simplicial complex to its boundary stratification: a vertex for each connected component of the codimension one stratum and a $k-1$-simplex joining $k$ vertices for each connected component of the intersection of the closures of the components.

In Section \ref{s:moving points in the leaf} we will prove that any non-zero form in $\Omega\mrs_g$ can be isoperiodically deformed to converge to a boundary point of the bordification $\Omega_0^*\mnc_g$. The proof follows by finding families of twin geodesics, that is, parallel geodesic paths of the singular metric  defined by a non-zero abelian differential, that start at a saddle point and have the same length. An appropriate surgery along them-- namely a Schiffer variation-- provide a piece of the desired path in $\Omega\mathcal{S}_g$. For forms with a single zero these pairs of twins are found on the boundary of cylinders of closed geodesics whose existence is guaranteed by Masur's Annulus Theorem (see \cite{Masur}). 

To prove the connectedness of the bordification of a period fiber we just need to connect all boundary points by isoperiodic deformations. An isoperiodic deformation in a boundary stratum of $\Omega^*_0\mnc_g$ can be thought as a product of isoperiodic deformations of forms on $\Omega^{*}\mrs_{h,n}$ with $h< g$. This is how we will use the inductive hypothesis to reduce the proof of the connectedness \textit{of the boundary}  to a combinatorial algebraic problem that we describe next.

Given $p\in\mathcal{H}_{g}$, we consider the dual boundary complex  $\mathcal{G}_p$ of the bordification of $\per^{-1}(p)$ in $\Omega^*_0\mnc_{g}$. The definition of the stratifications allow to define a simplicial continuous map  
\begin{equation}\label{eq:bdry inclusion}
    \mathcal{G}_p\rightarrow \mathcal{C}(\mnc_g)
\end{equation} that associates to each connected component of the boundary stratum of codimension one of the isoperiodic set, the component of the stratum of codimension one of the ambient space where it sits. It extends to simplices in the natural way thanks to the transversality of the local isoperiodic deformation space with boundary components. Its image tells us which components of boundary strata of $\mnc_{g}$ admit forms of periods $p$ without zero components.

In Section \ref{s: isoperiodic sets on curves with one node} we describe all vertices in the image of \eqref{eq:bdry inclusion} by characterizing them algebraically as follows:
$\mnc_{g}$ can be thought as the quotient of augmented Teichm\" uller space quotiented by the Torelli group $\mathcal{I}_g$. For  this cover the complex is the so called curve complex $\mathcal{C}_{g}$ of a surface of genus $g$, introduced by Harvey in \cite{Harvey} (see (see \cite[Chapter 4.1]{FM} for further details). It has a vertex for each homotopy class of essential simple closed curve on a genus $g$ surface and a $k-1$ simplex joining $k$ given vertices if the corresponding classes can be realized disjointly by simple closed curves. The subcomplex generated by separating curves is denoted by $\mathcal{C}^{\text{sep}}_{g}$. The mapping class group  of  the surface with marked points at the punctures  acts naturally on both complexes and we have  $$\mathcal{C}(\mnc_{g})\simeq\mathcal{C}_{g}/\mathcal{I}_{g}\text{ and }  \mathcal{C}(\mnc^{c}_{g})\simeq\mathcal{C}^{\text{sep}}_{g}/\mathcal{I}_{g} $$ 

The vertex of $\curvecomplex_g/\mathcal{I}_g$ corresponding to a Torelli class of a simple closed curve $c$ is characterized either by the primitive rank one submodule $\mathbb{Z}[c]$ in the first homology group of the surface $\mathbb{Z}^{2g}$ if this class in non-trivial, or by a splitting of $\mathbb{Z}^{2g}=V_1\oplus V_2$ into a direct sum of orthogonal symplectic submodules otherwise (i.e. $c$ is separating). Some simplexes connecting those vertices can be identified by simply using this algebraic information. As an instance, a splitting of $\mathbb{Z}^{2g}$ into three factors $V_{1}\oplus V_{2}\oplus V_{3}$ that are pairwise orthogonal symplectic submodules, allows to construct a marked stable curve with two separating nodes that induces the given splitting.  It determines at least an edge in $\curvecomplex_g/\mathcal{I}_g$  joining  the vertex  $V_{1}\oplus V_{1}^{\perp}$ and the one corresponding to $V_{3}^{\perp}\oplus V_{3}$.

A vertex of non-separating type lies in the image of \eqref{eq:bdry inclusion} if and only if $[c]\in \ker p\setminus 0$ and the map induced by $p$ on $[c]^{\perp}/\mathbb{Z}[c]$ is the period of some non-zero abelian differential on a smooth curve. We say that $[c]$ is pinched by $p$.  A vertex of the second type belongs the image of \eqref{eq:bdry inclusion} if and only if $p_{|V_i}$ is the period of some non-zero abelian differential on a smooth curve for $i=1,2$. We say that $V_1\oplus V_2$ is a $p$-admissible decomposition. Haupt conditions is what allows to algebrize the problem, even for other simplexes. We will exploit the following dichotomy to identify vertices:  either the image of $p$ has large rank -- and we can identify enough $p$-admissible splittings --- or $p$ has large kernel -- and we can find enough classes pinched by $p$. 

In Section \ref{s:period fibers with marked points}  we find, by using normalization of the node,  a homeomorphism of the isoperiodic set of period $p$ in a boundary component of $\Omega\mnc_g$ of codimension one with a product of isoperiodic sets of  forms (with marked points!) on lower genera. Using the inductive hypothesis and the cases of genera two and three it allows to determine a family of codimension one boundary components of $\Omega^{*}_{0}\mnc_g$ (or vertices of $\mathcal{C}_g/\mathcal{I}_g$), called $p$-simple, that contain a single isoperiodic connected component of period $p$ ( i.e. have a single vertex in the preimage by \eqref{eq:bdry inclusion}). 

When $p$ is injective every boundary stratum of codimension one of $\Omega\mnc_g$ has at most one isoperiodic component of period $p$. In fact, all boundary points of $\per^{-1}(p)$ are of compact type, the bordification is a complex manifold and the boundary a normal crossing divisor in it whose dual complex is $\mathcal{G}_p$. On the other hand, in subsection \ref{ss:connected boundary injective case} we prove the connectedness of the image of \eqref{eq:bdry inclusion} which actually lies in $\curvecomplex^{\text{sep}}_g/\mathcal{I}_g$ by analyzing the problem in terms of splittings and the volumes of the parts Haupt's conditions on the factors of a splitting are reduced to a condition on the volume in this case. 

When $p$ is not injective, it is not always true that the boundary divisor of compact type is connected. Examples can be constructed in high genus by taking stable curves with one separating node that have double covers over distinct  elliptic curves on each side. When they are double covers of the same elliptic curve and the total degree is two, the curve has necessarily genus two. This is due to the fact that the sum of the degrees is at most two, and the degree of each part is at least one. Therefore both parts have genus one. Hence the need of introducing the boundary components of non-compact type.

 Whenever $\deg(p)\geq 3$, we show that the restriction of \eqref{eq:bdry inclusion} to a  subcomplex $\mathcal{G}'_p$ of $\mathcal{G}_p$ spanned by simple vertices has  connected image. Here we need to deal with both of Haupt's conditions. On the other hand, non-simple vertices of $\mathcal{G}_p$ are always related to degree two coverings of elliptic curves and  can be isoperiodically degenerated to pinch enough different non-separating curves to be able to connect them to some $p$-simple component of the boundary.  This implies that $\mathcal{G}_p$ is connected finishing the proof of the inductive step of Theorem \ref{t:connectedness}.
 
 The boundary components of $\Omega\mnc_g$ with precisely one non-separating node play a fundamental role in the proof of the general case. One of the main technical difficulties of the paper is to control  the number of connected components of the intersection of one of them with the closure of a fiber of $\per^{-1}(p)$. This is done in Section  \ref{s:period fibers with marked points}. Here we use once again the bordification strategy, but this time we need to add some strata of the boundary of $\Omega\mnc_{g}$ where the forms of period $p\in\mathcal{H}_{g}$ have zero components --of genus one-- and where the local isperiodic deformations have, potentially, wilder singularities.  The bordification is best described by normalizing the non-separating node. \cbstart It is carried in subsections  \ref{s:bordification of Sg2} and  \ref{sss:singularities}. \cbend
 \begin{figure}[httb]
\centering
\def\svgwidth{\columnwidth}
\includegraphics[width=3in]{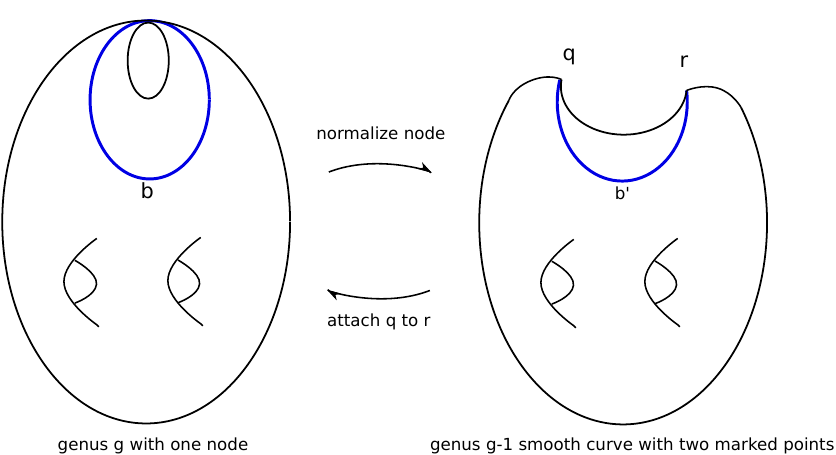}
 \caption{The closed cycle $b$ becomes a relative homology class $b'$ after  normalization} \label{fig:normalization}
\end{figure}

 The normalization map -- normalizing the nonseparating node-- sends the given set homeomorphically to a subset of $\Omega\mrs_{g-1,2}$  that lies in a fiber $F$ of the composition 
 \begin{equation}\label{eq:PeroFor}\Omega\mathcal{S}_{g-1,2}\overset{\text{For}}\rightarrow\Omega{\mathcal{S}}_{g-1}\overset{\per_{g-1}}{\longrightarrow}\mathcal{H}_{g-1}.\end{equation}
 
 \cbstart In fact this subset is precisely an analytic hypersurface of $F$, defined by the zeros of a holomorphic function $$h:\Omega\mathcal{S}_{g-1,2}\rightarrow\mathbb{C}$$ also given by integration-- on a relative cycle joining the two marked points as $b'$ in Figure \ref{fig:normalization}. The map $h$ extends holomorphically to the Hodge bundle over a smooth bordification $\U_{g-1,2}\subset\mnc_{g-1,2}$ of $\mrs_{g-1,2}$ described in subsection \ref{s:bordification of Sg2}. \cbend
 There is a particular isoperiodic deformation in $\Omega\mrs_{g-1,2}$ that consists in fixing the underlying marked form $(C,m,\omega)$ on a smooth genus $g-1$ curve and  moving the marked points in $C$ by preserving the value of the integral along the class $b'$. If this integral lies in the image of the periods of $\omega$ it might happen that  both marked points tend to the same point $q$ in $C$ along the isoperiodic deformation.  We are then approaching a point in a boundary stratum of $\Omega\mnc^{c}_{g-1,2}$ whose normalization has a zero component of genus zero with three marked points  that is  glued to  $(C,m,\omega)$ at a point $q\in C$ where  $\omega(q)=0$ (see Figure \ref{fig:normalize_degeneration} for a representation).
 \begin{figure}[httb]
\centering
\def\svgwidth{\columnwidth}
\includegraphics[width=4in]{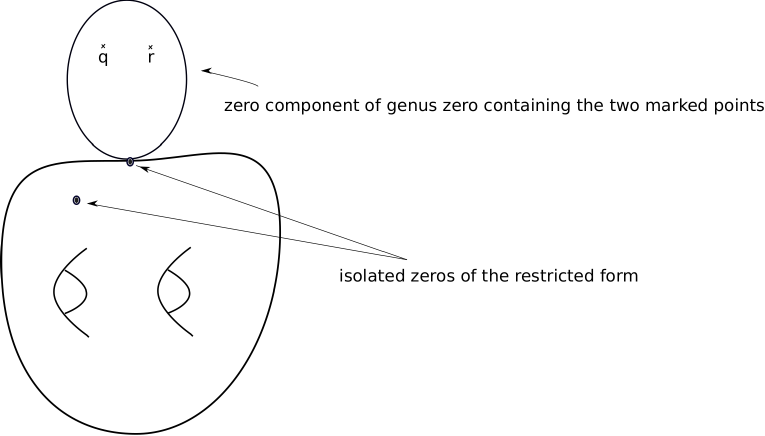}
 \caption{Example of a (singular) point of normal crossing of two regular local isoperiodic components  in \(\Omega\U_{2,2}\): the marked points lie in a zero component of genus zero that is glued to a simple isolated zero of a genus two component.} \label{fig:normalize_degeneration}
\end{figure}
 Those are the boundary points we add to bordify the isoperiodic set \cbstart (in subsection  \ref{s:bordification of Sg2}). Each of them has singular local isoperiodic deformation space (see subsection \ref{sss:singularities}). \cbend Indeed, by varying the point of gluing in the neighbourhood of $q$ we obtain another isoperiodic deformation lying on the boundary. However, as the zeros of $\omega$ are isolated, the neighbouring isoperiodic boundary points do not have local isoperiodic deformations that leave the boundary. 
We will need to consider those isoperiodic boundary components to deduce the connectedness of the bordification too.  

The problem of  relating the topology of a space with that of a hypersurface in it is reminiscient of Lefschetz hyperplane section Theorem. In the present context  the lack of an algebraic setting  for $F$ and $h$  do not allow to use it directly. 
 However, an application of Simpson's Theorem in  \cite{Simpson} --a generalization of Lefschetz theorem for functions defined via integration of a holomorphic one-form--  we prove that the hypersurface defined by $h$ on the generic (two dimensional) fiber of the forgetful map $\Omega\mnc_{g-1,2}^c\rightarrow\Omega\mnc_{g-1}$ is connected.  Moreover it is a nodal curve whose singular points lie in the boundary and have a local irreducible component completely contained in the boundary and the other transverse to the boundary. Each of those nodal points corresponds the stable forms with a zero component defined above. When the fiber $F$ of the map \eqref{eq:PeroFor} is over a homomorphism of degree at least three, we can use the inductive hypothesis and the fibration $\text{For}$ to deduce that the constructed bordification is contained in a connected analytic set in $\Omega\mnc^{c}_{g-1,2}$ having possibly some additional smooth irreducible components contained in the boundary. The intersection of each of those boundary components with the other components can be parametrized by a set of isoperiodic forms of genus $g-1$ with a marked simple zero. By analyzing how the zeros of a form can be permuted in the isoperiodic set we deduce that this intersection is connected. This allows to take the boundary irreducible components one by one without changing the total number of connected components, to deduce that the fiber of $h$ in $F$  is connected.

\section{The period map on the Hodge bundle over Torelli space}
\label{s:results without compactification}

\subsection{Torelli map and period fibers} 
\label{ss:Torelli and period map}Fix a reference closed connected and oriented surface $\Sigma_g$ of genus $g$ with $n\geq 0$ ordered marked points  $\Sigma_{g,n}=(\Sigma_g,q_1,\ldots,q_n)$. 
 A homotopical marking of a smooth genus $g$ compact complex curve $C$ with $n$ pairwise distinct ordered marked points, $P=(p_1,\ldots,p_n)$ is a homeomorphism $f:\Sigma_{g,n}\rightarrow (C,P)$ sending $q_i$ to $p_i$. Two such markings $f,f'$ of $(C,P)$ and $(C',P')$ are equivalent if there exists a biholomorphism $\varphi:(C,P)\rightarrow (C',P')$ such that $\varphi\circ f$ is isotopic to $f'$.
The Teichm\" uller space $\mathcal{T}_{g,n}=\mathcal{T}(\Sigma_{g,n})$ for $(g,n)$ satisfying $k=3g-3+n>0$ is the set of equivalence classes of homotopically marked genus $g$ smooth complex curves with $n$ ordered distinct points endowed with the Teichm\" uller topology, i.e. the weakest topology for which the length function associated to a homotopy class of closed curve on \(\Sigma_{g,n}\) is continuous. A point in $\mathcal{T}_{g,n}$ will be denoted by $[f:\Sigma_{g,n}\rightarrow (C,P)]$.

The mapping class group of $\Sigma_{g,n}$ is the group $\text{Mod}(\Sigma_{g,n})$ of isotopy classes of orientation preserving diffeomorphisms that fix each marked point. It acts on $\mathcal{T}_{g,n}$ by precomposition on the marking. 

Bers (\cite{Bers1}, Section 16) showed that whenever $k>0$, the space $\mathcal{T}_{g,n}$ can be embedded in $\mathbb{C}^k$ as a bounded open domain, inheriting a complex structure. Following Ahlfors \cite{Ahlfors}, this complex structure is the only over the given topology for which the coordinate functions of the period matrices of curves depend holomorphically on the curve. In more detail, given a symplectic basis $a_1,b_1, \ldots,a_g,b_g$ of $H_1(\Sigma_g,\mathbb{Z})$ (i.e. the only non-zero products of the cycles are $a_i\cdot b_i=1$ and $b_i\cdot a_i=-1$) we can choose, for each marked curve $f:\Sigma_{g,n}\rightarrow (C,P)$ of genus $g$ the unique basis $\omega_1,\ldots, \omega_g$ of $\Omega(C)$ such that $$\int_{f_*a_i}\omega_j=\delta_{ij}.$$ The maps $\tau_{ij}: \mathcal{T}_{g,n}\rightarrow\mathbb{C}$ defined by \begin{equation}\label{eq:period matrix coordinate}\tau_{ij}[f:\Sigma_{g,n}\rightarrow (C,P)]= \int_{f_*b_i}\omega_j\end{equation} are well defined and holomorphic. 
Let $\mathfrak{S}_g$ denote the Siegel space of genus $g$, i.e. the set of symmetric $g\times g$ matrices with complex entries whose imaginary part is positive definite. Riemann showed that the squared matrix of functions $(\tau_{ij})$ defines a holomorphic map 
$\mathcal{T}_{g,n}\rightarrow \mathfrak{S}_g$ that is invariant by the action of the Torelli group $\mathcal{I}_{g,n}$, kernel of the representation \[\text{Mod}(\Sigma_{g,n})\rightarrow \text{Aut}(H_1(\Sigma_g,q_1,\ldots,q_n,\mathbb{Z})).\] It induces the Torelli map on the Torelli space $\mrs_{g,n}:=\mathcal{T}_{g,n}/\mathcal{I}_{g,n}$, \begin{equation}\label{eq:torelli map}\mrs_{g,n}\rightarrow \mathfrak{S}_g.\end{equation} 
In the case of $n=0$ it was shown to be injective by Torelli(\cite{Torelli}).
Each point in $\mrs_{g,n}$ is characterized by a triple $(C,P,f_*)$ where $f_*$ is the isomorphism induced by $f$ in homology. Therefore we will denote a point in $\mrs_{g,n}$ simply as  $(C,p_1,\ldots,p_n, m)$ where $m:H_1(\Sigma_{g},q_1,\ldots,q_n,\mathbb{Z})\rightarrow H_1(C,p_1,\ldots,p_n,\mathbb{Z})$ is an isomorphism.  
The quotient $\mathcal{T}_{g,n}/\text{Mod}(\Sigma_{g,n})$ is the moduli space $\mathcal{M}_{g,n}$ of genus $g$ curves with $n$ marked points and the holomorphic structure induces an orbifold structure on $\mathcal{M}_{g,n}$.  The Hodge bundle is a holomorphic vector bundle $\Omega\mathcal{M}_{g,n}\rightarrow \mathcal{M}_{g,n}$ whose fiber over a point $(C,P)$ is the set of abelian differentials $\omega\in \Omega(C)$. It can be pulled back to a holomorphic bundle $\Omega\mrs_{g,n}\rightarrow\mrs_{g,n}$.

\begin{definition}
 The period map on $\Omega\mrs_{g,n}$ is  the holomorphic map $$\per_{g,n}:\Omega\mrs_{g,n}\rightarrow \text{Hom}(H_1(\Sigma_g,q_1,\ldots, q_n,\mathbb{Z}),\mathbb{C})$$ defined by $\per_{g,n}(C,p_1,\ldots,p_n,m,\omega)= \{\gamma\mapsto\int_{m(\gamma)}\omega\}$. When $n=0$ we write $\per_g=\per_{g,0}$. When there is no risk of confusion we omit all subindices and write $\per$.
\end{definition}
For instance given a homomorphism  $p:H_1(\Sigma_g,q_1,\ldots,q_n,\mathbb{Z})\rightarrow \mathbb{C}$ we denote $\per^{-1}(p)$ the fiber of $\per_{g,n}$ over $p$. 

Multiplying forms by a non-zero constant induces a biholomorphism between fibers of $\per_{g,n}$.

In the case $n=0$ the properties of the Torelli map and of the intersection form in $H_1(\Sigma_g,\mathbb{Z})$ have a nice consequence:
\begin{theorem}{\cite{McMullen}}\label{t:slice of siegel} Let $g\geq 2$ and $p:H_{1}(\Sigma_{g},\mathbb{Z})\rightarrow \mathbb{C}$ be a homomorphism with $\vol (p)>0$. 
Then $\per^{-1}_{g}(p)$ is biholomorphic to the intersection of a linear Siegel space $\mathfrak{S}_{g-1}\subset \mathfrak{S}_{g}$ and the image of the Torelli map -- the so-called  Schottky locus.-- \end{theorem}

\begin{proof}
We take the choices and notations of the definition of \eqref{eq:period matrix coordinate}.
Let  $(C,m,\omega)\in\per^{-1}(p)$. Its expression in the chosen basis reads \( \omega= p(a_1)  \omega_1 +\ldots + p(a_g) \omega_g\) and the following equations are satisfied:
\begin{equation}\label{eq:image of Torelli}
    m_{ij}=\tau_{ij}(C,m)  \text{ for } i,j=1,\ldots, g
\end{equation}
\begin{equation}\label{eq:slice of siegel}
    p(b_i) = p(a_1) m_{i1}+\ldots + p(a_g) m_{ig} \text{ for } i=1,\ldots, g.
\end{equation}
A point $Z=(m_{ij})\in\mathfrak{S}_{g}$  that satisfies  \eqref{eq:image of Torelli} is said to belong to the Schottky locus in $\mathfrak{S}_{g}$. 
If it furthermore satisfies \eqref{eq:slice of siegel}, then $\omega:=\sum p(a_{j})\omega_{j}$ is the unique abelian differential on $(C,m)$ having periods $p$. This proves that $\per^{-1}(p)$ is biholomorphic to the intersection of the Schottky locus with the set of solutions of \eqref{eq:slice of siegel}.

It remains to show that the condition $\vol (p)>0$ implies that the set of all solutions $Z=(m_{ij})\in\mathfrak{S}_{g}$ of \eqref{eq:slice of siegel} (that we call the set of matrices that admit $p$ as a period) is biholomorphic to $\mathfrak{S}_{g-1}$. 
Up to multiplying all forms in $\per^{-1}(p)$ by a constant we can suppose that all forms have volume one, i.e. $\vol(p)=1$. 

The symplectic automorphism group of $H_1(\Sigma_g,\mathbb{R})$ acts on Siegel space (by changing the marking) . Direct calculation shows that,  writing $T\in\text{Sp}(2g,\mathbb{R})$ in the basis $\{a_{i},b_{i}\}$ as block $g\times g$ real matrices with two lines $A,B$ and $C,D$ the action is defined by   $$Z'=T\star Z=(AZ+B)(CZ+D)^{-1}. $$ 

Moreover denoting $p_{\mathbb{R}}:H_{1}(\Sigma_g,\mathbb{R})\rightarrow \mathbb{C}$  the natural extension of $p:H_1(\Sigma_{g},\mathbb{Z}))\rightarrow \mathbb{C}$, the given action preserves the condition  $p_{\mathbb{R}}$ is a period of the matrix. 

By definition,  $\Re (p_{\mathbb{R}})\cdot\Im (p_{\mathbb{R}})=\Re (p)\cdot \Im (p)=\vol (p)=1$.  On the other hand, by duality there exist elements $a_{1},b_{1}\in H_{1}(\Sigma_{g},\mathbb{R})$ such that $a_{1}\cdot b_{1}=1$,  $a_{1}^{*}=\Re p_{\mathbb{R}}$ and $b_{1}^{*}=\Im p_{\mathbb{R}}$.  Now,  $\ker p_{\mathbb{R}}=\ker \Re (p_{\mathbb{R}})\cap\ker \Im p_{\mathbb{R}}$ and contains the rank $2g-2$ symplectic orthogonal of $\mathbb{R}a_{1}\oplus\mathbb{R}b_{1}$. We deduce $\ker p$ is symplectic of rank $2g-2$. Choose a symplectic basis $a_{2},b_{2},\ldots,a_{g},b_{g}$ of $\ker p_{\mathbb{R}}$. The matrix in Siegel space associated to the real basis $\{a_{i},b_{i}\}$  has two diagonal blocks:  $1\in\mathbb{C}$ and a matrix in $\mathfrak{S}_{g-1}$.  Every such matrix admits $p_{\mathbb{R}}$ as period, trivially. Therefore the set of solutions of \eqref{eq:slice of siegel} is also a linear Siegel subspace $\mathfrak{S}_{g-1}\subset\mathfrak{S}_{g}$. \end{proof}

\begin{theorem}\label{t:connectedness in genus 2 and 3}
 Let $g=2,3$ and $p\in H^1(\Sigma_g,\mathbb{C})$ with $\vol (p)>0$ and $\deg (p)>1$. Then  $\per^{-1}(p)$ is non-empty and connected. 
\end{theorem}
\begin{proof}
The Schottky locus is Zariski dense in $\mathfrak{S}_{g}$ for  $g=2,3$. Its complement corresponds to period matrices of curves of compact type. By Theorem \ref{t:slice of siegel}  we have that $\per^{-1}(p)$ is biholomorphic to a Zariski open set of $\mathfrak{S}_{g-1}$, hence connected. 

Suppose $\per^{-1}(p)$ is empty. Then, the linear subspace of Theorem \ref{t:slice of siegel} is completely contained in the complement of the Schottky locus. The same construction of forms of the proof of Theorem \ref{t:slice of siegel} (solutions of \eqref{eq:image of Torelli} and \eqref{eq:slice of siegel}) can be carried on Jacobians of curves of compact type to construct stable forms of periods $p$. 

From a stable form of compact type that has no zero component we can obtain a form on a smooth curve having the same periods.  Indeed by taking parallel slits instead of points to glue the different components we smoothen the curve, without changing the period map. 

If $g=2$, the curve has two components of genus one and we know that one of the components of the form is zero. This implies $\deg(p)=1$. A contradiction. 

If $g=3$ the curve has either a part of genus two and a part of genus one or three parts of genus one and the stable form is zero in one of the parts. If it it is zero on a genus two part, it implies $\text{deg}(p)=1$, contrary to assumption. If it is zero in just one part of genus one we can suppose that we have a form $\omega\neq 0$ on a smooth curve of genus two. It has at least one zero $z_{1}$. Consider the local map $z\mapsto \int_{z_{1}}^{z}\omega$ and the pre-image $\gamma$ of a small segment $[0,\varepsilon]$. It is a path with distinct endpoints satisfying $\int_{\gamma}\omega=0$. Gluing the endpoints produces a stable form on a curve that is not of compact type. If we take parallel slits at the endpoints  instead, the gluing produces a stable form on a smooth curve of genus three. Marking the curve appropriately we obtain that its period is $p$ . A contradiction. 
\end{proof}

 For $g\geq 4$ the image of the Torelli map is an analytic set of positive codimension (as an easy dimension count shows) and determining it  is known as the Schottky problem. In fact, as stated in Theorem \ref{t:degree 2}, there are some of the intersections given by Theorem \ref{t:slice of siegel}  that will not be connected. As for the non-emptiness of the fiber of $\per$ over points of positive volume and primitive degree at least two,  we will show inductively that they can be realized as period of a stable form without zero components as in the proof of Theorem \ref{t:connectedness in genus 2 and 3} and apply the surgeries. These surgeries will be extensively developed in Section \ref{s:moving points in the leaf}  to produce continuous isoperiodic  deformations of stable forms.  
 \subsection{Strata of holomorphic forms on smooth curves and isoperiodic foliations}
 \label{ss:strata}
It is well known that the spaces $\Omega^{*}\mathcal{M}_{g,n}$ and $\Omega^{*}\mrs_{g,n}$  are stratified by the properties of the zero sets of forms. Two points $(C,P,\omega)$, $(C',P',\omega')$  lie in the same stratum if there exists a homeomorphism $(C,P)\rightarrow (C',P')$  sending the zero divisor $(\omega)$ to the zero divisor $(\omega')$. In other words, each marked point is sent to a marked point, each zero to a zero, and the orders of the zeros are preserved. The generic stratum  $\Omega^{SZ}\mrs_{g,n}$ is formed by forms with ($2g-2$) simple zeros, none of which is a marked point. The minimal stratum $\Omega\mrs_{g,n}(2g-2)$ is formed by abelian differentials with a single zero (at some marked point if there are any). Veech (\cite{Veech}) and Masur (\cite{Masur2}) showed that there are local holomorphic coordinates defined on the stratum of a point $(C,P,\omega)$, with coordinates in $$\text{Hom}(H_1(C,Z(\omega)\cup P,\mathbb{Z}),\mathbb{C}), $$ defined by integration. As a consequence, the period map restricted to any stratum is a linear projection in the coordinates, and is therefore submersive and open. This implies that the fibers of the map $\per_{g,n}$ are regular and transverse to all strata different from the generic stratum. In particular, in restriction to  $\Omega\mrs_{g,n}(2g-2)$, the  map $\per_{g,n}$ is a local biholomorphism. On the other hand it also proves that the restriction of the map $\per_{g,n}$ to each stratum of stable forms defines a regular holomorphic foliation of codimension $2g$.

The natural action of $\text{Mod}(\Sigma_{g,n})$ on $\Omega\mrs_{g,n}$ preserves each stratum, and the map $\per_{g,n}$ is equivariant with respect to  the corresponding action on $\text{Hom}(H_1(\Sigma_{g},q_1,\ldots,q_n;\mathbb{Z}),\mathbb{C})$.
\begin{definition} The regular foliation induced by $\per_{g,n}$ on $\Omega^{*}\mrs_{g,n}$ (and its restriction to each stratum) descends to $\Omega^{*}\mathcal{M}_{g,n}$ as a regular holomorphic foliation called the isoperiodic foliation and denoted $\mathcal{F}_{g,n}$ (or $\mathcal{F}_g=\mathcal{F}_{g,0}$ when there are no marked points). 
\end{definition}
If $L\subset\Omega\mathcal{M}_g$ is the leaf of $\mathcal{F}_g$ corresponding to the periods $p\in\mathcal{H}_g$, the restriction of the Torelli cover to the fiber   \begin{equation}\label{eq:galois cover of leaf}
    \per^{-1}(p)\rightarrow L
\end{equation} is a Galois covering. We can give some information about the covering group.

\begin{definition}
Given a symplectic module $M$ over a ring and a  homomorphism $p:M\rightarrow\mathbb{C}$, the stabilizer of $p$ is the subgroup of the group of symplectic automorphisms of $M$ defined by  \[\text{Stab}(p)=\{h\in\text{Aut}(M):p\circ h=p\}\] When we want to stress the group it sits in we denote it $\text{Stab}_{\text{Aut}(M)}(p)$. 
\end{definition}
\begin{remark}
\label{ss:symmetries of periods} Let $p=\per (C,m,\omega)$ be the period map of a marked abelian differential $(C,m,\omega)\in\Omega\mrs_{g,0}$. Then the group $\text{Stab}(p)$ also stabilizes the fiber $\per^{-1}(p)$.  The isotopy class of a Dehn twist around any simple closed curve in $C$ defining a primitive element $a\in\ker p\setminus 0$, induces a non-trivial element $\delta_a\in\text{Stab}(p)$ whose action fixes no point. 
 \end{remark}

\begin{remark}\label{rem:fundamental group of leaves}

  The covering group of the map \eqref{eq:galois cover of leaf} is $\text{Stab}(p)$. In particular, if $\per^{-1}(p)$ is connected, the image of $\pi_1(L)$ in $\text{Sp}(2g,\mathbb{Z})$ is precisely $\text{Stab}(p)$, which is non-trivial as soon as $\ker (p)\neq 0$.
\end{remark}


\subsection{Disconnected fibers: proof of Theorem \ref{t:degree 2}}

Let $p\in\mathcal{H}_g$ with $\deg p=2$. We will define an invariant that is constant on each component of $\per^{-1}(p)$. On the other hand we will prove that if $g\geq 5$, the stabilizer of $p$ in the symplectic group $\text{Aut}(H_1(\Sigma_g,\mathbb{Z}))$ is large enough to guarantee that the invariant takes at least two values on any orbit of its action on $\per^{-1}(p)$. 

Consider the elliptic curve \(E:=\mathbb C/\text{Im}(p)\). To any \((C,m,\omega)\in \text{Per}^{-1}(p)\) we can  associate a double branched covering \(\pi : C\rightarrow E\) defined as the integral of \(\omega\) based at some point. It is well-defined up to post composition by a translation in \(E\). Denote by \(C(\pi)\) (resp. \(VC(\pi)\)) the set of critical points (resp. critical values) of \(\pi\). 

As observed by Arnold, see \cite{Arnold}, the map \(\pi_*: H_1(C\setminus C(\pi),\mathbb Z / 2\mathbb Z) \rightarrow H_1(E\setminus VC(\pi) ,\mathbb Z / 2\mathbb Z) \) extends as a homomorphism 
\begin{equation} \label{eq: pre-Arnold's map} \Pi : H_1(C,\mathbb Z / 2\mathbb Z) \rightarrow H_1(E\setminus VC(\pi) ,\mathbb Z / 2\mathbb Z) .\end{equation}
This is due to the fact that any cycle turning once around a critical point of \(\pi\) is mapped to a cycle turning twice around a critical value of \(\pi\). The image of \(\Pi\) is the kernel of the monodromy representation \(\varepsilon : H_1( E\setminus VC(\pi_0), \mathbb Z / 2\mathbb Z)\rightarrow \mathbb Z / 2\mathbb Z \) of the covering \( \pi\).

To define an invariant that does not depend on the choices made so far, we need to choose a reference for the homology and the critical values. Fix a reference subset \(VC_0\subset E\) of cardinality \(2g-2\). There exists a homeomorphism \(\varphi\) of \(E\) sending \(VC(\pi)\) to \( VC_0\), that is homotopic to the identity. Two such choices differ by  post-composition by an element of the braid group of the pair \((E, VC_0)\). We define Arnold's invariant of \( (C,m,\omega)\) as the map 
\[ Ar (C,m,\omega) := \varphi _* \circ \Pi \circ m \in \text{Hom} (H_1(\Sigma_g, \mathbb Z / 2\mathbb Z), H_1 (E\setminus VC_0,\mathbb Z / 2\mathbb Z))/G,\]
where \(\varphi_*: H_1 (E\setminus VC(\pi) ,\mathbb Z / 2\mathbb Z)\rightarrow H_1 (E\setminus VC_0,\mathbb Z / 2\mathbb Z)\) is the natural map induced by \(\varphi\), and where \(G \subset \text{Aut} (H_1(E\setminus VC_0,\mathbb Z / 2\mathbb Z))\) is the image of the natural representation of the braid group of the pair \( (E,VC_0) \) in the homology group \(H_1 (E\setminus VC_0,\mathbb Z / 2\mathbb Z)\), acting on \(\text{Hom} (H_1(\Sigma_g, \mathbb Z / 2\mathbb Z), H_1 (E\setminus VC_0,\mathbb Z / 2\mathbb Z))\) by post-composition. By construction the Arnold invariant is constant on every connected component of $\per^{-1}(p)$.

The group \(\text{Aut}(H_1(\Sigma_g, \mathbb Z / 2\mathbb Z))\) of linear automorphisms preserving the symplectic structure acts on \(\text{Hom} (H_1(\Sigma_g, \mathbb Z / 2\mathbb Z), H_1 (E\setminus VC_0,\mathbb Z / 2\mathbb Z))/G\) by precomposition. Now, if the image of the stabilizer of $p$ in $\text{Aut}(H_1(\Sigma_g,\mathbb{Z}))$ by the mod 2 reduction homomorphism \[ \text{Aut}( H_1(\Sigma_g,\mathbb{Z}))\rightarrow \text{Aut}(H_1(\Sigma_g,\mathbb{Z}/2\mathbb{Z}))\] does not stabilize the Arnold invariant of an element, it means that there are at least two values of the invariant in the orbit. We will compare the sizes of the stabilizer and the image to conclude.  

\begin{lemma}\label{l: bound stabilizer Arnold invariant}
For any \((C,m,\omega)\in \text{Per}^{-1}(p)\) the stabilizer of \(Ar(C,m,\omega)\) in \(\text{Aut}(H_1(\Sigma_g, \mathbb Z / 2\mathbb Z))\) has order at most $2^{4g-5} (2g-2)!$, i.e. 
\[| \text{Stab} _{\text{Aut}(H_1(\Sigma_g, \mathbb Z / 2\mathbb Z))} (Ar(C,m,\omega))|\leq 2^{4g-5} (2g-2)! \] 
\end{lemma}

\begin{proof}
Let \(\varphi\) be a homeomorphism of \(E\) sending \(VC(\pi)\) to \( VC_0\), that is homotopic to the identity, and define \( A:= \varphi_*\circ\Pi\circ m\), where \( \varphi_*: H_1 (E\setminus VC(\pi); \mathbb Z/2\mathbb Z) \rightarrow H_1 (E\setminus VC_0, \mathbb Z/2\mathbb Z) \) is the map induced by \(\varphi\) on homology.  The Arnold's invariant is the class of \(A\) modulo post composition by an element of \(G\). We need to count the number of elements $M\in Aut(H_1(\Sigma_g,\mathbb Z/2\mathbb Z))$ such that there exists an element $g\in G$ satisfying
$$A\circ M=g\circ A.$$

The map $g$ preserves the image of $A$, and $g|_{\rm{Im}(A)}$ is determined by $M$. So we have a well-defined representation \[\rho:\text{Stab}_{Aut(H_1(\Sigma_g,\mathbb Z/2\mathbb Z))}(A)\to \text{End}(\text{Im}(A)) \text{ given by }\rho(M):=g|_{\text{Im}(A)}.\] To bound the cardinality of \(\text{Stab}_{Aut(H_1(\Sigma_g,\mathbb Z/2\mathbb Z))}(A)\) we proceed to bound the cardinality of the image and kernel of \(\rho\). 


To bound the cardinality of the image of \(\rho\), let us analyze  the action of $G$ on $H_1(E\setminus VC_0,\mathbb Z/2\mathbb Z)$. Consider the natural map $i_*:H_1(E\setminus
  VC_0,\mathbb Z/2\mathbb Z)\to H_1(E,\mathbb Z/2\mathbb Z)$ given by the inclusion. Its kernel can be identified
  with the space $\left(\oplus_{v\in CV_0}(\mathbb Z/2\mathbb Z)v\right)/(\mathbb Z/2\mathbb Z)\sigma$, with
  $\sigma=\sum_{v\in CV_0}v$, where $v\in CV_0$ corresponds to the cycle turning once
  around $v$. We have the exact sequence
\begin{equation} \label{eq: homology punctured elliptic curve} 0\to \left(\oplus_{v\in CV_0}(\mathbb Z/2\mathbb Z)v\right)/(\mathbb Z/2\mathbb Z)\sigma \to H_1(E\setminus VC_0,\mathbb Z/2\mathbb Z)\stackrel{i_*}{\to} H_1(E,\mathbb Z/2\mathbb Z)\to 0. \end{equation}
The group $G$ preserves $\ker i_*$ and acts on it by permutations. Moreover, it acts trivially on the quotient $H_1(E\setminus VC_0,\mathbb Z/2\mathbb Z)/\ker i_*\simeq H_1(E,\mathbb Z/2\mathbb Z)$ because any element of the braid group is isotopic to the identity once marked points are forgotten.

Since the image of $\Pi$ is the kernel of the monodromy representation of the branched covering $\pi: C\to E$, the subspace $\text{Im} (A)\subset H_1( E\setminus VC_0, \mathbb Z/2\mathbb Z) $ is a hyperplane, whose intersection with $\ker i_*$ is the $(2g-4)$-dimensional vector space $\ker i_*\cap\rm(Im(A))\simeq \left(\oplus_{v\in CV(\pi)}(\mathbb Z/2\mathbb Z)v\right)^{even}/(\mathbb Z/2\mathbb Z)\sigma$ (i.e. the formal sums of an even number of critical values  of \(\pi\)). Moreover, $i_*(\text{Im}(A))=H_1(E,\mathbb Z/2\mathbb Z)$. So the exact sequence \eqref{eq: homology punctured elliptic curve} induces an exact sequence
\begin{equation} \label{eq: exact sequence Im(A)} 0\to \left(\oplus_{v\in CV_0}(\mathbb Z/2\mathbb Z)v\right)^{even}/(\mathbb Z/2\mathbb Z)\sigma \to \text{Im}(A)\stackrel{i_*}{\to} H_1(E,\mathbb Z/2\mathbb Z)\to 0. , \end{equation}
which is invariant by \(\rho\), the action of this latter on the left-hand module being made by permutations.



Since the group of automorphisms of the exact sequence \eqref{eq: exact sequence Im(A)} acting on its left hand side by permutations has cardinality at most $(2g-2)! \times 2^{2(2g-4)}$, we get the bound \begin{equation}\label{eq: bound image rho} |\text{Im}(\rho)|\leq (2g-2)! \times 2^{2 (2g-4)}. \end{equation}

Let us now bound the number of elements of the kernel of \(\rho\). Let \(M\in \text{ker}(\rho)\), which means that \( A\circ M=A\). Write \(M = I + \psi \) where \( \psi : H_1(\Sigma_g, \mathbb Z / 2\mathbb Z) \rightarrow \text{ker}(A)\). 

\vspace{0.2cm} 

\noindent {\bf Claim:} \textit{ \(\ker (A)\) is a two-dimensional isotropic subspace of \(H_1(\Sigma_g,\mathbb Z / 2\mathbb Z)\).}

\vspace{0.2cm} 

\begin{proof}[Proof of the claim]
From  \cite{BE} we deduce that,  up to composition by homeomorphisms in source and target, there is only one double  branched covering from a genus $g$ closed oriented connected surface  to \(E\).  It is therefore sufficient to verify the statement on an example that we depict in Figure \ref{fig:claimexample}.

\begin{figure}[httb]
\centering
 \begin{tikzpicture}[x=1ex,y=1ex]
\draw (0,0) to[out=up,in=left] (15,20) to[out=right,in=up] (30,0) to[out=down,in=right]
(15,-20) to[out=left,in=down] (0,0);
\draw (8,14) to[out=-40,in=220] (22,14);
\draw (10,12.65) to[out=40,in=-220] (20,12.65);

\draw (8,-13) to[out=-40,in=220] (22,-13);
\draw (10,-14.35) to[out=40,in=-220] (20,-14.35);

\begin{scope}[rotate={90}, shift={(-7.5,-12)}, scale=1/2]
\draw (8,14) to[out=-40,in=220] (22,14);
\draw (10,12.65) to[out=40,in=-220] (20,12.65);
\end{scope}

\begin{scope}[rotate={90}, shift={(-7.5,-18)}, scale=1/2]
\draw (8,14) to[out=-40,in=220] (22,14);
\draw (10,12.65) to[out=40,in=-220] (20,12.65);
\end{scope}

\begin{scope}[rotate={90}, shift={(-7.5,-24)}, scale=1/2]
\draw (8,14) to[out=-40,in=220] (22,14);
\draw (10,12.65) to[out=40,in=-220] (20,12.65);
\end{scope}

\begin{scope}[rotate={90}, shift={(-7.5,-30)}, scale=1/2]
\draw (8,14) to[out=-40,in=220] (22,14);
\draw (10,12.65) to[out=40,in=-220] (20,12.65);
\end{scope}

\draw[dashed, ->] (-10,0) -- (40,0);
\begin{scope}[x={(.5,0)}]
\draw[->] (-10,.8) arc (20:330:2);  
\end{scope}

\node[above] at (-8,2) {$180^\circ$};

\draw[red] (15,13) circle [x radius= 9, y radius =4];
\draw[blue] (15,20) to[out=240,in=120] (15,14.5);
\draw[dashed, blue] (15,20) to[out=300,in=60] (15,14.5); 
\node[below] at (15,9) {$a_g$};
\node[left] at (14.9,18.2) {$b_g$};

\draw[red] (15,-14) circle [x radius= 9, y radius =4];
\draw[blue] (15,-15.65) to[out=240,in=120] (15,-20);
\draw[dashed, blue] (15,-15.65) to[out=300,in=60] (15,-20); 
\node[above] at (15,-10) {$a_{g-1}$};
\node[below] at (15,-20) {$b_{g-1}$};

\end{tikzpicture}
 \caption{Example of double cover with bi-dimensional isotropic $\ker A$} \label{fig:claimexample}
\end{figure}
First we construct a \(2:1\) branched covering \( \pi ' : \Sigma' \rightarrow \mathbb P^1\) on a connected closed oriented surface of genus \(g-2\), whose associated Arnold's map \(A ' : H_1 (\Sigma ', \mathbb Z /2\mathbb Z ) \rightarrow H_1 (\mathbb P^1 \setminus VC(\pi' ) , \mathbb Z/2\mathbb Z)\) is injective.  Let \( z_1, z_2, \ldots, z_{2g-2}\in \mathbb P^1\) be distinct points in the Riemann sphere, and let \( I_0, \ldots, I_{g-2}\subset \mathbb P^1\) be pairwise disjoint segments, such that \(\partial I_k = \{z_ {2k+1}, z_{2k+2}\} \). Slit two copies of \( \mathbb P^1 \setminus \left( \cup _k I_k \right) \) and glue them together by the usual rule: the right side of \(I_k\) in one copy is glued to the left side of \(I_k\) in the second, and vice versa. We get a branched double cover \( \pi': \Sigma' \rightarrow \mathbb P^1\) whose critical values are the \( z_k\)'s. We let \( \alpha_1 , \ldots , \alpha_{g-2}\) be simple closed curves that are boundaries of small neighborhoods of the \(I_k\)'s. Let \(\{J_k\}_{k=1,\ldots , g-2}\) be a disjoint family of embedded segments in the Riemann sphere, the boundary of \(J_k\) intersecting the union \(\cup I_l \) only in its extremities \(z_{2k}\) and \(z_{2k+1}\). We denote by \( \beta_k\) the boundary of a small neighborhood of \(J_k\).  The curves \( \alpha_1, \beta_1, \ldots , \alpha_{g-2}, \beta_{g-2}\) lift to curves \(a_1 ', b_1 ', \ldots , a_{g-2} ', b_{g-2} ' \subset \Sigma'\) forming a symplectic basis of \(\Sigma'\). Notice that the  homology of \(\mathbb P^1\setminus \{z_1, \ldots, z_{2g-2}\}\) with coefficients in \(\mathbb Z/2\mathbb Z\) is generated by the cycles associated to small simple closed curves \(c_k '\) turning around \( z_k\), the only relation being that the sum of these cycles is zero.  By construction we have \(A'(a_k' ) = c_{2k+1}'+c_{2k+2}'\) and \( A'(b_k') = c_{2k}'+c_{2k+1}'\). From this we easily prove that \(A'\) is injective.

Now, take an open disc \(\Delta\) in the Riemann sphere  that contains the union of the intervals \(I_k\) and \(J_k\), and the curves \(\alpha_k, \beta_k, c_k \),  for \(k=1,\ldots, g\). The surface \( \Sigma \setminus (\pi') ^{-1} (\Delta) \) is a union of two discs corresponding to lifts of \( \mathbb P^1 \setminus \Delta\). Consider a torus \( E\), and an embedding \(i:  \Delta \rightarrow E\). Attach to \((\pi') ^{-1} (\Delta) \) two copies of \( E\setminus i (  \Delta ) \) along the identification  of their boundaries given by the maps \(i\) and \(\pi\). We get a surface \(\Sigma \) of genus \(g\), and a branched covering \( \pi : \Sigma \rightarrow E\) defined by \( \pi '\) in \( (\pi ')^{-1} (\Delta) \) and by the natural identification of the two components of \( \Sigma \setminus  (\pi ')^{-1} (\Delta)\) with \( E\setminus i (\Delta)\). Denoting by \( \alpha, \beta\subset E\) oriented simple closed curves that do not intersect \( i (\Delta) \) and that form a symplectic basis of \( H_1 (E, \mathbb Z) \), we let \( a_{g-1}, b_{g-1}\subset \Sigma\) and \( a_g, b_g\subset \Sigma\) the corresponding cycles in the two copies of \( E\setminus i (\Delta) \). We denote by \( a_1, b_1, \ldots, a_{g-2}, b_{g-2}\subset \Sigma\) the simple closed curves equal to the curves \( a_1', b_1',\ldots, a_{g-2}', b_{g-2}'\) in \( (\pi')^{-1} (\Delta)\). The cycles \( a_1, b_1, \ldots, a_g, b_g\) in \(H_ 1(\Sigma, \mathbb Z/2\mathbb Z) \) form a symplectic basis, and we have 
\[A (a_k ) = c_{2k+1}+c_{2k+2}, \ A(b_k ) = c_{2k}+c_{2k+1} \text{ if } k\leq g-2\]
and 
\[A(a_{g-1})=A(a_g) = a, A(b_{g-1}) = A(b_g) = b, \]
where, in these formula, we denote by \( c_k= i (c_k')\). The kernel of \( A\) is then the space generated by \( a_{g-1} - a_g \) and \( b_{g-1} - b_g\), which is isotropic modulo 2. The claim follows.
\end{proof}

Since $M$ is symplectic and $\ker A$ isotropic of dimension two,  \(\psi\) satisfies \( \psi+ \psi ^* =0\), and induces an anti-symmetric map \(H_1(\Sigma_g,\mathbb Z / 2\mathbb Z)/ \text{ker}(A)^\perp \rightarrow \text{ker}(A)\) (here the anti-symmetric character is relative to the natural duality between \(\text{ker}(A)\) and \(H_1(\Sigma_g,\mathbb Z / 2\mathbb Z)/ \text{ker}(A)^\perp\) given by the intersection form). Since \( \text{ker}(A)\) is two dimensional, there is a \(3\)-dimensional space of such maps \(\psi\) over \(\mathbb Z / 2\mathbb Z\). In particular, we have 
\begin{equation}\label{eq: bound kernel rho} |\text{ker} (\rho) | \leq 2^3. \end{equation} 
The lemma follows immediately from bounds \eqref{eq: bound image rho} and \eqref{eq: bound kernel rho}. \end{proof}

We denote by \(p[2]\) the reduction of \(p\) modulo two, namely the map 
\begin{equation}\label{eq: period mod two} p[2]: H_1(\Sigma_g, \mathbb Z / 2\mathbb Z) \rightarrow \text{Im}(p) / 2\text{Im} (p) \simeq H_1(E, \mathbb Z / 2\mathbb Z).\end{equation}
More geometrically, for an element \( (C,m,\omega)\in \text{Per}^{-1}(p)\), we have \( p[2]= i_* \circ \varphi_* \circ \Pi \circ m\), where \(i_*: H_1 (E\setminus VC_0, \mathbb Z/2\mathbb Z) \rightarrow H_1 (E , \mathbb Z/2\mathbb Z)\) is the map induced by inclusion. 

To conclude the proof of Theorem \ref{t:degree 2}, we use an algebraic result whose proof can be found in Appendix II (see section \ref{s:appendix1}).

\begin{lemma}\label{l: surjection symplectic modulo 2}
If $g\geq 3$, the image of the mod 2 reduction map \[\text{Stab}_{\text{Aut}(H_1(\Sigma_g, \mathbb Z))}(p) \rightarrow \text{Aut}(H_1(\Sigma_g,\mathbb{Z}/2\mathbb{Z}))\] is \(\text{Stab}_{\text{Aut}(H_1(\Sigma_g, \mathbb Z / 2\mathbb Z))}(p[2])\). Its cardinality is
\(2^{2g-2}\times 2^{2g-3}\times (2^{2g-4}-1) \times 2^{2g-5} \times \ldots \times (2^2-1) \times 2. \) 
\end{lemma}

The combination of Lemma \ref{l: bound stabilizer Arnold invariant} and Lemma \ref{l: surjection symplectic modulo 2} show that  Arnold's invariant takes at least
\[\frac{2^{2g-2}\times 2^{2g-3}\times (2^{2g-4}-1) \times 2^{2g-5} \times \ldots \times (2^2-1) \times 2}
{2^{4g-5} (2g-2)!}\] 
distinct values on any orbit of \(\text{Stab}_{\text{Aut}(H_1(\Sigma_g, \mathbb Z))}(p)\) in \(\per^{-1}(p)\). This number is strictly larger than \(1\) if \(g\geq 5\), hence the proof of Theorem \ref{t:degree 2} is complete.

In the case of genus $g=4$, a more detailed study allows to show that  the Arnold's invariant takes only one value on the isoperiodic spaces of degree two (and actually, realizes another instance of the exceptional isomorphism between \(\text{Sp} (4,\mathbb Z/2\mathbb Z)\) and \(S_6\)). It is plausible that these isoperiodic sets are connected indeed, but we could not prove it by using this invariant. 


\subsection{Dynamics of the action of  $\text{Sp}(2g,\mathbb{Z})$ on $\mathcal{H}_g$}
\label{s:dynamics}

In this subsection, we state the extension of the analysis of the linear action of $\Gamma=\text{Sp}(2g,\mathbb{Z})$ on the set of periods of positive volume, that were presented by Kapovich in \cite{Kapovich} for the cases $g\geq 3$ to the case of genus $g=2$. For the sake of completeness we include the proof in Appendix I (see section \ref{s:appendix2}). This will enable us to derive Theorem \ref{t:dynamics} from Theorem \ref{t:connectedness}.



The closure of the orbit of a period \(p\) depends heavily on the rationality properties of the real two dimensional symplectic subspace 
\begin{equation} \label{eq: W} W= \mathbb R \Re p + \mathbb R \Im p \subset \mathbb{R}^{2g}\end{equation} defined over smaller fields. Except for the case where $W$ is defined over a quadratic field and $W^{\sigma}=W^{\perp}$ for the Galois involution, the different cases can be characterized by the dimension of the maximal subspace of $W$ defined over the rationals. This dimension can also be detected by the topological properties of the closure of the $\mathbb{Z}$-submodule $\Lambda(p)$ of $\mathbb{C}$ generated by the entries of $p\in\mathbb{C}^{2g}$. It is two if $\Lambda(p)$ is discrete, one if it is isomorphic to $\mathbb{R}+i\mathbb{Z}$ and zero if it is dense. Notice that this submodule is invariant under the action of $\Gamma$ on $X$. The different possibilities in Theorem \ref{t:dynamics} arise from this analysis at the level of this submodule. We resume the analysis in the following

\vspace{0.3cm}

\begin{proposition}\label{p:Kapovich}
Assume $g>2$. For any $p\in \mathbb C^{2g}$ of positive volume, we have the trichotomy for $W=\mathbb R \Re p + \mathbb R \Im p$
\begin{itemize}
\item $W$ is defined over $\mathbb Q$. In this case, $\Lambda(p)$ is discrete and either $p$ is the period of a finite branched covering of the elliptic differential $(\mathbb C/ \Lambda(p), dz)$ or it is a collapse of $g-1$ handles. The set $\overline{\Gamma \cdot p}$ is the set of periods $q\in \mathbb C^{2g}$ of volume $V(q)=V(p)$ such that $\Lambda(q)=\Lambda(p)$.
\item $W$ is not rational but contains a rational line. In this case, up to the action of $\text{GL} (2,\mathbb{R})$, the set $\overline{\Lambda(p)}$ is $\mathbb R + i\mathbb Z$, and $\overline{\Gamma \cdot p}$ is the set of periods $q\in \mathbb C^{2g}$ of volume $V(q)=V(p)$ whose imaginary part are integer valued and primitive.
\item $W$ does not contain any rational subspace of positive dimension. In this case, $\overline{\Lambda(p)} = \mathbb C$, and $\overline{\Gamma \cdot p} $ is the set of periods $q \in \mathbb C^{2g}$ such that $V(q)=V(p)$.
\end{itemize}
In genus $g=2$, there is another possibility:
\begin{itemize}
\item $W$ is defined over a quadratic field $K$, and $W^\perp = W^\sigma$, where $\sigma$ is the Galois involution of $K$. In this case $\overline{\Gamma \cdot p}$ is the set of periods $q$ which differ from $p$ by post-composition by an element of $\text{GL}(2,\mathbb{R})$, and by pre-composition by an element of $\text{Sp}(4,\mathbb Z)$.
\end{itemize}
In any case, the action of $\Gamma$ is ergodic in $\overline{\Gamma \cdot p}$.
\end{proposition}


\subsection{Proof of (Theorem   \ref{t:connectedness})$_g \Rightarrow$(Theorem \ref{t:dynamics})$_{g}$} \label{ss: pf of thm dynamics} We suppose that we know that Theorem \ref{t:connectedness} is true for some genus $g\geq 2$ and hence by the Transfer Principle,  invariant sets of the isoperiodic foliation $\mathcal{F}_g:=\mathcal{F}_{g,0}$ correspond bijectively to  invariant sets of the $\text{Sp}(2g,\mathbb{Z})$ action on $\mathcal{H}_g$. We want to deduce Theorem $\ref{t:dynamics}$ for genus $g$. For the sake of simplicity we omit the sub-index of $\per_{g,0}$ and write $\per$. 

The closure $\overline{L}$ of the leaf $L=L(C,\omega)$ of $\mathcal{F}_g$ through $(C,\omega)\in\Omega\mathcal{M}_g$ is a closed $\mathcal{F}_g$-invariant set. By the Transfer principle (see Theorem \ref{t:connectedness}), after choosing a marking $m$ of $C$, it corresponds to the projection from $\Omega\mrs_g$ to $\Omega\mathcal{M}_g$ of all marked abelian differentials having periods in the closure of the $\text{Sp}(2g,\mathbb{Z})$-orbit of $p=\per(C,m,\omega)\in\mathcal{H}_g$. The final ergoditicy part of the statement of Theorem \ref{t:dynamics} follows from the final statement in Proposition \ref{p:Kapovich}. 

\emph{Case 1}: If $\omega$ is a genus two eigenform for real multiplication by a real quadratic order $\mathfrak{o}_D$ of discriminant $D>0$. Then the closure of $L$ is one of the Hilbert modular manifolds $\Omega X_D$ (see \cite{Calta} or \cite{McMullen3}). As shown in \cite[Case 3. of Theorem 5.1]{McMullen3} this case occurs if and only if $p$ has the properties described in the last possibility of Proposition \ref{p:Kapovich}.

In what follows we suppose that $p$ does not satisfy the last condition of Proposition \ref{p:Kapovich}.
The other possibilities can be characterized in terms of the closure of $\Lambda(p)$ in $\mathbb{C}$:

\emph{Case 2}: $\Lambda(p)$ is discrete. The image of the periods of any $(C',\omega')\in \overline{L}$ are precisely $\Lambda(p)$.

Each element in the fibre $\per^{-1}(p)\subset \Omega\mrs_g$ corresponds, by integration, to a unique primitive branched covering over $\mathbb{C}/\Lambda(p)$ of volume $V$. If we find a marking $m'$ of $C'$ where $\text{Per}(C',m',\omega')=p$, then by Theorem \ref{t:connectedness} we have $(C',\omega')\in L$. Therefore the leaf $L$ is closed. The first item of Proposition \ref{p:Kapovich} (or equivalently Lemma \ref{l:finite degree orbit}) shows that such a marking $m'$ on $(C',\omega')$ exists. 

In particular this provides an alternative proof of the connectedness of the Hurwitz space of primitive branched coverings over $\mathbb{C}/\Lambda(p)$ of volume $V$ (see \cite{BE} for the original proof).

\emph{Case 3}: $\overline{\Lambda(p)}$ is neither discrete nor dense in $\mathbb{C}$. Then its closure is an infinite union of parallel real lines and we fall in the second case of Proposition \ref{p:Kapovich}. The conclusion is that in the closure $\overline{L}$ we find all pairs $(C',\omega')$ of volume $V$ whose periods belong to $\Lambda(p)$ and contain elements in every real line where $\Lambda(p)$ is dense.

\emph{Case 4}: $\Lambda(p)$ is dense in $\mathbb{C}$. We fall in the third case of Proposition \ref{p:Kapovich}. In $\overline{L}$ we find all pairs $(C',\omega')$ of volume $V$.

\subsection{Remark on the transfer principle on strata}

The generic stratum is the complement of an analytic subset, and therefore its intersection with each fiber of $\per_{g}$ has the same number of connected components as the whole fiber. We can thus apply the Transfer principle to the generic stratum as well. On the other strata, the question seems to be much more delicate (see the discussion in section \ref{ss: notes and references}). At the other extreme, on the minimal stratum, the isoperiodic foliation has dimension zero, so we cannot expect to transfer dynamical properties of the symplectic group on the period domain to dynamical properties of the isoperiodic foliation. We even show here that 
 
\begin{proposition}
The intersection of a fiber of $\per_g$ with the minimal stratum is either empty or infinite discrete.
\end{proposition}

The connectedness of the intersection of the fibers of $\per_g$ with the connected components of all other strata remains open.

\begin{proof}
The restriction of the map $\per_{g}$ to the minimal stratum is a local diffeomorphism onto an open set in $\mathcal{H}_{g}$. Hence its fibers are discrete. 
Let $(C,m,\omega)\in\Omega\mrs_{g}$ be a point in the minimal stratum and $p=\per (C,m,\omega)$. Remark that the image of $\text{Aut}(C,\omega)\rightarrow \text{Sp}(H_1(C,\mathbb{Z}))$ is a finite subgroup $G$.
Therefore the group $$\text{Stab}(C,m,\omega)=\{M\in\text{Sp}(H_{1}(\Sigma_{g},\mathbb{Z})): (C,m\circ M,\omega)\sim (C,m,\omega)\}$$ is finite.
If $p$ has discrete image, then $\text{Stab}_{\text{Sp}(H_{1}(\Sigma_{g},\mathbb{Z}))}(p)$ is infinite and contains $G$. Hence there is an infinite number of points in the orbit of $(C,m,\omega)$ under this group. All of them lie in $\per^{-1}(p)$, so the result follows in this case.

Next suppose $p$ has non-discrete image. By Proposition \ref{p:Kapovich} there exists an infinite sequence of  $M_{n}\in\text{Sp}(H_{1}(\Sigma_{g},\mathbb{Z}))$ such that  $p_{n}=p\circ M_{n}$ is a sequence of pairwise distinct points converging to  $p$ in $\mathcal{H}_{g}$.  Up to taking a subsequence, these points correspond to points $(C_{n},m_{n},\omega_{n})$ in the minimal stratum via $\per_{g}$, that project to an infinite family of pairwise distinct points $(C_{n},\omega_{n})\in\Omega\mathcal{M}_{g}$. By construction, all points $(C_{n},m_{n}\circ M_{n}^{-1},\omega_{n})$ belong to the intersection of $\per^{-1}(p)$ with the minimal stratum and are pairwise distinct in $\Omega\mrs_{g}$. \end{proof}
\subsection{Monodromy of zeros along a fiber of  $\per_{g,n}$ in the generic stratum }

\begin{definition} Given a connected component $K$ of the intersection of a fiber of $\per_{g,n}$ with a stratum of $\Omega\mrs_{g,n}$ having $l$ distinct zeros, define the monodromy group of zeros associated to $K$ as the conjugacy class of the image group $G_{K}$ of the representation $$\pi_1(K,(C_0,P_0,m_0,\omega_0))\rightarrow S_l$$
that associates to each loop in $K$ the permutation induced on the zero set $z_1,\ldots,z_l$ of $\omega_0$.
\end{definition}

Our main interest will be the monodromy group of the zeros in the generic stratum. Recall that it is a Zariski open set in the fiber of the isoperiodic map and is therefore connected if and only if the fiber is. Turning around the divisors associated to other strata in the fiber already produces some monodromy of the zeros:

\begin{lemma}\label{l:non-trivial monodromy of zeros}
 Let $g\geq 2$ and $L$ be a connected component of a fiber of $\per_{g,n}$. Suppose it contains a point whose zero divisor is $d_1p_1+\ldots+d_lp_l$. Then, the monodromy group of the zeros in the generic stratum $L^{SZ}$ of $L$,  contains all transpositions of pairs of elements that belong to the same part of a partition $Z_1\sqcup\cdots\sqcup Z_l$ of the set of $2g-2$ elements with $|Z_i|=d_i$.
 \end{lemma}
 
 Before we proceed to the proof we recall the interpretation of an element $(C,p_1,\ldots,p_n,\omega)$ in $\Omega\mathcal{M}_g$  as a branched projective structure (see \cite{bps} for more details), or, more precisely, a branched translation structure. Indeed, around any point $q\in C$ we can define a local branched covering of degree $\text{ord}_q(\omega)+1$ by $\phi_q(z)=\int_{q}^{z}\omega$. At the intersection of domains of two such branched coverings differ by a translation, and therefore define a branched atlas on the topological curve $C$ with transitions in the set of translations (which are actually holomorphic). The branch points correspond to the zeros of the form.
 Reciprocally, if we are given an atlas $\{D_q, \phi_q\}$ on a compact topological surface of genus $g$ of (topological) branched coverings $\phi_q:D_q\rightarrow \mathbb{C}$  with a finite number of branch points whose transition maps are translations, it defines a complex structure on the surface and the form defined locally by  $\omega:=d\phi_q$ is globally defined. In section \ref{s:Schiffer variations} we will exploit this interpretation further. For the moment we will use it to deform a given point without changing the corresponding periods by changing a given atlas locally around a branched point.

 \begin{proof}[Proof of Lemma \ref{l:non-trivial monodromy of zeros}:]  Let $(C,r_1,\ldots,r_n,m,\omega)$ be the point satisfying $(\omega)_0=d_1p_1+\ldots+d_lp_l$. If all $d_i=1$ there is nothing to prove. Otherwise, $L^{SZ}$ is the complement of a normal crossing divisor in $L$ (see section \ref{ss:strata}). Choose a local chart $\phi_i:D_i\rightarrow\mathbb{C}$ of degree $d_i+1$ of the translation structure of $\omega$  around $p_i$.  We can suppose that none of the zeros of $\omega$ is a marked point $r_i$ and take discs $D_i$ not containing any marked point.

Recall from the appendix in \cite{bps} that we can construct a  germ of continuous  map $$\big(H(\phi_1)\times\cdots\times H(\phi_{d_l}),(\phi_1,\ldots,\phi_l)\big)\rightarrow (L,\omega)$$
where each $H(\phi_i)$ is a Hurwitz space of coverings of degree $d_i$ over a disc $\phi_i(D_i)$ up to an equivalence in the boundary that allows to glue each element with the translation surface defined by $\omega$ on $C\setminus (\cup D_i)$ to obtain a branched translation structure on a closed marked surface of genus $g$.  Since we can keep the marked points and a marking of the homology by avoiding the discs $D_i$, the gluing preserves all integrals over cycles of $H_1(\Sigma_{g},q_1,\ldots,q_n;\mathbb{Z})$.

By \cite{bps}[Lemma A7, p. 439], each $H(\phi_i)$ can be parametrized by the space of polynomials $P_a(z)=z ^{d_i}+a_{d_i-1}z^{d_i-1}+\ldots+a_0$ with $\sum a_j=0$ having critical values in the unit disc. The point with all coordinates $a_j=0$ corresponds to the initial point. Any choice on each Hurwitz space of a point that has $d_i$ distinct critical values $\{v_1,\ldots,v_{d_i}\}$ (therefore determining simple critical points on some set $Z_i$) has as image in $L$ a point with simple zeros.  On the other hand, the set of critical values determines the values of the $a_i$'s in the parameter space. Choose one such point $(C_0,\omega_0)\in L^{SZ}$. Since the germ of  $\mathbb{D}^{d_i}\setminus \{(z_1,\ldots,z_{d_i}): \text{ with }z_j=z_k\text{ for some } j\neq k\}$ at the origin is connected, there exists a path in it joining $(v_1,\ldots,v_{d_i})$ to any point obtained by permuting the order of the coordinates. Such a path will determine a closed path in the parameter space whose image in $L^{SZ}$ will permute the zeros of $\omega_0$ in each part $Z_i$ as desired. 

\end{proof}

\begin{lemma} \label{l: transitivity}
Let $g\geq 3$ and $p\in \mathcal{H}_g$ such that \( \per^{-1}(p) \) is connected. Then, the monodromy group of the zeros in the generic stratum of \( \per^{-1}(p)\) is $S_{2g-2}$ if \( \deg(p) \geq 3 \) and trivial if $\deg(p)=2$.
\end{lemma}

\begin{proof}
Assume first that \(p\) has infinite primitive degree, or equivalently non discrete image. In this case we claim there exists a form in \(\per ^{-1} (p) \) in the minimal stratum, namely having a unique zero of multiplicity \(2g-2\). From Lemma \ref{l:non-trivial monodromy of zeros}  we deduce that the whole group $S_{2g-2}$ is in the image of the representation. To prove the claim, observe that up to the action of the group \( \text{GL}^+(2,{\mathbb R})\), we can assume that the closure of the orbit of \(p\) under the action of the group \( \text{Aut} (H_1(\Sigma_g, \mathbb Z) )\) contains the set \(P\) of all periods \( q: H_1(\Sigma_g , \mathbb Z) \rightarrow \mathbb R +i {\mathbb Z} \) of volume \(\vol(q)= \vol(p)\) (see  Proposition \ref{p:Kapovich}). Since the period map restricted to the minimal stratum is a local biholomophism (see subsection 3.2), its image is an open set. It therefore contains points with discrete values in $\mathbb{Q}+i\mathbb{Q}$ that necessarily have finite degree.  Up to the action of $\text{GL}(2,\mathbb{R})$ we can find one whose periods lie in $P$. Moreover, there is a neighborhood of \(P\) in \(\text{Hom} (H_1(\Sigma_g , {\mathbb Z}), {\mathbb  C}) \) which consist of periods of forms in the minimal stratum. In particular, the orbit of \(p\) under \( \text{Aut} (H_1(\Sigma_g , {\mathbb Z}) )\) intersects this neighborhood, and thus \(p\) itself is the period of a form in the generic stratum. Hence the result follows in the case \(\deg(p)=\infty\). 

Assume now that \( \deg (p) < \infty\). In this case, consider the map \( R : \per^{-1} (p) \rightarrow H _{g,d} (E)\) taking values in the Hurwitz space \( H_{g,d} (E) \) of equivalence classes of primitive branched coverings over the elliptic curve \( E= {\mathbb C} / \text{Im} (p) \) of primitive degree \( d= \deg(p)\), the equivalence being the pre- and post-composition by automorphisms. The map \(R\) sends a form \( (C,\omega, m) \in \per^{-1} (p)\) to the equivalence class of the covering from \(C\) to \(E\) given by integrating \(\omega\). Denoting by \( U\subset \per  ^{-1} (p)^{SZ}\) the Zariski open subset consisting of forms with simple zeros whose integral between any pair of distinct zeros does not belong to the lattice \(\text{Im} (p)\), and by \( H_{g,d} (E)^{DC}\subset  H _{g,d} (E)\) the Zariski open subset consisting of coverings with $2g-2$ distinct critical values, the map \(R\) induces a Galois covering \( U \rightarrow H_{g,d}  (E)^{DC}\) with group \( \text{Aut} (H_1(\Sigma_g, {\mathbb  Z}))\). Consider the representation of permutations of the critical values \begin{equation}\label{eq:permutations of critical values}\pi _1 ( H_{g,d}  (E)^{DC} ) \rightarrow S_{2g-2}.\end{equation}  
It is well-known that  this representation is surjective, but by lack of references we reproduce a proof in the following paragraph. With this at hand, let $H\subset S_{2g-2}$ be the normal subgroup corresponding via \eqref{eq:permutations of critical values} to the image of the Galois covering $\pi_{1}(U)\rightarrow \pi _1 ( H_{g,d}  (E)^{DC} )$. We claim that if \(\deg(p)\geq 3\), and \(g\geq 3\), then $H$ contains a transposition, and hence $H=S_{2g-2}$.  We will find a transposition by applying  Lemma \ref{l:non-trivial monodromy of zeros} to a particular form with a double zero.  By the first item of Proposition \ref{p:Kapovich}, or equivalently Lemma 9.1 ,  it suffices to prove the existence of primitive genus $g$ and degree $d$ branched covering over $E$, branched over a set $\{e_1,\ldots,e_{2g-3}\}$ of $2g-3$ elements  with transitive (covering !) monodromy representation $\pi_1(E\setminus\{e_1,\ldots,e_{2g-3}\})\rightarrow S_d$ taking one of the peripherals to a cycle of order three (which corresponds to the local monodromy of the branched cover defined by integration around a double zero of a holomorphic form) and every other peripheral to a transposition (which corresponds to the local monodromy generator of the branched cover defined by integration around a simple zero of a holomorphic form). This is certainly possible if $d\geq 3$ and $2g-3\geq 2$, i.e. $g\geq 3$.  

Let us now  prove that the representation  $\eqref{eq:permutations of critical values}$ is onto. Let \( c : E'\rightarrow E\) be a non ramified cyclic \(d:1\) covering of \(E\). We denote by \( \tau \in \text{Aut} (E')\) a generator of the Galois group of \(c\).  Given $g-1$  smooth curves \(s_1, \ldots, s_{g-1}\)  in \(E'\), diffeomorphic to closed intervals, such that all the intervals \(c(s_1) , \ldots, c(s_{g-1})\) are disjoint we can construct an associated element in $H_{g,d}  (E)^{DC}$. Indeed,  all curves in the family \(s_1,\ldots, s_{g-1}, \tau (s_1), \ldots, \tau(s_{g-1}) \) are disjoint, and we can  slit \(E'\) along all the curves of this later family, and glue the right side of \(s_k\) to the left side of \(\tau(s_k)\), and vice versa, for every \(k=1,\ldots, g-1\). We obtain a Riemann surface \(C\) of genus \(g\), on which the map \(c\) induces a well-defined degree \(d\) branched covering \( r : C\rightarrow E\), which is primitive. Observe that the critical values of \( r\) are the $2g-2$ points lying at the extremities \(\iota_k,e_k\) of the curves \(c(s_k)\) for \(k=1,\ldots, g-1\).

The previous construction can be used to describe paths in $H_{g,d}  (E)^{DC}$ by deforming the chosen smooth curves continuously. Take a choice of $g-1$  smooth curves \(s_1, \ldots, s_{g-1}\)  in \(E'\) and consider the associated covering $r:C\rightarrow E$ . Observe, for instance,  that moving the curve \(s_k\) in $E$ without touching the other curves, and returning back to the same curve \(s_k\) but with opposite orientation, we draw a closed loop in \( H_{g,d}  (E)^{DC}\) based at $r$, and the permutation on the critical values is the transposition exchanging \(\iota_k\) and \(e_k\). It remains to show that transpositions of critical values, involving endpoints of distinct segments are also possible. Now, since the elements of \( H_{g,d}  (E)^{DC}\) have $2g-2$ critical values, all critical points of $r$ are simple and there is a one--to--one correspondence between critical values and critical points of $r$. Each critical point corresponds to a zero of $\omega=dr$, so if for each pair of distinct $k,l$ we manage to find a loop in \(U\) with base point $\omega$ whose associated monodromy of the zeros exchanges the zero associated to the endpoint $\iota_{k}$ of $s_{k}$ with the zero associated to an endpoint $e_{l}$ of $s_{l}$ and fixes every other zero, we will be done. As before, we will invoke Lemma \ref{l:non-trivial monodromy of zeros} which allows to deduce the property if we manage to find a path in $\per^{-1}(p)$ starting at $\omega$ and finishing at a point in a stratum of forms with precisely $2g-4$ simple zeros corresponding to the endpoints distinct from $\iota_{k}$ and $e_{l}$ and a double zero. This path can be constructed in \( H_{g,d}  (E)^{DC}\) by deforming \(s_k\) in such a way that it does not touch any other \(s_h\), apart at the very end of the movement where the only two points of distinct segments that coincide are \( \iota_k\) and \(e_l\). The set of permutations constructed in this way generate the whole symmetric group of the set \(\{\iota_1,e_1, \ldots, \iota_{g-1}, e_{g-1}\}\). 

It remains to prove the last statement, namely that if \(p\) has  primitive degree two, the monodromy of the zeros on \(\per^{-1} (p) \) is trivial. Remark that all forms in that set have simple zeros. We can use  done using the invariant of Arnold's introduced in the proof of Theorem \ref{t:degree 2}. Consider the (locally flat) fiber bundle  \(\mathcal E\rightarrow \per ^{-1} (p) \) by \( \mathbb Z/ 2\mathbb Z \)-vector spaces, whose fiber over an element \( (C, m , \omega) \) is the \(\mathbb Z/2\mathbb Z\) vector space generated by the zeros of \(\omega\), the only relation being that the sum of all zeros is trivial. Since the number of zeros of \(\omega\) is \(2g-2\geq 4\), we need to prove that the monodromy of \(\mathcal E\) is trivial when the primitive degree of \(p\) is equal to two. 

The fiber bundle \(\mathcal E\) embeds naturally in the locally flat \(\mathbb Z /2\mathbb Z\)-vector bundle \(\mathcal G\) over \(\per ^{-1} (p)\) whose fiber over a point \( (C, m , \omega) \) is the space \( H_1 (E\setminus \text{VC} (\pi), \mathbb Z/2\mathbb Z) \); the embedding sends a zero of \(\omega\) to the peripheral cycle turning around its \(\pi\)-image in \(E\).   

Fiberwise, the family of maps \( \Pi \circ m\), where \( (C,m,\omega) \) varies in \(\text{Per}^{-1} (p)\) and \(\Pi\) is the map \eqref{eq: pre-Arnold's map}, induce a map from the constant bundle \( H_1(\Sigma_g, \mathbb Z/2\mathbb Z) \) over \(\text{Per}^{-1} (p) \) to the bundle \(\mathcal G\). The image of the constant subbundle \( \text{Ker} (p[2])\) of \(H_1(\Sigma_g, \mathbb Z/2\mathbb Z) \) over \( \per ^{-1} (p)\) (recall that \(p[2]\) is the reduction of \(p\) modulo two, see \eqref{eq: period mod two}) fall inside the subbundle \(\mathcal E\) of \(\mathcal G\). Moreover,  it consists fiberwise of the subspaces of formal sums of an even number of peripheral cycles in \(H_1 (E\setminus \text{VC} (\pi) , \mathbb Z/2\mathbb Z)\).  

We infer that the monodromy of \(\mathcal E\) acts trivially on this subspace. Since a permutation of a set of at least three elements is determined by its action on the subsets of an even number of elements, we are done.\end{proof}

\section{Augmented Torelli space and extension of the period map}
\label{s:augmented Torelli and period maps}
In this section we review the topological and analytical properties of moduli spaces of curves and some of their coverings and (partial) bordifications. We will adopt the analytic point of view first developed by Abikoff in \cite{Abikoff}. It is based on the Augmented Teichm\"uller space (see \cite{ACG,Bers1,Bers2,Bers3,GH,Harris,Herrlich,HK,Masur3,WW, Wolpert1,Wolpert2,Yamada} for details and further references) and the bundles of stable forms over it (see the two recent papers \cite{5A18,5A19}). Most of the section will be devoted to describe the properties of the quotient of Augmented Teichm\" uller space by the Torelli group, the so-called Augmented Torelli space. We will also analyze the extension of the period map to the  points lying in the closure of fibers of $\per_{g,n}$.

\subsection{Marked stable curves }
\label{s:stable curves}
\begin{definition}
  A connected complex curve $C$ with $n$ marked distinct points $r_1,\ldots,r_n\in C$ is said to be stable if its singularities are nodes and do not coincide with any of the marked points,  and the closure $C_i$ of each component of $C^*:=C\setminus\text{Sing}(C)$, called a part of $C$, has a group of automorphisms that fix the marked points and the boundary points that is finite. The normalization of $C$ is the smooth curve $\hat{C}=\sqcup C_i$. A stable curve $C$ is said of compact type if every node separates $C$ in two components. Otherwise $C$ is said to be of non-compact type.
\end{definition}

The arithmetic genus of a stable curve is $g=h^1(C,\mathcal{O})$. When $C$ has $\delta$ nodes and its normalization has $\nu$ components of genera $g_1,\ldots, g_{\nu}$, the arithmetic genus satisfies (see \cite{Harris}[p. 48]) $$g=\sum_{i=1}^\nu (g_i-1)+\delta+1$$

As in section \ref{ss:Torelli and period map}, for $g,n\geq 0$ we fix a reference surface with marked points $\Sigma_{g,n}=(\Sigma_g,q_1,\ldots,q_n)$: a closed connected oriented surface $\Sigma_g$  of genus $g$ with a set of $n$ distinct ordered marked points $Q=(q_1,\ldots, q_n)\in\Sigma_g$. When $n=0$ we omit the subindex and write $\Sigma_g=\Sigma_{g,0}$.  

\begin{definition}
 A homotopical marking (or sometimes a collapse) of a connected genus $g$ stable curve $C$ with $n$ ordered pairwise distinct marked points $R=(r_1,\ldots, r_n)\in C^*$ is a continuous surjection $f:\Sigma_{g,n}\rightarrow (C, r_1,\ldots,r_n)$ such that $f(q_i)=r_i$, the preimage of each node is a simple closed curve on $\Sigma_g\setminus Q$ and on each component of $\Sigma_g\setminus f^{-1}(N)$ where $N$ is the set of nodes, the map $f$ is a homeomorphism onto a part of $C$ that preserves the orientation. 
\end{definition}

When $C$ is non-singular a collapse is a homeomorphism and the definition coincides with the one given in subsection \ref{ss:Torelli and period map}. 

\begin{definition}
 A homotopically marked stable curve with $n$ marked points is a marked stable curve $(C, r_1,\ldots,r_n)$ together with a homotopical marking $f:(\Sigma_{g},Q)\rightarrow (C,R)$. Two homotopically marked stable curves $f_i:(\Sigma_g,Q)\rightarrow (C_i,R_i)$ for $i=1,2$ are said to be equivalent if there exists a conformal isomorphism $g:C_1\rightarrow C_2$ such that $g\circ f_1$ is homotopic to $f_2$ relative to $Q$. The class of a $\Sigma_{g,n}$ marked stable curve will be denoted by $[f:\Sigma_{g,n}\rightarrow (C, r_1,\ldots, r_n)]$. 
\end{definition}

\begin{remark}\label{rem:homotopically deform to dehn}
If $\Delta:\Sigma_{g,n}\rightarrow\Sigma_{g,n}$ is a Dehn twist around a curve  in $\Sigma_g$ that is collapsed by the marking $f_1:(\Sigma_g,Q)\rightarrow (C,R)$ to a point, then $(C,R)$ marked by $f_1\circ\Delta$ is equivalent to the same curve marked by $f_1$. 
\end{remark}

\subsection{Augmented Teichm\" uller space and its stratification}
\label{ss:augmented teich}
\begin{definition}
The augmented Teichm\" uller space $\overline{\mathcal{T}}_{g,n}$ is the set of all homotopically marked stable genus $g$ curves with $n$ marked points  up to equivalence. 
\end{definition}

The Teichm\" uller space $\mathcal{T}_{g,n}$ is the subset of  $\overline{\mathcal{T}}_{g,n}$ formed by curves without nodes. Its complement, denoted by $\partial\mathcal{T}_{g,n}=\overline{\mathcal{T}}_{g,n}\setminus \mathcal{T}_{g,n}$ is called the boundary.

\begin{definition}
A curve system $c=\sqcup c_i$ in $\Sigma_{g,n}=(\Sigma_g,q_1,\ldots,q_n)$ is a disjoint collection of simple closed curves $c_i$ on $\Sigma_g\setminus \{q_1,\ldots,q_n\}$ none of which is isotopic to any other, to a point or to a cylinder in $\Sigma_g\setminus \{q_1,\ldots,q_n\}$. To a curve system $c$ we can associate the subset $B_{c}\subset \partial\mathcal{T}_g$ of the boundary consisting of homotopically marked stable curves topologically equivalent to a collapse $\Sigma_g\rightarrow \Sigma_g/c$ obtained by identifying each curve in $c$ to a point. 
\end{definition}

Given a curve system $c$, for each component $\Sigma^{i}$ of $\Sigma_{g}\setminus c$ we define $(\Sigma_{g_i},Q_i)$ to be the closed surface of genus $g_i$ with a set of marked points $Q_i$ obtained by collapsing each boundary component of $\Sigma^i$ to a (marked) point and keeping the marked points of $\Sigma_{g}$ lying on $\Sigma^i$ in $Q_i$. As we will see in subsection \ref{ss:attach&forget on augmented Torelli} there is a natural identification  \begin{equation}B_{c}\cong\Pi_{i}\mathcal{T}_{g_i,n_i}.\label{eq:curvesystemdecomposition}\end{equation}
 
The boundary $\partial\mathcal{T}_{g,n}$ is the disjoint union of all boundary strata $\sqcup_{c}B_{c}$ where 
$c$ varies in the set of nonempty curve systems. $\mathcal{T}_{g,n}$ corresponds to the empty curve system. 
Each stratum $B_c$ has a topology and complex structure given by the bijection (\ref{eq:curvesystemdecomposition}). All the curves that appear in a stratum have the same number of separating and non-separating nodes. When all the nodes are non-separating we say that the stratum is of compact type. 

Given a simple closed curve $c'\subset \Sigma_{g,n}$ we denote by $$D_{c'}=\bigsqcup_{c'\subset c}B_{c}$$ the union of all strata that collapse $c'$ to a node.

\cbstart
\begin{definition}
Given $\mathcal{K}\subset\overline{\mathcal{T}}_{g,n}$, a stratum of $\mathcal{K}$ is, by definition, the intersection $\mathcal{K}\cap B_c$ of $\mathcal{K}$ with a stratum  $B_c\subset\overline{\mathcal{T}}_{g,n}$. If the stratum lies in the boundary of $\mathcal{T}_{g,n}$ we call it a boundary stratum. The boundary strata of $\mathcal{K}$ is the set $\partial \mathcal{K}:= \mathcal{K}\cap\partial\mathcal{T}_{g,n}$.
\end{definition}
\cbend
\subsection{Topology and stratification of $\overline{\mathcal{T}}_{g,n}$}
\label{ss:topology of Augmented teichm} 

For a detailed description of the topology we refer to \cite[pp. 485-493]{ACG} and references therein. 




The restriction of the given topology to $\mathcal{T}_{g,n}$ produces the so-called conformal topology. Abikoff showed (in \cite{Abikoff}[Theorem 1]) that this topology is equivalent to the Teichm\"uller topology introduced in subsection \ref{ss:Torelli and period map}. 

The restriction of the topology to the boundary stratum $B_{c}$ corresponding to a curve system $c$ is equivalent to the product topology obtained from (\ref{eq:curvesystemdecomposition}).



The topology is not locally compact around any boundary point. Indeed, if $U$ is a neighbourhood of a point in $B_c$, the action of the Dehn twist $\Delta_a:\Sigma_g\rightarrow \Sigma_g$ around a simple closed curve $a\in c$ fixes all the points in $U$ for which the marking collapses $a$ to a point, but has infinite orbits at any other point in $U$. Therefore, there is no manifold structure in $\overline{\mathcal{T}}_{g,n}$ compatible with the given topology.

\begin{definition}
Given a curve system $c$ the distinguished neighbourhood of the stratum $B_{c}$ is the set   $$U_{c}=\bigsqcup_{c'\subset c}B_{c'}.$$
\end{definition}

\subsection{Local stratification and Deligne-Mumford compactification} 
The mapping class group of $\Sigma_{g,n}$, i.e. the group $\text{Mod}(\Sigma_{g,n})$ of isotopy classes of orientation preserving diffeomorphisms that fix each marked point, 
acts on $\overline{\mathcal{T}}_{g,n}$ by homeomorphisms that preserve the stratification and are holomorphic in restriction to any stratum. The action is defined by pre-composition on the marking. The quotient $$\overline{\mathcal{M}}_{g,n}=\overline{\mathcal{T}}_{g,n}/\text{Mod}(\Sigma_{g,n})$$ is a compact topological space. 
It can be endowed with a complex orbifold structure. 

Consider a curve system $c$ and define  $\Gamma_c$ the abelian group generated by Dehn twists around the curves in $c$. Following \cite{Bers1}, the quotient $U_c/\Gamma_c$ is equivalent to a bounded domain in $\mathbb{C}^{3g-3}$. Under this equivalence, each stratum $B_{c'}$ associated to a simple closed curve $c'\subset c$ has image contained in a regular divisor $D_{c'}$. These divisors intersect normally and their intersections define the other different strata: the stratum associated to $c'\subset c$ is the intersection of all the divisors associated to the simple curves in $c'$. The complement of this divisor is the stratum $B_{\emptyset}=\mathcal{T}_{g,n}$ formed by smooth marked curves. The union of all natural maps $U_c/\Gamma_c\rightarrow \overline{\mathcal{M}}_{g,n}$ induce a system of (orbifold) charts with holomorphic transition maps on $\overline{\mathcal{M}}_{g,n}$. 
 A a more precise local description of the quotient map $U_{c}\rightarrow U_{c}/\Gamma_{c}$ in the neighbourhood of a point of $B_{c}$ up to local homeomorphisms on source and target that can be found in  \cite[pp. 485-493]{ACG}. We recall the model here because it will be useful to determine the local properties of the isoperiodic deformations of stable forms. 
 
 Consider the set $X=\frac{\overline{\mathbb{R}}_{+}\times\mathbb{R}}{0\times\mathbb{R}}$ endowed with the (non-locally compact, first countable, Hausdorff) topology whose basis of open sets around $[0,0]$  are the sets of the form $\{[\rho,\theta]: \rho<\varepsilon\}$ with $\varepsilon>0$ and the usual round discs around any other point.  Let $n\geq 2$ and $0\leq l\leq n$ and the the normal crossing divisor $$D_l=\{(z_1,\ldots,z_n)\in\mathbb{C}^{n}: z_{1}\cdots z_{l}=0\}. $$ Consider the map with infinite ramification on $D_{l}$ \begin{equation}\label{eq:bldown1}\varphi_{l,n}:X^{l} \times \mathbb{C}^{n-l}\rightarrow\mathbb{C}^{n}\text{ defined by }\end{equation}  $(\rho_{1},\theta_{1},\ldots, \rho_{l},\theta_{l},z_{l+1},\ldots, z_{n})\mapsto (\rho_{1}e^{2i\pi\theta_{1}},\ldots, \rho_{l}e^{2i\pi\theta_{l}}, z_{l+1},\ldots, z_{n})$.
It is a continuous map with respect to the product topologies on source and target. On the complement of the  preimage of $D_{l}$  the map $\varphi_{l,n}$ is a topological cover with free abelian covering group-- that corresponds to the action of $\mathbb{Z}^{l}$ on the $\theta$ variables by translations, componentwise. This implies that it is a connected complex manifold. Given a subset $I\subset\{1,\ldots,l\}$ with $k$ elements, the restriction of $\varphi_{l,n}$ to the set $\{\rho_{j}=0: j\in I\}$ produces map that is equivalent to some other $\varphi_{l',n'}$, so the set $\{\rho_{j}=0: j\in I\}\cap\{\rho_{j}\neq 0:j\notin I \}$ is a connected complex manifold of codimension $k$ accumulating the origin. They fit together to form a stratification of $X^{l} \times \mathbb{C}^{n-l}$. The action of $\mathbb{Z}^{l}$  leaves \textit{each connected component} of a stratum  invariant. 
\begin{remark}\label{rem: normal crossing type}
Each connected component of a local stratum of positive codimension close to the origin is characterized by the set of local components of codimension one that accumulate to it. This property persists under quotients of  $X^{l} \times \mathbb{C}^{n-l}$ by subgroups $H\subset \mathbb{Z}^{l}$. 
\end{remark}

Coming back to the local descriptions of the map $U_{c}\rightarrow U_{c}/\Gamma_{c}$, around a point of the curve system $c$ with $l$ elements, it is locally equivalent to  $\varphi_{l,3g-3}$ at the origin. 

\begin{definition}
A stratified space is a local abelian ramified cover of a normal crossing divisor if the stratification is locally homeomorphic  to a quotient stratification of $X^{l} \times \mathbb{C}^{n-l}$ by some subgroup $H\subset \mathbb{Z}^{l}$ for some $n$ and $l\leq n$ at the origin.
\end{definition}

\begin{definition} 
To any stratified space $X$ that is a local abelian ramified cover of a normal crossing divisor on some fixed dimension $n\geq 2$ we can associate a boundary complex $\mathcal{C}(X)$ having a vertex for each connected component of the codimension one stratum and a simplex joining $k$ vertices for each connected component of the stratum of codimension $k$ contained in the closure of the corresponding components. 
\end{definition}

By the description of connected component of strata by classes of curve systems, we get an isomorphism of this dual boundary complex with the curve complex $\mathcal{C}(\overline{\mathcal{T}}_{g,n})\simeq\curvecomplex_{g,n}$ (see subsection \ref{ss: strategy}) for the definition . The group $\text{Mod}(\Sigma_{g,n})$ acts on $\mathcal{T}_{g,n}$ ( resp. $\mathcal{C}_{g,n}$) preserving strata (resp. the simplicial structure).

\subsection{Augmented Torelli space: a ramified covering of $\overline{\mathcal{M}}_{g,n}$.} The map induced in homology by an element of $\text{Mod}(\Sigma_{g,n})$ provides an exact sequence $$0\rightarrow\mathcal{I}_{g,n}\rightarrow\text{Mod}(\Sigma_{g,n})\rightarrow \text{Aut}(H_1(\Sigma_g,q_1,\ldots, q_n;\mathbb{Z}))$$
where $\mathcal{I}_{g,n}$ is the Torelli group of $\Sigma_{g,n}$ formed by isotopy classes of diffeomorphisms that act trivially on relative homology. The Augmented Torelli space is the topological quotient  $$\overline{\mathcal{S}}_{g,n}=\overline{\mathcal{T}}_{g,n}/\mathcal{I}_{g,n}.$$ 

Again, $\overline{\mathcal{S}}_{g,n}$ contains the Torelli space $\mathcal{S}_{g,n}=\mathcal{T}_{g,n}/\mathcal{I}_{g,n}$ as a proper set. Its boundary is the set $\partial \mrs_{g,n}=\mnc_{g,n}\setminus \mrs_{g,n}$. Two strata $B_{c}$ and $B_{c'}$ lie on the same class if and only if there exists an element in the Torelli group that sends the curve system $c$ to the curve system $c'$. Therefore, the boundary inherits a stratification 
\cbstart 
\begin{definition}
The equivalence class under the action of the Torelli group of a curve system $c$  in $\Sigma_{g,n}$ will be denoted by $\curvesystem$.  
\end{definition}
The boundary inherits a partition induced by the stratification
$$\partial \mrs_{g,n}=\bigsqcup_{\curvesystem\neq\emptyset}B_{\curvesystem}$$ 
where $\curvesystem$ runs over all equivalence classes of non-empty curve systems $c$ under the action of the Torelli group $\mathcal{I}_{g,n}$.

The complex structure in restriction to the stratum $B_c$ induces a complex structure on $B_{\curvesystem}$. However, the action of the Dehn twist around a  simple closed curve of non-trivial class in $H_1(\Sigma_{g,n})$ can act non-trivialy on $H_1(\Sigma_{g,n})$. For instance, in the case where $n=0$ and the simple closed curve is non-separating of class $a\in H_1(\Sigma_g)\setminus 0$ the action is  on homology is the non-trivial isomorphism $\delta_a:H_1(\Sigma_g,\mathbb{Z})\rightarrow H_1(\Sigma_g,\mathbb{Z})$ given by 
$$\delta_{a}(b)=b+(a\cdot b)a.$$

The topology of $\mnc_{g,n}$ is still non-locally compact around such points. Therefore, $\mnc_{g,n}$ does not admit a compatible manifold structure.

\begin{lemma}{\emph{(Local structure  around a stratum $B_{\curvesystem}$ and boundary complex of $\mnc_{g,n}$)}} \label{l:mncstratification}
The stratification of $\overline{\mathcal{T}}_{g,n}$ induces a stratification of $\mnc_{g,n}$ that is a local abelian ramified covering of a normal crossing divisor. The dual boundary complex of $\mnc_{g,n}$ is $$\mathcal{C}(\mnc_{g,n})\simeq\mathcal{C}_{g,n}/\mathcal{I}_{g,n}.$$ A neighbourhood of a stratum $B_{\curvesystem}$ in $\mnc_{g,n}$ admits a (compatible complex) manifold structure if and only if $\Gamma_c\subset\mathcal{I}_{g,n}$, i.e. the Dehn twist around any simple closed curve of $\curvesystem$ acts trivially on $H_1(\Sigma_{g,n})$.
\end{lemma}

\begin{proof}
  The stratification of the open set $U_{c}\subset\overline{\mathcal{T}}_{g,n}$ is invariant under the action of the subgroup $\Gamma_{c}\cap\mathcal{I}_{g,n}\subset \Gamma_{c}$ and a local abelian ramified cover of a normal crossing divisor. Therefore, the quotient  on $U_{c}/(\Gamma_{c}\cap\mathcal{I}_{g,n})$ has the same local property. On the other hand, as a consequence of the fact that any automorphism of a stable curve that acts trivially on homology is equivalent to some Dehn twist around the pinched curves, $U_{c}/(\Gamma_{c}\cap\mathcal{I}_{g,n})\rightarrow U_{c}/\mathcal{I}_{g,n}\subset \mnc_{g,n}$ is a local homeomorphism at every point of the stratum $B_{c}$. This proves all statements about the local structure around a stratum. 
  
 A connected component of a stratum of $\mnc_{g,n}$ corresponds to one orbit of connected components of a stratum of $\overline{\mathcal{T}}_{g,n}$ under the action of $\mathcal{I}_{g,n}$. The result follows from the isomorphism $\mathcal{C}(\overline{\mathcal{T}}_{g,n})\simeq \mathcal{C}_{g,n}$. 

\end{proof}
\cbend
Given a simple closed curve $c_1$, there might be distinct strata in the boundary component $D_{c_1}$ of $\overline{\mathcal{T}}_{g,n}$ that are identified in the quotient. For instance, suppose $c_2$ is a non-separating curve determining the same homology class as $c_1$ but distinct homotopy classes in $\Sigma_{g,0}$. 
By \cite{BF}[Section 1.3] there exists an element $\phi$ in the Torelli group $\mathcal{I}_{g,0}$ sending $c_2$ to $c_1$. The  
two curve systems $c=c_1\sqcup c_2$ and $c'=\phi(c)=c_1\sqcup \phi(c_1)$ are equivalent and define two distinct strata contained in $D_{c_1}$ provided $\phi(c_1)$ and $c_2$ define distinct isotopy classes in $\Sigma_{g,0}$. In terms of the dual boundary complex this means that there are edges that join the same vertex. 

\subsection{Extension of the Torelli map $\mrs_{g}\rightarrow\mathfrak{S}_{g}$ }\label{ss:extension of complex structure} The holomorphic orbifold structure of $\overline{\mathcal{M}}_{g}$ can be pulled back to a complex manifold structure in the open dense subset $\mct_{g}$ of marked curves of compact type in $\mnc_{g}$. Indeed, all the Dehn twists around the curves that are collapsed to the nodes lie already in the Torelli group and the ramification disappears. 

The  definition of the coordinate functions via \eqref {eq:period matrix coordinate} can be done word by  word for stable curves of compact type. In this manner we define an extension of \eqref{eq:torelli map} to a map $\mct_g\rightarrow \mathfrak{S}_g$ that is still holomorphic.  Its image is the so-called Schottky locus, an analytic set that has positive codimension as soon as $g\geq 4$. For $g\geq 3$ the extended Torelli map has fibers of positive dimension over boundary points. 

\begin{remark}\label{rem:extended torelli map}
In genus $g=2$ the added points $\mnc
^c_2\setminus \mrs_2$ correspond to products of two marked elliptic curves. The extension of the Torelli map \eqref{eq:torelli map} in this case is a biholomorphism (see \cite{Mess}). 
\end{remark}

\subsection{Homological markings} 
\label{ss:homological invariants}

Given  $[f:\Sigma_{g,n}\rightarrow (C,r_1,\ldots,r_n)]\in\mnc_{g,n}$ we can associate a surjective homomorphism $$f_*:H_1(\Sigma_g,q_1,\ldots,q_n;\mathbb{Z})\rightarrow H_1(C,r_1,\ldots,r_n;\mathbb{Z}).$$
When $\ker f_*=0$, any element $\phi\in\text{Mod}(\Sigma_{g,n})$ such that $(f\circ\phi)_*=f_*$ satisfies $\phi_*=\text{Id}$. Therefore $f_*$ determines the Torelli class of the marking $f$. In particular this is the case when each node of $C$ is separating. When there is at most one node that is non-separating we have a similar result:
\begin{lemma}
 Let $f:\Sigma_g\to C$ be a homotopical marking that pinches a simple closed,
non-separating curve $\gamma$ and some separating simple closed 
curves. Let $\phi\in \text{Mod}(\Sigma_{g,0})$ such that $(f\circ
\phi)_*=f_*:H_1(\Sigma_g;\mathbb Z)\to H_1(C,\mathbb Z)$. Then, up to Dehn twists along
$\gamma$, $\phi_*=id$, that is to say, it belongs to the Torelli group of $\Sigma_g$. 
\end{lemma}

\begin{proof} Set $a=[\gamma]$, then $\ker f_*=\mathbb{Z}a$. Fix $a=a_1,b_1,a_2,b_2\dots a_g,b_g$ a symplectic basis of $ H_1(\Sigma_g)$. Since $f_*\circ \phi_*=f_*$ we have 
$f_*(\phi_*-id)=0$. Therefore $\psi=\phi_*-id$ is a morphism from $H_1(\Sigma_g,\mathbb Z)$ to
$\ker f_*=\mathbb{Z}a$. So  we have 
$\psi(b_i)=\lambda_i a_1,  \psi(a_i)=\mu_ia_1$ for some $\lambda_i,\mu_i\in\mathbb{Z}$. From the fact
that $\phi$ is symplectic, we deduce $\mu_i=0$ for all $i$ and $\lambda_i=0$ for all $i\neq
1$. ($\delta_{i1}=a_i\cdot b_1=\phi(a_i)\cdot\phi(b_1)=\delta_{i1}+\mu_i$ and $0=b_i\cdot
b_1=\phi(b_i)\cdot \phi(b_1)=\lambda_i$). 
Up to Dehn twist along $\gamma$ we may
assume $\lambda_1=0$, so $\psi=0$ and $\phi_*=id$, that is to say $\phi\in\mathcal{I}_{g,0}$ is a Torelli mapping
class.
\end{proof}

An equivalent statement fails in general for curves $C$ with more than two non-separating nodes. Indeed, suppose $f$ is a marking that collapses two elements $a_1$ and $a_2$ of a symplectic basis $a_1,b_1,a_2,b_2,a_3,b_3,\dots a_g,b_g$ of $H_1(\Sigma_g)$. Define a morphism $\psi:H_1(\Sigma_g)\rightarrow \ker f_*$ by  $\psi(b_1)=a_2$ and $\psi(b_2)=a_1$ and
$\psi(a_i)=\psi(b_j)=0$ for all $i\geq 1$ and for all $j>2$. Then $Id+\psi\in
Sp(2g,\mathbb Z)$ and it is induced by a mapping class $\phi\in\text{Mod}(\Sigma_g)$. By construction $f_*=f_*\circ\phi_*$
but $\phi$ is not in the Torelli group, nor is a product of twists along loops whose class
belong to $\ker f_*=\mathbb{Z}a_1\oplus\mathbb{Z}a_2$ (because $\delta_{na_1+ma_2}(b_1)=b_1+na_1$ so all such twists leave
$<a_1,b_1>$ invariant, and $\phi$ does not).

\begin{definition}
A homological marking of a stable curve with $n$ marked points $(C,r_1,\ldots,r_n)$ is a surjective homomorphism $$m:H_1(\Sigma_g,q_1,\ldots,q_n;\mathbb{Z})\rightarrow H_1(C,r_1,\ldots,r_n;\mathbb{Z}).$$
The data $(C,r_1,\ldots,r_n,m)$ will be referred to as a homologically marked curve. 
\end{definition}
\subsection{Parametrization of the vertices of the dual boundary complex $\curvecomplex_{g}/\mathcal{I}_{g}$ of $\mnc_{g}$}\label{s:parametrization of torelli classes} 
With the use of homological markings we will provide a homological characterization of the Torelli classes of simple closed curves on $\Sigma_{g}$. This will allow us to give a special parametrization of the vertices of the dual boundary complex of $\mnc_{g}$.

 First we recall some concepts of homology theory. The intersection of two elements $a,b\in H_1(\Sigma_g,\mathbb Z)$ is denoted by $a\cdot b$ and the associated intersection form defines an integral unimodular symplectic structure on $H_1(\Sigma_g):=H_1(\Sigma_g,\mathbb Z)$.
Two submodules of a unimodular symplectic module $M$ are said to be orthogonal if the intersection of any element of one of them with an element of the other is zero. Given a submodule $N\subset M$ we denote by $N^{\perp}\subset M$ the orthogonal submodule of $N$ in $M$, that is, the set of elements in $M$ that have zero intersection with all elements of $N$.

\begin{definition}
A submodule $N$ of a unimodular symplectic module $M$ is said to be symplectic if the symplectic form of $M$ restricted to $N$ is still unimodular. A splitting $M=N_1\oplus\cdots\oplus N_k$ into pairwise orthogonal submodules is said to be symplectic if every $N_i$ is a symplectic submodule. 
\end{definition}

\begin{definition}
 A submodule $N$ of a $\mathbb{Z}$-module $M$ is primitive if whenever $zm\in M$ for some $m \in M$ and $z\in\mathbb{Z}$ then also $m\in N$.
\end{definition}
Let $c$ be a simple closed curve in $\Sigma_g$. If it is non-separating, it determines a primitive class $[c]\in H_1(\Sigma_g)$. If it is separating, the induced homology class is trivial and carries no interesting information. However, $c$ splits $\Sigma_g$ in two parts, and these induce a symplectic splitting of $H_1(\Sigma_g)$ that characterizes the Torelli class of the curve in the following sense: 

\begin{proposition}[\cite{FM}, Proposition 6.14]
Let $c,c'$ be two isotopy classes of simple closed curves in $\Sigma_g$. If $c$ and $c'$ are separating, then they are equivalent by some element of the Torelli group $\mathcal{I}_{g,0}$ if and only if the associated symplectic splittings coincide up to changing the order of the factors. If $c$ and $c'$ are non-separating, then they are equivalent by some element of the Torelli group $\mathcal{I}_{g,0}$ if and only if, up to sign, they determine the same homology class in $H_1(\Sigma_g)$. 
\end{proposition}
We can use this information to parametrize vertices of the dual boundary complex $\curvecomplex_{g}/\mathcal{I}_{g}$ of the augmented Torelli space  $\mnc_g$.
\begin{corollary}\label{c:parametrizing boundary}
The set of unordered pairs $\{V,V^{\perp}\}$ of non-trivial symplectic submodules of $H_1(\Sigma_g)$ such that $V\oplus V^{\perp}=H_1(\Sigma_g)$ is in 1-1 correspondence with Torelli classes of separating essential simple closed curves in $\Sigma_{g}$. 
The set of non-trivial primitive submodules $\mathbb{Z}[c] \subset H_1(\Sigma_g)$ is in 1-1 correpondence with the set of Torelli classes of  non-separating simple closed curves.

\end{corollary}
\subsection{Identifying some simplexes of $\mathcal{C}_g/\mathcal{I}_g$}

A separating node on a homologically marked stable curve $(C,m)$ induces a non-trivial symplectic splitting $V\oplus V^{\perp}$ of $H_1(\Sigma_g)$. A non-separating node induces a cyclic non-trivial primitive submodule $\mathbb{Z}a$ of $\ker m\subset H_1(\Sigma_g)$.

\begin{definition}
Two marked nodal curves $(C_i,m_i)$ for $i=1,2$ of genus $g$ are said to share a separating node if for each $i=1,2$ there exists a separating node $r_i\in C_i$ that induces the same splitting $V\oplus V^{\perp}$ of $H_1(\Sigma_g)$ up to the order of the factors. Equivalently, we say that they share a non-separating node if there exist non-separating nodes $r_i\in C_i$ whose associated cyclic primitive submodules of $H_1(\Sigma_g)$ coincide.
\end{definition}

Two marked stable curves that share a node define points that lie in the closure of the same bondary stratum of codimension one. In terms of the complex $\mathcal{C}_g/\mathcal{I}_g$ they represent simplexes that intersect at one vertex. 

This characterization allows to identify certain simplexes of $\mathcal{C}_g/\mathcal{I}_g$ by looking at the homological information of the vertices. For instance, given a splitting $V_1\oplus\ldots\oplus  V_{k+1}$ of $H_1(\Sigma_g)$ into symplectic submodules of rank at least two,  the vertices associated to $\{V_i,V_i^{\perp}\}$  $i=1\ldots k$ belong to a simplex of $\mathcal{C}_g/\mathcal{I}_g$. It suffices to construct a marked stable curve with $k$ separating  nodes inducing the given splitting. If in each $V_i$ we are given a primitive class of $H_1(\Sigma_g)$ we can construct a stable curve with $2k$ nodes, that defines a simplex joining the corresponding $2k$ vertices, etc. These simplexes will be very useful to analyze the topology of the boundary stratification of $\mnc_g$.

\subsection{Attaching and forgetful maps on augmented Torelli spaces}\label{ss:attach&forget on augmented Torelli}
It is well known that attaching maps, i.e. identification of distinct marked points, define holomorphic maps at the level of moduli spaces. The same is true for forgetful maps, i.e. forgetting part of the marked points followed by stabilization. In this subsection we will see how these maps can be defined at the level of augmented Teichm\" uller and Torelli spaces.

Let $\phi:\Sigma_{g,n}\rightarrow\Sigma_{g,n}/c$ be the collapse of a simple closed curve $c$ in $\Sigma_{g,n}\setminus\{q_1,\ldots,q_n\}$. We will define an attaching map associated to $\phi$.

If $c$ is a separating curve, the surface $\Sigma_{g,n}\setminus c$ has two components, each with a boundary component. Collapsing each boundary component to a new marked point produces two surfaces with marked points $\Sigma_{g_1,n_1}$ and $\Sigma_{g_2,n_2}$ where $g_1+g_2=g$,  $n_1+n_2-2=n$ and the new points are named $q_{n_1}$ and $q_{n_2}$ respectively. The attaching map $$A_{\phi}:\overline{\mathcal{T}}_{g_1,n_1}\times\overline{\mathcal{T}}_{g_2,n_2}\rightarrow \overline{\mathcal{T}}_{g,n}$$ is defined by the isotopy class of  the composition $(f_1\vee f_2)\circ \phi$ where $f_i:\Sigma_{g_i,n_i}\rightarrow (C_i,r^1_{1},\ldots,r^1_{n_i})$ is a representative of its class in $\overline{\mathcal{T}}_{g_i,n_i}$ and $f_1\vee f_2$ denotes the  map $\Sigma_{g,n}/c\rightarrow (C_1\vee C_2, q_1,\ldots, q_n)$ by collapsing $c$ to the new node $q_{n_1}=q_{n_2}$ of $C_1\vee C_2$ and applying the corresponding $f_i$ in each part of the complement of the node.
\begin{remark}
The isomorphism \eqref{eq:curvesystemdecomposition} is realized by attaching maps
\end{remark}
\begin{lemma}\label{l:attaching marked curves}
 The map $A_{\phi}$ induces a well defined continuous attaching map $$\mnc_{g_1,n_1}\times\mnc_{g_2,n_2}\rightarrow \mnc_{g,n}.$$
\end{lemma}
\begin{proof}
Remark that $\phi$ allows to define a  map $$\rho_{\phi}: \mathcal{I}_{g_1,n_1}\times \mathcal{I}_{g_2,n_2}\rightarrow \mathcal{I}_{g,n}$$ between Torelli groups. Indeed, every pair $(\psi_1,\psi_2)\in \mathcal{I}_{g_1,n_1}\times\mathcal{I}_{g_2,n_2}$ defines an automorphism $\psi_{c}$ of $\Sigma_{g,n}/c$ that acts trivially on its relative homology (with $n$ marked points). We need to lift the element to an element $\psi\in\mathcal{I}_{g,n}$. The lift via $\phi$ is a well defined map $\Sigma_{g,n}\setminus c\rightarrow \Sigma_{g,n}\setminus c$. Up to changing $\psi_{i}'s$ in the neighbourhood of the points $q_{n_1}$ and $q_{n_2}$ homotopically, we can guarantee that there exists an extension $\psi=\rho (\psi_1,\psi_2)$ of the lift to $\Sigma_{g,n}$ that fixes  $c$ pointwise. 

By construction the map $A_{\phi}$ is equivariant with respect to $\rho$ and therefore induces the desired map $\mnc_{g_1,n_1}\times\mnc_{g_2,n_2}\rightarrow \mnc_{g,n}$. 

\end{proof}

If $c$ is non-separating, then $\Sigma_{g,n}\setminus c$ is connected and has two boundary curves. Collapsing each boundary curve to a new marked point produces a surface of genus $g-1$ with $n+2$ marked points that we denote $\Sigma_{g-1,n+2}$. The new points will be denoted $q_{n+1}, q_{n+2}$. Remark that $\Sigma_{g,n}/c$ can be identified with $\Sigma_{g-1,n+2}/q_{n+1}\sim q_{n+2}$. The attaching map is  $$A_{\phi}:\overline{\mathcal{T}}_{g-1,n+2}\rightarrow \overline{\mathcal{T}}_{g,n}$$ defined by $(\vee f)\circ \phi$ where $\vee f:\Sigma_{g,n}/c\rightarrow (\vee C,r_1,\ldots,r_n)$ is the map induced by a marking  $f:\Sigma_{g-1,n+2}\rightarrow (C,r_1,\ldots, r_{n+1},r_{n+2})$
on the stable curve $\vee C$ (that has an extra non-separating node, compared to $C$) obtained by identifying $r_{n+1}$ and $r_{n+2}$ in $C$. 
\begin{lemma}\label{l:attaching at two points}
 The map $A_{\phi}$ induces a continuous attaching map $$\mnc_{g-1,n+2}\rightarrow \mnc_{g,n}.$$
\end{lemma}
\begin{proof}
As before $\phi$ induces a homomorphism between the relevant Torelli groups $$\rho_{\phi}:\mathcal{I}_{g-1,n+2}\rightarrow \mathcal{I}_{g,n}$$ and $A_{\phi}$ is $\rho_{\phi}$ equivariant by construction.  

To construct $\rho_{\phi}$, suppose $\psi:\Sigma_{g-1,n+2}\rightarrow \Sigma_{g-1,n+2}$ represents an element in $\mathcal{I}_{g-1,n+2}$. Then $\psi$ induces an automorphism of $\Sigma_{g-1,n+2}/q_{n+1}\sim q_{n+2}$. The equivalence between this space and $\Sigma_{g,n}/c$ provides an automorphism of $\Sigma_{g,n}/c$ that can be lifted to $\Sigma_{g,n}\setminus c$ via $\phi$. Up to changing the initial $\psi$ homotopically in a small neighbourhood of the points $q_{n+1}$ and $q_{n+2}$ we can suppose that the lift extends to an automorphism $\rho_{\phi}(\psi):\Sigma_{g,n}\rightarrow \Sigma_{g,n}$  that fixes $c$ pointwise. The action of $\rho(\psi)$ in $H_1(\Sigma_{g},\{ q_1,\ldots,q_n\}; \mathbb{Z})$ is trivial by construction and therefore $\rho_{\phi}(\psi)\in \mathcal{I}_{,g,n}$.  
\end{proof}

As for forgetful maps, given an inclusion $\phi:\Sigma_{g,n}\rightarrow \Sigma_{g,n+1}$ of the marked points,  we can define its associated  forgetful map \begin{equation}\label{eq:forgetful_map}
    \overline{\mathcal{T}}_{g,n+1}\rightarrow \overline{\mathcal{T}}_{g,n}
\end{equation} sending the class of $f:\Sigma_{g,n+1}\rightarrow (C,r_1,\ldots, r_{n+1})$ to the stabilization of $f\circ \phi$, i.e. if the component of $C$ that contains the forgotten point has an infinite group of automorphisms after deleting the point, we collapse the component of source and target to a point.  Again we have 
\begin{lemma}\label{l:forget a point}
 The map $\phi$ induces a continuous  forgetful map $\mnc_{g,n+1}\rightarrow \mnc_{g,n}$.
\end{lemma}
\begin{proof}
In this case, $\phi$ induces an inclusion $\rho_{\phi}:\mathcal{I}_{g,n+1}\rightarrow \mathcal{I}_{g,n}$ since any diffeomorphism that fixes $n+1$ points fixes any of the $n$ points that correspond to the image by  $\phi$ of the $n$ marked points of $\Sigma_{g,n}$.  The map (\ref{eq:forgetful_map}) is $\rho_{\phi}$-equivariant by construction. 
\end{proof}
\cbstart
\subsection{A useful bordification $\U_{g,2}$ of $\mrs_{g,2}$ in $\mnc_{g,2}$}\label{s:bordification of Sg2}
In this subsection we introduce a smooth partial bordification of $\mrs_{g,2}$ that will  be useful in Section \ref{s:period fibers with marked points}. We also describe its boundary complex. 

Consider the composition of two forgetful maps \begin{equation}
    \text{For}_{g,2}:\mnc_{g,2}\rightarrow \mnc_{g,0}=\mnc_g
\end{equation}
and the open set $$\U_{g,2}=\text{For}_{g,2}^{-1}(\mrs_{g})$$
of marked curves whose stabilization after forgetting both marked points is a smooth curve. 
The restriction of $\text{For}_{g,2}$ to $\U_{g,2}$ will be denoted by $\text{For}$.

Let $\mathfrak{C}(\mrs_g)\rightarrow\mrs_g$ denote the universal curve bundle over $\mrs_g$ and $\mathfrak{C}^2(\mrs_{g})\rightarrow \mrs_g$ denote the pull back of the square of the universal curve bundle $\mathfrak{C}(\mrs_g)\times\mathfrak{C}(\mrs_g)\rightarrow \mrs_g\times\mrs_g$ by the diagonal map $\mrs_g\rightarrow\mrs_g\times\mrs_g$.

\begin{proposition} \label{p:bordification of Sg2} 
There exists a covering map $R:\U_{g,2}\rightarrow\mathfrak{C}^2(\mrs_g)$ such that 
\begin{equation}
    \begin{tikzcd}
\U_{g,2} \arrow[r,"R"] \arrow[dr,"For"]
& \mathfrak{C}^2(\mrs_g) \arrow[d]\\
& \mrs_g
\end{tikzcd}
\end{equation}
commutes.
For each $(C,m)\in\mrs_{g}$,   $R_{|\text{For}^{-1}(C,m)}:\text{For}^{-1}(C,m)\rightarrow C\times C$ is an $H_1(\Sigma_g)$-cover with monodromy given by  
\begin{equation}\label{eq: abelian monodromy}  (\alpha_1, \alpha_2) \cdot \gamma = \gamma + \alpha_2 - \alpha_1,\quad for\quad  (\alpha_1,\alpha_2)\in H_1(\Sigma_g)^2\simeq H_1(C\times C), \gamma\in H_1(\Sigma_g).\end{equation} 
Moreover, the boundary strata in the complex manifold $\U_{g,2}$ form a smooth divisor whose components are in one to one correspondence with the set  $$\betas=\{\beta\in H_1(\Sigma_g,q_1,q_2;\mathbb{Z}): \partial\beta=q_2-q_1\}.$$
\end{proposition}

\begin{proof}[Proof of Proposition \ref{p:bordification of Sg2}:]

We split the proof in several parts:

\vspace{0.3cm} 
\textit{Action of \(\text{Mod}_{g,2}\) on relative homology.}

Let \(\text{Aut} (H_1 (\Sigma_g, \{q_1,q_2\}, \mathbb Z))\) be the subgroup of linear automorphisms of \(H_1 (\Sigma_g, \{q_1,q_2\}, \mathbb Z)\) that preserve the exact sequence 
\begin{equation} \label{eq: homology exact sequence}  0\rightarrow H_1 (\Sigma_g, \mathbb Z) \rightarrow H_1 (\Sigma_g, \{q_1,q_2\}, \mathbb Z) \stackrel {\partial}{\rightarrow} \mathbb Z (q_2-q_1) \rightarrow 0\end{equation} 
where \(\partial \) is the boundary operator. The group \( \text{Aut} (H_1 (\Sigma_g, \{q_1,q_2\}, \mathbb Z))\) splits into an exact sequence 
\[ 0\rightarrow G \rightarrow \text{Aut} (H_1 (\Sigma_g, \{q_1,q_2\}, \mathbb Z)) \rightarrow \text{Aut} (H_1(\Sigma_g, \mathbb Z) ) \rightarrow 0 \] 
where \(\text{Aut} (H_1(\Sigma_g, \mathbb Z)) \) is the group of linear automorphisms of \(H_1(\Sigma_g, \mathbb Z)\) that preserve the intersection form, and where \(G\) is the abelian unipotent subgroup of \(\text{Aut} (H_1 (\Sigma_g, \{q_1,q_2\}, \mathbb Z))\) formed by the transformations 
\[ G_\alpha ( \gamma ) := \gamma + n(\gamma) \alpha \] 
where \(\alpha \in H_1 (\Sigma_g, \mathbb Z) \), and where we write the boundary operator  as \( \partial = n (q_2- q_1)\). Notice that \(G\) naturally identifies with the absolute homology group \(H_1(\Sigma_g, \mathbb Z)\).

Since the action of \(\text{Mod}_{g,2}\) on \( H_1 (\Sigma_g, \{p,q\}, \mathbb Z)\) preserves \eqref{eq: homology exact sequence}, it gives rise to morphism 
\begin{equation} \label{eq: action on relative homology}\text{Mod} _{g,2} \rightarrow \text{Aut} (H_1 (\Sigma_g, \{q_1,q_2\}, \mathbb Z)). \end{equation}  

\vspace{0.3cm} 
\textit{Action of the braid group.} 

Consider the exact sequence 
\begin{equation} \label{eq: exact sequence brraid group} 0\rightarrow B(\Sigma_{g,2}) \rightarrow \text{Mod} (\Sigma_{g,2}) \rightarrow \text{Mod} (\Sigma_g) \rightarrow 0 \end{equation}
where  \(B(\Sigma_{g,2})\) is the braid group with two strings on \(\Sigma_g\), and where the right arrow is given by the forgetful map that forgets the fact that the mapping class fixes the two marked points. 
By Birman's theorem, the braid group is isomorphic to the fundamental group of the square of \(\Sigma_g \) deprived of its diagonal \(\Delta\).\footnote{Usually what is refered to as Birman's exact sequence is for a surface punctured at only one point, but the idea of the proof is the same in the case of two points: given an element of the braid group \([\varphi]\in B(\Sigma_{g,2})\), represented by an isotopy class of diffeomorphism \(\varphi\in \text{Diff} (\Sigma_g, q_1,q_2)\), there exists an isotopy \((\varphi_t)_{t\in [0,1]}\) in \(\text{Diff}(\Sigma_g)\) such that \(\varphi_0=\varphi \) and \(\varphi_1 = Id\). The loop \((\varphi_t (q_1),\varphi_t (q_2))\in \Sigma_g^2 \setminus \Delta\) defines an element \(\gamma\in \pi_1 (\Sigma_g^2 \setminus \Delta, (q_1,q_2))\) that only depends on \([\varphi]\) and the map \( [\varphi] \mapsto \gamma\) induces an isomorphism between \(B(\Sigma_{g,2})\) and \(\pi_1 (\Sigma_g^2 \setminus \Delta, (q_1,q_2))\).  } 

Its action on relative homology group (obtained by restriction of the action \eqref{eq: action on relative homology}) takes values in the abelian unipotent group \(G\)
, hence its descends to an action defined on the abelianization of \( B_{g,2}\), which is isomorphic to  \( H_1 (\Sigma_g^2 \setminus \Delta, \mathbb Z) \simeq H_1 (\Sigma_ g ^2 , \mathbb Z) \simeq H_1 (\Sigma_g)^2 \); one checks that under these identifications it is given by the morphism 
\begin{equation} \label{eq: braid group action on relative homology}  (\alpha_1, \alpha_2 ) \in H_1 (\Sigma_g)^2 \mapsto G_{\alpha_2- \alpha_1} \in \text{Aut} (H_1 (\Sigma_{g}, \{q_1, q_2\}, \mathbb Z) ) . \end{equation}

\textit{A first \(G\)-covering of $\mrs_{g,0}$.} 

The group \(G\) acts on \(\mrs_{g,2}\) by precomposition of the marking, and the action is free and proper (recall that the action of the group \( \text{Mod}_{g,2} / \mathcal I _{g,2} \simeq \text{Aut} (H_1(\Sigma_g,\{q_1,q_2\}, \mathbb Z))\) on \(\mrs_{g,2}\) in free and discontinuous, see subsection \ref{ss:Torelli and period map}). Denote by \(R'\) the covering \( \mrs_{g,2} \rightarrow \mrs_{g,2}/G \). The fundamental group of \( \mrs_{g,2}\) is isomorphic to \(\mathcal I_{g,2}\), the one of \( G\backslash  \mrs_{g,2}\) is the group generated by the Torelli group \(\mathcal I_{g,2}\) and the braid group \( B(\Sigma_{g,2})\), and the monodromy of \( R'\) is given by the quotient map 
\[ <B(\Sigma_{g,2}) , \mathcal I_{g,2} > \rightarrow <B(\Sigma_{g,2}) , \mathcal I_{g,2} >/ \mathcal I_{g,2} \simeq G \subset \text{Aut} (H_1(\Sigma_g,\{q_1,q_2\}, \mathbb Z))\simeq \text{Mod}_{g,2} / \mathcal I _{g,2}\]

The \(G\)-action preserves the restricted forgetful map \(\Forget_{|\mrs_{g,2}} :\mrs_{g,2} \rightarrow \mrs_{g}\), so there is a map \( \widetilde{\Forget} : \mrs_{g,2}/G\rightarrow \mrs_{g} \), we have \( \Forget_{|\mrs_{g,2}} = \widetilde{\Forget} \circ R'\).  The map \(\widetilde{\Forget}\) is a holomorphic fibration, whose fibers \( \widetilde{\Forget} ^{-1} (C, m)\) are biholomorphic to the square of \(C\) deprived of its diagonal, namely \( \widetilde{\Forget} ^{-1} (C, m)\simeq C\times C\setminus \Delta\) where \(\Delta:=\{(x,x)\ |\ x\in C\}\). The fibration at the level of the fundamental group gives the exact sequence  
\[ 0\rightarrow B (\Sigma_{g,2}  ) \rightarrow <B(\Sigma_{g,2}) , \mathcal I_{g,2} > \rightarrow \mathcal I_g \rightarrow 0 .\] 
Hence the fibers of \(\widetilde{\Forget}\) are homologically marked by \(H_1 (\Sigma_g, \mathbb Z)^2\)) and the monodromy of the \(G\)-covering \( R' \) in restriction to a fiber of \( \widetilde{\Forget} \) is given by \eqref{eq: braid group action on relative homology}.


\vspace{0.3cm} 

\textit{Construction of \(R\) as a bordification of the covering \(R'\).}

We now bordify the previous covering \(R'\) over \(\mrs_{g,2}/G\). The idea is to add in each fiber \(\widetilde{\Forget}^{-1} (C,m)\simeq C\times C \setminus \Delta\) the diagonal \(\Delta\). 

As before, the action of \( G\simeq H_1(\Sigma_g, \mathbb Z)\) on \( \mnc_{g,2}\) preserves each fiber of the forgetful map \(\Forget_{g,2} :\mnc_{g,2}\rightarrow \mnc_{g} \), hence it preserves the open subset \(\U_{g,2} := \Forget ^{-1}_{g,2} (\mrs_g) \subset \mnc _{g,2}\). The action is free on this subset. Indeed, recall that the stabilizer of a point \((C,r_1,r_2,m)\in \mnc_{g,2}\) in \(\text{Aut} (H_1 (\Sigma_g, \{q_1,q_2\}, \mathbb Z) ) \) is generated by the action on relative homology of the Dehn twists along the (Torelli classes of) pinched curves of \(m\). In the case of a point \((C,r_1,r_2,m)\in \U_{g,2}\), the only non trivial (Torelli class of) pinched curves of \(m\) are the ones that separate the surface \(\Sigma_{g, 2}\) into two domains, the first of genus zero containing the two marked points \(q_1,q_2\), the second being of genus \(g\) with a disc removed (a representation of the surface of genus two with the pinched curve is depicted in Figure \ref{fig:normalize_degeneration}). The Dehn twist along such a curve acts trivially on the relative homology group. On the one hand we deduce from Lemma \ref{l:mncstratification} that \(\U_{g,2}\) is a manifold and on the other that the action of \(H_1(\Sigma_g, \mathbb Z)\) on \(\U_{g,2}\) is free. 

We claim that the action of \(G\) on \(\U_{g,2}\) is moreover proper and discontinuous. Since \(\U_{g,2}\) is a manifold, and hence is locally compact, it suffices to prove that the action is discontinuous, namely that any point has a neighborhood which is disjoint from its images by \(G\).  We already know this is so on the open stratum \(\mrs_{g,2}\) of \(\U_{g,2}\) since there the action of the whole Torelli group is. So it remains to prove that the action is proper around any point of the boundary of \(\U_{g,2}\). For any point  \([(C,q_1,q_2,m)]\in \partial \U_{g,2} \) denote by \(\gamma\subset \Sigma_{g,2}\) a loop which is mapped to the node of \(C\) by some marking of \(C\) in the equivalence class of \(m\). There is a neighborhood \(\mathcal W\subset \mathcal U_{g,2}\) defined by the property that for every \( [(C',q_1',q_2',m')]\in \mathcal W\) the image of the loop \(\gamma\) in \(C'\) by some marking in the equivalence class of \(m'\) is of minimal length and this is the unique loop having this property.  In particular, if an element of \(G\) maps an element of \(\mathcal W\) to an element belonging to \(\mathcal W\), then some of its lift to the mapping class group of \( \Sigma_{g,2}\) needs to fix \(\gamma\) (by uniqueness) and this can happen only if the original element of \(G\) is the identity. Hence the claim follows.

We denote the quotient map  by \(R: \U_{g,2}\rightarrow \mathcal V:= \U_{g,2}/H_1 (\Sigma_g ,\mathbb Z)\), which is a \(H_1(\Sigma_g,\mathbb Z)\)-covering, the forgetful map \(\Forget_{g,2}\) transits via a map 
\[\overline{\Forget} : \mathcal V \rightarrow \mrs_g \]
satisfying \( \Forget= \overline{\Forget} \circ R\); it is a holomorphic fibration isomorphic to the square universal bundle $\mathfrak{C}^2(\mrs_g)\rightarrow\mrs_g$ by using Riemann's extension theorem applied on the added boundary points.
Notice that for a given  \( (C, m)\in \mrs_g\) the fiber  \(\overline{\Forget} ^{-1} (C,m)\) is biholomorphic to the complex surface  \(C\times C\) and inherits a natural identification of the homology \(m\times m:H_1(\Sigma_g)\times H_1(\Sigma_g)\rightarrow H_1(C)\times H_1(C)\).  The covering 
\[ R_{|\Forget ^{-1} (C,m)} : \Forget ^{-1} (C,m) \rightarrow \overline{\Forget} ^{-1} (C,m)\simeq C\times C \]
is a \(H_1 (\Sigma_g)\)-covering of monodromy given by \eqref{eq: abelian monodromy}.

\vspace{0.3cm} 
\textit{Boundary complex of $\U_{g,2}$.}

Each boundary stratum $B_{\curvesystem}$ of $\U_{g,2}$ is characterized by a Torelli class of a simple closed curve $c$ in $\Sigma_{g,2}$ that bounds a disc containing the marked points. The following Lemma proves that every such stratum can be characterized by a  homology class in $\betas$.
 
 \begin{lemma}
 Let $c$ be an isotopy class of separating simple closed curve in $\Sigma_{g,2}$ that splits $\Sigma_{g}$ into a disc containing both marked points and a genus $g$ component without marked points. Then, $c$ defines a unique class $\beta_c \in H_1(\Sigma_g, q_1,q_2; \mathbb{Z})$ represented by a cycle in the disc such that $\partial\beta_c=q_2-q_1$. 
 Two such curves $c,c'$ are equivalent under the Torelli group $\mathcal{I}_{g,2}$ if and only if $\beta_c=\beta_{c'}$. 
\end{lemma}
\begin{proof}
 If $c$ and $c'$ are equivalent by an element of the Torelli group it is obvious that $\beta_{c'}$ is the image of $\beta_{c}$ and since the action is trivial in homology we have $\beta_c=\beta_{c'}$. 

Suppose $\beta_c=\beta_{c'}$. Let $D$ and $D'$ be the open discs bounded by representatives of $c$ and $c'$ respectively. Find an orientation preserving diffeomorphism $\phi:\Sigma_g\rightarrow \Sigma_g$ such that $\phi(c)=c'$,  $\phi(D)=D'$ and $\phi_1(q_i)=q_i$. By definition the map $\phi_*$ induced by $\phi$ in $H_1(\Sigma_g,q_1,q_2,\mathbb{C})$ satisfies $\phi_*(\beta_c)=\beta_{c'}$. We will show that up to pre-composition of $\phi$ with an diffeomorphims $\psi$ that fixes the closure of $D$ pointwise, we can suppose that $(\phi\circ\psi)_*$ fixes every class in the homology group. 
Indeed, the restriction of $\phi_*$ to the image of the natural inclusion $0\rightarrow H_1(\Sigma_g,\mathbb{Z})\rightarrow H_1(\Sigma_g,q_1,q_2,\mathbb{Z})$ defines an element in the group $\text{Aut}(H_1(\Sigma_g,\mathbb{Z}))$. On the other hand, the action of the mapping class group of orientation preserving isotopy classes of diffeomorphisms of $\Sigma_g\setminus D$ that fix the boundary pointwise, provides a map $\text{Mod}(\Sigma_g\setminus D)\rightarrow \text{Aut}(H_1(\Sigma_g,\mathbb{Z}))$ that is surjective (see \cite{FM}{p.169 and Section 6.3.2} for details). Hence, there exists a $\psi:\Sigma_g\rightarrow\Sigma_g$ that is the identity on $D$, such that $\psi_*=(\phi_{*|H_1(\Sigma_g)})^{-1}$ and $\psi_*(\beta_c)=\beta_c$. By construction  $(\phi\circ\psi)_*=\text{Id}$ and $\phi\circ\psi(c)=c'$.
\end{proof}
This finishes the proof of Proposition \ref{p:bordification of Sg2}. 
\end{proof}
\cbend
\subsection{Stable forms on a stable curve}
\begin{definition}
 Let $C$ be a stable nodal curve. A stable one-form on $C$ is a section of its dualizing sheaf. In other words, a holomorphic $1$-form on $C^*$ that has at worst simple poles at the nodes and satisfies that the sums of the residues of the branches meeting at each node is zero. $\Omega(C)$ denotes the space of stable forms on $C$. A stable one form will be sometimes referred to as an abelian differential.
\end{definition}

By Riemann Roch's Theorem, the space of meromorphic 1-forms on a compact connected genus $g$ Riemann surface $X$ with at most simple poles at $k>0$ marked points has dimension $k+g-1$. If we apply this to each part of a connected stable nodal curve $C$ of genus $g$, we deduce that the dimension of the complex vector space $\Omega(C)$ is $g$ (see \cite{Harris}[p.82])
Remark that the restriction of a stable form $\omega\in\Omega(C)$ to a component $C_i$ of the normalization of $C$ can be the zero form. If this is the case we say that $\omega$ has a zero component. 

\begin{definition} Given a stable curve $C$ we denote
\begin{itemize}
\item $\Omega^*(C)\subset \Omega(C)$ the set of stable forms without zero components.
\item $\Omega_0(C)\subset\Omega(C)$ the vector subspace of stable forms with zero residue at any node
\item $\Omega^*_0(C)=\Omega^*(C)\cap\Omega_0(C)$ the set of stable forms without zero components, and whose residues at the nodes are zero.
\end{itemize}
\end{definition}

If $C$ is of compact type, then by the residue theorem applied to each part of $C$, all residues of all branches at the nodes have to be zero, so $\Omega_0(C)=\Omega(C)$. If $C_1,\cdots,C_k$ are the distinct parts of $C$, a stable form in $\Omega(C)$ will be written as $ \omega_1\vee\ldots\vee\omega_k$ where each $\omega_i\in\Omega(C_i)$ and the $\vee$ indicates that we glue them at the points corresponding to the nodes of $C$. 

\begin{definition}
The order of a stable form $(C,\omega)$ at a node $q$ is defined to be $$\text{ord}_q(\omega)=2+\text{ord}_q(\omega_1)+\text{ord}_q(\omega_2)$$ where $\omega_i$ denotes the restriction of $\omega$ to a branch of $C$ through $q$. 
\end{definition}
Note that the order of the node cannot be $1$ and $\text{ord}_q(\omega)\geq 0$ for any point $q\in C$.

Given a stable form $\omega\in\Omega^*(C)$ of genus $g$ we define its associated divisor $$(\omega)=\sum_{q\in C} \text{ord}_q(\omega)q$$ whose primitive degree satisfies $$\deg (\omega)=\sum_{q\in C} \text{ord}_q(\omega)=2g-2.$$
The support is a disjoint union of the zeros $Z(\omega)$ and the nodes $N(\omega)$ of $\omega$. 
As a consequence any stable curve $C$ with more than $g-1$ nodes has empty $\Omega_0^*(C)$.


\subsection{Stable one forms and singular translation structures}
On each component $C_i$ where a stable form $(C,\omega)$ is not identically zero it defines naturally a singular translation structure. Indeed, around a point $q\in C_i^*$ we can locally define a holomorphic function $\phi_q(z)=\int_q^{z}\omega$ that is a branched covering of degree $\text{ord}_q(\omega)+1$, ramified over $0$ if the degree is at least two. At the intersection of domains two such maps $\phi_q$ and $\phi_p$ satisfy $$\phi_p=\phi_q+\const.$$ 

Reciprocally if we are given a cover $U_{\alpha}$ of a compact (possibly disconnected) topological surface $\Sigma$, and finite branched coverings $\phi_{\alpha}:U_{\alpha}\rightarrow V_{\alpha}\subset \mathbb{C}$ satisfying $\phi_{\alpha}=\phi_{\beta}+\const$ at the intersections $U_{\alpha}\cap U_{\beta}$, we can define a complex structure on $\Sigma$ by declaring that the $\phi_{\alpha}$'s are holomorphic. The abelian differential $\omega$ defined locally by $d\phi_{\alpha}$ is well defined on the obtained Riemann surface $\hat{C}$. By identifying pairs of points in $\hat{C}$, we obtain all nodal curves $C$ that are normalized by $\hat{C}$.

\subsection{Singular flat metric and geodesic foliations}
\label{ss:singular flat metric}
Denote by $Z(\omega)\cup N(\omega)$ the support of zeros and nodes of the divisor $(\omega)=\sum_{q\in C} \text{ord}_q(\omega) q$ of a stable form $(C,\omega)$ having components  $(C_i,\omega_{C_i})$. Any object invariant by translations in $\mathbb{C}$ can be pulled back to $C_i^*\setminus Z(\omega_{C_i})\cup N(\omega_{C_i})$ with singularities at the points of $Z(\omega)\cup N(\omega)$. 

In particular the pull back of the flat metric in $\mathbb{C}$ produces a  a singular flat metric on any non-zero component $(C_i,\omega_i)$ defined by $\omega\otimes\overline{\omega}$.  At a branch $C_i$ of $C$ around a point $q\in C_i$ of non-negative order for the restriction $\omega_{|C_i}$, the metric is equivalent to a standard conical point of angle $2\pi(\text{ord}_q(\omega_{C_i})+1)$. Around a point with non-zero residue $a\in\mathbb{C}^*$, the metric is a semi-infinite cylinder equivalent to one of the ends of $\mathbb{C}^*/a\mathbb{Z}$.  
The volume of a stable form $\omega\in\Omega(C)$ is defined as \begin{equation}\vol(\omega)=\frac{i}{2}\int_C\omega\wedge\overline{\omega}.\end{equation} In particular, $0\leq \vol(\omega)\leq\infty$ and it is finite if and only if all the residues of $\omega$ at the nodes of $C$ are zero. If $\vol(\omega)=0$ then $\omega$ is the zero form.

The oriented geodesic directional foliation of $\mathbb{C}$ given by an angle $\theta\in\mathbb{S}^1$ is also invariant by translations, so we can also lift it to a singular oriented directional foliation $\mathcal{G}_{\theta}$ on $C_i$. Its leaves are geodesics for the metric induced by the form $\omega$. At a zero $q$ of $\omega_{C_i}$ the foliation has a saddle with $2(\text{ord}_q(\omega_{C_i})+1)$ separatrices, that alternatively enter and leave the singularity by forming an angle of $\pi$ (see Figure \ref{fig:saddle}). At any other point the foliation is regular.

\begin{figure}[httb]
\centering
\def\svgwidth{\columnwidth}
\includegraphics[width=1in]{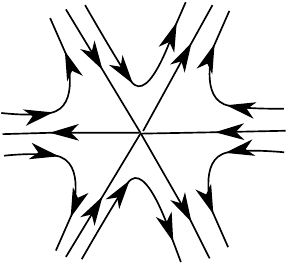}
 \caption{A saddle point of a directional foliation at a zero of order two} \label{fig:saddle}
\end{figure}

Some subsets of stable forms where these geometric objects encode important information will be of special importance for us and we introduce appropriate notation.

\begin{definition}
Given a family $\mathcal{K}$ of (marked or unmarked) stable curves, we define 
\begin{itemize}
\item $\Omega_0\mathcal{K}\subset \Omega\mathcal{K}$ to be the set of stable forms over curves in $\mathcal{K}$ whose residues at all nodes are zero (or equivalently the volume of the underlying metric is finite) and 
\item $\Omega^*_0\mathcal{C}\subset \Omega_0\mathcal{C}$ the set of stable forms over curves in $\mathcal{C}$ that have no zero components and zero residue at all nodes (or equivalently the underlying metric has finite volume and isolated singularities) 
\end{itemize}
\end{definition}
\subsection{The Hodge bundle over the Deligne-Mumford compactification of $\mathcal{M}_{g,n}$}
Recall (see \cite{Harris}, Chapter 4) that over the Deligne-Mumford compactification $\overline{\mathcal{M}}_{g,n}$ we can define the universal curve bundle $$\overline{\mathcal{C}}\mathcal{M}_{g,n}\rightarrow\overline{\mathcal{M}}_{g,n}$$ which is a holomorphic map between compact complex orbifolds whose fiber over the point $(C,r_1,\ldots, r_n)\in\overline{\mathcal{M}}_{g,n}$ is biholomorphic to the curve with marked points $(C,r_1,\ldots,r_n)$. The cotangent bundle to this fibration is well defined on the complement of the set of nodes. It extends as a line bundle $\mathcal{L}$ to the whole universal curve , called the relative cotangent bundle. A section of the restriction $\mathcal{L}_{|C}$ corresponds precisely to a stable form on $C$. The constancy of the dimension of this space of sections implies that the sheaf obtained by pushing it forward to $\overline{\mathcal{M}}_{g,n}$ is locally free of rank $g$ (see \cite{Hartshorne}[Exercise 5.8]). Thus it defines a holomorphic vector bundle $$\pi:\Omega\overline{\mathcal{M}}_{g,n}\rightarrow\overline{\mathcal{M}}_{g,n}$$ called the Hodge bundle. The fiber of $\pi$ over a point $(C,r_1,\ldots,r_n)\in\overline{\mathcal{M}}_{g,n}$ corresponds precisely to the set of stable forms $\Omega(C)$. When $n=0$, the Hodge bundle trivializes over any of the preferred neighbourhoods $U=U_c/\Gamma_c$. A choice of a Lagrangian $\Lambda\subset H_1(\Sigma_g)$ containing all the classes induced by the curves in the curve system $c$ allows to define a trivialization $$\Omega U\simeq U \times \text{Hom}(\Lambda,\mathbb{C})$$ 

\cbstart
\begin{definition}
The boundary strata of $\Omega\overline{\mathcal{M}}_{g,n}$ is the set of stable forms over stable curves with marked points lying the boundary strata $\partial\overline{\mathcal{M}}_{g,n}$.
\end{definition}
\cbend

The pull back of the compactified Hodge bundle $\Omega\overline{\mathcal{M}}_{g,n}\rightarrow \overline{\mathcal{M}}_{g,n}$ by the  map $\overline{\mathcal{T}}_{g,n}\rightarrow \overline{\mathcal{M}}_{g,n}$ defines the (topological) complex vector bundle $\Omega\overline{\mathcal{T}}_{g,n}\rightarrow \overline{\mathcal{T}}_{g,n}$ of homotopically marked stable forms of genus $g$ with $n$ marked points. An element in $\Omega\overline{\mathcal{T}}_{g,n}$ is a pair $([f:\Sigma_{g,n}\rightarrow (C,r_1,\ldots,r_n)], \omega)$ where $[f:\Sigma_{g,n}\rightarrow (C,r_1,\ldots,r_n)]\in\overline{\mathcal{T}}_{g,n}$ and $\omega\in\Omega(C)$. The Torelli group $\mathcal{I}_{g,n}$ acts on the bundle $\Omega\overline{\mathcal{T}}_{g,n}\rightarrow\overline{\mathcal{T}}_{g,n} $. The quotient defines a bundle $\Omega\mnc_{g,n}\rightarrow \mnc_{g,n}$. These will be referred to as the Hodge bundle over augmented Teichm\" uller or augmented Torelli space, depending on the case.   

 Their boundaries are, by definition, the preimage of the boundary of $\Omega\overline{\mathcal{M}}_g$ by the corresponding projection,i.e. the stable forms over curves with some node. 

Given a collection of (marked or unmarked) stable curves $\mathcal{K}$ we define $\Omega\mathcal{K}$ to be the set of stable forms over $\mathcal{K}$, with the induced topology in the corresponding vector bundle of forms. 

Given a subset of some space of forms $L\subset \Omega\mathcal{K}$ we denote by $\partial L$ the intersection of $L$ with the boundary strata of the ambient space.

Endowed with the complex structure in each boundary stratum $B$, $\Omega B$ becomes a holomorphic vector bundle over $B$. We will refer to it as a boundary stratum of the space of forms where $B$ lies. 

\begin{definition}
 A homologically marked stable form of genus $g\geq 0$ with $n$ marked points is a tuple $(C,r_1,\ldots,r_n,m,\omega)\in\Omega\mnc_{g,n}$ where $\omega$ is a stable one-form on a homologically marked stable curve $(C,r_1,\ldots,r_n,m)$ of genus $g$ with $n$ marked points.  By abuse of language we say that $\omega$ pinches $a\in H_1(\Sigma_g, q_1,\ldots,q_n,\mathbb{Z})\setminus 0$ if $a$ is primitive and $m(a)=0$. We say that $(C,r_1,\ldots,r_n,m,\omega)$ belongs to the boundary strata if $C$ has some node. 
\end{definition}
\subsection{Attaching and forgetful maps on Hodge bundles}\label{ss:attach&forget on hodge bundles}
The continuous attaching maps defined at the level of marked curves in subsection \ref{ss:attach&forget on augmented Torelli} 
can be lifted to attaching maps on the corresponding Hodge bundles. 

In particular, using Lemmas \ref{l:attaching marked curves} and \ref{l:attaching at two points} the attaching maps can be lifted as continuous maps to the corresponding Hodge bundles: 
\begin{equation}\Omega\mnc_{g_1,n_1}\times\Omega\mnc_{g_2,n_2}\rightarrow\Omega\mnc_{g_1+g_2, n_1+n_2-2}\text { and} \label{eq:attach two components}\end{equation}
\begin{equation}\label{eq: attach two points in one component}\Omega \mnc_{g-1,n+2}\rightarrow \Omega\mnc_{g,n}\end{equation} by simply considering the stable form over the image curve that coincides with the given restriction on each part. The image will only contain forms with zero residue at the node produced by the attaching map.
More generally, we can use attaching maps and the the biholomorphism \eqref{eq:curvesystemdecomposition} to define, via the attaching maps a decomposition of forms with zero residues on a stratum $B_c$ of $\overline{\mathcal{T}}_{g,n}$: 
\begin{equation}\Omega_0B_{c}\cong\Pi_{i}\Omega \mathcal{T}_{g_i,n_i}\label{eq:zero_residue_formdecomposition}\end{equation}

On the other hand we can also lift  the forgetful maps defined by Lemma \ref{l:forget a point} to the Hodge bundle. To construct the natural lift $$\Omega\mnc_{g,n}\rightarrow \Omega\mnc_{g,n-1}$$it suffices to restrict the form to the image curve. Remark that on the eventually contracted components after forgetting the point the form is of genus zero and uniquely determined by the value of one of the residues.

\subsection{Sub-stratification of Hodge bundles and period coordinates on strata}

The boundary stratification of the Hodge bundle over $\overline{\mathcal{M}}_{g,n}$ is substratified by the topological properties of the zero divisor and polar set of the forms.
Two elements in $\Omega \overline{\mathcal{M}}_{g,n}$ are said to belong to the same substratum if there exists a homeomorphism between the underlying marked curves  preserving the following data of the forms:

\begin{enumerate}
\item zero components 
\item nodes with non-zero residue
\item the associated zero divisor on each non-zero component
\end{enumerate}
In particular, if an element has a zero of order $k$ at a marked point, then any  element in its substratum will have a zero at of order $k$ at a marked point. The same happens with poles with zero residue: if the restriction of the form to one of the branches has a zero of order $k$, then any form in its stratum will have this property on the corresponding branch of a node.
 
This stratification can be lifted to $\Omega\mnc_{g,n}$ and $\Omega\overline{\mathcal{T}}_{g,n}$. It is a substratification of the stratification induced by $\sqcup B_{c}$ on $\overline{\mathcal{T}}_{g,n}$. 
The generic stratum in $\Omega B_{c}$ is the only substratum that is open in $\Omega B_{c}$, and it is Zariski open too. Indeed, the union of all other substrata in $\Omega B_{c}$ form an analytic subset of $\Omega B_c$.
A substratum in $\Omega B_c$ is minimal if it has minimal dimension among substrata of $\Omega B_c$. In particular, the generic stratum in $\Omega B_{\emptyset}=\Omega \mathcal{T}_{g,n}$  is the stratum having only simple zeros at unmarked points and is also dense in $\Omega\overline{\mathcal{T}}_{g,n}$. The minimal stratum in $\Omega \mathcal{T}_{g,n}$ is formed by forms with a single zero of order $2g-2$ at some marked point (for $n\geq 1$). 

The restriction of an isomorphism of type  \eqref{eq:zero_residue_formdecomposition} to all possible products of substrata of $\Omega\mathcal{T}_{g_i,n_i}$ describe all the diferent substrata of forms with zero residues in $\Omega_0 B_{c}$. Let us define local holomorphic coordinates on them.


The zero substratum of $\Omega\mathcal{M}_{g,n}$ is isomorphic to $\mathcal{M}_{g,n}$. On the other hand,  Veech \cite{Veech} and Masur in \cite{Masur2} proved that local holomorphic coordinates can be given in any substratum $R$ of $\Omega\mathcal{M}_{g,n}$ around a point $(C,P,\omega)$ with $\omega\neq 0$. They are defined by integration of the form on cycles in $H_1(C,Z(\omega)\cup P ;\mathbb{Z})$ and therefore lie in the vector space $H^1(C,Z(\omega)\cup P ,\mathbb{C})$. These endow the stratum with the structure of a complex manifold compatible with the one induced by the complex structure of $\Omega\mathcal{T}_{g,n}$. 

We conclude from the equivalence \eqref{eq:zero_residue_formdecomposition} and the description of strata, on smooth curves, that the substratum of stable forms  containing a point $(C,P,f,\omega)\in\Omega^*_0\overline{\mathcal{T}}_{g,n}$ without zero components and with zero residues at the nodes is \emph{locally} isomorphic to the product $$\prod_{i}H^1\big(C_i,(Z(\omega)\cup N(C)\cup P)\cap C_i,\mathbb{C}\big)$$
where each $C_i$ corresponds to a connected component of the normalization of $C$ and $N(C)$ is the union of all points of attaching in the components $C_i$ to obtain $C$. These local coordinates will be referred to as period coordinates on the substratum.




\subsection{Periods of marked stable forms with zero residues at the nodes}

A stable form on $(C,r_1,\ldots,r_n)$ is holomorphic on $C^*$ and can thus be integrated along paths in $C^*$. For closed paths the value of the integral does only depend on the homology class in $H_1(C^*,r_1,\ldots,r_n;\mathbb{Z})$ and it is called the period of the class. If the residues at nodes are all zero, we can also integrate along paths passing {\em through} the nodes, and the integral along a closed path depends only on its class in $H_1(C,\mathbb{Z})$.

For any marked stable form $([f:\Sigma_{g,n}\rightarrow (C,r_1,\ldots,r_n)],\omega)\in\Omega_0\overline{\mathcal{T}}_{g,n}$ of genus $g$ with $n$ marked points and  zero residues at the nodes we have a well defined notion of {\bf period homomorphism} $$\text{Per}([f:\Sigma_{g,n}\rightarrow (C,r_1,\ldots,r_n)],\omega):H_1(\Sigma_g, q_1,\ldots,q_n;\mathbb{Z})\rightarrow\mathbb{C}$$ defined by
\begin{equation}
 \per([f:\Sigma_{g,n}\rightarrow (C,r_1,\ldots,r_n)],\omega)(\gamma)=\int_{
 f_*(\gamma)}\omega\text{ for }\gamma\in H_1(\Sigma_g, q_1,\ldots,q_n;\mathbb{Z}).
\end{equation}
Since the period homomorphism depends only on the homological marking induced by the homotopical marking, we can also define the period on a homologically marked stable curve $(C,r_1,\ldots, r_n,m,\omega)$ with marked points by $\{\gamma\mapsto \int_{m(\gamma)}\omega\}$. 
Any homologically non-trivial curve $a$ in $\Sigma_g$ pinched by the marking, i.e. belonging to
$\text{Ker}(f_*)$, also belongs to the kernel of the period homomorphism.
Remark that if the period homomorphism is injective then the marking is an isomorphism, and thus the curve $C$ must be of compact type.

\begin{definition}
The period map associated to $(\Sigma_g,q_1,\ldots,q_n)$ is the map $$\Omega_0\overline{\mathcal{T}}_{g,n}\rightarrow H^1(\Sigma_g,q_1,\ldots,q_n,\mathbb{C})$$ that associates to any  stable form with zero residues marked by $(\Sigma_g,q_1,..,q_n)$, its period homomorphism. It is invariant by the Torelli group $\mathcal{I}_{g,n}$ action and descends to the quotient 
\cbstart$$\Omega_0\mnc_{g,n}\rightarrow H^1(\Sigma_g,q_1,\ldots,q_n,\mathbb{C}).$$\cbend
\end{definition} 
One of the difficulties that we will encounter is related to the fact that the domain of definition of the period map is neither open nor closed. It is a union of substrata of stable forms.  It contains the (dense) open subset of stable forms on smooth curves, but in the boundary strata we need to restrict to the closed set of strata that have no residues at the nodes. 
\begin{definition}
An isoperiodic deformation of a point in $\Omega_0\overline{\mathcal{T}}_0$ or $\Omega_0\mnc_{g,n}$ is a continous deformation in the ambient space that has constant value for the period map. 
\end{definition}

The following notation will be useful
\begin{definition}
Given a homomorphism $p:H_1(\Sigma_{g},q_1,\ldots,q_n;\mathbb{Z})\rightarrow \mathbb{C}$  and a family $\mathcal{K}\subset\Omega_0 \mnc_{g,n}$ we denote by $\mathcal{K}(p)$ the subset of marked stable forms in $\mathcal{K}$ having period homomorphism $p$, i.e. the intersection of the the period fiber over $p$  with the set $\mathcal{K}$. 
\end{definition}

\subsection{Local structure of the fibers of the period map}

A local fiber of the period map can be thought as a disjoint union of sets in the different strata of the ambient space. The next theorem shows that this partition is actually a nice stratification
\begin{theorem}\label{p:local structure of fiber at bdry point}

A local fiber $L$ of the period map in $\Omega^{*}_{0}\mnc_{g,n}$ projects to the orbifold chart of $\Omega\nc_{g,n}$ as a complex manifold transverse to all boundary divisors through the point. Therefore, $L$ is  an abelian ramified cover of a normal crossing divisor in $(\mathbb{C}^{2g+n-3},0)$ having precisely one component of codimension one in each component of codimension one of the ambient space $\Omega^{*}_{0}\mnc_{g,n}$ through the point. 
\end{theorem}

The statement is not true in general for forms with zero components (see subsection \ref{sss:singularities}).
\begin{proof}

A local period fiber on $\Omega_{0}\mnc_{g}$ can be lifted to a local period fiber in $\Omega_{0}\overline{\mathcal{T}}_{g}$. If the lift of the point belongs to the stratum $\Omega B_{c}$, the covering group is the subgroup $\Gamma_{c}\cap\mathcal{I}_{g}\subset \Gamma_{c}$. Hence, we just need to prove that the period fiber at the level of the Hodge bundle over augmented Teichm\" uller space has a stratification that is locally equivalent to some map $\varphi_{l,2g+n-3}$ as in \eqref{eq:bldown1} where $l$ denotes the number of simple closed curves in the curve system $c$.  It suffices in fact to show that at the level of the quotient $\Omega_{0}(U_{c}/\Gamma_{c})$ the period fiber is a holomorphic manifold transverse to every boundary component passing through the point. Thanks to the normal crossing condition of the ambient space, this is guaranteed if the period map is submersive in restriction to the (regular) stratum of the normal crossing divisor where the point belongs to. In fact we will not use the boundary stratum but a smaller regular submanifold, consisting of the substratum through the point.
\cbstart
\begin{proposition}\label{p:local period structure teich}
Consider $([f_0:\Sigma_{g,n}\rightarrow (C_0,P_0)],\omega_0)\in\Omega_0\overline{\mathcal{T}}_{g,n}$ and   $\Omega U_c$ its preferred neighbourhood. Denote $m_0:=f_{0*}: H_1(\Sigma_{g},q_1,\ldots,q_n;\mathbb{Z})\rightarrow H_1(C_0,P_0;\mathbb{Z})$ and define $H$ as the germ of the zero set of the map $\Omega U_c\rightarrow \emph{Hom}(\ker m_0,\mathbb{C})$ defined by integration. Then $H\subset \Omega_0 U_c$, $H/\Gamma_c$ is a proper holomorphic submanifold of $\Omega_0(U_c/\Gamma_c)$ and the period map induces a holomorphic map $\Omega_0(H/\Gamma_c)\rightarrow H^1(C_0,P_0,\mathbb{C})$. 
If the form $\omega_0$ has no zero components, the holomorphic map thus defined is submersive, even by restricting it to the substratum containing $(C_0,\omega_0)$. Hence, the isoperiodic foliation extends along $H/\Gamma_{c}$ as a regular foliation. 
\end{proposition}
\cbend
\begin{proof}

We first need a lemma about the holomorphicity of integration over a cycle:

\begin{lemma}\label{l:holomorphic integral}
Let $U_c/\Gamma_c\rightarrow U\subset\overline{\mathcal{M}}_{g,n}$ be a (orbifold) chart around a point $(C_0,P_0)\in\overline{\mathcal{M}}_{g,n}$ and $\gamma$ be a path in $\Sigma_g\setminus c$. 
Then, the map $\Omega U_c\rightarrow \mathbb{C}$ defined by $$([f:\Sigma_{g,n}\rightarrow (C,P)],\omega)\mapsto \int_{f_*(\gamma)}\omega$$ induces a well defined holomorphic map on $\Omega U\subset \Omega \overline{\mathcal{M}}_{g,n}$.
\end{lemma}

\begin{proof}
The function is well defined on $U_c/\Gamma_c$ because $\gamma$ does not intersect any of the simple closed curves in $c$. It is holomorphic outside the boundary and bounded in the neighbourhood of every boundary point, hence holomorphic by Riemann extension Theorem.  
\end{proof}

 Remark first that there is a non-trivial exact sequence \begin{equation}0\rightarrow \text{Hom}(H_1(C_{0},P_0;\mathbb{Z}), \mathbb{C})\rightarrow \text{Hom}(H_1(\Sigma_{g},q_1,\ldots,q_n;\mathbb{Z}), \mathbb{C})\rightarrow \text{Hom}(\ker m_{0},\mathbb{C})\rightarrow 0\label{eq:exact sequence of marking}\end{equation} induced by that defined by the marking $$0\rightarrow \ker m_0\rightarrow H_{1}(\Sigma_{g},q_1,\ldots,q_n;\mathbb{Z})\rightarrow H_{1}(C_{0},P_0;\mathbb{Z})\rightarrow 0.$$
The map $$\Omega U_{\curvesystem}\rightarrow \text{Hom}(\ker m_{0},\mathbb{C})$$ given by $([f:\Sigma_{g,n}\rightarrow (C,r_1,\ldots,r_n)],\omega)\mapsto \{c\mapsto\int_{f_*(c)}\omega\}$ is well defined and invariant by the action of $\Gamma_{\curvesystem}$ on $\Omega U_{\curvesystem}$. Lemma \ref{l:holomorphic integral} implies that it induces a holomorphic map \begin{equation} \label{eq:submersion at boundary points} h:\Omega (U_{\curvesystem}/\Gamma_{\curvesystem})\rightarrow\text{Hom}(\ker m_{0},\mathbb{C}).\end{equation} We claim that $h$ is submersive at $[f_0:\Sigma_{g,n}\rightarrow (C_0,P_0)]$. In fact, its restriction to the Hodge bundle fiber $\Omega C_0$ is already submersive. It associates to each node, the value of its residue. It is therefore linear. The kernel of this restriction is precisely $\Omega_0C_0$, the space of stable forms with zero residues on $C_0$. Let $C_1,\ldots,C_l$ be the components of the normalization of $C_0$ and define $N(C_i)\subset C_i$ the points corresponding to the nodes of $C_0$. 
The restriction of any stable form in $\Omega C_{0}$ to the component $C_{i}$ is a meromorphic form with at worst simple poles on $N(C_i)$. It satisfies the residue theorem. On the other hand, by Riemann-Roch theorem applied to $C_i$ with $i\geq 1$ if we denote $m_i=|N(C_i)|$ and $g_i=g(C_i)$ the genus of $C_i$ we get 
\begin{enumerate}
\item the space $M_i$ of meromorphic forms on $C_i$ that have at worst simple poles at the marked points has dimension $g_i+m_i-1$ and 
\item the space $\Omega C_i$ of holomorphic forms on $C_i$ has dimension $g_i$. 
\end{enumerate}
Denote $L_i\subset\mathbb{C}^{m_i}$ the codimension one space defined by $x_1+\ldots+x_{m_i}=0$. 
By the previous dimension calculations, the sequence $0\rightarrow \Omega C_i\rightarrow M_i\rightarrow L_i\rightarrow 0$ is exact.
Recall that $\sum m_i=2m$ where $m$ denotes the number of nodes of $C_0$. 
Consider $L\subset\mathbb{C}^{2m}$ the codimension $l$ subspace $L_1\oplus\cdots\oplus L_l$.

The sequence 
\begin{equation} \label{eq:exact seq} 0\rightarrow \Omega_0 C_{0}\rightarrow M_{1}\oplus\cdots\oplus M_{l}\rightarrow L \rightarrow 0\end{equation} is still exact. 
A necesssary and sufficient condition for an element of $M_{1}\oplus\cdots \oplus M_l$ to define a form in $\Omega C_0$ is that the residues at the points that are attached to produce a node in $C_0$ have opposite sign. This corresponds to cutting $L$ with $m$ linear equations of type $x_i-x_j=0$ in $\mathbb{C}^{2m}$ where $m$ is the number of nodes of $C_0$. When we intersect $m-1$ of those with $L$, we easily obtain the last of such equations. In fact, the dimension of the intersection is $2m-l-(m-1)=m-l+1$. The points in this intersection are in one to one correspondence with elements in $\text{Hom}(\ker m_0, \mathbb{C})$. Indeed, the non-separating curves in $\curvesystem$ generate $\ker m_0$, and the only relations they have to satisfy are the relations given by the fact that some sums of the generators are boundaries of components. This relations coincide with those defining the relevant subspace of $L$. On the other hand the exactness of (\ref{eq:exact seq}) tells us that we can find a form in $\Omega C_0$ with any given residue homomorphism $\ker m_0\rightarrow \mathbb{C}$. This proves that $h_{|\Omega C_{0}}$ is surjective. Since it is linear it is also submersive, and so are $h$ and the restriction of $h$ to each boundary component containing $(C_0,m_0,\omega_0)$. The levels of $h$ are regular subvarieties transverse to each boundary component and invariant by $\per$. The zero level set $$H=\{([f:\Sigma_{g,n}\rightarrow (C,P)],\omega)\in \Omega U_{\curvesystem}: \int_{f_*(c)}\omega=0\quad \forall c\in\ker m_{0}\}$$ contains the marked stable form $([f_0:\Sigma_{g,n}\rightarrow (C_0,P_0)],\omega_0)$ and is invariant under the action of $\Gamma_{\curvesystem}$.
By the exact sequence (\ref{eq:exact sequence of marking}), the restriction of the period map  to $H$ can be thought as a map $H\rightarrow H^{1}(C_{0},\mathbb{C})$ that is invariant by $\Gamma_{\curvesystem}$. The induced map $$h_2: H/\Gamma_{\curvesystem}\rightarrow \text{Hom}(H_1(C_{0},P_0;\mathbb{Z}),\mathbb{C})\text{ is holomorphic. }$$

This proves that the fibers of the period map in $\Omega_0\mnc_{g,n}$ are preimages under the branched covering $\Omega U_{\curvesystem}\rightarrow \Omega (U_{\curvesystem}/\Gamma_{\curvesystem})$ of analytic sets.

It remains to prove that $h_2$ is submersive in restriction to the substratum $R$ containing $([f_0:\Sigma_{g,n}\rightarrow (C_0,P_0)],\omega_0)$ whenever $\omega_0$ has no zero components.

Recall that $C_1,\ldots, C_l$ denote the components of the normalization of $C_0$ and define 

\begin{itemize}
\item $\omega_{i}$ the restriction of $\omega_0$ to $C_i$ 
 \item $g_i=g(C_i)$ the genus of $C_i$
 \item $P_i=(P\cap C_i)\cup N(C_i)\cup Z(\omega_{0i}))$
\end{itemize}

 The local coordinates in the substratum $R_{i}$ of the stable form $(C_i,P_i,\omega_{i})$. The latter are given by vectors in $W_{i}=H^{1}(C_{i},P_i,\mathbb{C})$.

 By using the attaching maps used to obtain $C_{0}$ from $C_{1},\ldots, C_{l}$ we define a surjective holomorphic map $$W_{1}\oplus \cdots\oplus W_{l}\rightarrow R.$$ If we prove that the composition of this map with the period map $h_2$ is submersive onto $H^{1}(C_{0}, \mathbb{C})$ we will be done. The said composition is a map $$W_{1}\oplus \cdots\oplus W_{l}\rightarrow \text{Hom}(H_{1}(C_{0},P_0,\mathbb{Z}),\mathbb{C})$$ that factors through $$\widehat{W}_{1}\oplus\cdots\oplus\widehat{W}_{l}\text{ where }\widehat{W}_{i}=H^{1}(C_{i},N(C_i),\mathbb{C}).$$ The map $W_{1}\oplus \cdots\oplus W_{l}\rightarrow \widehat{W}_{1}\oplus\cdots\oplus\widehat{W}_{l}$ and the map $\widehat{W}_{1}\oplus\cdots\oplus\widehat{W}_{l}\rightarrow \text{Hom}(H_1(C_{0},P_0;\mathbb{Z}),\mathbb{C})$ are obviously submersive.
\end{proof}
 
\end{proof}
\begin{definition} The map 
 $$\percompact_{g,n}:\Omega_0^* \mnc_{g,n}\rightarrow \text{Hom}(H_1(\Sigma_{g,n});\mathbb{C})$$ denotes the extension to the bordification  $\Omega_0^* \mnc_{g,n}$ of the period map $\per_{g,n}$ defined on $\Omega\mrs_{g,n}$.  When there is no risk of confusion we omit the subindex and write $\percompact$ and $\per$ respectively. 
\end{definition}

 \begin{definition}
  Let $p\in\mathcal{H}_g$ and denote $\mathcal{G}_{p}$ the dual boundary complex of the bordification $\percompact^{-1}(p)$ of $\per^{-1}(p)$ defined by its closure in $\Omega^{*}_{0}\mnc_{g}$ .  The final statement of Theorem \ref{p:local structure of fiber at bdry point} and Remark \ref{rem: normal crossing type} allow to define a continuous map of dual boundary complexes\begin{equation}
      \mathcal{G}_p\rightarrow \mathcal{C}(\mnc_g)\simeq\mathcal{C}_g/\mathcal{I}_g
  \end{equation}
   \end{definition}

\begin{corollary}\label{c:holomorphic extension to compact type}
The period map $\percompact_{g,n}$ induces a regular holomorphic foliation $\overline{\mathcal{F}}_{g,n}$ on the moduli space $\Omega^*\overline{\mathcal{M}}^c_{g,n}$ of stable forms without zero components on curves of compact type  called the isoperiodic foliation. It is transverse to all boundary components and to the substratum  passing through the point. Its restriction to $\Omega^*\mathcal{M}_{g,n}$ will be denoted by $\mathcal{F}_{g,n}$. 
\end{corollary}
\begin{proof}
 All the residues of all forms in a neighbourhood of a point in $\Omega_0\mnc^c_{g,n}$  are zero. Hence, the map $\percompact$ is holomorphic in some neighbourhood. Proposition \ref{p:local period structure teich} guarantees that it is also submersive there. On the other hand $\percompact_{g,n}$ is equivariant with respect to the natural action of $\text{Mod}(\Sigma_{g,n})$ on source and target. Hence the underlying foliation is well defined in the quotient. 
\end{proof}

\begin{remark}
 The isoperiodic foliation $\overline{\mathcal{F}}_{g,n}$ extends to the whole space $\Omega\overline{\mathcal{M}}_{g,n}$ as a singular holomorphic foliation. In fact, the zero section of the Hodge bundle is part of the singular set of this foliation, as are the points of non-compact type in $\Omega_0^*\overline{\mathcal{M}}_{g,n}$, and most strata with zero components (see section \ref{p:singularity of isoperiodic set} for an example). To construct the extension isoperiodic foliation to the missing strata (those with non-zero residue at the nodes) we need to define the notion of isoperiodic family  of  \emph{meromorphic} forms on smooth curves. This subject will be developed in a forthcoming paper. 
\end{remark}

\cbstart
\subsection{Some singularities of the isoperiodic set on regular points of $\Omega\mnc_{g,2}$}

\label{sss:singularities} 
In this subsection we analyze the local isoperiodic deformation spaces at some points of the Hodge bundle over the complex manifold  $\U_{g,2}\subset\mnc_{g,2}$, bordification of $\mrs_{g,2}$,  introduced in Section \ref{s:bordification of Sg2}. The contents of this section will be useful in Section \ref{s:period fibers with marked points}.

Recall from Section \ref{p:bordification of Sg2} that the boundary points in $\U_{g,2}$ consist on stable curves having precisely one separating node that leaves both marked points on a genus zero component (in Figure \ref{fig:normalize_degeneration} we depict such a stable curve for $g=2$). The strata of  $\partial\U_{g,2}$ form a smooth divisor having a connected component $B_{\beta}$ for each class $\beta\in\betas=\{\beta\in H_1(\Sigma_g,q_1,q_2;\mathbb{Z}): \partial\beta=q_2-q_1\}$ 



 The bundle $\Omega \U_{g,2}\rightarrow \U_{g,2}$ is a complex manifold. The stratification of the ambient space induces a stratification on $\Omega \U_{g,2}$ that has the following boundary strata:  $$\bigsqcup_{\beta\in\betas}\Omega B_\beta. $$ 
 
 Each stable form in $\Omega B_\beta$ has  a zero component of genus zero and lies in the domain $\Omega_0\mnc_{g,2}$ of the period map $\per_{g,2}$. Moreover the restriction of this map to $\Omega\U_{g,2}$ is holomorphic (see Proposition \ref{p:local period structure teich}) .  
 We want to analyze the local analytic set formed by a  period fiber around generic points of $\Omega B_\beta$. 

 By abuse of language we write $\Omega^{SZ}B_\beta\subset \Omega B_\beta $ the forms that have simple zeros on the smooth genus $g$ component and a zero component of genus zero, and $\Omega^{SZ,1}B_\beta\subset \Omega^{SZ}B_\beta$ those where the zero component is glued at a simple zero of the non-zero component (as in Figure \ref{fig:normalize_degeneration}) . 

\begin{proposition}\label{p:singularity of isoperiodic set}
 Let $(C^0,r^0_1,r^0_2,m^0,\omega^0)\in\Omega^{SZ}B_\beta\subset\Omega\U_{g,2}$. The local isoperiodic deformation space at $(C^0,r^0_1,r^0_2,m^0,\omega^0)$ is
 \begin{enumerate}
     
     \item a normal crossing of two smooth manifolds, precisely one of which is contained in the boundary strata, if  $(C^0,r^0_1,r^0_2,m^0,\omega^0)\in\Omega^{SZ,1}B_\beta$ or 

\item a smooth manifold contained in the boundary if $(C^0,r^0_1,r^0_2,m^0,\omega^0)\notin\Omega^{SZ,1}B_\beta$
 \end{enumerate}
\end{proposition}

\begin{proof}[Proof of Proposition \ref{p:singularity of isoperiodic set}] 
Along the proof we write $U=\U_{g,2}$. Let $(C,r_1,r_2,m,\omega)\in\Omega^{SZ}B_\beta$. We consider the contraction of the spherical component of $\omega$, namely the tuple $(\widehat{C},\widehat{q},\widehat{m},\widehat{\omega}))\in\Omega^{SZ}\mrs_{g,1}$ satisfying the following relations
\begin{equation}
    \widehat{m}=m\circ\iota\text{ where }\iota:H_1(\Sigma_{g,1})\rightarrow H_1(\Sigma_{g,2}) \text{ is the natural inclusion}
\end{equation}\begin{equation}\label{eq:boundary points in Sg2} 
(C,r_1,r_2,m)=(\widehat{C},\widehat{q},\widehat{m})\underset{\widehat{q}=\infty}{\vee}(\mathbb{C}P^1,0,1,\infty)\text{ for some } \widehat{q}\in \widehat{C}\end{equation} and 
 \begin{equation}\label{eq:stable form with two marked}\omega=\widehat{\omega}\vee 0\text{ on }(\widehat{C},\widehat{q},\widehat{m})\underset{\widehat{q}=\infty}{\vee}(\mathbb{C}P^1,0,1,\infty).\end{equation}
 The point lies in $\Omega^{SZ,1}B_\beta$ if and only if $\widehat{\omega}(\widehat{q})=0$. 
 
We can define a holomorphic map 

\begin{equation}
    h_{\beta}:\Omega^{SZ}U\rightarrow \mathbb{C} \text{ by } h_\beta(C,r_1,r_2,m,\omega)=\int_{m(\beta)}\omega
\end{equation}
Its zero set $H_\beta=\{h_\beta=0\}$ contains $\Omega^{SZ}B_\beta$ and is invariant by the isoperiodic foliation on $\Omega^{SZ}U$.

Let us provide adapted coordinates of $\Omega^{SZ}U$ around the point $(C^0,r^0_1,r^0_2,m^0,\omega^0)\in\Omega^{SZ}B_\beta$. Denote 
\begin{equation}\label{eq:restricted forgetful}\text{For}:\Omega^{SZ}U\rightarrow\Omega^{SZ}\mrs_{g}\end{equation}
the restriction
of the forgetful map \(\text{For}_{g,2}:\Omega\mnc_{g,2}\rightarrow\Omega\mrs_{g}\) to the open set \(\Omega^{SZ}U\). 

Given a small neighbourhood $V^0$ of $\text{For}(C^0,r^0_1,r^0_2,m^0,\omega^0)$ in $\Omega\mrs_{g}$ suppose we are given a local holomorphic section $q:V^0\rightarrow \Omega^{SZ}\mrs_{g,1}$ of  $\Omega^{SZ}\mrs_{g,1}\rightarrow \Omega^{SZ}\mrs_g$ such that 
\begin{enumerate}
    \item $q(\widehat{C^0},\widehat{m^0}, \widehat{\omega^0})=\widehat{q^0}\in\widehat{C^0}$ is the point corresponding to the node of $C^0$, and 
     \item the order  of the form $\hat{\omega}$ at the point $q(C,m,\omega)$ is independent of $(C,m,\omega)$ and (therefore) has value $k-1\in\{0,1\}$
\end{enumerate} 

The germs of holomorphic functions $$r_1\mapsto \int_{q\circ\text{For}(C,r_1,r_2,m,\omega)}^{r_1}\widehat{\omega}\quad\text{ and }\quad r_2\mapsto \int_{q\circ\text{For}(C,r_1,r_2,m,\omega)}^{r_2}\widehat{\omega} $$
at $q(C,m,\omega)$ have each degree $k$ and can therefore be written as $$r_1\mapsto \big(\zeta_1(C,r_1,r_2,m,\omega)\big)^k\quad\text{ and }\quad r_2\mapsto \big(\zeta_2(C,r_1,r_2,m,\omega)\big)^k$$ for some germ of biholomorphisms $$\zeta_1,\zeta_2:(C,q(C,m,\omega))\rightarrow (\mathbb{C},0)$$  
In a small neighbourhood $W^0$ of $(C^0,r^0_1,r^0_2,m^0,\omega^0)$ in $\text{For}^{-1}(V^0)$ we consider the holomorphic map $\Psi_q: W^0\rightarrow V^0\times\mathbb{C}^2$  defined by 
\begin{equation}
\Psi_q(C,r_1,r_2,m,\omega)=\big(\text{For}(C,r_1,r_2,m,\omega),\zeta_1(C,r_1,r_2,m,\omega),\zeta_2(C,r_1,r_2,m,\omega)\big).\end{equation}
It is locally injective and therefore it is locally a biholomorphism. In these coordinates, the boundary strata of $\Omega^{SZ}U$ are defined by  $\{\zeta_1=\zeta_2\}$ and the map $h_\beta$ and $\per_{g,2}$ by $$ h_\beta\big((C,m,\omega), \zeta_1,\zeta_2\big)=\zeta_2^k-\zeta_1^k\quad\text{ and }\quad \per_{g,2}((C,m,\omega),\zeta_1,\zeta_2)=\big(\per_{g,0}(C,m,\omega),\zeta_2^k-\zeta_1^k\big).$$

In particular $$\per_{g,2}^{-1}(p_0,0)=\bigcup_{\{\alpha\in\mathbb{C}:\alpha^k=1\}}\per_{g,0}^{-1}(p_0)\times\{\zeta_2=\alpha\zeta_1\}$$ 

When $k=1$ the germ of isoperiodic set at $(C^0,r^0_1,r^0_2,m^0,\omega^0)$ is smooth contained in the boundary.

When $k=2$, the holomorphic section $q:V_0\rightarrow \Omega^{SZ}\mrs_{g,1}$ is uniquely defined by the implicit condition $\omega\big(q(C,m,\omega)\big)=0$. There are two smooth isoperiodic components intersecting transversely at $(C^0,r^0_1,r^0_2,m^0,\omega^0)$; precisely one lies in the boundary.
\end{proof}
\cbend

\section{Surgeries on stable forms and models of degeneracy}\label{s:moving points in the leaf}

In this section we recall a surgery called Schiffer variation, that allows to construct isoperiodic deformations of forms without zero components on stable curves. We use them mainly to find boundary points that have some node zero residues and no zero components in the closure of any connected component of a fiber of $\per$.

\subsection{Schiffer variations on stable forms}
\label{s:Schiffer variations} 
A Schiffer variation is a continuous deformation of stable forms with some branch point or node. It can be best described in terms of the associated translation structures. The surgery changes the translation structure associated to a form on some small neighbourhood in the surface without changing it on the complement. By taking representatives of the cycles in the homology group that avoid that open set, it is easy to see that the period maps before and after the surgery coincide. There are several instances of the surgery that produce differences on the underlying surface and the translation structure. For instance it allows to deform the projective structure without varying the orders and number of zeros and nodes. It also allows to split a multiple branch point into simpler branch points; at last it allows to produce a smooth surface with several branch points from a node. The surgery operation is invertible (the surgery is actually involutive under suitable restrictions) so the inverse operations can be used to join branch points and to produce nodal curves. They were first considered by Schiffer in~\cite{Schiffer}. A detailed discussion of Schiffer variations on projective structures the can be found in ~\cite{bps}. We will introduce it only for the case of branched translation structures and will use them to prove that any abelian differential on a smooth curve can be joined to an abelian differential on a stable curve with some node by a sequence of such surgeries. 

Let $(C_0,m_0,\omega_0)$ be a marked abelian differential on a nodal curve $C_0$ and $q$ be point where $0<\text{ord}_q(\omega_0)< \infty$. Remark that the chart $\phi=\phi_q$ defined by $\omega_0$ around $q$ can be analytically continued along any path in $C_0$. Let $\gamma_1$ and $\gamma_2$ be two embedded paths in $C_0$ starting at $q$ parametrized by $[0,1]$ whose images do not intersect outside the endpoints. For our purposes it is important to stress that both endpoints \emph{can} coincide with each other, and/or with the starting point. We say that $\gamma_1$ and $\gamma_2$ are {\bf twin paths} if none passes through two distinct nodes and the continuation $\phi_i$ of $\phi$ along $\gamma_i$ satisfies that $\phi_1\circ \gamma_1(t)=\phi_2\circ\gamma_2(t)\in\mathbb{C}$ for all $t$ and $t\mapsto \phi_i\circ \gamma_i(t)$ is a simple path in $\mathbb{C}$. Since the chart $\phi_q$ is a branched covering of degree $\text{ord}_q(\omega_0)+1$ each path starting at $q$ has $\text{ord}_q(\omega_0)$ candidates to be its twin paths. 
Remark that the restriction of a pair of twin paths parametrized by $[0,1]$ to some sub-interval $[0,t_0]$ is still a pair of twin paths.

Given a pair of twin paths $\gamma_1,\gamma_2$ in $C_0$, we can use the equivalence between abelian differentials and atlases formed by branched coverings over open sets in $\mathbb{C}$ and with transitions in the set of translations $z\mapsto z+ \const$ to produce a new abelian differential $\omega_1$ by cutting and pasting the twin paths as follows. Let $U_0$ be a small neighbourhood of $\gamma_1\cup\gamma_2$ where the $\phi_i$'s are defined. Cut $U_0$ along the segments $\gamma_1$ and $\gamma_2$ and glue the boundary on the left (resp. on the right) of $\gamma_1$ to the boundary on the right (resp. on the left) of $\gamma_2$ by identifying points that have the same image for $\phi_i$(see Figure \ref{fig:example1}). The fact that none of the twins passes through distinct nodes implies that at any point $m$, the new germ $(V,m)$ of singular surface has the property that $V\setminus m$ has one or two components. This implies that the new singular surface has only regular points and nodes. Let $U_1$ be the set of points obtained from $U_0$ after the gluing. The new nodal surface is equipped with a family of local branched coverings. Indeed, on the complement of $U_1$ we consider the family of branched coverings given by integrating $\omega$. On $U_1$ we consider the branched covering defined by the $\phi_i$'s after the new identification. This family of branched coverings is translation invariant and defines a stable form $\omega_1$ on some stable curve $C_1$, that has a non-zero component whenever the component has points belonging to $U_1$. The number of points where the local covering has degree at least two is finite. The order of $\omega_1$ at a point that has not been glued is the same as in $\omega_0$. The zeros or nodes of $\omega_1$ that lie on the glued points, can only appear at points corresponding to endpoints of the twin paths, or to zeros of $\omega_0$ lying on the twins. By construction, the total order of the zero divisor of $\omega_1$ on $U_1$ is the same as that of $\omega_0$ on $U_0$. Therefore, the genus of $C_1$ is the same as the genus of $C_0$.


Remark that the twin paths $\gamma_1,\gamma_2$ that we cut in $\omega_0$ produce a pair of paths in $C_1$. After inverting their orientation we get a pair of twin paths $\tilde{\gamma}_1,\tilde{\gamma}_2$ for $\omega_1$. Cutting and pasting the twin paths $\tilde{\gamma}_1,\tilde{\gamma}_2$ in $\omega_1$ we obtain $(C_0,m_0,\omega_0)$ back. Therefore the inverse surgery is naturally defined in the same manner. We will use this inversion of surgery especially when the base point of the twin paths is a node.

Suppose $\gamma_1,\gamma_2$ are twin paths of a stable form $\omega_0$ starting at a point $q$ and ending at points $q_1$ and $q_2$ respectively. Suppose that only the starting point and endpoints can be zeros or nodes of $\omega_0$. We are going to describe the structure of the twin paths $\tilde{\gamma}_1,\tilde{\gamma}_2$ of the form $\omega_1$ in some cases. Denote by $\tilde{q}$ the starting point of $\tilde{\gamma}_i$ and $\tilde{q}_i$ the final endpoint of $\tilde{\gamma}_i$. 

\underline{Example 1}: Suppose $q$, $q_1$ and $q_2$ are pairwise distinct points and $q_1$ and $q_2$ lie in the same component of $C_0^*$. Then the same is true for $\tilde{q}$, $\tilde{q}_1$ and $\tilde{q}_2$ (see Figure \ref{fig:example1}).

\begin{figure}[httb]
\centering
\def\svgwidth{\columnwidth}
\includegraphics[width=5in]{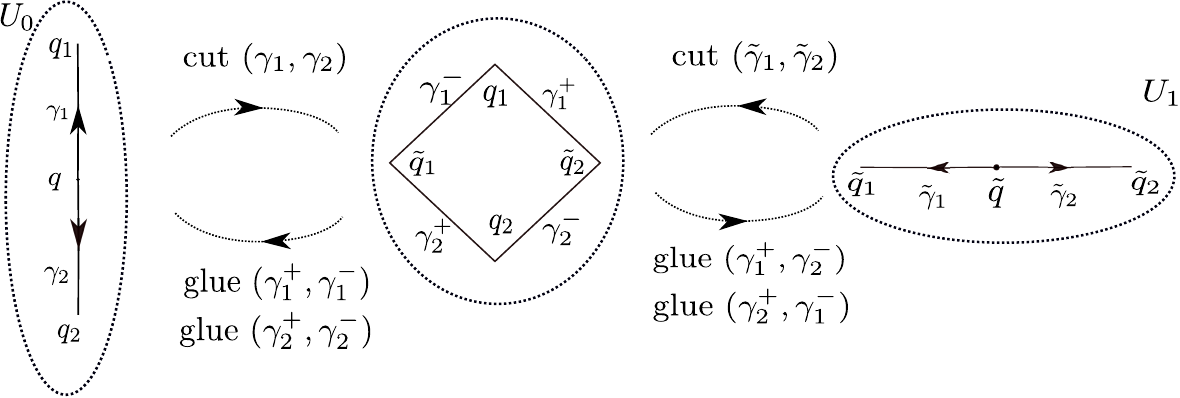}
 \caption{Cut and paste along twin paths contained in one part of the stable curve} \label{fig:example1}
\end{figure}

Comparing the local degrees of the branched coverings at each point we get: $$\text{ord}_{\tilde{q}}(\omega_1)+1=(\text{ord}_{q_1}(\omega_0)+1)+(\text{ord}_{q_2}(\omega_0)+1);$$ $$\text{ord}_q(\omega_0)+1=(\text{ord}_{\tilde{q}_1}(\omega_1)+1)+(\text{ord}_{\tilde{q}_2}(\omega_1)+1).$$
By taking appropriate combinations of twins and zeros we can do several types of changes to the zero divisor. For instance, if $q$ is a simple zero and $q_1,q_2$ are regular points, the surgery simply changes the position of the simple zero. We can also {\bf split a multiple zero}, meaning that two or three zeros of $\omega_1$ are produced from a single zero of $\omega_0$. It suffices to take any pair of twins starting at that zero with distinct endpoints. 

The inverse operation of the latter will be referred to as {\bf joining} two or three different zeros in one zero. This is the case, for instance, when two twins leaving a simple zero have distinct endpoints and at least one of them is a zero.

\underline{Example 2}: Suppose that $\gamma_1$ and $\gamma_2$ leave $q$ following different branches of $C_0$ at $q$. Then $q$ is a node and there are three interesting sub-cases that will be referred to as {\bf smoothing of a node}: 

\begin{enumerate}

\item \label{it:example node1} If $q_1,q_2$ and $q$ are pairwise distinct, then $\tilde{q_1}=\tilde{q}_2\neq \tilde{q}$, none of them is a node, and they are zeros of $\omega_1$. Each $\tilde{\gamma}_i$ joins the same pair of distinct zeros of $\omega_1$ (see Figure \ref{fig:example2_1}). 

 \begin{figure}[httb]
\centering
\def\svgwidth{\columnwidth}
\includegraphics[width=4in]{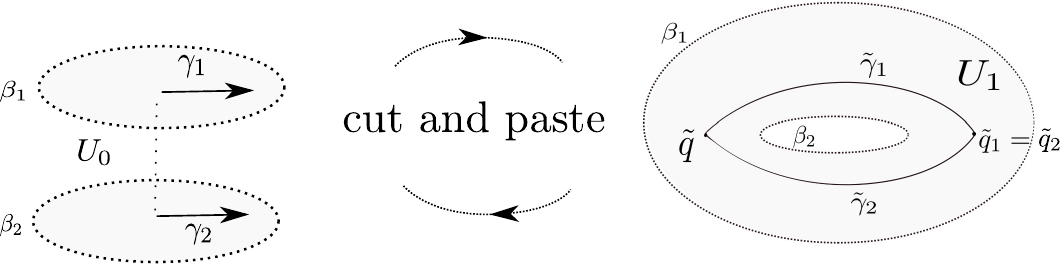}
 \caption{Cut and paste along twin paths contained in different branches of a node to smoothen the node and obtain a couple of twins between zeros} \label{fig:example2_1}
\end{figure}

\item If $q_1=q\neq q_2$ and $\gamma_1$ does not come back to $q$ through a different branch of $C_0$, then $\tilde{q}$ is not a node of $\omega_1$ and $\tilde{q}=\tilde{q}_1=\tilde{q}_2$, i.e. both $\tilde{\gamma}_1$ and $\tilde{\gamma}_2$ are closed loops based at a zero of $\omega_1$ (see Figure \ref{fig:example2_2}). Moreover, by considering the intersection of the twins with a neighbourhood of $\tilde{q}$ we observe that among the oriented segments there are two consecutive segments in the cyclic order with the same sign (either both enter or both leave the singularity). 

\begin{figure}[httb]
\centering
\def\svgwidth{\columnwidth}
\includegraphics[width=3.5in]{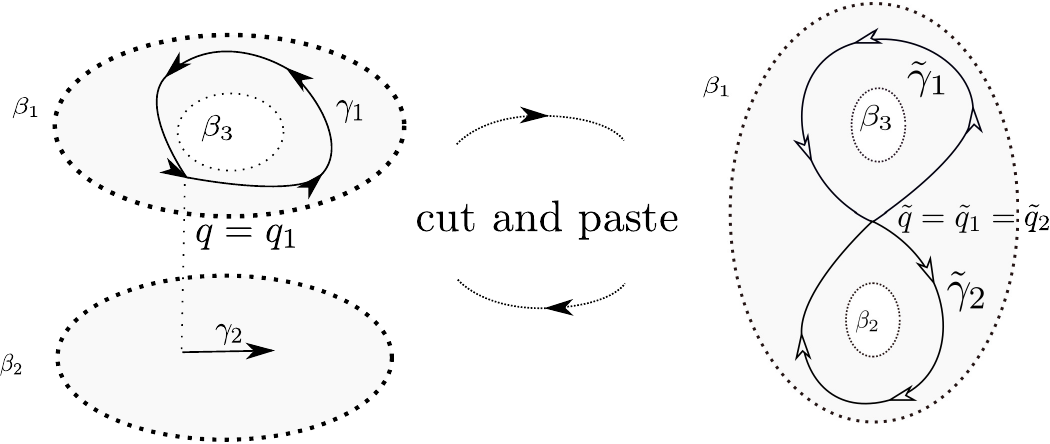}
\caption{$\gamma_1$ is a closed path that leaves and arrives to the node through one of the branches. After the surgery, there is no node, and a couple of closed twin paths based at a zero $\tilde{q}$.}
\label{fig:example2_2}
\end{figure}

\item If $q_1=q\neq q_2$ and $\gamma_1$ returns to $q$ through the other branch of the node, then the conclusion is the same as in case (2). The only difference with that case is that the cyclic order changes (see Figure \ref{fig:example2_3}).

 \begin{figure}[httb]
\centering
\def\svgwidth{\columnwidth}
\includegraphics[width=4in]{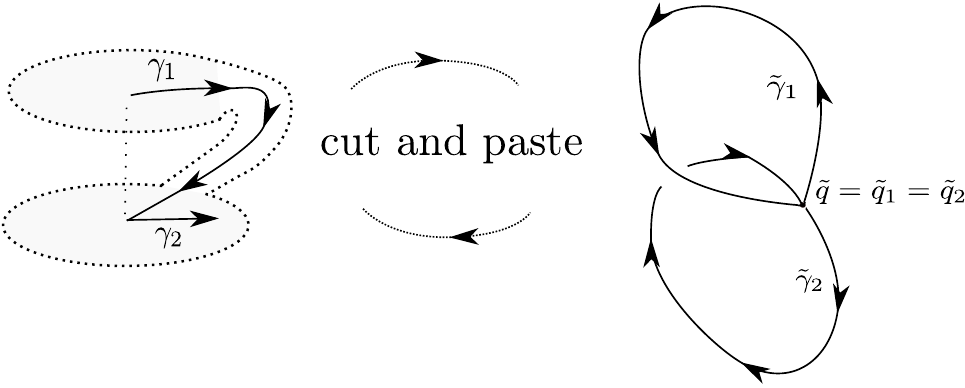}
\caption{$\gamma_1$ is a closed path based at the node. Its starting segment and end segment lie in different branches. After the surgery, there is no node, and a couple of closed twin paths based at a zero $\tilde{q}$}
\label{fig:example2_3}
\end{figure}
\end{enumerate}

\cbstart
The inverse surgeries of those examples allow to construct stable forms on nodal curves starting from certain stable forms on smooth curves. In Figure \ref{fig:degeneration genus 2} we depict an example in genus two joining a form in the minimal stratum with a form with a node from the point of view of the description of a form as a translation surface

 \begin{figure}[httb]
\centering
\def\svgwidth{\columnwidth}
\includegraphics[width=5in]{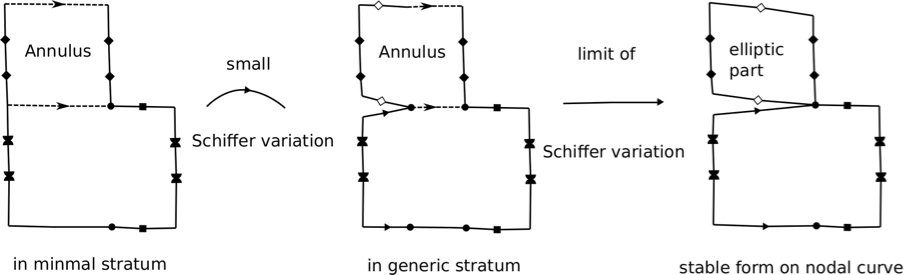}
\caption{The dotted lines represent  pairs of twins in the smooth genus two surface. After the Schiffer variation along them we obtain a stable form with a node.}
\label{fig:degeneration genus 2}
\end{figure}

In figure \ref{fig:twins for degeneration} we provide an example of combinatorics of a pair of closed twins whose Schiffer variation does not lead to a nodal curve.  

 \begin{figure}[httb]
\centering
\def\svgwidth{\columnwidth}
\includegraphics[width=3in]{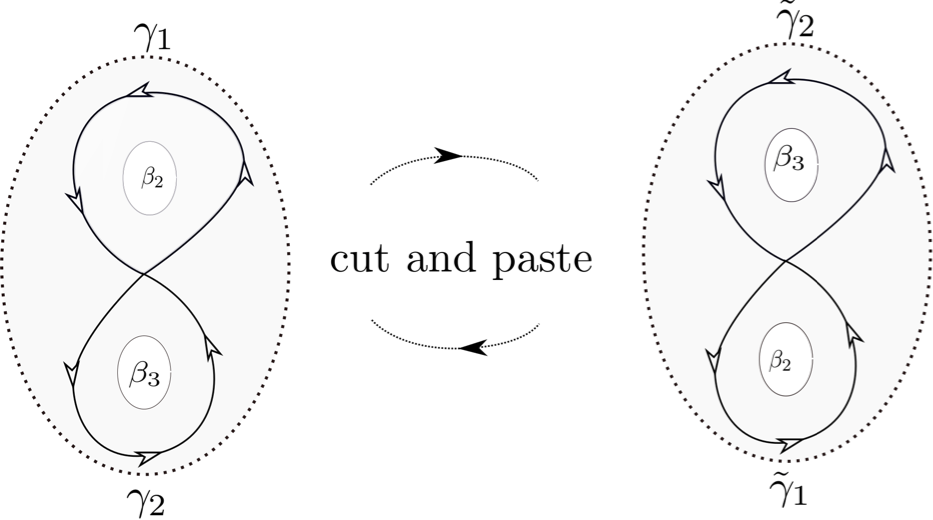}
\caption{The Schiffer variation along a pair of closed twins as in Lemma \ref{l:twins for degeneration} exchanges the relative position of $\beta_2$ and $\beta_3$, but no node appears.}
\label{fig:twins for degeneration}
\end{figure}

For future reference we state: 

\begin{lemma}\label{l:twins for degeneration}
If a pair of twin paths based at a zero share starting point and endpoint then either the cut and paste surgery produces a stable form with an additional node, or both start and endpoint coincide, and in the cyclic order at that point, any pair of consecutive oriented segments of the twins have opposite sign (as in Figure \ref{fig:twins for degeneration}) . 
\end{lemma}
\cbend

\begin{proof}
The cases with a node correspond to the inverse surgeries of the smoothing of a node (Example 2). The last case described in the lemma does not produce a node, since the base point obtained after the surgery does not separate the germ of surface in two connected components. 
\end{proof}

The surgery can be generalized to families of twin paths. We call a family of paths $\Gamma=\{\gamma_1,\ldots, \gamma_k\}$ a family of twin paths for a form $\omega_0$ if every pair forms a pair of twin paths for $\omega_0$. The cut and paste surgery associated to the family $\Gamma$ cuts each $\gamma_i$ producing two sides of each cut, and glues each to its adjacent twin side. \cbstart  When $\Gamma$ is maximal and all points in the paths are regular except for the initial point, that is a zero of order $k-1$, the surgery provides a deformation of the form in its substratum. We can use the families of twins to {\bf join} zeros of any order:\cbend 
\cbstart
\begin{lemma}\label{l:joining zeros}
Let $\Gamma=\{\gamma_1,\ldots,\gamma_k\}$ be a \emph{maximal} family of twin paths starting at a zero of order $k-1$ of a stable form $\omega_0$ . Suppose that all points in the paths are regular except for their common starting point and the endpoint of $\gamma_1$, that is a zero of $\omega_0$ distinct from the starting point. Then the form $\omega_1$ obtained from the cut and paste surgery associated to $\Gamma$ has less zeros than $\omega_0$ (See Figure \ref{fig:joining_zeros} for an example)
\end{lemma}

 \begin{figure}[httb]
\centering
\def\svgwidth{\columnwidth}
\includegraphics[width=4.5in]{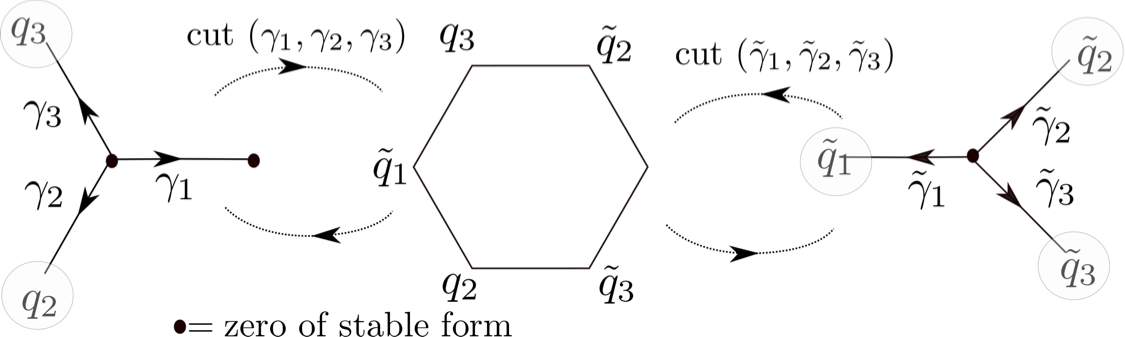}
 \caption{On the left a maximal family of twin paths at a double zero. Both endpoints of $\gamma_1$ are zeros. The maximality implies that, after the Schiffer variation, the points $\widetilde{q_j}$ have angle $2\pi$ (the circle around them represents this angle). The obtained form has one zero less.} \label{fig:joining_zeros}
\end{figure}
\cbend
\begin{proof}
After the surgery all the endpoints of the initial twins are glued together to form a single zero. The angle between any pair of adjacent twins is $2\pi$. Therefore the starting point of the twins splits into $k$ regular points after the surgery. All other glued points are regular and the conclusion follows.
\end{proof}

Since the restriction of a family of twin paths parametrized by $[0,1]$ to some interval starting at time zero is still a family of twin paths, the previous construction gives a stable form for each $t\in[0,1]$. We thus obtain a parametrized family $\{(C_t,\omega_t)\}_{t\in[0,1]}$ of stable forms. Moreover, at any $t\in(0,1)$ there are two natural families of twin paths on $(C_t,\omega_t)$. On the one hand there is the family that inverts the surgery from $(C_0,\omega_0)$ to $(C_t,\omega_t)$; on the other, there is the family of twin paths that correspond to the segments of the original twin paths parametrized by $[t,1]$. The cut and paste surgery applied to the latter family of twin paths on $(C_t,\omega_t)$ produces the same stable form $(C_1,\omega_1)$. With this descritption at hand, items (ii) and (iii) of Example 2 are just a concatenation of Example 1 and the inverse of item (i) in Example 2. To compare the periods of the constructed stable forms we need to produce a (homological) marking $m_t$ on each $C_t$, starting from a marking $m_0$ on $C_0$. It suffices to do that under the hypothesis of Example 1 and of the inverse of item (i) in Example 2.

In the case of Example 1 (see Figure \ref{fig:example1}), the neighbourhood $U_0$ of the twins $\gamma_1\cup \gamma_2$ can be taken as a topological disc. By construction there is a natural identification between $C_0\setminus U_0$ and $C_1\setminus U_1$. The map $\partial U_0\rightarrow \partial U_1$ has a unique extension up to isotopy to a homeomorphism $\overline{U_0}\rightarrow\overline{ U_1}$. The gluing of those two identifications produces a homeomorphism $C_0\rightarrow C_1$ that induces, by postcomposing it at the homology to $m_0$, a marking $m_1$ on $C_1$. With this marking it is readily verified that $$\percompact (C_0,m_0,\omega_0)=\percompact (C_1,m_1,\omega_1).$$
In the case of the inverse of item (i) in Example 2, we have the couple of twins $\tilde{\gamma}_1,\tilde{\gamma}_2$ that join two distinct zeros of a stable form $(C_1,\omega_1)$ on a smooth curve $C_1$ (the component they belong to). The neighbourhood $U_1$ is the annulus depicted in Figure \ref{fig:example2_1} surrounding the twins. The neighbourhood $U_0$ is a neighbourhood of the node that contains both twins $\gamma_0\cup \gamma_1$. There is a natural identification $C_1\setminus U_1\rightarrow C_0\setminus U_0$. Each boundary component of $U_1$ is sent to the boundary of a disc contained in one of the branches of the node, giving an orientation preserving homeomorphism $\partial U_1\rightarrow \partial U_0$. By arguing as before with each side of the pair of twins, up to isotopy there is a unique collapse $\overline{U_1}\rightarrow \overline{U_0}$ that sends the twins to the node and has the given values on the boundary. If there is a marking $m_1$ on $C_1$, it can be postcomposed to the action of the collapse $C_1\rightarrow C_0$ on homology, giving a homological marking $m_0$ of $C_0$. Again, the equality $\percompact (C_0,m_0,\omega_0)=\percompact (C_1,m_1,\omega_1)$ is easily verified.

With the given markings the parametrized family produces a map $[0,1]\rightarrow\Omega\mnc_g$ whose composition with the projection $\Omega\mnc_g\rightarrow \mnc_g$ to the space of marked curves $t\mapsto (C_t,m_t)\in\mnc_g$ is continuous by construction. On the other hand we claim that there exists a local trivialization of the fiber bundle $\Omega\mnc_g\rightarrow\mnc_g$ around the point $(C_t,m_t)$ where isoperiodic sets are contained in the constant sections of the local trivialization. Indeed, $\ker m_t$ is isotropic and we can find a symplectic basis $\{a_i,b_i: i=1,\ldots,g\}$ of $H_1(\Sigma_g)$ such that $\ker m_t$ is contained in the Lagrangian generated by $a_1,\ldots, a_g$.
The map $\Psi: \Omega U\rightarrow U\times\mathbb{C}^{g}$ defined by $$\Psi(C,m,\omega)=((C,m), \int_{m(a_1)}\omega,\ldots,\int_{m(a_g)}\omega)$$ is continuous thanks to Lemma \ref{l:holomorphic integral}. It preserves the fibers and it is also linear in restriction to each fiber. This restriction is injective since every form satisfying $\Psi(C,m,\omega)=(C,m, 0,\ldots,0)$ has zero residues at all non-separating nodes, hence at all nodes. In section \ref{ss: image of the period map} we will see that the condition implies that $\vol{\omega}=0$ and this volume is also the sum $\sum \vol(\omega_i)$ where $\vol(\omega_i)$ denotes the (positive) volume of the restriction $\omega_i$ of $\omega$ to the part $C_i$ of $C$. Hence $\vol(\omega_i)=0$ for all $i$. However a holomorphic form $\omega_i$ on a smooth curve $C_i$ has zero volume if and only if it is the zero form. Hence, the only possibility is that $\omega$ is zero everywhere. The map $\Psi$ is therefore a trivialization of the bundle $\Omega U\rightarrow U$. In this local trivialization the fibers of $\percompact$ are contained in the constant horizontal sections, which are continuous. Hence, the map $[0,1]\mapsto \Omega\mnc_g$ is continuous. The same argument works for a family of twin paths. 

\cbstart
\begin{definition}Let $\Gamma=\{\gamma_1,\ldots, \gamma_k\}$ be a (not necessarily maximal)  family of twin paths for $(C_0,m_0,\omega_0)$ parametrized by $t\in[0,1]$ and set $p=\percompact(C_0,m_0,\omega_0)$. The {\bf Schiffer variation} along $\Gamma$ is the continuous path \begin{equation} t\mapsto(C_t,m_t,\omega_t)\in\percompact ^{-1}(p)\label{eq:path_schiffer_variation}\end{equation}
where $(C_t,m_t, \omega_t)$ is obtained from $(C_0,m_0,\omega_0)$ by cutting and pasting the family of twins $\{\gamma_{1|[0,t]},\ldots, \gamma_{k|[0,t]}\}$. 
\end{definition}
\cbend
In the case of translation structures there is naturally a distinguished family of twin paths, namely those that have image in straight lines in $\mathbb{C}$. This produces a sub-family of paths in the fibers of $\percompact$. 
However, connectedness by paths in either of the families is equivalent to connectedness in the case of fibers of $\percompact$. The advantage of using straight paths is that the candidates to twins are invariant of the associated directional foliation. This allows to have some control on the embedding/intersection properties. In particular, a pair of geodesic segments of leaves of one of the directional foliations leaving the same singularity and having the same length, form a pair of twin paths. Let us analyze with more detail the oriented directional foliations.

\subsection{Periodic annuli}\label{ss:directionalfoliations}

The dynamics of each oriented directional foliation $\mathcal{G}_{\theta}$ induced by an abelian differential $\omega$ without residues or zero components on a smooth curve $C$ is well known. Indeed, by Maier's Theorem (see \cite{Maier} or \cite{Strebel}) there exist a finite number of saddle connections, that is, leaves $\gamma_1,\ldots, \gamma_n$ such that each $\gamma_{i}$ converges to some singular point in the positive direction and to some singular point in the opposite direction. Each component of $C\setminus\cup\gamma_i$ is saturated by $\mathcal{F}$ and is either a {\bf periodic annulus}, i.e. an annulus formed of closed leaves/geodesics, or {\bf minimal }, i.e. each leaf in the component is dense in the component. 

The length of all leaves in a periodic annulus is the same and coincides with the length of each of its boundary components. We orient each boundary component with the orientation of the foliation (not that induced by the orientation of the annulus). We call $b^{+}$ the boundary component that has the annulus to its right, and $b^{-}$ the boundary component that has the annulus on its left. Each boundary component of the annulus is identified with an ordered cycle of saddle connections $(\gamma_1^{\pm},\ldots,\gamma_k^{\pm})$ each starting at the singular point where the former ended. The angle that a saddle connection forms with the following is of $\pi$ to the corresponding side (right for $b^{+}$, left for $b^{-}$). Thus the intersection of the annulus with a neighbourhood of a zero of the abelian differential is a (possibly empty) family of sectors of angle $\pi$.

We recall an important existence result: 
\begin{theorem}[Masur, \cite{Masur}]
 For any non-zero abelian differential on a smooth curve there exists a periodic annulus. 
\end{theorem}

\subsection{Degeneracy}
We will use Schiffer variations to show the following
\begin{proposition}\label{p:degeneration}
 Let $g\geq 2$ . Given  $(C,m,\omega)\in\Omega^{*}\mrs_{g}$ there exists a finite sequence of Schiffer variations joining it to a stable form with one node and no zero components. If  $\omega$ has a single zero we can suppose that the node separates the stable curve, and has an elliptic component.\end{proposition}

\begin{proof}
Let $(C,m,\omega)\in\Omega^{*}\mrs_g$. We will apply Schiffer variations to connect it to a point in $\Omega\mrs_g$ having a configuration of twin paths as in the first alternative of Lemma \ref{l:twins for degeneration}: two twin paths joining the same pair of zeros with a pair of non-alternating signs whenever the zeros coincide. In fact the twins that we will find are geodesic saddle connections. The application of the Lemma then allows to connect the latter to a point in the boundary strata via the associated Schiffer variations. Since Schiffer variations are paths in the fibers of $\percompact$, the proposition follows.

 First we will treat three cases where the boundary of some periodic annulus for $\omega$ provides the twins.  In those three cases the Schiffer variation leads to a stable form on a curve with one separating node and an elliptic component. Denote by $b^+$ and $b^-$ the boundary components of some periodic annulus $A$ of $\omega$.
 
 \underline {Case 1:} If $b^+$ and $b^-$ are closed saddle connections at a zero $q$ of $\omega$, meaning that each is formed by a single saddle connection starting and ending at $q$. Then they cannot both coincide. Indeed, if they did, the surface would have genus one, which is not the case. Therefore they are different and the extension of the chart of $\omega$ at $q$ along the saddle connections sends each of them to a segments of oriented straight line of the same length and direction. This implies that they form a pair of twin paths for $\omega$. To be able to apply Lemma \ref{l:twins for degeneration} it remains to check that there is a consecutive pair of separatrices at $q$ with the same sign. Since the annulus lies to the right of $b^+$ and to the left of $b^-$ the pair of twin segments leaving the singularity are consecutive in the cyclic order and we conclude. \cbstart
 The situation is similar to the one depicted in Figure \ref{fig:degeneration genus 2}.
 \cbend
 
 \underline{Case 2:} Suppose one of the saddle connections of $b^{+}$ is closed, and coincides with one of the saddle connections of $b^{-}$. Call the saddle connection $\gamma$. In this case we can find another periodic annulus satisfying the conditions of Case 1. Think of the universal cover of $A$ as an infinite horizontal band $\widetilde{A}\subset \mathbb{C}$. In each boundary component of $\widetilde{A}$ we have copies of $\gamma$ and they all point in the same direction (see Figure \ref{fig:annulus_cover}). Choose a copy on each boundary component of the band, and draw the parallel geodesics (straight lines) in the band joining corresponding points in the chosen copies of $\gamma$. By construction, for the foliation given by the direction of the geodesics we get a periodic annulus in $\omega$ whose boundary components are each formed by one closed saddle connection, as in Case 1.

\begin{figure}[httb]
\centering
\def\svgwidth{\columnwidth}
\includegraphics[width=3in]{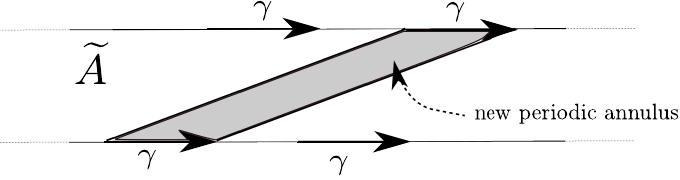}
 \caption{Case 2: Universal cover of the initial annulus $A$} \label{fig:annulus_cover}
\end{figure}

\underline{Case 3:} If $b^+$ and $b^-$ only pass through one zero $q$ of $\omega$, and at least one of them is formed by two saddle connections. The only case that is left is the case where none of the saddle connections of $b^+$ and $b^-$ coincide. We focus on the distribution of $\pi$-sectors described by the annulus at the zero $q$ (see Figure \ref{fig:saddle_colours}). 
 \begin{figure}[httb]
\centering
\def\svgwidth{\columnwidth}
\includegraphics[width=2in]{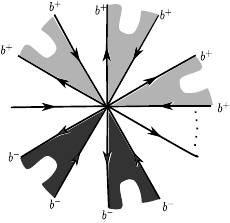}
 \caption{Case 3: The $\pi$-sectors corresponding to the boundary of $A$ at the saddle $q$ are coloured black if they correspond to sectors of $b^-$ and light grey if they correspond to sectors of $b^+$. Consecutive sectors of distinct colour have an even number of white $\pi$-sectors between them. If they are coloured the same there is an uneven number of white sectors between them.} \label{fig:saddle_colours}
\end{figure}

Some of the $\pi$-sectors correspond to $b^+$ and some to $b^-$. Draw them in two different colors, say black and grey. By hypothesis the intersections of the closures of any two such sectors are empty, since otherwise we would be in Case 2. In the cyclic order at $q$ there is a pair of such $\pi$-sectors of distinct color that are consecutive (among the coloured sectors). Between two such consecutive sectors of different colours there is necessarily an even number number of uncoloured $\pi$-sectors. Hence, up to changing the orientation of the directional foliation if needed, we can consider that the consecutive saddle connections $\gamma^+$ and $\gamma^-$ that lie consecutively in boundaries of those two sectors of different colors both leave the singularity. If the lengths of $\gamma^+$ and $\gamma^{-}$ coincide, they are twin paths satisfying all the conditions of the first alternative of Lemma \ref{l:twins for degeneration}. Otherwise, we consider the shortest among $\gamma^+$ and $\gamma^-$ and its twin path starting along the other. The Schiffer variation along this pair of twin paths produces a form that has a periodic annulus as in Case 2. 

This said we proceed to prove the Proposition by induction on the number of zeros of the form. 

Suppose first that $\omega$ has one zero. Then all the saddle connections appearing in the boundary of a periodic annulus start and end at the zero. Therefore we fall in one of Cases 1,2 or 3 described above.  Since in all cases we reached a stable form on a curve of compact type this proves the final statement of Proposition \ref{p:degeneration}. 

Suppose that we have proved Proposition \ref{p:degeneration} for all marked forms with $n$ zeros and let us prove it for a form $\omega$ with $n+1$ zeros on a smooth curve $C$. 

Let $z_1,\ldots,z_{n+1}$ be the zeros of $\omega$. Up to moving them via Schiffer variations we can suppose that the number $$d_{min}=\min_{j\neq k}\{\text{dist}(z_j,z_k)\}$$ is realized by a single pair of zeros, say $z_1,z_2$. We can further move the point $z_2$ to guarantee that $\int_{z_1}^{z_2}\omega$ is not in the countable set $\{\int_a\omega\in \mathbb{C}:a\in H_1(C)\}$.

Consider one of the geodesics $\gamma_1$ between $z_1$ and $z_2$ that realizes the distance between them. It is a saddle connection. If the zero $z_1$ has order $k$ there are $k+1$ oriented geodesics leaving $z_1$ in the same direction as $\gamma_1$. We extend each until distance $d_{min}$. By construction they do not pass through zeros of $\omega$, except for (maybe) at their endpoint, where they can only reach $z_2$. Indeed, if either reached $z_j$ with $j\neq 2$ one of the conditions of the definition of $z_1,z_2$ would be violated.

If one of the twins ends at $z_2$ we are done by Lemma \ref{l:twins for degeneration}. Otherwise the family of twin paths satisfies the hypothesis of Lemma \ref{l:joining zeros}. We use the Lemma to connect $\omega$ to a form with less zeros. The inductive hypothesis then concludes the proof. 
\end{proof}
For future reference we state 
\begin{corollary}\label{c:nonempty boundary iff haupt}
Let $p:H_1(\Sigma_g)\rightarrow \mathbb{C}$ be a homomorphism. Then there exists an abelian differential $(C,m,\omega)$ on a smooth curve $C$ satisfying $p=\per(C,m,\omega)$ if and only if the same is true for some stable form without zero components defined on a singular curve. 
\end{corollary}
\begin{proof}
Suppose first that there exists a marked abelian differential $(C,m,\omega)$ on a smooth curve $C$ satisfying $\per(C,m,\omega)=p$. Then Proposition \ref{p:degeneration} implies that there also exists an abelian differential with no zero component on a marked stable curve with a node, having periods $p$. \newline
For the converse, suppose that there exists a point in a boundary stratum contained in $\percompact^{-1}(p)$. By smoothening the nodes isoperiodically we obtain a form on a smooth curve. Alternatively, we can argue by Theorem \ref{p:local structure of fiber at bdry point}: the local fiber $L$ of $\percompact$ around this point satisfies that $L\setminus \partial L$ is nonempty. A point in the latter set provides a marked abelian differential on a smooth curve with period homomorphism given by $p$. 
\end{proof}

\section{Period homomorphisms of stable forms}

\label{s: isoperiodic sets on curves with one node}

\subsection{The image of the period map $\percompact_{g,0}$}
\label{ss: image of the period map}
Recall from subsection \ref{ss:singular flat metric} that a stable form $(C,\omega)$ induces a flat singular metric $\omega\otimes\overline{\omega}$ on $C^*$ with volume \[\vol(\omega)=\frac{i}{2}\int_C\omega\wedge\overline{\omega}.\] In particular, $0\leq \vol(\omega)\leq\infty$ and it is finite if and only if all the residues of $\omega$ at the nodes of $C$ are zero. If $\vol(\omega)=0$ then $\omega$ is the zero form.

A second fact more directly related with periods is that whenever $\omega$ has no zero components on $C$, and the set of periods $\Lambda=\{\int_a\omega\in\mathbb{C}: a\in H_1(C)\}$ is a lattice in $\mathbb{C}$,  integration along paths produces a branched covering \begin{equation}\pi:C\rightarrow \mathbb{C}/\Lambda \text{ whose primitive degree satisfies } \deg (\pi)=\frac{\vol(\omega)}{\vol(\mathbb{C}/\Lambda)}.\end{equation} In particular if the genus of $C$ is at least 2, $\deg (\pi)>1$.

Now, by Riemann's relations, the volume of $\omega$ can be calculated by using the periods of $\omega$. Indeed, for any symplectic basis $\{a_1,b_1,\ldots, a_g,b_g\}$ of $H_1(C)$ we have
 \begin{equation*}
 \vol(\omega)=\int _{C}\Re \omega \wedge \Im \omega = \frac{i}{2}\int_{C}\omega\wedge\overline{\omega}=\frac{i}{2}\sum_{j=1}^g\int_{a_j}\omega\int_{b_j}\overline{\omega}-\int_{a_j}\overline{\omega}\int_{b_j}\omega= \end{equation*} $$=\sum_{j}\Im(\overline{\int_{a_j}\omega}\int_{b_j}\omega))$$
 
 The positivity of both terms implies that if the periods of a form are discrete, they cannot be contained in a real line and thus define a lattice.
 
 We extend the definitions of volume and degree to homomorphisms on symplectic modules:

 \begin{definition}
 Given a unimodular symplectic $\mathbb{Z}$-module $M$ and a homomorphism $p:M\rightarrow \mathbb{C}$ we define 
 \begin{itemize}
 \item The volume of $p$ as $\vol(p):=\Re (p)\cdot \Im(p)$ where the intersection is on the dual space $\text{Hom}(M,\mathbb{R})$. Up to a choice of a symplectic basis $\{a_j,b_j\}$ of $M$, the volume can be calculated as $$\vol(p)=\sum_{j}\Im(\overline{p(a_j)}p(b_j))$$
 \item Define the primitive degree of $p$, denoted $\deg(p)$ as $\infty$ if $p(M)$ is non-discrete and as $$\deg (p)=\frac{\vol(p)}{\vol(\mathbb{C}/p(M))}\text{ if } p(M)\subset \mathbb{C} \text{ is discrete. }$$ 
 \end{itemize}
For any symplectic submodule $V\subset M$ we denote $\vol_p(V):=\vol(p_{|V})$
\end{definition}

Remark that if $0<\deg(p)<\infty$, the kernel of $p$ has necessarily co-rank two. 

If $(C,m,\omega)$ is a homologically marked stable form with zero residues and no zero components on a marked stable curve $(C,m)$, then the homomorphism $p=\percompact(C,m,\omega)\in H^1(\Sigma_g,\mathbb{C})$ has the following properties:
\begin{itemize}
 \item[($H_1$)] $\vol (p) >0$ and
 \item [($H_2$)]If $g\geq 2$ , $\deg(p)>1$. 
 \end{itemize}

Haupt proved in~\cite{Haupt} that conditions $(H_1)$ and $(H_2)$ are in fact also sufficient for a homomorphism
$p:H_1(\Sigma_g)\rightarrow\mathbb{C}$ to be the period of a non-zero abelian differential on a smooth curve. For genera $g=2,3$ we have already given an argument in Theorem \ref{t:connectedness in genus 2 and 3}. We will give an alternative proof of this theorem in subsection \ref{ss:proofs of Haupt}.

If we allow the non-zero form to have some zero components, there is a case where $(H_2)$ is not satisfied: when there is a single non-zero component and it has genus one.
\subsection{Haupt homomorphisms}

Condition $(H_2)$ of a homomorphism has other equivalent statements that will be useful, and show that the  said exception is unique: 

\begin{proposition}
\label{p:Haupt characterizations}
Let $M$ be a unimodular symplectic module of rank $2g$ and $p:M\rightarrow \mathbb{C}$ be a homomorphism. Suppose $\vol(p)>0$ and $\Lambda=p(M)\subset\mathbb{C}$ discrete. Then the following are equivalent:
\begin{enumerate}
\item $\vol (p)\leq \vol (\mathbb{C}/\Lambda)$;
\item $\deg (p)=1$;
\item $\ker p$ is a symplectic submodule of $M$ (of rank $2g-2$).
\end{enumerate}
\end{proposition}

\begin{proof}

First remark that by virtue of Riemann's relations we have $\vol(p)\in\mathbb{Z}\vol(\mathbb{C}/\Lambda)$.
If we assume $(1)$ then the positivity of $\vol(p)$ implies $(2)$. Obviously $(2)$ implies $(1)$.
For the proof of $(2)\Rightarrow (3)$, we can normalize $p$ by post-composing it by a real linear map to suppose that $\Lambda=\mathbb{Z}^2\subset \mathbb{C}$. The kernels before and after this composition coincide. Up to taking a symplectic basis of $M$, we have $p:\mathbb{Z}^{2g}\rightarrow \mathbb{Z}^2\subset \mathbb{C}$ a homomorphism of positive volume. Then $\vol(p)=p_1\cdot p_2$ where $p=(p_1,p_2)$ and $p_i\in(\mathbb{Z}^{2g})^*$. On the other hand $\ker p=\ker p_1\cap \ker p_2$.
Now, the dual of $p_i$ is an element $u_i\in\mathbb{Z}^{2g}$ and we have the equality $p_1\cdot p_2=u_1\cdot u_2$. If $\vol(p)=\vol(\mathbb{C}/\mathbb{Z}^2)=1$, then $u_1,u_2$ generate a symplectic submodule of rank two. Its orthogonal is a symplectic submodule of rank $2g-2$ that coincides with $\ker p$.



Next suppose $(3)$ and let us prove $(2)$. Recall that $\ker p$ has rank $2g-2$. Since $\ker p$ is a symplectic unimodular submodule of $M$ we can take a symplectic basis $a_1,b_1,\ldots, a_{g-1},b_{g-1}$ of $\ker p$ and complete it to a symplectic basis of $M$ by adding $a_g,b_g$. Apply Riemann's relations to deduce that $0<\vol(p)=\Im(p(a_g)\overline{p(b_g)})$. This implies that $p(a_g)$ and $p(b_g)$ generate a lattice; since they also generate $\Lambda$, we have $\vol(p)=\vol(\mathbb{C}/\Lambda)$.


\end{proof}

\begin{definition}
 Let $M$ be a unimodular symplectic module of rank $2g$ and $p:M\rightarrow \mathbb{C}$ be a homomorphism. We say that $p$ is a Haupt homomorphism if $\vol(p)>0$ and either $g=1$ (and $\deg p=1$) or $g\geq 2$ and $\deg (p)>1$. The set of Haupt homomorphisms in $H^1(\Sigma_g,\mathbb{C})=\text{Hom}(H_1(\Sigma_g),\mathbb{C})$ will be denoted by $\mathcal{H}_g$.
 \end{definition}

\begin{corollary}
 A homomorphism $p:M\rightarrow\mathbb{C}$ on a unimodular module of rank $2g$ with $\vol (p)>0$ is \emph{not} Haupt if and only if $g\geq 2$ and $\ker p$ is a symplectic module of rank $2g-2$.
\end{corollary} 
\subsection{Homological invariants of isoperiodic boundary components}

The image of the extended period map $\percompact_g$ is contained in $\mathcal{H}_g\cup 0$. The fundamental remark for the proof of Theorem \ref{t:connectedness} is related to the fact that the normalization of some node of a form in $\Omega_0^*\mnc_g$ is again a collection of forms in some $\Omega_0^*\mnc_{h,n}$ with $h<g$ and $n\geq 1$. The periods of those forms are intimately related and are subject to satisfy Haupt's conditions. In more detail: 

Suppose $c$ is a simple closed curve in $\Sigma_g$ that is collapsed by the marking of a stable form $(C,m,\omega)\in\Omega_0^*\mnc_g$ to a node and set $p=\percompact(C,m,\omega)\in \mathcal{H}_g$. 

If $[c]=0$ in homology, then the curve is separating and  $C$ is a union $C_1\vee C_2$ of stable curves and the form $\omega=\omega_1\vee \omega_2$ where $\omega_i\in\Omega_0^*(C_i)$. On the other hand $m$ and the decomposition $H_1(C)=H_1(C_1)\oplus H_1(C_2)$ define a symplectic splitting $V_1\oplus V_2$ of $H_1(\Sigma_g)$. The map $m_i=m_{|V_i}$ is a marking of $H_1(C_i)$. The periods of $(C_i,m_i,\omega_i)$ are precisely $p_i:=p_{|V_i}$, which are Haupt homomorphisms of lower genus. Therefore we have a decomposition $$p=p_1\oplus p_2$$
of the Haupt homomorphism $p$ into two Haupt homomorphisms. In this case $$\vol(p)=\vol_p(V_1)+\vol_p (V_2)=\vol(p_1)+\vol(p_2)$$ and each factor of the sum is positive.
On the other hand $$\deg p=\deg p_1\cdot \left \lvert \frac{\image(p)}{\image(p_1)}\right \rvert+\deg p_2\cdot \left \lvert \frac{\image(p)}{\image(p_2)}\right \rvert.$$
If $[c]\neq 0$ then the curve is non-serparating and $p([c])=0$. The normalization of the node produces a stable curve $C_2$ of genus $g-1$ equipped with a form $\omega_2\in\Omega_0^*\mnc_{g-1}$. Let $V_1$ be a rank two symplectic submodule of $H_1(\Sigma_g)$ that contains $[c]$ and $V_1\oplus V_2$ its associated symplectic splitting of $H_1(\Sigma_g)$. Then, $m_{|V_2}$ defines a marking of $H_1(C_2)$. The map $p_2=p_{|V_2}$ corresponds to the periods of $(C_2,m_2,\omega_2)$ and is therefore a Haupt homomorphism. In this case the volume of $p_1=p_{|V_1}$ is equal to zero and we have $$\vol(p)=\vol(p_1)+\vol(p_2)=\vol(p_2).$$ The positivity of $\vol(p_2)$ is guaranteed by that of $\vol(p)$. However for $p_2$ to be a Haupt homomorphism we need to verify the primitive degree condition $(H_2)$. 

\begin{remark}\label{rem:recognizing pinchables}
Given any symplectic submodule $V$ of rank two containing a primitive element $a\in\ker p$, there is a natural symplectic isomorphism between $V^{\perp}$ and $a^{\perp}/\mathbb{Z}a$. Under this identification $p_{|V^{\perp}}$ is equal to the map $$p_a:a^{\perp}/\mathbb{Z}a\rightarrow\mathbb{C}$$ induced by $p$ on the quotient. In particular, if one is a Haupt homomorphism then so is the other.
\end{remark}
\begin{definition}
Let $p:M\rightarrow \mathbb{C}$ be a Haupt homomorphism defined on a symplectic unimodular module. Given a symplectic splitting $V_1\oplus\cdots\oplus V_k$ of $M$, we say that it is $p$-admissible if every $p_{|V_i}$ is a Haupt homomorphism. We say that a primitive element $a\in \ker (p)\setminus 0$ is pinched by $p$ if the induced homomorphism $p_a$ is a Haupt homomorphism. 
\end{definition}

\begin{remark}\label{rem:symplectic in ker is pinchable}
If $p:M\rightarrow\mathbb{C}$ is a Haupt homomorphism, then every primitive element of a symplectic submodule of $\ker p$ is pinched by $p$. 
\end{remark}

In Corollary \ref{c:realization of decompositions} we will prove inductively that all those decompositions and elements of $\ker p$ appear as induced by  marked stable forms.

Factors of rank two of $p$-admissible splittings and the primitive elements they contain will be important in the sequel.

\begin{definition}
Given a Haupt homomorphism $p:M\rightarrow\mathbb{C}$ define 
$$\mathcal{V}_p :=\{V\subset M: \text{rank}(V)=2, V\oplus V^{\perp} \text{ is a } p\text{ -admissible splitting of }M\}.$$ 
A primitive element $a\in M$ is said to be $p$-admissible if it belongs to some $V\in\mathcal{V}_p$. 
\end{definition}

\subsection{Volumes of symplectic submodules and $p$-admissible elements}
 A necessary condition for a symplectic splitting $V\oplus V^{\perp}$ to be $p$-admissible is that $\vol_p(V)\in (0,\vol (p))$. In this subsection we will analyze the possibilities for the volumes for rank two symplectic submodules containing a given element $a$ from an algebraic point of view. The idea behind the Lemmas is that we can parametrize such submodules by $a^{\perp}$. The volume function is then affine with linear part $$L_a:a^{\perp}\rightarrow\mathbb{R}\text{ defined by } L_a(c)=\Im(p(c)\overline{p(a)}).$$ Up to an isomorphism $\mathbb{R}
^2\cong\mathbb{C}$, $L_a(c)$ is the projection of $p(c)\in\mathbb{C}$ along the line $\ell=p(a)\mathbb{R}$ to its orthogonal line. Whenever the image is discrete (which amounts to having many elements having $p$-image in $\ell$) there will be none, or a finite number of values of the affine function in the interval $(0,\vol(p))$. In the other cases, the values are dense in $\mathbb{R}$. 

\begin{lemma}\label{l:rank}
 Let $W$ be a unimodular symplectic module of rank $2g\geq 4$ and $p:W\rightarrow\mathbb{C}$ a non-trivial homomorphism. Let $a\in W\setminus\ker p$ and suppose that one of the following conditions hold:
 \begin{enumerate}

 \item $\emph{rank}(a^{\perp}\cap p^{-1}(\mathbb{R}p(a)))\leq 2g-3$ or
 \item there exists a real line $\ell\neq \mathbb{R}p(a)$ in $\mathbb{C}$ containing $0$ and satisfying $$\emph{rank}(p(W)\cap \ell)>2.$$
\end{enumerate}
 Then for every $\varepsilon_1<\varepsilon_2$ there exists a symplectic submodule $V\subset W$ of rank $2$ such that $a\in V$ and $\varepsilon_1<\vol_p(V)<\varepsilon_2$. In particular if $a$ is primitive, it is also $p$-admissible.
\end{lemma}
\begin{proof} First suppose $a$ is primitive and choose $b\in W$ be such that $a\cdot b= 1$. For each $e\in a^\perp$ define $b'= b+e$. Denote $\alpha=p(a), \beta=p(b)$. The volume of $V= \mathbb Z a + \mathbb Z b'$ is given by
\begin{equation}\label{eq: volume of W} \text{vol}_p (V)= \Im (\beta \overline{\alpha}) + \Im (p(e) \overline{\alpha}). \end{equation}

If $1)$ holds, the form $e\in a^\perp \mapsto \Im (p(e) \overline{\alpha}) \in \mathbb R$ has image a submodule of $\mathbb{R}$ of rank
$$\rank( a^\perp) - \rank (a^{\perp}\cap p^{-1}(\mathbb R \alpha))\geq 2g-1-(2g-3)) \geq 2.$$ Therefore its image is dense in $\mathbb{R}$ and we can find the desired $e$ for any given $\varepsilon$'s. On the other hand, $2)$ implies $1)$ so the same conclusion holds.

If $a$ is not primitive, it is an integer multiple of some primitive $a_1$. If one of the conditions $(1)$ or $(2)$ is valid for $a$, then it is also valid for $a_1$. Since the lemma is valid for $a_1$ we obtain $V$ of appropriate volume containing $a_1$. By construction $a\in V$.

\end{proof}

\begin{proposition}\label{l:padmissible elements} Let $W$ be a unimodular symplectic module of rank $2g\geq 4$, and $p: W \rightarrow \mathbb C$ be a homomorphism whose image is not contained in a real line. Suppose that either $p$ is injective or $\rank (p)\geq 5$. Then at least one of the following possibilities occur
\begin{enumerate}
\item there exists an element $a\in W\setminus \ker p$ such that for any pair of real numbers $\varepsilon_1<\varepsilon_2$, there exists a rank two symplectic submodule $V\subset W$ containing $a$ such that $$\varepsilon_1< \vol_p(V) < \varepsilon_2.$$
\item $g= 2$, and for every real line $l\subset \mathbb C$, the preimage $p^{-1} (l)$ is either $\{0\}$ or a Lagrangian submodule of $W$.
\end{enumerate}
Moreover, if $g\geq 3$ there exists a submodule $\ker p\subset I\subset W$ of positive co-rank such that the conclusion is true for every primitive $a\in W\setminus I$. If $I=\ker p$ does not have the property, then there exists a unique real line $\ell\subset\mathbb{C}$ such that $\rank(p(W)\cap \ell)>2$. In this case, the module $I=p^{-1}(\ell)$ does the job.
\end{proposition}

\begin{proof} We first treat the case $g\geq 3$. Assume that for every real line $l\subset \mathbb C$, $p(W)\cap l$ has rank $\leq 2$. Take $a\in W\setminus\ker p$. Then $$\rank(a^{\perp}\cap p^{-1}(\mathbb{R}p(a))\leq \rank p^{-1}(\mathbb{R}p(a))\leq \rank \ker p+ \rank (p(W)\cap \mathbb{R}p(a))\leq$$ $$\leq 2g-5+2=2g-3$$ and we conclude by Lemma \ref{l:rank}. Therefore in this case the conclusion with $I=\ker p$ is valid.
If there exists a real line $l\subset \mathbb C$ such that $p(W)\cap l$ has rank $> 2$ therefore  item (1) in Proposition \ref{l:padmissible elements} follows by item (1) of Lemma \ref{l:rank}. In this case the submodule $I=p^{-1}(l)$ does the job.
For the uniqueness of the module $I$ as defined: suppose that there exists $a\in W\setminus \ker p$ and an interval $(\varepsilon_1,\varepsilon_2)$ in $\mathbb{R}$ such that no symplectic submodule $V\subset W$ containing $a$ satisfies $\varepsilon_1<\vol_p(V)<\varepsilon_2$. Then there exists a real line $l\subset \mathbb C$, $p(W)\cap l$ has rank $> 2$. On the other hand, no real line $l\neq\mathbb{R}p(a)$ can have this property, since otherwise $a$ would belong to rank two submodules of $W$ of arbitrary volume. Hence the only possibility is that $l=\mathbb{R}p(a)$. The submodule $I=p^{-1}(l)$ is the only of this type that has the desired property.

Next suppose $g=2$. Then $p$ is injective by assumption. If there exists a real line $l=\mathbb{R}p(w)\subset \mathbb C$ with $\rank (a^{\perp}\cap p^{-1} (l)) =1$ or $\rank(p^{-1}(l))>2$, we can use Lemma \ref{l:rank} to find the desired element $a\in W\setminus 0$. Otherwise we have that for every $a\in W\setminus 0$ $$\rank(a^{\perp}\cap p^{-1}(\mathbb{R}p(a)))=2\text{ and } \rank p^{-1}(\mathbb{R}p(a))\leq 2.$$ By injectivity of $p$ this means that $p^{-1}(\mathbb{R}p(a))\subset a^{\perp}$ for every $a$, so $p^{-1}(\mathbb{R}p(a))$ is a Lagrangian.

\end{proof}
\begin{corollary}\label{c:collinear non admissible}
 Let $W$ be a unimodular symplectic module of rank $2g\geq 6$, and $p: W \rightarrow \mathbb C$ be a homomorphism whose image is not contained in a real line such that $\rank (p)\geq 5$. Then, either all elements in $W\setminus\ker p$ are $p$-admissible or all elements that are not $p$-admissible have image in a real line $\ell\subset\mathbb{C}$ containing $0$ and lie in a submodule of rank $2g-2$ or $2g-1$ of $W$. Consequently for every real line $l\neq \ell$ passing through $0$ we have  $\rank p^{-1}(l) \leq 2$. 
\end{corollary}
\begin{example}\label{ex:genus 2 exception}
In $\mathcal{H}_2$ there are examples of injective homomorphisms for which the volume of symplectic submodules can only take discrete values. They necessarily correspond to case (2) in Proposition \ref{l:padmissible elements}. These examples correspond precisely to the periods of forms belonging to Hilbert modular invariant submanifolds (see Theorem \ref{t:dynamics}). Let $W$ be a rank $4$ unimodular symplectic module and the homomorphism $p:W\rightarrow\mathbb{C}$ given on a symplectic basis $a_1,b_1,a_2,b_2$ by
$\alpha_1 = 1,\ \beta_1 = i \sqrt{D}, \ \alpha_2 = \sqrt{D}, \ \beta_2 = i, $
where $D\geq 2$ is an integer. Then, direct calculation shows that for any symplectic submodule $V$ of $W$ we have $\text{vol} _p (V) \in \sqrt{D} + \mathbb Z$.
Even if the possible volumes of symplectic submodules form a discrete set, there are an infinite number of elements in $\mathcal{V}_p$, all having volumes in a finite set of values. 
\end{example}
\subsection{Existence of pinched elements for non-injective $p\in\mathcal{H}_g$}

\begin{lemma}\label{l:adding handles to Haupt}
Suppose $V_1$ is a symplectic module of rank $\geq 4$, and $V_2$ one of rank $\geq 2$. Let $p_i:V_i\rightarrow\mathbb{C}$ be a homomorphism for $i=1,2$. Suppose $p_1$ is Haupt and $\vol(p_2)\geq 0$. Then $p=p_1\oplus p_2: V_1\oplus V_2\rightarrow\mathbb{C}$ is a Haupt homomorphism.
\end{lemma}
\begin{proof} We already have $\vol(p)=\vol(p_1)+\vol(p_2)\geq\vol(p_1)>0$.
If $p$ were not Haupt, then $\vol(p)=\vol(\mathbb{C}/\image(p))$. On the other hand $\image(p_1)\subset \image(p)$ are discrete and therefore $$\vol(\mathbb{C}/\image(p))\leq\vol(\mathbb{C}/\image(p_1))<\vol(p_1)\leq \vol(p)$$ where the strict inequality comes from the Haupt condition for $p_1$.

\end{proof}
\begin{lemma}\label{lem:symplectic submodules of kerp}
 Let $p:V\rightarrow \mathbb{C}$ be a homomorphism on a rank $2g\geq 6$ unimodular symplectic module $V$ and suppose $W\subset \ker p$ is a symplectic submodule of rank at most $2g-4$. Then $p$ is Haupt if and only if $p_{|W^{\perp}}$ is Haupt. If both are Haupt, then
 \begin{enumerate}
 \item all primitive elements in $W$ are pinched by $p$;
 \item a primitive $a\in\ker (p)\cap W^{\perp}$ is pinched by $p$ if and only if it is pinched by $p_{|W^{\perp}}$.
 \end{enumerate}
 Whenever  $p$ is Haupt and $\deg (p)<\infty$ there exists such a $W$ of (maximal) rank $2g-4$.
\end{lemma}
\begin{proof}

 We have that $V=W\oplus W^{\perp}$ is a symplectic splitting and $\vol_p(W)=0$. Hence $\vol_p=\vol_p(W^{\perp})$ is positive as soon as one of them is. As for the second of Haupt's conditions, remark that $\ker p$ is a symplectic submodule of rank $2g-2$ if and only if $\ker p\cap W^{\perp}$ is a symplectic submodule of rank $\text{rank}(W^{\perp})-2$. By Proposition \ref{p:Haupt characterizations} we have that under the positivity of the volume hypothesis, $p$ is not Haupt if and only if $p_{|W^{\perp}}$ is not Haupt.

 To prove item (1), we remark that for any $a\in W$, $p_a$ is a Haupt homomorphism if and only if $p_{|W^{\perp}}$ is. As for item (2), take $a\in \ker p\cap W^{\perp}$ and $W_1\subset W^{\perp}$ a symplectic submodule of rank two containing $a$. Take the splitting $W^{\perp}=W_1\oplus W_2$. If $a$ is pinched by $p_{|W^{\perp}}$, then $p_{|W_2}$ is a Haupt homomorphism. By Lemma \ref{l:adding handles to Haupt} $p_{|W_2\oplus W}$ is also a Haupt homomorphism which implies that $a$ is pinched by $p$. If $a$ is not pinched by $p_{|W^{\perp}}$ then $W_2\cap\ker p$ is a symplectic submodule of rank $2g-6$ and thus $\ker p\cap W_1^{\perp}$ is a symplectic submodule of rank $2g-4$. This implies that $p_a$ is not Haupt, and thus $a$ is not pinched by $p$.

Using the normal form for homomorphisms of finite primitive degree, see Lemma \ref{l:finite degree orbit} proved in Appendix \ref{s:appendix2}, we find a rank $2g-4$ symplectic submodule $W\subset\ker p$.  
\end{proof}

\begin{lemma}\label{l:existence of pinched classes}

Let $g\geq 2$ and $p\in\mathcal{H}_g$ be a Haupt homomorphism with $\ker p\neq 0$. Then either all primitive elements of $\ker p$ are pinched by $p$, or, $g\geq 3$, $\ker p$ has rank $\geq 2g-3$ and contains a symplectic submodule of rank $2g-4$ (whose primitive elements are pinched by $p$).
\end{lemma}

\begin{proof}
Suppose that there exists a primitive element $a\in\ker p$ such that $p_a$ is not Haupt and take a rank two symplectic submodule $V$ containing $a$. Then $p(V)\neq 0$ and $0<\vol(p)=\vol_p(V)+\vol_p(V^{\perp})=\vol_p(V^{\perp})$.

If $g=2$ then $V^{\perp}$ has rank two, and $p_{V^{\perp}}$ is a Haupt homomorphism. A contradiction. Thus in genus $g=2$ every $a\in\ker p$ satisfies that $p_a$ is Haupt, therefore pinched by $p$.

If $g\geq 3$, then $p_{|V^{\perp}}$ is of positive volume and primitive degree one, therefore contains a symplectic submodule $W\subset \ker p$ of rank $2g-4$. Moreover $W\oplus\mathbb{Z}a\subset\ker p$.
\end{proof}


\subsection{Alternative proofs of Haupt's Theorem}\label{ss:proofs of Haupt}
 In this section we provide an alternative proof of Haupt's Theorem  in \cite{Haupt}. 
\begin{theorem}[Haupt's Theorem]\label{c:Haupt}
A character $p\in H^1(\Sigma_g,\mathbb{C})$ is the period of some marked abelian differential on a smooth curve if and only if it is a Haupt homomorphism.
\end{theorem}
\begin{proof} 
For genus one and $p\in \mathcal{H}_1$ the condition $\vol(p)>0$ implies that $p$ is injective and the image of $p$ is a lattice $\Lambda\subset \mathbb{C}$. The abelian differential $dz$ on $\mathbb{C}/\Lambda$ has periods $p$.

For genus $2$ (and even three) we can apply Theorem \ref{t:connectedness in genus 2 and 3}. 

 For $g\geq 3$, we can proceed by induction. Suppose Haupt's Theorem is true for all genera up to $g-1\geq 2$ and take $p\in\mathcal{H}_g$.

\underline{Case 1:} If $\ker p=0$. By Proposition \ref{l:padmissible elements} applied for $g\geq 3$ we deduce that there exists a $p$-admissible splitting $V_1\oplus V_2$. The restriction $p_i=p_{|V_i}$ is an injective Haupt homomorphism. By inductive hypothesis we can realize $p_i$ as the periods of a marked abelian differential  on a smooth curve $(C_i,m_i,\omega_i)$. The marked nodal form $(C_1\vee C_2, m_1\oplus m_2, \omega_1\vee\omega_2)$ has period $p$, no zero components and induces the splitting $V_1\oplus V_2$. 

\cbstart
Before proceeding to the proof of the non-injective case, we state a well-known lemma from the theory of Riemann surfaces: 
\begin{lemma}\label{lbz}
 Let $(C,\omega)$ be a non-zero abelian differential on a smooth curve of genus $g\geq 1$. For any $z\in\mathbb C$ there is an
 immersed arc $\delta$ in $C$ so that $\int_\delta\omega=z$. Moreover, if $g\geq 2$, the arc
 $\delta$ can be chosen to be embedded with distinct endpoints.
\end{lemma}
\proof
Let $\tilde C$ be the universal cover of $C$ and $\tilde\omega$ the lift of
$\omega$ to $\tilde C$. The map $I:\tilde C\to \mathbb C$ defined by $x\mapsto
\int_{x_0}^x\tilde\omega$ is holomorphic, non-constant and equivariant with respect to the action of $\pi_{1}(C)$ on $\mathbb{C}$ by translations under elements of $\Lambda=\{\int_{\gamma}\omega:\gamma\in\pi_{1}(C)\}$. The image of $I$ is therefore open and invariant by $\Lambda$. We claim that the image is $\mathbb{C}$. Indeed, if $\Lambda$ is dense we are done . If $\Lambda$ is discrete, it is a lattice and the map $I$ induces a branched covering $C\rightarrow \mathbb{C}/\Lambda$ onto a torus. The equivariance then tells us that the image of $I$ is the universal cover $\mathbb{C}$ of the torus. The only other possibility is that the closure of $\Lambda\subset \mathbb{C}$ is isomorphic to $\mathbb{R}\times\mathbb{Z}$. Up to post-composing $I$ with a real linear map we can suppose that $\overline{\Lambda}=\mathbb{R}+\mathbb{Z}i$. The imaginary part $\Im I$ induces a surjective map $C\rightarrow \mathbb{R}/\mathbb{Z}$. Therefore the image of $I$ covers a neighbourhood $U$ of the imaginary axis. The $\Lambda$ invariant set containing $U$ is then $\mathbb{C}$. 

Let $z\in \mathbb{C}$ be the chosen point. By the surjectivity of $I$, there exists an immersed arc $\delta$ in $C$, starting from
$x_0$, such that $\int_\delta\omega=z$. If $g\geq 2$ we can require $x_0$ to be a zero of
$\omega$. In this case the map $I$ is a branched covering of degree $>1$ near $x_0$. Thus,
if the endpoints of both $\delta$ coincide with $x_0$, we can
move them a little so that $\int_\delta\omega$ does not change and they become distinct. Now,
since $\delta$ is an arc, via an homotopy relative to endpoints we can eliminate all
self-intersections so that $\delta$ becomes embedded.\qed
\cbend

\underline{Case 2:} If $\ker p\neq 0$. Use Lemma \ref{l:existence of pinched classes} to consider a primitive element $a\in\ker p$ such that $p_a$ is Haupt. Choose $b\in H_1(\Sigma_g)$ such that $a\cdot b=1$, define $V_1=\mathbb{Z}a\oplus\mathbb{Z}b$, $V_2=V_1^{\perp}$, and $p_i=p_{|V_i}$. By construction $p_2$ is a Haupt homomorphism. By inductive hypothesis, let $\omega_2$ be an abelian differential on a smooth curve of periods $p_2$. Since $g-1\geq 2$, by Lemma \ref{lbz} applied to $\omega_2$ we can find an embedded arc $\beta$ with distinct endpoints in $\omega_2$ such that $\int_{\beta}\omega_2=p(b)$. Glue the endpoints and mark the obtained nodal curve to guarantee that the period homomorphism of the stable form is $p$ and the class $a$ is pinched to the node.

In either case smoothing the node produces a form of period $p$ on a smooth curve. \end{proof}

Kapovich gave a proof of Haupt's Theorem for genus $g\geq 3$ as a corollary of Proposition \ref{p:Kapovich} in \cite{Kapovich}.  Indeed, he remarked that the image of the map $\text{Per}_g$ is an open set in $H^{1}(\Sigma_g,\mathbb{C})$ invariant by the action of $\text{Sp}(2g,\mathbb{Z})$. For genus $g\geq 3$ he deduced that it is the set $\mathcal{H}_g$ by finding forms with periods in each of the closed invariant sets defined by Proposition \ref{p:Kapovich}, except of course, for the collapse of $g-1$ handles.

Using the constructions of Cases 1 and 2 of the proof of Haupt's Theorem we can easily construct stable forms with more nodes and given period map:

\begin{corollary} \label{c:realization of decompositions} Let $p:H_1(\Sigma_g)\rightarrow\mathbb{C}$ be a Haupt homomorphism 
and $H_1(\Sigma_g)=V_1\oplus\ldots\oplus V_k$ a $p$-admissible splitting. Then $k\leq g$ and there exists a stable form in $\Omega_0^{*}\mnc_g(p)$ with $k-1$ separating nodes such that each $p_{|V_j}$ corresponds to the periods of a component of its normalization.
Moreover, if we choose up to $g-k\geq 0$ classes $a_j\in\ker p\setminus 0$, each belonging to a distinct $V_j$ and such that $a_j$ is pinched by $p_{|V_j}$, then there exists a stable form as before that shares the separating nodes with the form of the previous case, and has non-separating nodes that pinch precisely the chosen classes $a_j$. 

\end{corollary}

\subsection{ Non-admissible splittings for injective $p\in\mathcal{H}_g$ and disconnected covers}

Since the primitive degree condition ($H_2$) of a Haupt homomorphism is automatically satisfied for  injective $p\in\mathcal{H}_g$, a symplectic submodule $V\subset H_1(\Sigma_g)$ induces a $p$-admissible splitting $V\oplus V^{\perp}$ if and only if the intermediate volume condition  $0<\vol_p(V)<\vol(p)$  is satisfied. Using the bijection between symplectic splittings and boundary components of the ambient space of compact type given in section \ref{ss:homological invariants} we prove

\begin{proposition}

\label{c:line ranks for injective}Let $g\geq 2$. If $p\in\mathcal{H}_g$ is injective, then $\Omega\mnc^c_g(p)$ cuts an infinite number of boundary strata of $\Omega\mnc_g^c$ and also avoids an infinite number of them. 
\end{proposition}
\begin{proof}
If $p$ falls into case $(1)$ of Proposition \ref{l:padmissible elements}, for each choice of interval $(\varepsilon_1,\varepsilon_2)$ in $\mathbb{R}$ we can construct an example of symplectic splitting $V\oplus V^{\perp}$ of $H_1(\Sigma_g)$ with $V$ of rank two and $\vol_p(V)\in (\varepsilon_1,\varepsilon_2)$ by using that proposition. 
 If the interval belongs to $(0,\vol(p))$ the splitting is $p$-admissible. If it belongs to $\mathbb{R}\setminus [0,\vol(p)]$ it is not $p$-admissible. Since two splittings having factors of different volumes are different and an infinite number of disjoint intervals can be chosen, we conclude.

If $p$ falls into case (2) of Proposition \ref{l:padmissible elements} we have that the genus $g=2$. Haupt's Theorem guarantees that $\per^{-1}(p)$ contains a stable form of period $p$ on a smooth curve. By Corollary \ref{c:nonempty boundary iff haupt}, $\Omega\mnc_g(p)$ contains a point in some boundary stratum of codimension one. It corresponds necessarily to a $p$-admissible splitting $H_1(\Sigma_g)=V\oplus V^{\perp}$ with factors of rank two. Furthermore, the detailed description of the isoperiodic set of marked stable forms associated to $p$ can be found in \cite[Theorem 1.2, p. 2274]{McMullen}. Among other properties we find an infinite number of elements defined over nodal curves with a (necessarily separating) node, and inducing distinct $p$-admissible splittings $H_1(\Sigma_g)=V\oplus V^{\perp}$ with factors of rank two.

To construct splittings that are not $p$-admissible we take any $p$-admissible splitting $V\oplus V^{\perp}$ with $V=\mathbb{Z}a\oplus \mathbb{Z} b$ and $c\in a^{\perp}$ with $p(c)\notin \mathbb{R}p(a)$. For each $k\in\mathbb{Z}$ consider $V_k=\mathbb{Z}a\oplus \mathbb{Z}(b+kc)$. Then only for a finite number of $k$'s we have $\vol_p(V_k)\in (0,\vol(p))$. Since the volumes of the $V_k$'s are all different, we have found an infinite family of distinct splittings $V_k\oplus V_k^{\perp}$ that are not $p$-admissible.

Any $p$-admissible splitting is realized by some singular stable form via Corollary \ref{c:realization of decompositions}. \\ 
\end{proof}

\begin{corollary}\label{c: fibers on T_2 are disconnected}
The generic fiber of the period map on $\Omega\mathcal{T}_2$ is disconnected. 
\end{corollary}
\begin{proof}
By \cite{Mess} there is a free generating family of the Torelli group $\mathcal{I}_2$ that can be thought in Siegel space as  the family of cycles around every boundary component of Torelli space in Siegel space. These boundary components are parametrized by symplectic splittings $V\oplus V^{\perp}$. 
Let $p\in\mathcal{H}_2$ be injective (which is generic). From the proof of Theorem \ref{t:connectedness in genus 2 and 3}, $\per^{-1}(p)$ is a slice of dimension one of Torelli space that is isomorphic to $\mathbb{D}\setminus B$ where $B$ is the intersection of the slice with the boundary components. Inclusion induces a map $\pi_1(\mathbb{D}\setminus B)\rightarrow \mathcal{I}_2$. By Corollary \ref{c:line ranks for injective} this map is not surjective. Hence the lift of $\per
^{-1}(p)$ to $\mathcal{T}_2$ is disconnected. 
\end{proof}

\subsection{Non-admissible splittings for non-injective $p\in\mathcal{H}_{g}$ and compact type boundary}
\cbstart
\begin{proposition}\label{ex:non compact type}
Let $g\geq 2$ and $p\in\mathcal{H}_{g}$.  There exists a $p$-admissible splitting of $H_{1}(\Sigma_{g})$  if and only if either $g=2$ or $g\geq 3$ and $\deg (p)>2$.  
\end{proposition}

\begin{proof}
First remark that if there are no $p$-admissible splittings, then there do not exist forms of compact type with no zero components and periods $p$.  By Proposition \ref{p:degeneration} there do also not exist forms with a single zero (i.e. in the minimal stratum) of period $p$.

As was already shown in the beginning of the proof of Lemma \ref{l: transitivity}, the image of the minimal stratum contains all points with $\deg(p)=\infty$, so we are left with the cases of finite degree. 

For $g=2$ there are always $p$-admissible splittings. Indeed, if $\deg(p)=2$ there cannot exist a form with a single zero of period $p$, since the local degree of the covering at the zero would be three so we cannot use this argument. However the stable form with a separating node $dz\vee dz$ on $\frac{\mathbb{C}}{\Lambda}\vee \frac{\mathbb{C}}{\Lambda}$ where $\Lambda:=p(H_1(\Sigma_g))$ defines a $p$-admissible splitting. For $\deg(p)\geq 3$ there are forms in the minimal stratum with period $p$ (\cite{Le Fils}, \cite{BJJP} for constructive proofs).

If $g\geq 3$ and $3\leq\deg(p)<\infty$ we proceed by induction on the genus to construct $p$-admissible splittings.  There is a symplectic submodule $W\subset \ker p$ of rank two, the map $p_{|W^{\perp}}$ has the same degree and volume as $p$.  If $\deg (p)\geq 3$ then there is a  $p_{|W^{\perp}}$-admissible splitting of $W^{\perp}$. Adding $W$ to one of the factors produces a $p$-admissible splitting. 

It remains to treat the case $g\geq 3$ and $\deg (p)=2$. 
Without loss of generality we can suppose that the image of $p$ is the group  of  Gaussian integers. Using the normal form of period of finite primitive degree (see Lemma \ref{l:finite degree orbit} proved in Appendix \ref{s:appendix2})
we can find  a symplectic basis $\{a_i,b_i\}$ by $p(a_1)=p(a_2)=1$, $p(b_1)=p(b_2)=i$ and zero elsewhere. In this case $\vol(p)=2$ and $\deg p=2$. The $p$- volume of any symplectic submodule of $H_1(\Sigma_g)$ is an integer. Suppose that there exists a $p$-admissible splitting $V_1\oplus V_2$ of $H_1(\Sigma_g)$. The volume of each component is a positive integer. Hence, the only possibility is that each component has volume one and thus $\vol(V_i)=\vol(\mathbb{C}/\text{Im}(p))$. Since $g\geq 3$, one of both factors, say $V_1$, has even rank $\geq 4$ so $p|_{V_1}$ is not a Haupt homomorphism, contradicting the definition of $p$-admissible decompositon. 
\end{proof}
\cbend
\begin{corollary}
If $p\in\mathcal{H}_{3}$ has $\deg(p)=2$ then  $\per^{-1}_{3}(p)\simeq \mathfrak{S}_2$.  Therefore its lift to $\Omega\mathcal{T}_{3}$ has infinitely many connected components.  
\end{corollary}
\begin{proof}
The proof of Theorem \ref{t:connectedness in genus 2 and 3} together with the fact that there are no $p$-admissible splittings, imply that  $\per^{-1}(p)$ is isomorphic to Siegel space $\mathfrak{S}_2$. By simple connectedness of Siegel space we deduce the final statement. \end{proof}
\begin{corollary}
 For  $p\in\mathcal{H}_{g}$ with $\deg(p)\geq 3$ that can be decomposed as a direct sum of primitive degree two homomorphisms on modules of rank at least ten,  the set of boundary points of compact type of $\per^{-1}(p)$ in $\Omega^{*}_{0}\mrs_{g}$ is disconnected. 
\end{corollary}
\begin{proof} Let $V_1\oplus V_2$ be the splitting such that $p_i=p_{|V_i}$ has degree two for $i=1,2$. By Proposition \ref{ex:non compact type} there are no $p_i$-admissible decoompositions of $V_i$. 

In particular there are no forms of compact type with two separating nodes and periods $p$ such that one of the nodes induces the splitting $V_1\oplus V_2$ so we cannot approach a point of intersection with another irreducible boundary component of $\per^{-1}(p)$ of compact type. On the other hand, the intersection of the boundary component with the fiber $\percompact^{-1}(p)$ is homeomorphic to a product of two fibers as in  Theorem \ref{t:degree 2} via the normalization map of the node. Hence it is a disconnected set. \end{proof}
This corollary indicates that to prove the connectedness of the bordification the boundary of compact type does not solve the problem, since it is not connected. 

\section{Detecting connected components of isoperiodic forms with one node}
\label{s:period fibers with marked points}
Attaching and forgetful maps have nice relations with respect to the fibers of period maps. For instance,  the restriction of a map of type \eqref{eq:attach two components} with $n_1=n_2=1$ to a product of fibers produces a map \[\Omega\mrs_{g_1,1}(p_1)\times \Omega\mrs_{g_2,1}(p_2) \rightarrow \Omega\mnc_{g_1+g_2,0}\] whose image points have one node and constant periods $p$ that can be decomposed as $p_1\oplus p_2$. Equivalently, the restriction of a map of type \eqref{eq: attach two points in one component} with $n=0$ to a fiber \(\Omega\mrs_{g-1,2}(p)\) has its image in the set of forms with a unique non-separating node in a fiber of the period map in \(\Omega\mnc_{g,0}\) (see Figure \ref{fig:normalization}). In fact, thanks to the characterization given in Corollary \ref{c:parametrizing boundary}, these maps provide continuous parametrizations of the components of the intersections of fibers of $\percompact$ with some boundary stratum $\Omega B_c$ of stable forms with one node. 
Remark that the hypothesis of Theorem \ref{t:connectedness} does not apply to isoperiodic spaces of forms \textit{with} some marked points, so to control the number of connected components of the image of the attaching maps, we will first relate  isoperiodic sets on spaces of curves with marked points with isoperiodic spaces without marked points. 

\cbstart
\subsection{Connectedness of the fibers of  $\per_{g,2}$} 
Recall that the relative homology long exact sequence gives a natural injection $$0\rightarrow H_1(\Sigma_{g},\mathbb{Z})\rightarrow H_1(\Sigma_{g},q_1,q_2;\mathbb{Z}).$$   


\begin{theorem}\label{p:add marked points} Let $g\geq 1$ and    
  $p:H_1(\Sigma_{g},q_1,q_2;\mathbb{Z})\rightarrow \mathbb{C}$ be a homomorphism such that  $0\neq p_0=p_{|H_1(\Sigma_{g},\mathbb{Z})}$ has degree at least three and $\Omega\mrs_{g,0}(p_0)$ is connected. Then, 
 $\Omega\mrs_{g,2}(p)$ is non-empty and connected. 
\end{theorem}

Since non-generic sub-strata of forms in $\Omega\mrs_{g,2}(p)$ form analytic subsets, the claim of Theorem \ref{p:add marked points} is equivalent to proving that the open subset $$\Omega^{SZ}\mrs_{g,2}(p)=\{(C,m,\omega)\in\Omega\mrs_{g,2}(p):\omega\text{ has simple zeros }\}$$ is connected . To prove Theorem \ref{p:add marked points} we will bordify $\Omega^{SZ}\mrs_{g,2}(p)$ in the complex manifold $\Omega^{SZ} \U_{g,2}\subset \Omega\mnc_{g,2}$ defined in subsection \ref{sss:singularities} by adding the limit points. Thanks to Proposition \ref{p:singularity of isoperiodic set} the closure is a smooth complex manifold where the boundary points form a divisor. 

Recall that $\U_{g,2}=\text{For}_{g,2}^{-1}(\mrs_{g,0})\subset \mnc_{g,2}$ where $\text{For}_{g,2}:\mnc_{g,2}\rightarrow \mnc_{g,0}$ is the map that forgets both marked points and stabilizes the obtained curve. Every point in some boundary stratum in $\U_{g,2}$ has one node and a part of genus zero containing both marked points. To simplify the notations along the proofs we fix $g$ and write $\U:=\U_{g,2}$. 

The fibration $\Omega^{SZ}\U\rightarrow \U$ of forms having only simple zeros in the smooth genus $g$ part is a complex manifold and the restriction of the period map defined on $\Omega^{SZ}\U$ is holomorphic (see Lemma \ref{l:holomorphic integral}).

Given a homomorphism $p:H_1(\Sigma_{g},q_1,q_2;\mathbb{Z})\rightarrow \mathbb{C}$, we define  the analytic subset $$X_p=\{(C,r_1,r_2,m,\omega)\in\Omega^{SZ}\U_{g,2}: \per_{g,2}(C,r_1,r_2,m,\omega)=p\}. $$

To understand the topology of $X_p$ we will slice it by fibers of the restriction of the forgetful map to \(\Omega^{SZ}\U\) introduced in equation \eqref{eq:restricted forgetful}
\begin{equation}
    \text{For}:\Omega^{SZ}\U\rightarrow\Omega^{SZ}\mrs_g.
\end{equation}
The fiber of $\text{For}$ over a point \((C,m,\omega)\in \Omega\mrs_g\)  is  biholomorphic to the covering of the surface $C\times C$ described in Proposition \ref{p:bordification of Sg2}. 

Recall that  $p_0:=p_{|H_1(\Sigma_{g})}$. Thanks to Lemma \ref{lbz}, we have $$\text{For}(X_p)=\Omega^{SZ}\mrs_g(p_0).$$

\begin{lemma} \label{l: connectedness section isoperiodic set} 
For any \( (C,m_0 ,\omega_0 )\in \Omega^{SZ}\mrs_g(p_0)\),  the analytic subset \[X_p \cap \Forget^{-1} (C,m_0,\omega_0) \] is a connected singular curve.
\end{lemma} 

\begin{proof}
Fix a homology class \( \alpha_0 \in H_1 (\Sigma_g, \{q_1, q_2\}, \mathbb Z) \) with boundary \(\partial \alpha _0 = q_2- q_1\). Define the map \(f: \Forget^{-1} ( C,m_0, \omega_0) \rightarrow \mathbb C\) by 
\[ f( C', r_1,r_2, m , \omega) = \int _{m_* \alpha_0} \omega\quad\text{ where }\quad C'=C\quad\text{ or }\quad C'=C\vee\mathbb{P}^1 \]
Since \(H_1 (\Sigma_g, \{q_1, q_2\}, \mathbb Z)\) is generated by \(H_1 (\Sigma_g, \mathbb Z)\) and \(\alpha_0\), the set \(X_p \cap \Forget^{-1} (C,m_0,\omega_0) \) is equal to the fiber \( f^{-1} (z_0) \) with \(z_0=p(\alpha_0)\in \mathbb C\). We will actually prove that all fibers of \(f\) are connected to conclude the proof. 

The function \(f\) is a primitive of the holomorphic \(1\)-form \(r^* \Omega \) where
\begin{itemize}
    \item[(i)] \(r=R_{|\text{For}^{-1}(C_,m_0,\omega_0)}:\text{For}^{-1}(C,m_0,\omega_0)\rightarrow C\times C\) is the $H_1(\Sigma_g)$-covering of monodromy \[ (\alpha_1, \alpha_2) \cdot \alpha= \alpha+ \alpha_2 - \alpha_1 ,\]described in Proposition \ref{p:bordification of Sg2} and 
\item [(ii)]\(\Omega\) is the form on \( C\times C\) defined by \(\Omega= pr_2 ^* \omega_0 - pr_1^* \omega_0\), where \(pr_i : C\times C\rightarrow C\) for \(i=1,2\) are the projections on the first and second factors respectively. 
\end{itemize}

A result of Simpson \cite{Simpson} shows that either 
\begin{enumerate} 
\item the fibers of \(f\) are connected, or
\item there exists a complex curve \(E\) equipped with a holomorphic one form \(\eta\), and a holomorphic map \( h : C\times C\rightarrow E\) with connected fibers such that \(\Omega= h^* \eta\). 
\end{enumerate} 
So it remains to prove the Lemma in the second case. 

Since the set of zeros of \(\Omega\) is finite (it is the square of the set of zeroes of \(\omega_0\) in \( C\times C\)), and that fibers of \( h\) are one dimensional subvarieties, the form \(\eta\) does not vanish, hence \(E\) is an elliptic curve. So we have \( E=\mathbb C/ \Lambda\) for a certain lattice \(\Lambda \subset \mathbb C\), and \(\eta= dz\) for \(z\in \mathbb C\) the coordinate. 

Notice that the set of absolute periods of \(\omega, \Omega\) and \(\eta\) coincide with \(\Lambda\); indeed, it is clear that the set \(\Lambda'\) of absolute periods of \(\omega\) and \(\Omega\) coincide and that it is contained in the set of periods of \(\eta\), which is equal to \(\Lambda\). Now, assuming for contradiction that  \(\Lambda'\) is strictly contained in \( \Lambda\), the map \( h\) could be factored in the form \( h =  r' \circ h' \) where \( h ' : C\times C \rightarrow \mathbb C/ \Lambda'\) and \( r' : \mathbb C/ \Lambda' \rightarrow \mathbb C / \Lambda \simeq E\) a covering of degree \(>1\). But \(h'\) is onto since it is a non constant holomorphic map, so this contradicts the fact that the fibers of \(h\) are connected. Hence \(\Lambda '= \Lambda\). This discussion shows that up to changing the coordinate \(z\) by a translation if necessary, we have the following commutative diagram
\begin{equation} \label{eq: quotient of period map} 
\begin{tikzcd}
\text{For}^{-1}(C,m_0,\omega_0) \arrow[r, "f"] \arrow[d,"r"]
& \mathbb{C} \arrow[d] \\
C\times C \arrow[r,"h"]
& \mathbb{C}/\Lambda
\end{tikzcd}
\end{equation} 
Let us fix $z\in\mathbb{C}$ and analyze the restriction \begin{equation}\label{eq:cover of fiber of f}
r_{|f^{-1}(z)}:f^{-1}(z)\rightarrow h^{-1}(z\text{ mod }\Lambda)\end{equation} 

Recall from Proposition \ref{p:bordification of Sg2} that the covering group $G\simeq H_1(\Sigma_g)$ of $r$ is presented as  \( G \subset \text{Aut} ( H_1 (\Sigma_, \{q_1, q_2\}, \mathbb Z)) \) acting on \(\alpha_0\) by \(\alpha \cdot \alpha_0= \alpha_0 + \alpha\). Therefore we have  
 \[ f\circ \alpha = f + p_0 (\alpha).\]
Together with the diagram \eqref{eq: quotient of period map}, this shows that \(r_{|f^{-1}(z)}\) is a \(\text{Ker} (p_0) \)-covering.  Its monodromy is the composition of the inclusion map \( H_1 (h^{-1} (z\text{ mod } \Lambda), \mathbb Z) \rightarrow H_1 (C\times C, \mathbb Z)\simeq H_1 (C,\mathbb Z) ^2 \) with that of the \(H_1 (\Sigma_g, \mathbb Z)\)-covering \(r\). The connectedness of $f^{-1}(z)$ is equivalent to the  transitivity of the action of the monodromy on the fiber of $r_{|f^{-1}(z)}$.





 Picard-Lefschetz theory can be applied to analyze this monodromy. We include a statement of a  particular instance of it that is needed here, in Corollary \ref{c: exact sequence} of Appendix \ref{a: Picard-Lefschetz}. It shows that the image of \( H_1 (h^{-1} (z\text{ mod } \Lambda), \mathbb Z) \rightarrow H_1 (C\times C, \mathbb Z)\) is the kernel of \( h_* = : H_1 (C\times C, \mathbb Z)\simeq H_1 (C,\mathbb Z)^2 \rightarrow H_1 (E, \mathbb Z)\simeq \Lambda\). Moreover, 
the map \( h \) at the level of the homology is given by the periods of \( \Omega\), namely 
by \( h_* =p_0 \circ (pr_2 -  pr_1). \)
In conclusion, the monodromy of the \(\text{Ker} (p_0) \)-covering \(r_{|f^{-1} (z)}: f^{-1} (z)\rightarrow h^{-1} (z\text{ mod } \Lambda) \) acts transitively, and this concludes the proof that \( f^{-1} (z)\) is connected, hence Lemma \ref{l: connectedness section isoperiodic set}. \end{proof}

\begin{corollary}\label{c:cor simpson}
$X_p$ is connected if and only if $\text{For}(X_p)$ is connected.
\end{corollary}
\begin{proof}[Proof of Corollary \ref{c:cor simpson}]
Suppose $\text{For}(X_p)=\Omega^{SZ}\mrs_{g}(p_0)$ is connected. By Lemma \ref{l: connectedness section isoperiodic set} , each connected component of $X_p$ is a union of fibers of $\text{For}_{|X_p}$. Suppose $X_p=X^1\sqcup X^2$ where $X^1$ and $X^2$ are disjoint open sets and $X^1\neq\emptyset$ is connected. Then $X^1$ is a union of fibers of $\text{For}_{|X_p}$, and so is its complement $X^2$. Moreover $\text{For}_{|X_p}:X_p\rightarrow\Omega^{SZ}\mrs_g(p_0)$ is an open map, since on each fiber there is some point where $\text{For}_{|X_p}$ is a submersion. Then $\text{For}(X_p)=\text{For}(X^1)\sqcup\text{For}(X^2)$. By connectedness of $\text{For}(X_p)$ we deduce $\text{For}(X^2)=\emptyset$, and therefore $X^2=\emptyset$, which implies that $X_p=X^1$ is connected. The other implication is obvious.

\end{proof}
When $X_p$ contains a point in some boundary stratum, there are, by Proposition \ref{p:singularity of isoperiodic set}, irreducible components of $X_p$ strictly contained in the boundary. To deduce the connectedness of the subset of points of $X_p$ that lie outside the boundary (to prove Theorem \ref{p:add marked points}) we need to show that taking out those components we still have a connected set. We will apply the following general 

\begin{figure}[httb]
\centering
\def\svgwidth{\columnwidth}
\includegraphics[width=4in]{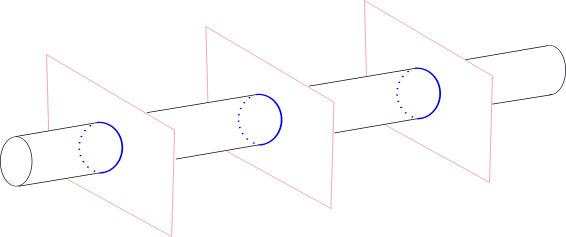}
 \caption{Real representation for the set $X$ in Lemma \ref{l:connectedness of analytic subset}: each plane represents component of $Z$. Each circle in a component of $Z$ represents a connected component of $\text{Sing}(X)$ that forms a divisor in $Y$, the closure of $X\setminus Z$.} \label{fig:connanset}
\end{figure}

\begin{lemma}\label{l:connectedness of analytic subset}
 Let $X$ be a connected analytic set on a complex manifold and suppose $X=Y\cup Z$ where both $Y$ and $Z$ are unions of irreducible components of $X$ that satisfy 
 \begin{enumerate}
     \item around each point of $\text{Sing}(X)$ the set $X$ is a normal crossing of two smooth manifolds, one of which lies in $Y$ and the other in $Z$, and 
     \item the intersection of each irreducible component of $Z$ with $\text{Sing}(X)$ is connected.
 \end{enumerate}
 Then $X\setminus Z\subset Y$ is connected. 
\end{lemma}
\begin{proof} The diagram in Figure \ref{fig:connanset} helps to understand the ideas behind the proof. 
The properties of $\text{Sing}(X)$ imply that $Y$ and $Z$ are smooth manifolds. Moreover,  each connected component of $\text{Sing}(X)$ is a smooth complex submanifold of $Y$ and of $Z$ that coincides with $Z_i\cap Y$ for some irreducible component $Z_i$ of $Z$.  In particular, the connectedness of $X\setminus Z$ is equivalent to the connectedness of $Y$. 

On the other hand, for $X$,  (and also for the manifolds  $Y$ and $\text{Sing}(X))$) connectedness is equivalent to path connectedness. 

Given $y_1,y_2\in Y\subset X$ we can, by hypothesis, find a path $\gamma:I\rightarrow X$ starting at $y_1$ and ending at $y_2$.  If the path lies in $Y$ we are done. Otherwise $I\setminus\gamma^{-1}(Y)$ is a finite family of disjoint segments $I_1,\ldots,I_k\subset I$ where $\gamma_{|I_k}$ is a path contained in a unique $Z_{k}$ with endpoints in $\text{Sing}(X)$. Substitute $\gamma_{|I_k}$ by a path in the connected manifold $Z_k\cap Y\subset \text{Sing}(X)$ joining the same endpoints. The constructed path lies in $Y$. Therefore, $Y$ is path connected and so is $X\setminus Z$. 
\end{proof}

\subsubsection{Proof of Theorem \ref{p:add marked points}}

 Recall from subsection \ref{sss:singularities} that the boundary strata in $\Omega^{SZ}\U_{g,2}$ is  precisely the disjoint union  $\bigsqcup_{\beta\in\betas}\Omega^{SZ}B_{\beta}$, where $\betas=\{\beta\in H_1(\Sigma_{g},q_1,q_2,\mathbb{Z}): \partial\beta=q_2-q_1\}$. The boundary strata in $X_p$ are 

$$Z_p=\bigsqcup_{\beta\in\betas}\Omega^{SZ}B_{\beta}(p).$$ Let $Y_p$ be the union of irreducible components of $X_p$ that have some point outside the boundary.

From Proposition \ref{p:singularity of isoperiodic set} we deduce that $$\text{Sing}(X_p)=\bigsqcup_{\beta\in\betas}\Omega^{SZ,1}B_\beta(p).$$ 
where $\Omega^{SZ,1}B_{\beta}$ is the substratum of forms where the points corresponding to the node are zeros of the form on both sides of the node. Moreover, the local structure of $X_p$ at a singular point is a normal crossing of two smooth components: one lying in $Z_p$ and the other in $Y_p$.

\textit{Claim 1:} $X_p$ is connected.\\
\textit{Proof of Claim 1:} Since $\text{deg}(p_0)\geq 3$ and $\per_{g,0}^{-1}(p_0)$ is connected by hypothesis, so is the Zariski open set $\per_{g,0}^{-1}(p_0)\cap\Omega^{SZ}\mrs_{g,0}=\text{For}(X_p)$.  Corollary \ref{c:cor simpson} then implies that $X_p$ is connected.  \\

\textit{Claim 2:} Each $\Omega^{SZ}B_\beta(p)$ is a smooth irreducible component of $Z_p$.\\
\textit{Proof of Claim 2:} Proposition \ref{p:singularity of isoperiodic set} shows that the set $\Omega^{SZ}B_\beta(p)$ is a smooth manifold. 
As for connectedness, the attaching map 
 \begin{equation}\label{eq:attaching P1}
\Omega^{SZ}\mrs_{g,1}\cong\Omega^{SZ}\mrs_{g,1}\times\Omega\mrs_{0,3}\rightarrow \Omega^{SZ} B_{\beta}
\end{equation} that sends a form with simple zeros on a \emph{smooth} genus $g$ curve with a marked point to the form that attaches a zero component of genus zero at the marked point, is a biholomorphism. Define $p_{1}=p_{|H_1(\Sigma_{g,1})}$. Then, the set $\Omega^{SZ}B_\beta(p)$ corresponds to $\Omega^{SZ}\mrs_{g,1}(p_1)$ under this biholomorphism. Remark that the forgetful map $\mrs_{g,1}\rightarrow \mrs_g$ is a fiber bundle that is equivalent to the universal curve bundle $\mathcal{C}\mrs_{g}\rightarrow\mrs_g$. The restriction of the forgetful map $\mathcal{C}\Omega\mrs_{g,1}\rightarrow \Omega\mrs_{g}$ to   $\Omega^{SZ}\mrs_{g,1}(p_1)\subset \Omega^{SZ}\mrs_{g,1}$ coincides with the restriction of the universal curve bundle to the open set $\Omega^{SZ}\mrs_g(p_0)\subset \Omega\mrs_g$. Since the latter is connected by hypothesis, so is $\Omega^{SZ}\mrs_{g,1}(p_1)\subset \Omega^{SZ}\mrs_{g,1}$.

\textit{Claim 3:}  Each $\Omega^{SZ,1}B_\beta(p)$ is a connected set. 
\\
\textit{Proof of Claim 3:}  By the biholomorphism \eqref{eq:attaching P1}, the set $\Omega^{SZ,1}B_\beta(p)$ corresponds to the subset $\Omega^{SZ,1}\mrs_{g,1}(p)\subset\Omega^{SZ}\mrs_{g,1}$ where the marked point coincides with one of the simple zeros.  This set can be interpreted as the multi-section described by the set of $2g-2$ simple zeros on each fiber of the universal curve bundle $$\mathcal{C}\Omega^{SZ}\mrs_{g,0}(p_0)\rightarrow \Omega^{SZ}\mrs_{g,0}(p_0). $$ Since $\Omega^{SZ}\mrs_{g,0}(p_0)$ is connected and  $\deg p_0\geq 3$, we deduce from Lemma \ref{l: transitivity} that this multi-section is connected.

From Claims 1,2 and 3 we deduce that the hypotheses of Lemma \ref{l:connectedness of analytic subset} are satisfied for $X_p=Y_p\cup Z_p$. We deduce from it  that $\Omega\mrs_{g,2}(p)=X_p\setminus Z_p$ is connected. 

\cbend
\subsection{Some connected components of isoperiodic sets on curves with one node}

\label{ss:isoperiodic components with one node}
\begin{proposition}\label{p:inductive step compact type}
Let \( g\geq 2 \), \( p\in H^1( \Sigma_g, \mathbb C )\), and $\curvesystem$ be a Torelli class of a simple separating curve in \( \Sigma_g\) such that \(\Omega^*B_{\curvesystem}(p)\) is non empty. Let \( H_1(\Sigma_g, \mathbb Z) = V_1 \oplus V_2\) the symplectic splitting associated to \(\curvesystem\) and denote $p_i=p_{V_i}$. Suppose that, after a symplectic identification of \( V_i \) with \( H_1(\Sigma_{g_i}, \mathbb Z)\) for corresponding \(g_i\)'s, the set \(\Omega\mrs_{g_i}(p_i)\) is connected for \( i=1,2\). Then, $\Omega_0^*B_{\curvesystem}(p)$ is connected as well. 
\end{proposition}


\begin{proof}

First remark that $p_i\neq 0$ and $\Omega\mrs_{g_i,0}(p_i)$ is non-empty, since the restriction of an element $(C_1\vee C_2,m,\omega_1\vee\omega_2)\in\Omega^*B_{\curvesystem}(p)$ to a part $(C_i,m_{|V_i},\omega_i)$ belongs to that set.  By Theorem \ref{p:add marked points} the set $\Omega\mrs_{g_i,1}(p_i)$ is also non-empty and connected.  
The natural attaching map $$\Omega\mrs_{g_1,1}(p_1)\times \Omega\mrs_{g_2,1}(p_2)\rightarrow \Omega^*B_{\curvesystem}(p)$$ that attaches the marked curves and the forms at the marked points 
is a surjective continuous map (see Subsection \ref{ss:attach&forget on hodge bundles}). Therefore the image is connected. \end{proof}

\begin{proposition}\label{p:inductive step non compact type}
Let \( g\geq 3 \), \( p\in H^1( \Sigma_g, \mathbb C )\), and $\curvesystem$ be a Torelli class of a simple non separating curve in \( \Sigma_g\) such that \(\Omega_0 B_{\curvesystem}(p)\)  is non empty. Suppose that, after a symplectic identification of \( [\curvesystem]^\perp / \mathbb Z [\curvesystem] \) with \( H_1(\Sigma_{g-1}, \mathbb Z) \), the primitive degree of \( p_{[\curvesystem]} \) is at least three, and  \(\Omega\mrs_{g-1}(p_{[\curvesystem]})\) is non-empty and connected. Then, $\Omega_0 B_{\curvesystem}(p)$ is connected as well. 
\end{proposition}

\begin{proof}
Take a representative of \( \curvesystem \) and introduce the closed connected oriented topological surface \(\Sigma_{g-1}\) with two marked points \(q_1\neq  q_2\in \Sigma_{g-1}\) defined by cutting \(\Sigma_g\) along \(\curvesystem\) and identifying each component of the geometric completion to a point, one defining the point \(q_1\) and the other the point \(q_2\).  Notice that we have a natural isomorphism 
\begin{equation} \label{eq: isomorphism relative homology} H_1 (\Sigma_{g-1}, q_1,q_2; \mathbb Z) \simeq H_ 1 (\Sigma_g, \mathbb Z)/ \mathbb Z [c], \end{equation}  
where \([\curvesystem]\) denotes the homology class of \( \curvesystem \). The inclusion \( H_1 (\Sigma_{g-1} , \mathbb Z)\hookrightarrow H_1 (\Sigma_{g-1},q_1,q_2; \mathbb Z)  \) coming from the long exact sequence of relative homology translates under the isomorphism \eqref{eq: isomorphism relative homology} into the inclusion \( [\curvesystem]^{\perp} / \mathbb{Z} [\curvesystem] \hookrightarrow H_1 (\Sigma_g, \mathbb Z) /\mathbb Z [\curvesystem]\). 

Since \([\curvesystem]\) belongs to the kernel of \( p: H_1(\Sigma_g, \mathbb Z) \rightarrow \mathbb C\), this latter defines a period \( \prel : H_1 (\Sigma_{g-1} , q_1,q_2; \mathbb Z) \rightarrow \mathbb C\). 
The attaching map that identifies the two marked points provides an isomorphism 
\[ \Omega_0 B_{\curvesystem}(p) \simeq \Omega \mrs_{g-1, 2} (\prel). \]
Now, the restricition $(p_{rel})_0$ of $p_{rel}$ to $H_1(\Sigma_{g-1},\mathbb{Z})$ coincides with $p_{[\curvesystem]}$ up to the given identification. In particular $\deg((p_{rel})_0)\geq 3$ and $\Omega\mrs_{g-1,0}((p_{rel})_0)$ is (non-emtpy and) connected. By Theorem \ref{p:add marked points} applied to $p_{rel}$ we deduce that $\Omega\mrs_{g-1,2}(p_{rel})$ is connected.   
 \end{proof}

\section{Connectedness of the boundary of a fiber of $\percompact$}
\label{s:connectedness of the boundary}
This section will be devoted to show the following
\begin{theorem}\label{t:connected boundary}
Let $g\geq 4$ and suppose that Theorem \ref{t:connectedness} is true up to genus $g-1$. Then for any $p\in\mathcal{H}_g$ with $\deg p\geq 3$ the boundary  $\partial \per^{-1}(p)$ in $\Omega^{*}_0\mnc_{g}$ is connected. 
\end{theorem}
Its proof will be split in two separate parts, depending on whether \(p\in \text{Hom}(H_1(\Sigma_g, \mathbb Z)  ,\mathbb C) \) is injective or not.

\subsection{Proof of Theorem \ref{t:connected boundary} for injective $p\in\mathcal{H}_g$}\label{ss:connected boundary injective case} 
In the case of injective $p\in\mathcal{H}_g$ all points in $\Omega\mnc_g(p)$ are contained in the smooth manifold $\Omega^*\mnc_g^ c$ formed by stable forms without zero components on curves of compact type. Moreover,  Corollary \ref{c:holomorphic extension to compact type} guarantees that $\Omega\mnc_g(p)$ is a smooth manifold, and the boundary strata
 $\partial\percompact^{-1}(p)$ form a normal crossing divisor.  The  connectedness of this divisor is equivalent to the connectedness of its dual complex  $\mathcal{G}_p$. 

\begin{remark}\label{rem:multiplying_periods}
If $L\in\text{GL}_2^+(\mathbb{R})$ is an orientation preserving real linear map, the complexes associated to $p$ and to $L\circ p$ are isomorphic. In particular, for any $c\in\mathbb{C}^*$, $\mathcal{G}_p$ and $\mathcal{G}_{cp}$ are isomorphic.  
\end{remark}

In the rest of this section we will prove the connectedness of the complex $\mathcal{G}_{p}$. 

Since each  component of a boundary stratum of $\percompact^{-1}(p)$ lies in a component of a stratum of the  boundary of $\Omega\mnc^{*}_{g}$ we can define, using the dual boundary complexes,  a continuous map  of simplicial complexes 
\begin{equation}\label{eq:bdry inclusion ss}
    \mathcal{G}_p\rightarrow \curvecomplex^{\text{sep}}_g/\mathcal{I}_g. 
\end{equation}
We claim that if $p$ is injective and $g\geq 4$ then so is the map \eqref{eq:bdry inclusion ss} at the level of vertices. Indeed, suppose $D_{1},D_{2}$ are connected components of $\partial \per^{-1}(p)$ lying on the boundary component of  $\Omega\mnc^{*}$  corresponding to a symplectic splitting $V_{1}\oplus V_{2}$. Since $\ker p=0$, $\deg(p_{|V_{i}})<\infty$ if and only if $\rank(V_i)=2$. The inductive hypothesis of Theorem \ref{t:connectedness} and Proposition   \ref{p:inductive step compact type} guarantee that there is only one isoperiodic irreducible component of period $p$ in the boundary stratum of $\Omega\mnc_{g}$ corresponding to $V_{1}\oplus V_{2}$. Hence, $D_{1}=D_{2}$. The connectedness of $\mathcal{G}_{p}$ is therefore equivalent to that of the image of \eqref{eq:bdry inclusion ss}. Two vertices of the image of \eqref{eq:bdry inclusion ss} are said to be equivalent if they lie in the same connected component of the image. Suppose that for $i=1,2$,  $\{V_i,V_i^{\perp}\}$ are the modules associated to two vertices. We write $V_1\sim V_2$ if the vertices lie in the same connected component of the image of  \eqref{eq:bdry inclusion ss}. In particular $V\sim V^{\perp}$ for any pair $\{V,V^{\perp}\}$ representing a vertex in the image. In fact, more generally, if $V_1,V_2$ are orthogonal submodules of intermediate volume $\vol_p(V_i)\in (0,\vol(p))$ such that $$0<\vol_p(V_1)+\vol_p(V_2)<\vol(p)\text{  then }V_1\sim V_2.$$ Indeed, they are two of the factors of a $p$-admissible splitting $V_1\oplus (V_1\oplus V_2)^{\perp}\oplus V_2$. By Corollary \ref{c:realization of decompositions}, this splitting occurs as the period of a stable form without zero components on a curve with two separating nodes. Moreover it lies in the intersection of the closure of the components of $V_1$ and $V_2$.

We first prove that every vertex  given by a $p$-admissible splitting $V\oplus V^{\perp}$ is equivalent to another $V_1\oplus V_1^{\perp}$ where $V_1\in\mathcal{V}_p$, i.e. $V_1$ is of rank two. Indeed, if $V$ is of rank strictly between $2$ and $2g-2$ then $p_{|V}$ is an injective Haupt homomorphism and by Proposition \ref{c:line ranks for injective} there exists a $p_{|V}$-admissible splitting $V_1\oplus V_2$ of $V$ with a factor $V_1$ of rank $2$. Therefore $V_1\oplus V_2\oplus V^{\perp}$ is also a $p$-admissible splitting and $V\sim V^{\perp}\sim V_1$. It remains to connect the vertices corresponding to rank two $p$-admissible submodules.

\begin{proposition}\label{p: injective periods}
If $g\geq 4$, $p\in\mathcal{H}_g$ is injective and $V,V'\in\mathcal{V}_p$, we have $V\sim V'$. Therefore the image of \eqref{eq:bdry inclusion ss} is connected. 
\end{proposition}
We will split the proof in several Lemmas :

\begin{lemma}\label{c:rank one intersection}
Let $g\geq 4$ and $p\in\mathcal{H}_g$ be injective. If $V,V'\in\mathcal{V}_p$ satisfy $V\cap V'\neq 0$, then $$V\sim V'.$$
\end{lemma}
\begin{proof}
If $V=V'$ we are done. Suppose that $V\neq V'$

\vspace{0.5cm}

\noindent \textit{First step: there is a symplectic basis $a_1, b_1,\ldots , a_g, b_g$ such that $V= \mathbb Z a_1 + \mathbb Z b_1$ and $V' = \mathbb Z a_1 + \mathbb Z (b_1 + m_2 ' a_2)$ for a certain integer $m_2 '$.}

\vspace{0.2cm}

\begin{proof}[Proof of the first step] The intersection $V\cap V'$ is a primitive submodule of $H_1$, since both $V$ and $V'$ are primitive. Being of rank $1$, we have $V\cap V' = \mathbb Z a_1$ with $a_1$ primitive. Let $b_1\in V$ (resp. $b_1'\in V'$) such that $a_1 \cdot b_1 = 1$ (resp. $a_1 \cdot b_1' = 1$). These elements exist since $V$ and $V'$ are unimodular. The element $b_1 ' - b_1$ belongs to $a_1^\perp$. For a certain integer $n$, the element $b_1' + n a_1 - b_1 $ is also orthogonal to $b_1$. Change $b_1$ to $b_1 - na_1$ if necessary. We then have that $b_1' - b_1$ is orthogonal to $V = \mathbb Z a_1 + \mathbb Z b_1$. Write $b'_1 - b_1 = m_2 ' a_2$ where $a_2$ is a primitive element of $V^\perp$. Completing $a_2$ into a symplectic basis $a_2,b_2,\ldots, a_g, b_g$ of $V^\perp$ gives the desired statement.
\end{proof}
\vspace{0.5cm}
\noindent \textit{Second step: if the periods of $(V + V')^\perp$ do not lie in a real line of $ \mathbb C$, there exists a symplectic rank two submodule $W\subset H_1(\Sigma_g)$ such that $V \perp W$, $V' \perp W$ and
\begin{equation} \label{eq: volume condition}
0 < \vol_p (W) < \inf (\vol_p(V^{\perp}), \vol_p ((V')^\perp).\end{equation}
In particular, $V \sim W \sim V'$. }

\vspace{0.2cm}

\begin{proof}[Proof of second step] In the coordinates of the first step, we have $(V + V' )^\perp = \mathbb Z a_2 + X$ where $X := \sum _{k\geq 3} \mathbb Z a_k +\mathbb Z b_k$. We apply Proposition \ref{l:padmissible elements} to $p_{|X}$. If the restriction of $p$ to $X$ belongs to case (1) of that proposition, we are done. If it belongs to case (2), we use the

\vspace{0.2cm}

\begin{lemma} Let $X$ be a unimodular symplectic module of rank $4$. For every Lagrangian subspace $L\subset X$, there exists a symplectic rank two submodule $Y\subset X$ such that $L \cap Y = \{ 0 \}$.
\end{lemma}
\vspace{0.3cm}

\begin{proof} We can assume that $L= \mathbb Z a + \mathbb Z a'$ is primitive. Let $a_1= a$ and $b_1$ be an element of $X$ such that $a_1 \cdot b_1 = 1$. We have $a' = m_1 a_1 + c$ where $c\in (\mathbb Z a_1 + \mathbb Z b_1)^\perp$ and $m_1 \in \mathbb Z$. Up to replacing $a' $ by $a' - m_1 a_1$, we can assume that $m_1=0$. Since $L$ is primitive, so is $c$, so that we can extend the family $a_1, b_1, a_2=c$ to a symplectic basis of $X$. The symplectic submodule $Y= \mathbb Z (a_1 + b_2) + \mathbb Z b_1$ has the desired properties.
\end{proof}

\vspace{0.5cm}
From now on a Greek letter will denote the period of the corresponding Latin letter.
Let $l= \mathbb R \alpha_2$, and $L= X \cap p^{-1} (l)$. This space is either $\{0 \}$ or a Lagrangian subspace of $X$ since we assume the restriction of $p$ to $X$ is in case (2) of Proposition \ref{l:padmissible elements}. By the preceding Lemma, there exists a symplectic rank two submodule $Y \subset X$ such that $Y \cap p^{-1}(l) = \{ 0 \}$. Let $a', b'$ be a symplectic basis of $Y$, and let
$$ a = a' + A a_2, \ \ \ \ b = b' + B a_2,$$
for some $A,B \in \mathbb Z$. We have $a\cdot b=a'\cdot b' = 1$, and the volume of $W= \mathbb Z a + \mathbb Z b$ is given by
\[ \text{vol}_p (W) = \Im ( (\beta' + B\alpha_2) \overline{(\alpha ' + A \alpha_2)} ) = \Im ( \beta' \overline{\alpha'}) + \Im ( ( A \beta' - B \alpha') \overline{\alpha_2}). \]
By construction none of the cycles of $Y$ are mapped by $p$ to an element of the line $l=\mathbb R \alpha_2$, so the linear form $(A,B) \in \mathbb Z^2 \mapsto \Im ( A \beta' - B \alpha') \overline{\alpha_2}) \in \mathbb R$ is injective, and thus the volume of $W$ can approximate any real value. Since $W$ is orthogonal to both $V$ and $V'$, choosing $W$ with $\vol_p(W)\in (0,\vol(p))$ gives the solution to step 2.

\end{proof}
\vspace{0.5cm}

\noindent \textit{Third step: assume that the periods of $(V+V')^\perp$ lie on a real line $l\subset \mathbb C$. Then $V \sim V'$.}

\vspace{0.2cm}

\begin{proof}[Proof of the third step] Thanks to Remark \ref{rem:multiplying_periods} it suffices to show that $V$ is equivalent to $V'$ for some $wp$ where $w\in\mathbb{C}^*$. By choosing an appropriate $w$ of modulus one we can suppose that $wp((V+V')^{\perp})\subset \mathbb{R}$. If we prove Step 3 for the case $l=\mathbb{R}$ we will be done. Recall that $X = \sum_{i\geq 3} \mathbb Z a_i + \mathbb Z b_i \subset (V+V')^\perp$. Let $c\in X$ and define $V_{c} := \mathbb Z a_1 + \mathbb Z (b_1 + c)$. The volume of $V_{c}$ is given by
$$\text{vol} _p (V_{c}) = \text{vol}_p (V) + \Im ( p(c)\overline{\alpha_1}).$$
If $\alpha_{1}\in\mathbb{R}$, then $V_{c}$ is automatically $p$-admissible. Otherwise $p(\alpha_{1})\notin\mathbb{R}$ and $V_{c}$ is $p$-admissible as soon as
$$ -\text{vol}_p (V) < \Im ( p(c)\overline{\alpha_1}) < \text{vol}(p) - \text{vol}_p (V). $$
Now $X$ has rank at least $4$ and so has the image of the homomorphism $X\rightarrow \mathbb{R}$ $$c\mapsto\Im (p(c)\overline{\alpha}_1).$$ Therefore, the inequality has an infinite number of solutions $c\in X \setminus \{ 0 \}$. We can even impose a further condition that $c\cdot a_3=1$. Define $V''=V_{c}$ for such a solution $c$. 

We claim that $V'' \sim V$. Indeed, the space $(V+V'')^\perp$ contains the element $b_2 $. Observe that the period $\beta_2$ of $b_2$ is not real, since otherwise all the periods of $V^\perp$ would be real, and so we would have $\text{vol}_p (V) = \text{vol}(p)$ which contradicts $V\in \mathcal V_p$. On the other hand, the submodule $c^\perp \cap X$ has rank $\geq 3$ and is contained in $(V+V')^\perp$. Since the periods of $X$ are real, this proves that some periods of $(V+V'')^\perp$ are real. We can thus apply Step 2 to the couple $(V,V'')$ to infer $V''\sim V$.

To prove that $V'' \sim V'$ we will show that $(V'+V^{''})^{\perp}$ does not have all its periods in a line and apply Step 2. Consider the element $ b_2 + m_2'(a_1+a_3)$. It belongs to $(V' + V^{''})^\perp$. If it has real period, then $b_2+m_2'a_1$ has also real period, which implies that $(V')^{\perp}$ has only real periods. This is in contradiction with the fact that $V'$ is $p$-admissible. Therefore $b_2+m_2'(a_1+a_3)$ has non-real period. On the other hand there are elements in $(V'+V^{''})^{\perp}\cap X\setminus 0$ and their periods are non-zero real numbers.
\end{proof}\end{proof}
Lemma \ref{c:rank one intersection}, allows to reduce the equivalence relation $\sim$ on submodules in $\mathcal{V}_p$ to an equivalence relation on the elements that belong to those submodules.

\begin{definition} Let $p\in \mathcal{H}_g$. Recall that a primitive element $w\in H_1(\Sigma_g)$ is said to be $p$-admissible if it is contained in some module $V\in\mathcal{V}_p$.
Two $p$-admissible elements $w,w'$ are equivalent and denoted $w\sim w'$ if there exist $V,V'\in\mathcal{V}_p$ containing $w$ and $w'$ respectively such that $V\sim V'$.
\end{definition}
The transitivity property of this equivalence relation is proven by the use of Lemma \ref{c:rank one intersection}.

In particular, we already know that if $V\cap V'\neq 0$ then any pair of primitive elements in $V\cup V'$ are equivalent.

If $V$ and $W$ belong to $\mathcal{V}_p$ and there exists some elements $v\in V$ and $w\in W$ such that $v\sim w$, then $V\sim W$. Indeed, we can find $V',W'\in \mathcal{V}_p$ such that $v\in V'$, $w\in W'$ and $V'\sim W'$. By Lemma \ref{c:rank one intersection}, $V\sim V'$ and $W\sim W'$, so $V\sim W$.

Let us analyze the $p$-admissible elements.

\begin{lemma}\label{l:connecting primitive elements} Let $p:H_{1}(\Sigma_{g})\rightarrow\mathbb{C}$ be an injective Haupt homomorphism. Given $w_1,w_2,w_3\in H_1(\Sigma_g)$ such that
\begin{enumerate} \item $w_i\cdot w_{i+1}=1$ for $i=1,2$,
\item $p(w_3)\notin\mathbb{R} p(w_1)$ and
\item for every real line $\ell\subset\mathbb{C}$ containing $0$ $$\text{rank}(p^{-1}(\ell)\cap w_1^{\perp}\cap w_3^{\perp})<2g-3.$$
 \end{enumerate}
 Then there exists $w_2'\in H_1(\Sigma_g)$ such that $w_1\cdot w_2'=w_2'\cdot w_3=1$ and $\mathbb{Z}w_1\oplus\mathbb{Z}w_2'$ and $\mathbb{Z}w_2'\oplus\mathbb{Z}w_3$ belong to $\mathcal{V}_p$. Therefore $w_1$ and $w_3$ are $p$-admissible and $w_1\sim w_3$.
\end{lemma}

\begin{proof}
Write $w_2'=w_2+z$ where $z\in w_1^{\perp}\cap w_3^{\perp}$. If we show that the image of the map $$w_1^{\perp}\cap w_3^{\perp}\rightarrow\mathbb{R}^2$$ defined by $z\mapsto (\text{vol}_p(\mathbb{Z}w_1\oplus\mathbb{Z}(w_2+z)),\text{vol}_p(\mathbb{Z}(w_2+z)\oplus\mathbb{Z}w_3))$ has a point in the square $(0,\text{vol}(p))\times (0,\text{vol}(p))$ we will be done. The previous map is affine, with linear part $$\varphi(z)=(\Im(p(z)\overline{p(w_1)}),\Im(p(w_3) \overline{p(z)})).$$
Since $p(w_1)$ and $p(w_3)$ are not $\mathbb{R}$-colinear, $\text{Ker} (\varphi)=0$ and therefore $\text{rank}(\text{Im}\varphi)=2g-2$. The topological closure of $\text{Im}\varphi$ in $\mathbb{R}^2$ is either $\mathbb{R}$, $\mathbb{Z}\times\mathbb{R}$ or $\mathbb{R}^2$. Suppose it is not $\mathbb{R}^2$. Then there exists a submodule $H\subset w_1^{\perp}\cap w_3^{\perp}$ such that $\varphi(H)\subset\ell\subset\mathbb{R}^2$ for some real line $\ell$ passing through the origin and $\text{rank}H\geq (2g-2)-1=2g-3$. Write $\ell=\{(x,y): \alpha x+\beta y=0\}$ and then for each $z\in H$, $$\Im (p(z)\overline{(\alpha p(w_1)-\beta p(w_3))}=0.$$
Hence $p(H)\subset \mathbb{R}(\alpha p(w_1)-\beta p(w_3))$ is a submodule of rank at least $2g-3$ and we reach a contradiction with the rank hypothesis.
\end{proof}

The role played by the rank of $p$ on lines in Lemma \ref{l:rank}, Lemma \ref{l:connecting primitive elements} and Proposition \ref{l:padmissible elements} makes it useful to introduce the following

\begin{definition}
Given a homomorphism $p:W\rightarrow\mathbb{C}$ from a $\mathbb{Z}$-module $W$ we define its \textbf{line rank} as $$r(p)=\max_{a\in\mathbb{S}^1}\rank_{\mathbb{Z}}(p^{-1}(a\mathbb{R})).$$ \end{definition}

Remark that if $W$ is symplectic and $\text{vol}(p)>0$ then $r(p)< 2g$. Also, if $p$ is injective and $r(p)>g$, then the maximum is attained by a unique real line $\ell_{max}\subset\mathbb{C}$ containing $0$.

As far as  Lemma \ref{l:connecting primitive elements} is concerned, the rank condition is automatically satisfied if $r(p)<2g-3$. For $g\geq 4$ and $r(p)\geq 2g-3$ we have to check the rank condition only for $\ell_{max}$. In case $r(p)=2g-1$ the rank condition cannot be attained. In the other two cases it depends on whether there exists a rank two submodule of $w_1^{\perp}\cap w_3^{\perp}$ with periods outside $\ell$ or not.

\begin{lemma}\label{l:suite de primitives}
For $g\geq 2$, given primitive $w_1, w_4\in H_1(\Sigma_g)$ such that $w_1\cdot w_4=0$, there exists $w_2, w_3\in H_1(\Sigma_g)$ such that $$w_i\cdot w_{i+1}=1\text{ for }i=1,2, 3 $$
\end{lemma}

\begin{proof}
Let $w_2'$ satisfy $w_1\cdot w_2'=1$ and $w_3'$ satisfy $w_3' \cdot w_4=1$. Since $w_1\in w_4^{\perp}$ for any $k$ we have $(w_3'+kw_1)\cdot w_4=1$. Choose $k$ as to have $w_2' \cdot (w_3'+kw_1) =0$ and define $w_3=w_3' +k w_1$. It is primitive and we can take $w_2''$ satisfying $w_2''\cdot w_3=1$. Since $w_2'\in w_3^{\perp}$ there exists $l$ such that $w_2=(w_2'' +l w_2')$ satisfies $w_1\cdot w_2=1$.
\end{proof}

\begin{lemma}\label{l: low line rank}
 If $g\geq 4$, $p:H_1(\Sigma_g)\rightarrow\mathbb{C}$ is an injective homomorphism of positive volume and $r(p)<2g-3$, then any pair of primitive elements in $H_1(\Sigma_g)$ are equivalent.
\end{lemma}

\begin{proof}
 Let $v,w\in H_1(\Sigma_g)$ be two primitive elements. By taking $z\in v^{\perp}\cap w^{\perp}$ and applying twice Lemma \ref{l:suite de primitives} we can consider a sequence $w_0,w_1,\ldots,w_6\in H_1(\Sigma_g)$ such that $w_i\cdot w_{i+1}=1$ for $i=0,\ldots, 5$, $w_0=v$, $w_3=z$ and $w_6=w$. 
 
 We claim that there exist $c_2\in w_1^{\perp}\cap w_3^{\perp}$ and $c_4\in w_3^{\perp}\cap w_5^{\perp}$ such that for $w_0'=w_0$, $w_2'=w_2+c_2$, $w_4'=w_4+c_4$ and $w_6'=w_6$ we can apply Lemma \ref{l:connecting primitive elements} to each of the triples $w'_i,w'_{i+1},w'_{i+2}$ for $i=0,2,4$ and conclude. By the line rank assumption on \(p\), we need to verify 
 \begin{equation}
 \Im(p(w_2')\overline{p(w_0)})\neq 0, \quad
 \Im(p(w_4')\overline{p(w'_2)})\neq 0, \quad
 \Im(p(w_6')\overline{p(w'_4)})\neq 0,
 \end{equation}
 for this claim to hold. 
 
 If \(p(w_2')\) is always colinear to \( p(w_0)\) for any choice of \(c_2\in w_1^\perp\cap w_3^\perp\), this means that \(p(w_1^\perp\cap w_3^\perp)\) is contained in the line \( \mathbb R p(w_0)\), which contradicts the line rank assumption. Let us fix \(c_2\in w_1^\perp\cap w_3^\perp\) (and the corresponding \(w_2'\)) such that \( p(w_2')\) is not colinear with \(p(w_0)\). 
 
 Now, the same argument shows that each affine subspace 
 \[ \{ c_ 4 \in w_3 ^\perp\cap w_5^\perp:\  \Im(p(w_4')\overline{p(w'_2)}) =0 \} \text{ and } 
 \{ c_ 4 \in w_3 ^\perp\cap w_5^\perp:\  \Im(p(w_6')\overline{p(w'_4)}) =0 \} \]
have positive corank. Hence, the union of these subspaces does not fill the whole \( w_3 ^\perp\cap w_5^\perp\). Choosing \(c_4\) outside the union of these subspaces gives the solution to Lemma \ref{l: low line rank}.\end{proof}

 Lemma \ref{l: low line rank} provides a proof of Proposition \ref{p: injective periods} for the case of $r(p) < 2g-3$. In fact, a similar reasoning, but more elaborated, permits to cover also the case \( r(p)= 2g-3\). However, we will treat this case, together with the cases \( r(p) \geq 2g-3\), by different methods, in the next family of lemmata:

\begin{lemma} \label{l:non interdits perp}
Let $g\geq 4$ and $p:H_1(\Sigma_g)\rightarrow\mathbb{C}$ an injective Haupt homomorphism with $r(p)\geq 2g-3$. Write $I=p^{-1}(\ell_{max})$ and
suppose $v,w\in H_1(\Sigma_g)\setminus I$ are primitive elements such that $v\cdot w=0$ and $[w]\in v^{\perp}/\mathbb{Z}v$ is also primitive. Then $v\sim w$.
\end{lemma}
\begin{proof}By applying Remark \ref{rem:multiplying_periods} we can suppose without loss of generality that $\ell_{max}=\mathbb{R}$. Choose $b\in w^{\perp}$ such that $v\cdot b=1$. We claim that, up to changing $b$ by $b+e$ for some $e\in I\cap w^{\perp}\cap v^{\perp}$ we can suppose that $V=\mathbb{Z}v\oplus\mathbb{Z}b$ belongs to $\mathcal{V}_p$. Indeed, since $p(e)\in\mathbb{R}$,
$$\vol(\mathbb{Z}v\oplus\mathbb{Z}(b+e))=\vol(\mathbb{Z}v\oplus\mathbb{Z} b)+p(e)\Im(\overline{p(v)}).$$
By hypothesis the rank of $p(I\cap w^{\perp}\cap v^{\perp})$ is at least $2g-5\geq 3$ for $g\geq 4$, so the value of the volume of $V$ can be chosen arbitrarily close to any desired value.

Next take $c\in V^{\perp}$ such that $w\cdot c=1$. Given $f\in w^\perp\cap V^{\perp}\cap I$, we have $$\vol(\mathbb{Z}w\oplus\mathbb{Z}(c+f))=\vol(\mathbb{Z}w\oplus\mathbb{Z} c)+p(f)\Im(\overline{p(w)}).$$
Again, since the rank of $w^\perp\cap V^{\perp}\cap I$ is at least $2g-6\geq 2$ for $g\geq 4$, we can suppose that $c$ is chosen so that $W=\mathbb{Z}w\oplus\mathbb{Z} c$ belongs to $\mathcal{V}_{p_{|V^{\perp}}}$. By construction $V\perp W$ and $0<\vol (V)+\vol(W)<\vol (p)$. Therefore $V\sim W$ and also $v\sim w$.
\end{proof}

\begin{lemma}\label{l:non interdits generique}
Let $g\geq 4$ and $p:H_1(\Sigma_g)\rightarrow\mathbb{C}$ an injective Haupt homomorphism with $r(p)\geq 2g-3$. Define $I=p^{-1}(\ell_{max})$. If $v,w\in H_1(\Sigma_g)\setminus I$ are primitive such that $$v^{\perp}\cap w^{\perp}\nsubseteq I$$ then $v\sim w$.
\end{lemma}
\begin{proof}
Take a symplectic basis $a_i,b_i$ of $H_1(\Sigma_g)$ such that $a_1=v$ and $$w=m_1a_1+n_1b_1+m_2 a_2.$$
Let $X:=\mathbb{Z}a_3\oplus\mathbb{Z}b_3\oplus\cdots\oplus\mathbb{Z}a_g\oplus\mathbb{Z}b_g$. If $X\nsubseteq I$ then choose $z\in X\setminus (I\cap X)$ primitive. By Lemma \ref{l:non interdits perp} $v\sim z$ and $w\sim z$ therefore $v\sim w$. If $X\subset I$ take $c\in (v^{\perp}\cap w^{\perp})\setminus I$ and write $c=c_{X}+c_{\perp}$ where $c_{\perp}\in X^{\perp}$. Then $c_{\perp}\notin I$ and $c_{\perp}\cdot v=c_{\perp}\cdot w=0$, so $v,w\in (X')
^{\perp}$ where $X':=\mathbb{Z}(a_3+c_{\perp})\oplus\mathbb{Z}b_3\oplus\cdots\oplus\mathbb{Z}a_g\oplus\mathbb{Z}b_g$. Restart the argument of the proof with a symplectic basis of $(X')^{\perp}$ whose first element is still $a_1=v$ and complete it to a basis of $H_1(\Sigma_g)$ with the chosen basis of $X'$. Since $X'\nsubseteq I$ the argument will fall in the previous case and we will be done.
\end{proof}
\begin{lemma}\label{l:interdits non generique}
 Let $g\geq 4$ and $p:H_1(\Sigma_g)\rightarrow\mathbb{C}$ an injective Haupt homomorphism with $r(p)\geq 2g-3$. Define $I=p^{-1}(\ell_{max})$. Let $v\in H_1(\Sigma_g)\setminus I$ and define $$I_v=\{z\in H_1(\Sigma_g): z^{\perp}\cap v^{\perp}\subset I\}.$$ Then there exists a proper submodule $J\varsubsetneq H_1(\Sigma_g)$ such that $I_v\subset J$.
\end{lemma}
\begin{proof}
 If $I_v=\emptyset$, the module $J=0$ does the job. Otherwise take $z\in I_v$. Then $z^{\perp}\cap v^{\perp}\subset v^{\perp}\cap I$. We also have $$2g-2\leq \text{rank}(v^{\perp}\cap I)<\text{rank} (v^{\perp})=2g-1$$ where the strict inequality comes from the fact that $v\notin I$ and $I$ is a primitive module. Therefore $\text{rank}(v^{\perp}\cap I)=2g-2$. Its primitive submodule $z^{\perp}\cap v^{\perp}$ has also rank $2g-2$ so the only possibility is that $z^{\perp}\cap v^{\perp}=v^{\perp}\cap I$. Therefore $z\in (v^{\perp}\cap I)^{\perp}=:J$
\end{proof}

\begin{lemma}
 Let $g\geq 4$ and $p:H_1(\Sigma_g)\rightarrow\mathbb{C}$ an injective Haupt homomorphism with $r(p)\geq 2g-3$. Then for any $V,W\in\mathcal{V}_p$ we have $V\sim W$.
\end{lemma}
\begin{proof}
Again we suppose $\ell_{max}=\mathbb{R}$ and define $I=p^{-1}(\mathbb{R})$. Since $V$ and $W$ are of positive volume we can find primitive elements $v\in V\cap I^{c}$ and $w\in W\cap I^{c}$. If $w^{\perp}\cap v^{\perp}\nsubseteq I$ we have $v\sim w$ by Lemma \ref{l:non interdits generique}. Therefore $V\sim W$.

If $w^{\perp}\cap v^{\perp}\subset I$ we can consider the union $I\cup I_v\cup I_w$. Since by Lemma \ref{l:interdits non generique} it is contained in a \emph{union} of proper submodules, it cannot cover the whole of $H_1(\Sigma_g)$. Take $z\in H_1(\Sigma_g)\setminus (I\cup I_v\cup I_w)$. Then by Lemma \ref{l:non interdits generique} $v\sim z\sim w$, which as before implies that $V\sim W$.
\end{proof}

Proposition \ref{p: injective periods} is now proven for all possible ranks of an injective $p\in\mathcal{H}_g$.

\subsection{Proof of Theorem \ref{t:connected boundary} for non-injective $p\in\mathcal{H}_g$}

In the case of non-injective $p$ we also consider the dual boundary complex $\mathcal{G}_p$ of the stratification of $\percompact^{-1}(p)$. 

 Remark that each connected component of a codimension one boundary stratum in  $\percompact^{-1}(p)$ is contained in a component of the boundary of $\Omega\mnc_{g}$ and again we can define a continuous map of dual complexes
 \begin{equation}\label{eq:bdry inclusion ssker}
    \mathcal{G}_p\rightarrow \curvecomplex_g/\mathcal{I}_g. 
\end{equation}
However, in the case $\ker p\neq 0$ we know that there will be some image points that correspond to Torelli classes of non-separating curves (see Lemma \ref{l:existence of pinched classes}). In terms of the parametrization of the vertices of  $\curvecomplex_g/\mathcal{I}_g$ given in subsection \ref{s:parametrization of torelli classes},  these correspond to primitive classes in $\ker p\setminus 0$ pinched by $p$.

In contrast with the case of injective periods, when $p$ has large kernel, there may be boundary components of codimension one of $\Omega\mnc_g$ whose intersection with $\percompact
^{-1}(p)$ is disconnected (when $\ker p$ contains symplectic submodules of large rank). For instance, the boundary component of $\Omega\mnc_{g}$ corresponding to a $p$-admissible splitting $V_1\oplus V_2$ with  $\rank V_1\geq 10$ and $\deg( p_{|V_1})=2$ provides an example, via attaching maps and Theorem \ref{t:degree 2}. This implies that the map \eqref{eq:bdry inclusion ssker} is no longer injective. However, we can use the inductive hypothesis to identify a particular subcomplex $\mathcal{G}_{p}'\subset\mathcal{G}_{p}$ spanned by \textit{simple vertices} where the restriction of \eqref{eq:bdry inclusion ssker} is injective: 

\begin{definition}
Let $p\in\mathcal{H}_{g}$. A $p$-admissible splitting $V_1\oplus V_2$ is $\spl$ if it satisfies  $\rank (V_i)\leq 6$ or $\deg(p_{|V_{i}})\geq 3$ for $i=1,2$. A pinched class $a\in\ker p\setminus 0$ is $\spl$ if either $g=3,4$ or $\deg p_a\geq 3$. A vertex of $\mathcal{C}_g/\mathcal{I}_g$ is $p$-simple if the corresponding $p$-admissible splitting or pinched class is $p$-admissible and simple. 
\end{definition}

\begin{lemma}\label{l:existence of good vertices}
Let $g\geq 3$. For any $p\in\mathcal{H}_{g}$ with $\ker p\neq 0$ and $\deg p\geq 3$ there exist $p$-$\spl$ vertices of non-compact type in $\mathcal{C}_g/\mathcal{I}_g$. 
\end{lemma}
\begin{proof} For $g=3,4$ the statement is equivalent to the existence of pinched classes (see Lemma \ref{l:existence of pinched classes}). Suppose $g\geq 5$ and that there is a vertex that is not $\spl$. If it is of compact type, it means that one of the factors of the associated $p$-admissible splitting is of rank at least eight and the restriction of $p$ to it is of primitive degree two. If it is of non-compact type, there is a symplectic submodule of rank at least 8 where the restriction of $p$ has degree two. In either case there is a symplectic submodule of $\ker p$ of rank at least two. Each primitive element in that submodule determines a simple vertex of non-compact type of $\mathcal{G}_p$.  
\end{proof}
\begin{proof}[Proof of Theorem \ref{t:connected boundary} in the non injective case]

We claim that the restriction of \eqref{eq:bdry inclusion ssker} to the subcomplex $\mathcal{G}_{p}'$--spanned by simple vertices-- is injective. Indeed,  suppose $D_{1}$ and $D_{2}$ are irreducible boundary components in boundary strata contained in  $\percompact^{-1}(p)$  lying in a boundary component of $\Omega\mnc_{g}$ corresponding to a simple $p$-admissible splitting or primitive class in $\ker p\setminus 0$.  Propositions \ref{p:inductive step compact type} and \ref{p:inductive step non compact type} describe this set  as the image of products of isoperiodic sets under attaching maps. However, by inductive hypothesis (or by  Theorem \ref{t:connectedness in genus 2 and 3} in the cases of genus two and three and degree two) the products of isoperiodic sets under consideration are connected. Hence $D_{1}=D_{2}$. We can therefore talk about simple vertices of $\mathcal{G}_p$, namely, those lying over $p$-simple vertices of $\mathcal{C}_g/\mathcal{I}_g$.  
The connectedness of the complex $\mathcal{G}_p$ in the non injective case proceeds in three basic steps:

\vspace{0.3cm}

\noindent \emph{Step 1:} Each vertex of compact type is equivalent in $\mathcal{G}_p$ to some vertex of non-compact type.\\ 
\emph{Step 2:} Each vertex of non-compact type is equivalent in $\mathcal{G}_p$ to some $\spl$ vertex of non-compact type.\\
\emph{Step 3:} Any pair of $\spl$ vertices of non-compact type are equivalent in $\mathcal{G}_{p}'$. 

\subsubsection{Proof of Step 1: from compact type to non-compact type:}
Let $V_{1}\oplus V_{2}=H_{1}(\Sigma_{g})$ be the $p$-admissible splitting associated to a separating simple closed curve $c$ in $\Sigma_{g}$ and denote $p_i=p_{|V_{i}}$. If one of the $p_{i}$'s is non-injective we claim that any stable form $$(C,m,\omega)=(C_{1}\vee C_{2}, m_{1}\oplus m_{2},\omega_{1}\vee \omega_{2})$$ of period $p$ sitting in the boundary component associated to $V_1\oplus V_2$ is equivalent to a stable form with a non-separating node. Indeed, suppose without loss of generality, that $\ker p_{2}\neq 0$. Then $V_{2}$ has at least rank $4$. 
Suppose first that the vertex defined by $(C,m,\omega)$ is $\spl$. 
Then, under the hypotheses of Proposition \ref{t:connected boundary}, $\percompact^{-1}(p_{2})$ is connected, contains $\omega_{2}$ and a stable form $\omega_{3}$ over a curve with a non-separating node. An isoperiodic path $\eta_{t}$ between them provides an isoperiodic path $\omega_{1}\vee\eta_{t}$ between 
$\omega=\omega_{1}\vee\omega_{2}$ and $\omega_{1}\vee\omega_{3}$. The latter is a form in $\Omega_{0}^{*}\mrs_{g}$ with periods $p$ and two nodes, one of which is non-separating. Therefore it defines an edge between the vertex defined by $(C,m,\omega)$ and a vertex of non-compact type of $\mathcal{G}_p$.

Next suppose that the vertex defined by $(C,m,\omega)$ is not $\spl$. Then $\rank V_{2}\geq 8$ and $\deg p_2=2$. Integration of $\omega_{2}$ along $C_{2}$ provides a branched degree two cover $$C_{2}\rightarrow \mathbb{C}/p_{2}(V_{2})=:E_{2}$$ having at least two distinct critical values. The preimage of a path in $E_{2}$ joining these critical values provides a pair of twin paths for $\omega_{2}$ sharing the same endpoints. The Schiffer variation along this pair of twin paths describes an isoperiodic deformation $\eta_{t}$ joining $\omega_{2}$ with a stable form $\omega_{3}$ with a non-separating node. Indeed, if it would join the two critical points to a single critical point, the local degree of the branched covering around the latter would be at least three, contradicting the degree two hypothesis. If it would collapse to a separating node, the degree of the restriction of the period map to each part of the curve would be one, and since one of the parts has genus at least two it would contradict Haupt's Theorem. 

It remains to treat the case where both $p_1$ and $p_2$ are injective. We will prove that the (\spl) vertex is equivalent to another vertex of compact type where the restriction of $p$ to one of the parts has non-trivial kernel. 

If both $p_1$ and $p_2$ are injective, we can find $p$-admissible splittings of $V_1$ and $V_2$ with all factors of rank $2$ by using Corollary \ref{c:line ranks for injective} inductively to $p_1$ and $p_2$. In particular the vertex associated to $(C,m,\omega)$ is equivalent to a vertex defined by a $p$-admissible splitting having a rank two factor, so we can suppose $\rank V_1=2$. If $p_2$ is non-injective we are done. It remains to treat the case where $p_2$ is injective.

Take a primitive element $a\in\ker p$ and write $a=n_1a_1+n_2a_2$ for primitive $a_i\in V_i$ and co-prime $n_1,n_2\in\mathbb{N}^*$.
If $a_2\in V_2$ is contained in a factor $V$ of a $p_2$-admissible splitting of $V_2$, define $V_1'=V_1\oplus V$. The splitting $V_1'\oplus V_1'^{\perp}$ is also $p$-admissible and satisfies $\ker (p_{|V_1'})\neq 0$. It defines a vertex that is equivalent to the one defined by $V_1\oplus V_2$. In particular this argument works if $a_2$ is $p_2$-admissible.

So suppose $p_2$ is injective and $a_2\neq 0$ is not $p_2$-admissible. Then by Lemma \ref{l:rank} applied to $p_2$, the rank of $p_2 ^{-1} ( \mathbb R p (a_2)) \cap a_2 ^{\perp}$ is at least $(2g-2)-3+1=2g-4$. Completing $a_2$ to a symplectic basis $a_2, b_2, a_3, b_3, \ldots , a_g, b_g$ of $V_2$, and denoting $V_3 :=\mathbb Z a_3 + \mathbb Z b_3 + \ldots + \mathbb Z a_g + \mathbb Z b_g$, we conclude that $H = (p_2 ^{-1} (\mathbb R p(a_2) )\cap a_2^{\perp})\cap b_2^{\perp}=p ^{-1} (\mathbb R p(a_2) ) \cap V_3$ is either $V_3$ or a co-rank one primitive submodule of $V_3$.

In the latter case, by considering an element $w\in V_3$ such that $p(w) $ does not belong to $\ell:=\mathbb R p(a_2)$, we apply Lemma \ref{l:rank} to $\ell$ and $V_3$ to construct a symplectic rank two submodule $W \subset V_3$ containing $w$ with $0<\vol_p(W)<\vol_p(V_2)$. Since $p_2$ is injective, this implies that $W\in\mathcal{V}_{p_2}$. The splitting $V_1\oplus W\oplus (W^{\perp}\cap V_2)$ of $H_1(\Sigma_g)$ is $p$-admissible, and in particular so is $W\oplus W^{\perp}$. On the other hand $a\in W^{\perp}\cap\ker p$ so the restriction $p_{|W^{\perp}}$ is non-injective and we are reduced to one of the previous cases.

It remains to treat the case where $H = V_3$, namely $p(V_3) \subset \mathbb R p(a_2)$. In this case we will find a new initial $p$-admissible splitting $V_1\oplus V_2$ that falls in one of the previous cases and that defines an equivalent splitting. The new $V_1$ that we want to construct, call it $V_1'$, will be a rank two factor of a $p_2$-admissible splitting of $V_2$.

The candidates to $V_1'$ satisfy $0< \text{vol}_p (V_1') < \text{vol} _p (V_2)$. The initial splitting of $a$ will change to $$a= n_1 '' a_1 '' + n_2 '' a_2''$$ with primitive $a_1 '' \in V_1'$ and $a_2'' \in V_2' := (V_1 ' ) ^\perp$ and co-prime $n_1''$ and $n_2''$ (the reason for this notation $a_i''$ instead of $a_i'$ will become clear hereafter). We claim that it is possible to find $V_1'$ so that either $a_2''$ is $p_{V_2'}$-admissible or $p_2'=p_{|V_2'}$ is not injective. We already explained that this would conclude the proof.

Take $b_1\in V_1$ such that $a_1\cdot b_1=1$.
Up to composing $p$ with a $\mathbb R$-linear orientation preserving equivalence from $\mathbb C$ to $\mathbb C$, we can assume that
$$ p(a_1)= n_2, \ \Im p(b_1) = 1 , \ p(a_2) = -n_1, \ \Im p(b_2) < 0 ,$$ and that
$$ p(a_k) = \alpha _k \in \mathbb R, \ p(b_k) = \beta_k \in \mathbb R \text{ for } k\geq 3. $$

By injectivity of $p_2$, the numbers $\alpha_3,\ldots,\alpha_g,\beta_3,\ldots, \beta_g$ are linearly independent over $\mathbb Q$.

We are going to look for the module $V_1 '$ as being generated by the elements $a_1 '$ and $b_1 '$, where
$$ a_1 ' = a_2 + \sum _{k\geq 3} m_k a_k + n_k b_k, \ b_1 ' = b_2 , \ a_2 ' = a_1, \ b_2 ' = b_1 $$
and for $k\geq 3$
$$ a_k ' = a_k + n_k b_2 , \ b_k' = b_k - m_k b_2 $$ form a new symplectic basis for $H_1(\Sigma_g)$.
Here $m_k, n_k$ are integers that have to be determined for $k\geq 3$. We have
$$ \text{vol} _p (V_1 ' ) = -\Im(p(b_2)) ( n_1 - \sum _{k\geq 3} m_k \alpha_k + n_k \beta_k ).$$
Observe also that $\text{vol} _p ( V_2 ) = - n_1 \Im (p(b_2))$.
We will choose $m_k, n_k$ multiple of $n_1$, so we write $m_k = m_k' n_1$, $n_k= n_k' n_1$ with $m_k', n_k'$ integers. We then have
$$ \text{vol} _p (V_1 ' ) = \varepsilon \text{vol}_p (V_2) \text{ with } \varepsilon = 1 - \sum _{k\geq 3} m_k' \alpha_k + n_k' \beta_k. $$
Since $p_2$ is injective, $V_1'$ is $p_2$-admissible if and only if $0< \varepsilon < 1$. By rational independence of $\alpha_3,\ldots,\alpha_g,\beta_3,\ldots, \beta_g$ we can find $m'_k,n'_k\in\mathbb{Z}$ satisfying this condition. If for such a choice the homomorphism $p_2':=p_{|V_2'}$ is not injective, we are done. Suppose that for all choices we have that $p_2'$ is injective. In this case we will refine the solution to have $a_2''$ to be $p_2'$-admissible.

Let us understand how the class $a$ decomposes according to the splitting $V_1 ' + (V_1 ') ^\perp$: it is given by $ a = n_2 a_1 '' + n_1 a_2'' $ with $ a_1'' = a_1' $ and $$a_2'' = a_2 ' - n_2 \sum _{k\geq 3} m_k ' a_k' + n_k' b_k' . $$
Hence, it suffices to find $n'_k,m'_k\in\mathbb{Z}$ such that $a_2 '' $ is $p_2'$-admissible. The volume of the symplectic rank two submodule $\mathbb Z a_2 '' + \mathbb Z b_2 ' \subset V_2 '$ (containing $a_2''$) is
$$ \text{vol}_p(\mathbb Z a_2 '' + \mathbb Z b_2 ' ) = \Im ( p(b_2 ' ) \overline{ p (a_2'') } ) = n_2 \big(1- \sum_{k\geq 3} m_k ' \alpha_k + n_k' \beta_k \big) = n_2 \varepsilon, $$
while
$$ \text{vol}_p ( V_2 ' ) = \text{vol}_p - \text{vol}_p (V_1') = \text{vol} _p - \varepsilon \text{vol}_p (V_2) . $$
Hence, as soon as $0 < \varepsilon < \frac{\text{vol}_p}{n_2 + \text{vol} _p (V_2)}$ one concludes that $a_2''$ is $p_2'$-admissible.
The rational independence of $\alpha_3,\ldots,\alpha_g,\beta_3,\ldots, \beta_g$ allows to find a solution, again.

\subsubsection{Proof of Step 2: From non-compact type to $\spl$ of non-compact type}
\begin{lemma}\label{l:degree_normalization}
Let $(C,\omega)$ be a stable form with a non-separating node having $(\widehat{C},\widehat{\omega})$ as its normalization and $n_1,n_2 \in \hat{C}$ the points corresponding to the node in $C$. Let $p$ (resp. $\widehat{p}$) be the period homomorphism associated to $\omega$ (resp. $\widehat{\omega}$). Assume that $\deg(\widehat{p})<\infty$, and let $\widehat{F}:\widehat{C}\rightarrow \mathbb{C}/\widehat{p}(H_1(\widehat{C}))$ be the map defined by integration of $\widehat{\omega}$. Then $$\deg(\widehat{p})=\deg (p)\Leftrightarrow \widehat{F}(n_1)=\widehat{F}(n_2) .$$
\end{lemma}
\begin{proof}
The proof is straightforward once one realizes that $\vol(\hat{p})=\vol(p)$. 
\end{proof}
Continuing with the notations of Lemma \ref{l:degree_normalization} we consider a marked stable form $(C,m,\omega)$ with period map $p$ of primitive degree at least three, and a non-separating node that collapses a curve $c$ such that $\deg(p_c)=2$. The period map $\widehat{p}$ of its normalization $(\widehat{C},\widehat{\omega})$ has therefore precisely degree two. By Lemma \ref{l:degree_normalization}, $\widehat{F}(n_1)\neq \widehat{F}(n_2)$ for the integral $$\widehat{F}:\widehat{C}\rightarrow \mathbb{C}/\widehat{p}(H_1(\widehat{C}))$$ of $\widehat{\omega}$. The map $\widehat{F}$ is a branched covering of degree two. All its branch points are therefore simple and have distinct values for $\widehat{F}$. Choose a path in \(\mathbb{C}/\widehat{p}(H_1(\widehat{C}))\) that avoids the two points $\widehat{F}(n_1)$ and $\widehat{F}(n_2)$ and joins two distinct critical values. Its preimage by $\widehat{F}$ defines a pair of twins that join two critical points in $\widehat{C}$ and also two zeros of $\omega$ in $C$. The closed curve formed by both twins cannot be separating. Indeed, if it was separating, it would induce a $\widehat{p}$-admissible splitting of $H_1(\Sigma_{g-1})$. However, these do not exist for $\deg(\hat{p})=2$ and $g-1\geq 3$ (see Example \ref{ex:non compact type}). Therefore, the curve formed by both twins is non-separating. After the Schiffer variation on \((\widehat{C},\widehat{\omega})\) along these twins, we obtain a new stable form $(\widetilde{C},\widetilde{\omega})$ of period $\widehat{p}$ with a node  $\tilde{n}$ that is non separating. It is a degree two cover of an elliptic differential. By Lemma \ref{l:degree_normalization}, identifying the points \(n_1 \) and \(n_2\) on \((\tilde{C},\tilde{\omega})\) produces a stable form whose period map has finite primitive degree different from $\deg (\widehat{p})=2$ (and from one by Haupt's Theorem). Hence it has degree at least three.  Again by the same Lemma, the normalization of $\tilde{n}$ on this form has that same degree, larger than three.  



\subsubsection{Proof of Step 3: connecting $\spl$ vertices of non-compact type in $\mathcal{G}_p'$ }

Recall that each $\spl$ vertex of non-compact type of the complex $\mathcal{G}_p$ corresponds to a unique cyclic primitive submodule $\mathbb{Z}a\subset \ker p$ such that $\deg p_a\geq 3$. We will parametrize those vertices by the primitive elements $\pm a\in\ker p\setminus 0$ generating the submodule. We say that two such primitive elements $a, a'\in\ker p\setminus 0$ are equivalent if the corresponding vertices in $\mathcal{G}'_p$ lie in the same connected component of the subcomplex $\mathcal{G}'_p$ of $\mathcal{G}_p$ spanned by simple vertices . 
\begin{lemma}\label{l:equivalent vertices in kernel}
Given $g\geq 4$, a homomorphism $p:H_1(\Sigma_g)\rightarrow \mathbb{C}$ with $\vol p>0$, $\deg p\geq 3$ and 
\begin{enumerate}
 \item a symplectic submodule $W\subset \ker p$ of rank two, then every pair of $\spl$ vertices of non-compact type of $\mathcal{G}'_p$ corresponding to classes in $W\cup W^{\perp}$ are equivalent. 
 \item an isotropic primitive submodule $L\subset\ker p$ of rank two, then every pair of $\spl$ vertices of non-compact type of $\mathcal{G}_p'$ corresponding to classes in $L$ are equivalent. 
 \item a $p$-admissible splitting $H_1(\Sigma_g)=V_1\oplus V_2$. The primitive elements $a\in V_i\cap\ker p$ such that $\deg\big((p|_{V_i})_a\big)\geq 2$ satisfy $\deg(p_a)\geq 3$. \end{enumerate}

\end{lemma}
\begin{proof}
$(1)$ First remark that the homomorphism $p_{1}:=p_{|W^{\perp}}$ is defined on a module of rank at least $6$ and has the same volume and primitive degree as $p$. By Lemma \ref{l:existence of good vertices} applied to $p_{1}$ we deduce that there exists a $\spl$ vertex for $\mathcal{G}_{p_1}$. Thanks to Lemma \ref{l:adding handles to Haupt} and the fact that $W\subset \ker p$, it also determines a simple vertex for $\mathcal{G}_p$. Using Corollary \ref{c:realization of decompositions} we construct a marked stable form $(C_{1},m_{1},\omega_{1})$ of genus $g-1$ and period $p_1$, with a node that induces the given $\spl$ vertex and -- thanks to the genus hypothesis and up to some Schiffer variations -- a simple zero at some regular point. Take a pair of embedded twins with distinct endpoints starting at the simple zero and glue the endpoints, thus obtaining a stable form on a genus $g$ curve with two nodes. After the gluing, the union of the twins forms a closed loop in the nodal curve whose period vanishes. Given any symplectic basis $a,b$ of $W$ we mark the obtained nodal form by collapsing $a$ to the node, associating $b$ to the loop obtained from the twins, and keeping the marking on $W^{\perp}$ as it was for $(C_{1},\omega_{1},m_{1})$. The given form with two nodes defines an edge in $\mathcal{G}_p$ joining two $\spl$ vertices. One of them corresponds to the node of the initial $\omega_1$ and is independent of the chosen $a$. In particular all primitive classes in $W$ define equivalent $\spl$ vertices. Moreover, if the chosen node for $\omega_{1 }$ is determined by a primitive element $\ker p\cap W^{\perp}$ defining a $\spl$ vertex for $\mathcal{G}_p$ it is also a $\spl$ vertex for $\mathcal{G}_{p_1}$. This vertex is equivalent to any vertex defined by a primitive element of $W$. This shows that all simple vertices of non-compact type of $\mathcal{G}_p$ corresponding to primitive classes in $W\cup W^{\perp}$ are equivalent. \\
(2) The module $L^{\perp}/L$ has rank $2g-4$ and the ambient symplectic form on $H_{1}(\Sigma_{g})$ induces a symplectic form on it. Since $L\subset \ker p$, the map $p$ induces a homomorphism $$p_{L}:L^{\perp}/L\rightarrow\mathbb{C}.$$ 
If $\ker p_{L}$ contains a non-trivial symplectic submodule, then by item $(1)$ of this Lemma, every pair of $\spl$ vertices in $L$ are equivalent. Item $(2)$ is proven for the cases $\deg p_{L}=1$ (where the kernel is symplectic of corank two) or $g\geq 5$ and $\deg p_{L}<\infty$ (thanks to Lemma \ref{lem:symplectic submodules of kerp} ). It remains to treat the cases $g=4$ and $\deg p_{L}\geq 2$ and $g\geq 5$ and $\deg p_{L}=\infty$. 

Next suppose $a_{1},a_{2}$ forms a basis of $ L$ corresponding to $\spl$ vertices and $\deg p_{L}\geq 2$. We claim that we can construct a marked stable form of period $p$ with precisely two non-separating nodes that correspond to the given classes. The corresponding vertices are thus joined by an edge in $\mathcal{G}_{p}$. 
Indeed, take a form $(C_{2},m_{2},\omega_{2})$ of genus $g-2$ and periods $p_{L}$ with two simple zeros. Complete the elements $a_{1 },a_{2}$ to a symplectic basis $a_{i},b_{i}$ of $H_{1}(\Sigma_{g})$. For $i=1,2$, apply Lemma \ref{lbz} to $\omega_{2}$ to choose an embedded path $\beta_{i}$ in $C_{2}$ having distinct endpoints and satisfying 
$$p(b_i)=\int_{\beta_{i}}\omega_{2}.$$ We can further assume that $\beta_{1}$ and $\beta_{2}$ are disjoint. Indeed, if $p(b_1)=p(b_2)=0$ it suffices to take pairs of short twins at distinct zeros of $\omega_{2}$. If only one of them is non-zero, we can take a very short pair of twins to realize the zero period so as to avoid the path of non-zero length. If both are non-zero and we have initially taken two paths that intersect, we change one of the paths in its homotopy class with fixed endpoints to avoid the intersections. Gluing the endpoints of the said paths and marking the form by collapsing $a_1$ and $a_2$ to the nodes, and associating $b_i$ to the corresponding loop $\beta_{i}$ we obtain a marked stable form of period $p$ with two non-separating nodes as was claimed. 
In full generality it is not true that every primitive element of $L$ corresponds to a simple vertex. However, if $\deg p_{L}\geq 3$ we have that every primitive element of $L$ corresponds to a $\spl$ vertex. The previous argument shows that we only have to check that for any pair of primitive elements $a,a'\in L$ we can construct a sequence $a=a_{0}, \ldots, a_{n}=a'$ of primitive elements in $L$ such that $a_{i},a_{i+1}$ form a basis of $L$ for each $i=0,\ldots, n-1$. This is guaranteed by Gauss algorithm in $L$. 

It remains to treat the case $g=4$ and $\deg p_{L}=2$. In this case, $p_{L}$ is defined on a rank four symplectic module and has a lattice $\Lambda\subset \mathbb{C}$ as image. The normal form for periods of finite primitive degree, see Lemma \ref{l:finite degree orbit} proved in Appendix \ref{s:appendix2}, provides a symplectic splitting of $L^{\perp}/L$ into rank two submodules $W_{1}\oplus W_{2}$ satisfying $p_{L}(W_{i})=\Lambda$. Let $V\subset H_{1}(\Sigma_{g})$ be a rank four symplectic submodule containing $L$ and consider a symplectic splitting of $V^{\perp}=V_{1}\oplus V_{2}$ that induces $W_{1}\oplus W_{2}$ on $L^{\perp}/L$. Let $V_{3}=V\oplus V_{1}$. The symplectic splitting $H_{1}(\Sigma_{g})=V_{3}\oplus V_{2}$ defines a $\spl$ vertex in $\mathcal{G}_{p}$, since all factors are of rank at most $6$. We claim that there is an edge of $\mathcal{G}_{p}$ between this $\spl$ vertex and any $\spl$ vertex corresponding to a primitive element in $L$. To construct it, it suffices to construct a marked stable form with two nodes one of which induces the symplectic splitting $V_{3}\oplus V_{2}$ and the other collapses the class $a\in L\subset V_{3}$ satisfying $\deg p_{a}\geq 3$ to a node. On the other hand we have $\text{Im}(p_{|V_{3}})_{a}=\text{Im}(p_{a})$, $\vol_{p}(V)=0$, and $\vol_{p}(V_{2})=\vol_{p}(V_{1})$. Therefore $$\deg \big( (p_{|V_{3}})_{a}\big)=\frac{\vol \big((p_{|V_{3}})_{a}\big)}{ \vol \big(\mathbb{C}/\text{Im}(p_{|V_{3}})_{a}\big)}=\frac{\vol \big(p_{a}\big)/2}{ \vol \big(\mathbb{C}/\text{Im}(p_{a})\big)}=\frac{1}{2}\deg(p_{a})\geq \frac{3}{2}>1.$$ 
By using Corollary \ref{c:realization of decompositions} we can construct a stable form with two nodes,   one of which corresponds to $a$ and the other to  the splitting $V_3\oplus V_2$.
\\

(3) Without loss of generality suppose $a\in V_{1}$. Remark that by definition $\text{Im}\big( (p_{|V_{i}})_{a}\big)\subset \text{Im} (p_{a})$. Therefore $$\deg p_{a}=\frac{\vol p_{a}}{\vol \big(\mathbb{C}/\text{Im}(p_{a})\big)}\geq \frac{\vol p_{a}}{\vol \big(\mathbb{C}/\text{Im}(p_{|V_{1}})_{a}\big)}=\frac{\vol\big( (p_{|V_{1}})_{a}\big)+ \vol (p_{|V_{2}})}{\vol \big(\mathbb{C}/\text{Im}(p_{|V_{1}})_{a}\big)}>$$ $$>\deg \big( (p_{|V_{1}})_{a}\big)\geq 2.$$ 

\end{proof}
\begin{corollary}\label{c:cor_1}
If $c,c'\in\ker p$ are primitive elements satisfying $c\cdot c'=0$ and $\deg(p_c),\deg(p_{c'})\geq 3$ and  then $c\sim c'$. 
\end{corollary}
\begin{proof}
Write $a_1=c$ and choose $b_1\in H_1(\Sigma_g)$ such that $a_1\cdot b_1=1$. Define $W_1=\mathbb{Z}a_1\oplus\mathbb{Z}b_1$ and write $$c'=m_1a_1+n_1b_1+m_2a_2$$ where $a_2\in W_1^{\perp}$ is a primitive element and $m_1,n_1,m_2\in\mathbb{Z}$. Since $c\cdot c'=0$ we have $n_1=0$. If $m_2=0$ then $c'=\pm c$ and we are done. Otherwise $a_2\in\ker p$ and we can apply item (2) of Lemma \ref{l:equivalent vertices in kernel} to $L=\mathbb{Z}a_1\oplus\mathbb{Z}a_2$ to conclude $c\sim c'$. 
\end{proof}
\begin{corollary}\label{c:cor_2}
If there exists a non-trivial symplectic submodule $W\subset\ker p$ then every pair of primitive $c,c'\in\ker p$ such that $\deg(p_c),\deg(p_{c'})\geq 3$ are equivalent. 
\end{corollary}
\begin{proof}
For every primitive $c\in W$ $\deg (p_c)=\deg p\geq 3$. Every pair of primitive elements in $W$ are equivalent by (1) in Lemma \ref{l:equivalent vertices in kernel}. Take any primitive $c\in\ker p$ with $\deg (p_c)\geq 3$. Then $\rank (c ^{\perp}\cap W)\geq 1$ and for any primitive element $c'\in c ^{\perp}\cap W$ we have $c'\sim c$ by Corollary \ref{c:cor_1} we deduce $c\sim c'$. 
\end{proof}
\begin{corollary}\label{c:cor_3}
If there exists a primitive $a\in\ker p$ such that $\deg (p_a)<\infty$ then for every pair of primitive $c,c'\in\ker p$ such that $\deg (p_c),\deg(p_{c'})\geq 3$ we have $c\sim c'$. 
\end{corollary}
\begin{proof}
$\deg (p_a)<\infty$ implies that $\rank \ker p=2g-2$. Since $g\geq 4$ we have that $\ker p$ contains a non-trivial symplectic submodule and conclude by Corollary \ref{c:cor_2}. 
\end{proof}
The following result improves Corollary \ref{c:cor_1}:
\begin{corollary}\label{c:cor_4}
If $\rank(\ker p)\geq 3$ then every pair of primitive $c,c'\in\ker p$ such that $\deg (p_c),\deg (p_{c'})\geq 3$ satisfy $c\sim c'$. 
\end{corollary}

\begin{proof}
By Corollary \ref{c:cor_3} if there exists an element $c$ such that $\deg(p_c)<\infty$ we are done. So we can suppose $\deg (p_c)=\infty$ for every primitive $c\in\ker p$. Let $c,c'\in\ker p$ be two primitive elements. Then $\rank(c^{\perp}\cap c'^{\perp}\cap \ker p)\geq 1$ and a primitive element $a\in c^{\perp}\cap c'^{\perp}\cap \ker p$ will satisfy $c\sim a\sim c'$ by Corollary \ref{c:cor_1}. 
\end{proof}
\begin{lemma}
If $\rank (\ker p) \leq 2$ then every pair of primitive $a,a'\in\ker p$ such that \[\deg (p_a),\deg (p_{a'})\geq 3\text{  satisfy }a\sim a'.\]
\end{lemma}

\begin{proof}
If $\ker p$ has rank one, or $a\cdot a'=0$ or there exists a symplectic module in $\ker p$ containing $a$ and $a'$  we are done by Corollaries \ref{c:cor_1} and \ref{c:cor_2}.

Let $a_1=a$ and $b_1$ such that $a_1\cdot b_1=1$. Write $W_1=\mathbb{Z}a_1\oplus \mathbb{Z}b_1$ and \begin{equation}a'=m_1a_1+n_1b_1+m_2a_2\label{eq:decompostion pinchable} \end{equation} for a primitive $a_2\in W_1^{\perp}$ and integers $m_1,n_1,m_2$. We can assume $n_1=a\cdot a'\neq 0$ and $m_2p(a_2)\neq 0$ (otherwise we contradict the rank hypothesis on $\ker p$ or we fall in one of the initial cases). Define $p_1=p_{|W_1^{\perp}}$. 

If $a_2$ is $p_1$-admissible, i.e., it belongs to a symplectic module of rank two $V_2\in \mathcal{V}_{p_1}$, there exists a $p$-admissible splitting $W_1^{\perp}=V_2\oplus V_3$. Then $a_1,a'\in V_4:=W_1\oplus V_2$ and $\vol_p(V_4)=\vol_p(V_2)>0$. The restriction of \(p\) to the orthogonal \( V_4^\perp= V_3\) has infinite degree, since it has rank at least four, and the kernel of the restriction of \(p\) to \(V_4^\perp\) has rank bounded by one (otherwise the whole period would have a rank bounded from below by three). Since $a_1\in V_4$ does not belong to any symplectic submodule contained in \(\text{ker}(p)\), $p_{|V_4}$ is a Haupt homomorphism on a rank four symplectic module. By Lemma \ref{l:existence of pinched classes} applied for $g=2$, all primitive elements \(a''\) in $\ker p_{|V_4}$ are pinched by $p_{|V_4}$. 
Applying Corollary \ref{c:realization of decompositions} we can find a form with two nodes: one that pinches $a''$ and another that induces the splitting $V_4\oplus V_4^{\perp}$, which is simple, hence the vertex of $\mathcal{G}_p$ associated to $a''$ is equivalent to the vertex associated to $V_4\oplus V_4^{\perp}$. Since $a,a'$ are primitive elements in $V_4\cap \ker p$ we have $a'\sim a$.

By Corollary \ref{c:collinear non admissible} applied to $W_1^{\perp}$, there exists a proper submodule $I\subset W_1^{\perp}$ containing all elements that are not $p_1$-admissible, i.e. that do not belong to any symplectic module $V_2\in\mathcal{V}_{p_1}$. So if $a_2\notin I$ we are done. Next we are going to show that, up to changing the initial choice of $b_1$, we can guarantee that we fall in the previous case.

Indeed, suppose $a_2\in I\setminus \ker p$ is not admissible. Denote $\ell=\mathbb{R}p(a_2)$. We know by Corollary \ref{c:collinear non admissible} that $I=p^{-1}(\ell)\cap W_1^{\perp}$ has rank at least $2g-4$ and for every other real line $\ell'$, $\ell'\cap p(W_1^{\perp})$ has rank at most $2$. Since $p(a_1),p(b_1)\in \ell$, $p^{-1}(\ell)$ has rank at least $2g-2\geq 6$. On the other hand, for every other real line $\ell'\subset\mathbb{C}$ containing $0$ we have $\rank(p^{-1}(\ell')\cap W_1^{\perp})\leq 2+\rank(\ker p_1)\leq 3$. Therefore $p^{-1}(\ell')$ has rank at most $5$. If we manage to find a splitting as in equation \eqref{eq:decompostion pinchable} where the image of the $a_2$ is outside $\ell$ we will be done. We are going to show that, up to changing the initial $b_1$, we can suppose that we fall in this case or one of the previous cases.

Given $w\in a_1^{\perp}$ define $b_1'=b_1+w$ and $W_1'=\mathbb{Z}a_1\oplus\mathbb{Z}b_1'$. Then $$a=m_1a_1+n_1b_1'+m'_2a_2'$$ where $m_2'a_2'=m_2a_2-n_1w$.
If we manage to guarantee that
\begin{itemize}
 \item $a_2'\in W_1'^{\perp}$, or equivalently $0=-n_1(b_1\cdot w)+m_2(w\cdot a_2)$
 \item $p(a_2')\notin\ell=\mathbb{R}p(a_2)$ or equivalently $p(w)\notin\ell$,
\end{itemize}
we will be done: $a_2'\in W_1'^{\perp}$ will be $p_{|W_1'^{\perp}}$-admissible.

If there exists $w\in a_2^{\perp}\cap W_1^{\perp}\setminus p^{-1}(\ell)$, it constitutes a solution. Otherwise $a_2^{\perp}\cap W_1^{\perp}\subset p^{-1}(\ell)\cap W_1^{\perp}$, and since $W_1^{\perp}$ has positive volume, any $b_2\in W_1^{\perp}$ satisfying $a_2\cdot b_2=1$ satisfies $p(b_2)\notin\ell$. In this case the element $w=m_2a_1+n_1b_2$ provides a solution.
\end{proof}

\end{proof}
\subsection{Proof of Theorem \ref{t:connectedness}}
\label{s:proof of inductive step of transfer}
The Theorem is true for $g=2,3$ by Theorem \ref{t:connectedness in genus 2 and 3} . We proceed by induction on the genus.
Fix some $g\geq 4$ and suppose that Theorem \ref{t:connectedness} is true up to genus $g-1$. Take $p\in\mathcal{H}_g$ with $\deg(p)\geq 3$. By Proposition \ref{p:degeneration} every connected component of $\percompact^{-1}(p)$ has points in the boundary. The boundary $\percompact^{-1}(p)\setminus \per^{-1}(p)$ is connected thanks to Theorem \ref{t:connected boundary}. Therefore $\percompact^{-1}(p)$ is connected. On the other hand, by Theorem \ref{p:local structure of fiber at bdry point}, the boundary points do not locally separate  $\per^{-1}(p)$, and the latter is therefore also connected.

\section{Appendix I: proof of Proposition \ref{p:Kapovich}}\label{s:appendix2}

We first begin by providing a normal form for periods of positive volume and finite primitive degree, which implies  the first item of Proposition \ref{p:Kapovich} in the case where the subspace \eqref{eq: W} is rational. 

\begin{lemma}\label{l:finite degree orbit}
 Given a surjective homomorphism $p:\mathbb{Z}^{2g}\rightarrow \mathbb{Z}+ i \mathbb{Z}$ of volume (and primitive degree) $d\geq 2$, there exists $M\in\text{Sp}(2g,\mathbb{Z})$ and a symplectic basis $\{a_1,b_1,\ldots, a_g,b_g\}$ of $\mathbb{Z}^{2g}$ such that $$p\circ M=a_1^*+i(db_1^*+a_2^*)$$ where the star denotes the symplectic dual of the given element. 
\end{lemma}

\begin{proof} Denote by $x=\Re(p)$ and $y=\Im (p)$ the elements in $(\mathbb{Z}^{2g})^*$. They satisfy $x\cdot y=d$. Choose $a_1\in\mathbb{Z}^{2g}$ the symplectic dual of $x$. Choose some $b_1\in\mathbb{Z}^{2g}$ such that $a_1\cdot b_1=1$ and write \begin{equation} y=m_1a_1^*+db_1^{*}+m_2a_2^{*}\label{eq:general form for degree d discrete}\end{equation} where $a_2$ is a primitive element satisfying $a_1\cdot a_2=b_1\cdot a_2=0$ and $m_1,d$ and $m_2$ are co-prime.
Complete those elements to a symplectic basis $a_1,b_1, \ldots, a_g,b_g$ of $\mathbb{Z}^{2g}$. The image of an element $u=\sum_{k\geq 1}\alpha_ka_k+\beta_kb_k$ under $p$ is $$p(u)=(a_{1}^{*}(u), (m_1a_1^*+db_1^{*}+m_2a_2^{*})(u))=(\alpha_1, m_1\alpha_1+d\beta_1+m_2\alpha_2).$$
Therefore $$p(u)=1\Leftrightarrow\alpha_1=1\text{ and } m_1+d\beta_2+m_2\alpha_2=0$$ $$p(u)=i\Leftrightarrow \alpha_1=0\text{ and } d\beta_1+m_2\alpha_2=1.$$

Hence, $p$ is surjective if and only if $d$ and $m_2$ are co-prime.

To conclude, we are going to show that under the hypothesis of $d$ and $m_2$ co-prime, there exists a choice of $b_1$ such that the coefficients $m_1$ and $m_2$ in the splitting given by (\ref{eq:general form for degree d discrete}) are $0$ and $1$ respectively.

Let us analyze the effect of a change of the first given $b_1$. Write $\widetilde{a_1}^*=a_1^*$ and $$\widetilde{b_1}^*=u_1a_1^*+b_1^*+\sum_{k\geq 2}u_ka_k^*+v_kb_k^*$$ for some $u_k,v_k\in \mathbb{Z}$. The new decomposition $y=\widetilde{m_1}\widetilde{a_1}^*+d\widetilde{b_1}^*+\widetilde{m_2}\widetilde{a_2}^*$ has the properties $$(y-\widetilde{m_1}\widetilde{a_1}^*-d\widetilde{b_1}^*)\cdot \widetilde{a_1}^*=0\text{ and } (y-\widetilde{m_1}\widetilde{a_1}^*-d\widetilde{b_1}^*)\cdot \widetilde{b_1}^*=0.$$ The first equation is automatically satisfied, and the second gives
\begin{equation}\label{eq:zero of m1}
 \widetilde{m_1}=m_1-du_1+m_2v_2
\end{equation}
Since $d$ and $m_2$ are co-prime we can already choose $u_1$ and $v_2$ to get $\widetilde{m_1}=0$ and restart the argument by supposing $m_1=0$.

For this choice, we have $$\widetilde{m_2}^*\widetilde{a_2}^*=y-\widetilde{m_1}\widetilde{a_1}^*-d\widetilde{b_1}^*=m_2v_2 a_1^{*}+(m_2-du_2)a_2^*-dv_2b_2^*-d(\sum_{k\geq 3}u_ka_k^*+v_kb_k^*)$$ and from this $$\widetilde{m_2}=\text{gcd}(-m_2v_2, m_2-du_2, -dv_2, -du_3, -dv_3, \ldots, -d u_g, -dv_g)$$

If we still want $\widetilde{m_1}=0$ we need to impose $m_2v_2=du_1$ by equation (\ref{eq:zero of m1}). The choice $v_2=d$, $u_1=m_2$ and all other coefficients equal to zero gives $\widetilde{m_1}=0$ and $\widetilde{m_2}=1$, as desired.
\end{proof}

We continue the proof of Proposition \ref{p:Kapovich}, which is reminiscent of Ratner's theory. 

Equipp \(\mathbb R^{2g}\) with its canonical symplectic form  $\omega (x,y) = \sum_{1\leq k\leq g} x_{2k}y_{2k+1}-x_{2k+1}y_{2k}$. The volume of a period $p\in \mathbb C^{2g}$ is the symplectic product \( V(p)=  \omega (\Re p , \Im p ) \).
Since the action of $\Gamma$ is linear, and that the volume is multiplicative, namely $V(\lambda p ) = |\lambda|^2 V(p)$ for every $\lambda \in \mathbb C$ and $p\in \mathbb C^{2g}$, we can restrict our attention to the action of $\Gamma$ on the subset $X\subset \mathbb C^{2g}$ whose elements have volume $1$. In real and imaginary coordinates the set of periods of volume \(1\) is then the set of pairs $(x,y)\in\mathbb{R}^{2g}\times\mathbb{R}^{2g}$ such that $\omega(x,y)=1$.

The simple real Lie group $G = \text{Sp}(2g, \mathbb R)$ acts transitively on the set of couples $(x,y)\in (\mathbb R^{2g})^2$ such that $\omega(x,y)=1$, and that the stabilizer of the couple \[\big( (1,0,\ldots,0), (0,1,0,\ldots,0) \big)\text{ is the group }
 \left( \begin{array}{ccc}
1 & & \\
& 1 & \\
& & \text{Sp}(2g-2,\mathbb R)\end{array} \right)\]
that we will denote by $U$ in the sequel. Our set $X$ is isomorphic to the homogeneous space $G/U$. The linear action of $\Gamma $ on $X$ is under the isomorphism $X \simeq G/U$ given by left multiplication on $G/U$.

Since the group $G$ is simple, that $U$ is generated by unipotent elements, and that $\Gamma $ is a lattice in $G$, Ratner's theorem \cite{Ratner} tells us that the closure of the $\Gamma$-orbits on $X$ are homogeneous in the following sense

\begin{theorem} [Ratner]\label{thm:Ratner}
For every $p\in X$ of the form $p = gU$, there exists a closed subgroup $H$ of $G$ containing $U^g= gUg^{-1}$, such that $\Gamma \cap H$ is a lattice in $H$, and such that $\overline{\Gamma \cdot p} = \Gamma H p$.
\end{theorem}


Notice that in our situation, we have $U^g = I_{|W} \oplus \text{Sp}(W^\perp) \simeq \text{Sp} (2g-2,\mathbb R)$ where $W=\mathbb{R}\Re p+\mathbb{R}\Im p \subset \mathbb{R}^{2g}$ is the symplectic subspace associated to the volume one $p\in\mathbb{C}^2$.

Let $H_0$ be the connected component of $H$ containing the identity: then $\Gamma \cap H_0$ is still a lattice in $H_0$, and $U^g $ is contained in $H_0$.

If $H_0=G$ then $\overline{\Gamma \cdot p}=G$ and we deduce that the orbit closure is dense in $X$. Since the closure of $\Lambda(p)$ contains all the $\Lambda(q)$ of elements $q\in \overline{\Gamma \cdot p}$, we have $\overline{\Lambda(p)}=\mathbb{C}$.

If $H_0$ is a proper subgroup of $G$, Kapovich observes that it falls into two categories
\begin{itemize}
\item (Semi-simple case) $H_0$ is of the form $S \oplus \text{Sp} (W^\perp)$, where $S$ is a Lie subgroup of $\text{Sp}(W)$.
\item (Non semi-simple case) $H_0$ is not semi-simple and preserves a line $L \subset W$.
\end{itemize}
The proof of this dichotomy can be found in \cite[p. 12]{Kapovich}, and is based on Dynkin's classification of maximal connected complex Lie subgroups of $\text{Sp}(2g,\mathbb C)$, see \cite{Dynkin}. Let $L$ be a maximal complex Lie subgroup of $\text{Sp}(2g,\mathbb C)$ which contains $H_0$. If $H_0\neq \text{Sp} (2g,\mathbb R)$, its Zariski closure in the complex domain is a strict subgroup of $\text{Sp}(2g,\mathbb C)$, so it is contained in a maximal complex Lie (strict) subgroup of $\text{Sp}(2g,\mathbb C)$. It satisfies one of the following properties (see \cite[Ch. 6, Thm 3.1, 3.2]{GOV}):
\begin{enumerate}
\item $L= \text{Sp} (V) \oplus \text{Sp} (V^\perp)$ for some complex symplectic subspace $V\subset \mathbb C^{2g}$,
\item $L$ is conjugated to $\text{Sp} (s,\mathbb C) \otimes \text{SO} (t,\mathbb C)$ where $2g = st$, $s\geq 2$, $t\geq 3$, $t\neq 4$ or $t=4 $ and $s=2$,
\item $L$ preserves a line of $\mathbb C^{2g}$.
\end{enumerate}
Since $H_0$ contains $U^g$, $L$ contains the complexification of $U^g$, which is nothing but $\text{Id}_{W_ \mathbb C}\oplus \text{Sp} (W_ \mathbb C ^\perp)$, where $W_\mathbb C$ denotes the complexification $W \otimes_\mathbb R \mathbb C$ of $W$. In case (1), the only possibility is that up to permutation of $V$ and $V^\perp$, we have $W_\mathbb C= V$. In particular, $H_0$ is a subgroup of $\text{Sp} (W) \oplus \text{Sp}(W^\perp)$. Since it contains $\text{Id}_{|W} \oplus \text{Sp}(W^\perp)$, it must be of the form $S \oplus \text{Sp} (W^\perp)$, where $S$ is a Lie subgroup of $\text{Sp}(W)$. Case (2) cannot occur. In case (3), observe that the line $L$ needs to be in $W_\mathbb C$, since the group $\text{Id}_{|W_\mathbb C} \oplus \text{Sp} (W_\mathbb C ^\perp)$ preserves this line. If $L$ is defined over the reals, we are done. If not, both $L$ and $\overline{L}$ (the image of $L $ by the complex conjugation) are preserved by $H_0$, and thus $H_0$ is a subgroup of $\text{Sp} (W) \oplus \text{Sp} (W^\perp)$. As before, because it contains $\text{Id}_W \oplus \text{Sp} (W^\perp)$, it must be of the form $S \oplus \text{Sp}(W^\perp)$, where $S$ is a Lie subgroup of $\text{Sp}(W)$.

\vspace{0.2cm}

\begin{remark}\label{rem:rational W}
If $\Gamma_W:= \Gamma \cap (\text{Id}_W \oplus \text{Sp}(W^\perp))$ is a lattice in $\text{Id}_W \oplus \text{Sp}(W^\perp)$ then $W$ is defined over $\mathbb{Q}$ and the conclusion of the first item of Proposition \ref{p:Kapovich} follows easily. Indeed, $\Gamma _W $ acts by the identity on $W^\sigma$ for every Galois automorphism $\sigma$. The Zariski closure of $\Gamma_W$ being $\text{Id}_W \oplus \text{Sp} (W^\perp)$ (by Borel density theorem, see \cite{Zimmer}), $\text{Id}_W \oplus \text{Sp}(W^\perp)$ acts by the identity on $W^\sigma$ as well. This implies that $W^\sigma= W$ for every $\sigma$ (otherwise $\text{Id}\oplus\text{Sp}(W^{\perp})$ acts by the identity on the non-trivial subspace $(W+W^{\sigma})\cap W^{\perp}$), and so $W$ is rational.
\end{remark}

Let us proceed to analyze the different cases.

\textit{Semi-simple case.} Since the group $H_0 = S\oplus \text{Sp} (W^\perp)$ contains a lattice, it must be unimodular. In particular, either $S$ is the trivial group, or a $1$-parameter subgroup, or the whole $\text{Sp} (W)$.
If $S$ is trivial, then Ratner's Theorem tells us $\Gamma_W:= \Gamma \cap (\text{Id}_W \oplus \text{Sp}(W^\perp))$ is a lattice in $\text{Id}_W \oplus \text{Sp}(W^\perp)$ and we fall in case (1) by Remark \ref{rem:rational W}.

If $S$ is $1$-dimensional, $S\oplus \text{Id}_{W^\perp}$ would be the radical of $H_0$, and a theorem of Wolf and Raghunathan, see \cite{Raghunathan}, shows that it would intersect $\Gamma$ in a lattice. This implies that the intersection of $\Gamma$ with $\text{Id}_W \oplus \text{Sp}(W^\perp)$ is also a lattice. Indeed, let $\overline{\Gamma}$ denote the natural projection of $\Gamma$ on $\text{Sp}(W^{\perp})$. Since the sequence $$0\rightarrow S\cap\Gamma\rightarrow \Gamma\rightarrow\overline{\Gamma}\rightarrow 1$$ is exact, the map $\Gamma\rightarrow (S\cap\Gamma)\backslash S$ induced by the projection of $S\oplus\text{Sp}(W^{\perp})$ to $S$ induces a morphism $\overline{\Gamma}\rightarrow (\Gamma\cap S)\backslash S$. The group $(\Gamma\cap S)\backslash S$ is abelian, and since $\overline{\Gamma}$ has Kazdhan property (T), the image group in $(\Gamma\cap S)\backslash S$ is finite. We thus conclude that the image of $\Gamma$ on the factor $S$ is a lattice and therefore $\Gamma\cap(\text{Id}_W \oplus \text{Sp}(W^\perp))$ as well. We conclude as in Remark \ref{rem:rational W}.

Finally, it remains to treat the case where $S= \text{Sp}(W)$. This case splits into two subcases, depending on the lattice $\Gamma \cap \text{Sp} (W) \oplus \text{Sp} (W^\perp)$ being reducible or irreducible. If it is reducible, this implies that $\Gamma \cap (\text{Id}_W \oplus \text{Sp} (W^\perp))$ is a lattice, and then $W$ must be rational by the above considerations. Assume now that we are in the irreducible case. Then $g=2$, by a theorem of Margulis \cite{Margulis}. Assume $W$ is not rational, otherwise we are done. Let $\sigma$ be a Galois automorphism such that $W^\sigma \neq W$. The group $\Gamma \cap \text{Sp} (W) \oplus \text{Sp} (W^\perp)$ preserves the splitting $W^\sigma \oplus (W^\sigma)^\perp$, since $\Gamma$ and the symplectic form are defined over the rationals. Borel density theorem applied to the lattice $\Gamma \cap \text{Sp} (W) \oplus \text{Sp} (W^\perp)$ shows that $\text{Sp}(W) \oplus \text{Sp} (W^\perp)$ preserves the splitting $W^\sigma \oplus (W^\sigma)^\perp$. This implies that $W^\sigma = W^\perp$ and $(W^\sigma)^\sigma= W$. This being true for every Galois automorphism, this means that $W$ is defined over a totally real quadratic field $K$, and we have $W^\sigma= W^\perp$ where $\sigma$ is the Galois automorphism of $K$. This is the only situation where we fall in the last case of Proposition \ref{p:Kapovich}.

\vspace{0.2cm}

\textit{Non semi-simple case.} We suppose that $H_0$ is not semi-simple and prove in this case that the periods $p$ satisfy the second case of Proposition \ref{p:Kapovich}. We already know that there is a line $L\subset W$ that is invariant by the action.

For this, we will first need to understand in detail the subgroup $B$ of $\text{Sp}(2g,\mathbb R) $ formed by all elements that stabilize the line $L$, see \cite[p. 10]{Kapovich}. To unscrew the structure of $B$, notice that any element of $B$ stabilizes both $L$ and $L^\perp$ so that we have an exact sequence
\[ CH_{2g}\rightarrow B \rightarrow \text{Sp} (L^\perp /L)\simeq \text{Sp}(2g-2) \]
The group $ CH_{2g}$ is then the set of elements $M\in \text{Sp} (2g)$ which induce the identity map on $L^\perp/L$.

We now have another exact sequence
\begin{equation}\label{eq:exact sequence2} H_{2g-1} \rightarrow CH_{2g} \rightarrow \text{GL}(L) \simeq \mathbb R^*, \end{equation}
the last arrow being given by the restriction of an element $M\in CH_{2g}$ to the line $L$. Hence the subgroup $H_{2g-1}\subset CH_{2g}$ is the group of elements $M\in \text{Sp} (2g)$ which act as the identity on $L$ and on $L^\perp / L$. Such $M$ are easily seen to be of the form $M_{\varphi, \alpha}$, for some $\varphi \in (L^\perp/L)^*$
and $\alpha \in \mathbb R$, where
\begin{itemize}
\item the restriction of $M_{\varphi, \alpha}$ to $L^\perp$ equals $id_{|L^\perp} + \varphi a_1$
\item $ M_{\varphi, \alpha} (b_1) = \alpha a_1 + b_1 + \sum _{k\geq 2} \varphi (b_k) a_k -\varphi (a_k) b_k$,
\end{itemize}
where $a_1,b_1,\ldots, a_g,b_g$ is a symplectic basis such that $L= \mathbb R a_1$.
The group structure on $H_{2g-1}$ is then given by the following relation
\begin{equation}\label{eq:group structure} M_{\varphi, \alpha} M_{\varphi ' , \alpha '} = M_{\varphi + \varphi', \alpha+ \alpha' + \omega(\varphi, \varphi')}, \end{equation}
where $\omega (\varphi, \varphi')$ is the natural symplectic product induced by $\omega$ on $(L^\perp/L)^*$, namely
\[ \omega(\varphi, \varphi') = \sum_{k\geq 2} \varphi(a_k)\varphi'(b_k)-\varphi'(a_k) \varphi(b_k). \]
Equation \eqref{eq:group structure} is a straightforward computation. An equivalent formulation is that $H_{2g-1}$ is the central extension
\[ \mathbb R \rightarrow H_{2g-1} \rightarrow (L^\perp / L)^*, \]
defined by the $2$-cocycle $ (\varphi, \varphi') \mapsto \omega (\varphi , \varphi')$. The group $H_3$ is isomorphic to the classical Heisenberg group of upper triangular real matrices of size $3\times 3$ with $1$'s on the diagonal.

Now $CH_{2g} $ is a semi-direct product of $\mathbb R^*$ by $H_{2g-1}$, see \eqref{eq:exact sequence2}. To understand its structure, we introduce for every $\lambda$, one of its lift $S_\lambda \in CH_{2g}$ defined by
\[ S_\lambda (a_1) = \lambda a_1, \ S_\lambda (b_1) = \frac{1}{\lambda} b_1, \ S_\lambda (a_k) = a_k,\ S_\lambda(b_k)= b_k\text{ for } k\geq 2.\]
A trivial computation shows that for any $\lambda \in \mathbb R^*$, every $\varphi \in (L^\perp /L)^*$ and every $\alpha \in \mathbb R$, we have
\begin{equation} S_\lambda M_{\varphi, \alpha} S_\lambda^{-1} = M_{\lambda \varphi, \lambda^2 \alpha}. \label{eq:action avoiding unimodularity}\end{equation}
This shows that $CH_{2g}$ is not unimodular, and consequently does not contain any lattice.

By construction, our group $H_0$ is contained in $B$. We have an exact sequence $CH_{2g} \rightarrow B \rightarrow \text{Sp}(L^\perp/L, \omega) $. The image of $H_0$ by the right arrow is onto since $H_0$ contains $U^g$, so that $H_0$ itself splits as an exact sequence $CH_{2g} \cap H_0 \rightarrow H_0 \rightarrow \text{Sp} (L^\perp/L, \omega)$. The group $CH_{2g}\cap H_0$ is invariant under the action by conjugation of $\text{Sp} (L^\perp/L, \omega)\simeq \text{Sp}(W^{\perp})$. The restriction of this action on $H_{2g-1}$ can be described explicitly: for $U\in\text{Sp}(L^{\perp}/L)\simeq \text{Sp}(W^{\perp})$ denote $s(U)=Id_W\oplus U\in H_0\subset B$. Then $$s(U)M_{\varphi,\alpha}s(U)^{-1}=M_{\varphi\circ s(U)^{-1},\alpha}.$$

\begin{lemma}\label{l:subgroups of CH}
 The closed non-trivial connected subgroups of $CH_{2g}$ invariant by $\text{Sp}(W^{\perp})$ are
 \begin{enumerate}
 \item $Z(H_{2g-1})$
 \item lifts of $\text{GL}^+(L)$ in $CH_{2g}$
 \item $H_{2g-1}$
 \item $\text{GL}^{+}(L)\ltimes Z(H_{2g-1})$
 \item $CH_{2g}$
 \end{enumerate}
\end{lemma}
The proof of Lemma \ref{l:subgroups of CH} is an easy consequence of the previous exact sequences and calculations.

Now $H_0\cap CH_{2g}$ cannot fall in cases $(1)$ and $(2)$ of Lemma \ref{l:subgroups of CH} since in either of those, $H_0$ would be semi-simple, contrary to hypothesis. It can neither fall in cases $(4)$ or $(5)$, since in those cases equation (\ref{eq:action avoiding unimodularity}) does not allow $H_0$ to be unimodular. Hence we are left with the possibility $H_0\cap CH_{2g}=H_{2g-1}$ (and $H_0\simeq\text{Sp}(W^{\perp})\ltimes H_{2g-1}$). In this case we will show that the invariant line $L$ is rational, and thus we fall in the second possibility of Proposition \ref{p:Kapovich}.

The theorem of Raghunathan and Wolf cited above tells us that $\Gamma \cap H_{2g-1}$ is a lattice in $ H_{2g-1}$. By using Borel's density Theorem in \cite[p.91]{Gromov} we deduce that its Zariski closure is $H_{2g-1}$. We have $L\subset K:=\bigcap _{\gamma \in \Gamma \cap H_{2g-1} } \text{Ker} (\gamma - I)$ which is an intersection of rational spaces. If the inclusion is proper, then the Zariski closure of $\Gamma\cap H_{2g-1}$ would not be the whole of $H_{2g-1}$. This shows that $L=K$ and it is a rational one-dimensional subspace of $\mathbb R^{2g}$.

Up to a real affine change of coordinates on $\mathbb C$, we can assume that the imaginary part of $p$ generates $L$, and that it is a primitive element of $\mathbb Z^{2g}$. Since the group $H_{2g-1}$ acts transitively on the set of vectors $v\in \mathbb R^{2g}$ such that $v\cdot \Im p= 1$, while keeping the period $\Im p$ fixed, we see that $H\cdot p$ already contains all the periods $q$ such that $\Im q = \Im p $ and such that $V(q)= V(p)= 1$. Since, $\Gamma$ acts transitively on the set of primitive elements of $\mathbb Z^{2g}$, we infer that $\Gamma H p = \overline{\Gamma \cdot p}$ contains all the periods $q$ with volume $V(p) = 1$ and with a primitive integer imaginary part. Since any periods of $ \overline{\Gamma \cdot p}$ is of this form, we deduce that this situation is exactly the second case of the proposition. The proof of this latter is now complete.

Applying Moore's ergodic theorem in \cite{Moore} to each case $H$ above we deduce the final ergodicity part of Proposition \ref{p:Kapovich}. 
\section{Appendix II: proof of Lemma \ref{l: surjection symplectic modulo 2}}\label{s:appendix1}

Up to composing \( p\) by an element of \(\text{GL}_2^+ (\mathbb R)\) we can assume that the image of \(p\) is the set of Gaussian integers. By Lemma \ref{l:finite degree orbit}, we can assume that there exists a symplectic basis \(a_1, b_1, \ldots , a_g, b_g\) in which the period \(p\) has the following form: 
\[ p(a_1)=p(a_2)=1, \ p(b_1)=p(b_2)=i \text{ and } p(a_k)= p(b_k)=0 \text{ for } k\geq 3.\]
This basis permits to identify \( H_1(\Sigma_g,\mathbb Z) \) with \( \mathbb Z^{2g} \) equipped with the symplectic form 
\[ u\cdot v = u_1 v_2 - u_2 v_1 +\ldots + u_{2g-1} v_{2g} - u_{2g} v_{2g-1},\]
and the group \(\text{Aut}(H_1(\Sigma_g, \mathbb Z))\) with \(\text{Sp}(2g,\mathbb Z)\). In these coordinates we have 
\[ p (u_1, \ldots, u_{2g}) = (u_1+u_3) + (u_2+u_4)i .\]
The form \( u_1+u_3 \in (\mathbb Z^{2g})^*\) is dual to the vector \(P_1= -(b_1+b_2)\) (meaning that \(u_1+u_3= - (b_1+b_2)\cdot u\)), whereas the form \( u_2+u_4\in (\mathbb Z^{2g})^*\) is dual to \(P_2= a_1+a_2\) (meaning \(u_2+u_4= (a_1+a_2)\cdot u\)). We then have that an edomorphism of $\mathbb{Z}^{2g}$ given by  a matrix \(M\) stabilizes \(p\) if and only if \(M (P_k)=P_k\) for \(k=1,2\), and similarly an endomorphism of $(\mathbb Z/2\mathbb Z)^{2g}$ given by a matrix \(M[2]\) stabilizes $p[2]$ if and only if \(M[2] (P_k[2])=P_k[2]\) for \(k=1,2\).
The condition can be read in the columns of the matrices. If $C_k$ (resp. $C_k[2]$) denotes the $k$-th columun of $M$ (resp. $M[2]$) then \begin{equation}\label{eq: invariance of p} C_1 +C_3= a_1+a_2 \text{ and } C_2 +C_4=b_1+b_2.\end{equation} \begin{equation}\label{eq:invariance of p[2]}(\text{resp. } C_1[2]+C_3[2]=a_1[2]+a_2[2] \text{ and } C_2[2]+C_4[2]=b_1[2]+b_2[2])\end{equation} 

Let \(M[2]\in \text{Stab} _{\text{Sp}(2g,\mathbb Z / 2\mathbb Z)} (p[2])\) a symplectic isomorphism whose columns \( C_1[2] , \ldots , C_{2g}[2]\) satisfy \eqref{eq:invariance of p[2]}. 
Let \( C_k' \in \mathbb Z^{2g}\) be representatives of the classes \(C_k[2]\in (\mathbb Z / 2\mathbb Z)^{2g}\), for \(k=1,\dots, 2g\), and  \(M' \) be the square matrix whose columns are the \(C_k'\). Up to changing the representative in each class, we can suppose \eqref{eq: invariance of p} is valid for the columns of $M'$. We will assume in the sequel that these equations are always satisfied.

Our goal is to modify the vectors \( C_k '\) by defining 
\[ C_k := C_k ' + 2E_k, \text{ with } E_k\in \mathbb Z^{2g}\]
in such a way that the matrix \( M:= (C_1,\ldots , C_{2g})\) not only stabilizes $p$ (satisfying \eqref{eq: invariance of p}) but also belongs to \( \text{Sp} (2g,\mathbb Z)\). It is therefore necessary to impose
\begin{equation}\label{eq: Ek} E_1+E_3=0 \text{ and } E_2+E_4=0.\end{equation}
\begin{equation}
    C_1,\ldots, C_{2g}\text{  forms a symplectic basis of } \mathbb{Z}^{2g} 
\end{equation}
\textit{Main step: construction of \(C_1,C_2,C_3,C_4\).} The conditions we need to satisfy are 
\begin{enumerate}
\item \( C_1\cdot C_2= 1\) 
\item \( C_1\cdot C_3=0 \), or equivalently \( C_1 \cdot (a_1+a_2) =0\)
\item \( C_1\cdot C_4= 0\), or equivalently, knowing (1): \( C_1\cdot (b_1+b_2)=C_1\cdot C_2 =1\)
\item \(C_2\cdot C_3=0\), or equivalently, knowing (1): \( C_2 \cdot (a_1+a_2)=C_2\cdot C_1=-1\) 
\item \(C_2\cdot C_4=0\), or equivalently, \( C_2 \cdot (b_1+b_2) =0\)
\item \(C_3\cdot C_4=1\), or equivalently, \( ((a_1+a_2) -C_1)\cdot ((b_1+b_2) - C_2) =1\) 
\end{enumerate}
 If \eqref{eq: invariance of p} and (1)--(5) are true, we have that (6) is automatically satisfied 
\[ ((a_1+a_2) -C_1)\cdot ((b_1+b_2) - C_2) =\]\[=(a_1+a_2)\cdot (b_1+b_2) -(a_1+a_2) \cdot C_2 -C_1 \cdot (b_1+b_2) +C_1\cdot C_2= 2-1-1+1=1.\]
Let us first find \(C_1, C_2, C_3, C_4\) so that conditions (2)--(5) are satisfied. In real and imaginary coordinates this is equivalent to
\[ p (E_1) = \left( \frac{1-C_1'\cdot (b_1+b_2) } {2}, \frac{C_1' \cdot (a_1+a_2)}{2} \right),\] 
and 
\[ p(E_2) = \left(\frac{(b_1+b_2)\cdot C_2' }{2} , \frac{1+C_2'\cdot (a_1+a_2)}{2} \right). \]
Observe that the right hand sides of the last two equations are integers because \(C_k'\)'s are lifts of - the elements of the symplectic basis- \(C_k[2]\)'s and satisfy \eqref{eq: invariance of p}. By surjectivity of \(p\) there exist solutions to the equations (2)--(5), and so we can assume that the \(C_k'\)'s satisfy (2)--(5). We can now replace \(C_k'\) by \( C_k = C_k'+2E_k\) for \(k=1,2\) with 
\[ E_k \in (a_1+a_2)^{\perp} \cap (b_1+b_2) ^\perp \text{ for } k=1,2. \]
to preserve conditions (2)--(5). Condition (1) is equivalent to 
\[1= C_1 \cdot (C_2' +2 E_2) = C_1\cdot C_2' + 2 C_1\cdot E_2,\]
and we know that \( C_1 \cdot C_2'\) is odd. So we are done if we can choose \(C_1\) so that the map 
\begin{equation}\label{eq: onto?} E_2 \in (a_1+a_2)^{\perp} \cap (b_1+b_2) ^\perp \mapsto C_1\cdot E_2\in \mathbb Z\end{equation}
is onto. We have \( (a_1+a_2)^{\perp} \cap (b_1+b_2) ^\perp= \mathbb Z (a_1-a_2) +\mathbb Z (b_1-b_2) +\sum_{k\geq 3} \mathbb Z a_k+\mathbb Z b_k\)
So the map \eqref{eq: onto?} is onto if and only if 
\begin{equation}\label{eq: gcd} \text{gcd} ( C_1 \cdot (a_1-a_2) , C_1 \cdot (b_1-b_2) , C_1 \cdot a_3 , C_1\cdot b_3, \ldots, C_1 \cdot a_g , C_1\cdot b_g) =1.\end{equation}
However, observe the following:
\begin{itemize}
\item \(C_1 \cdot (b_1-b_2)= C_1 \cdot (b_1+b_2) [2]= 1[2] \) is odd, and 
\item If $g\geq 3$, choosing \(E_1\) appropriately, we can assume that the value of \( C_1 \cdot a_3 \) is either \( 1\text{ or } 2\).
\end{itemize} 
So  \eqref{eq: gcd} is satisfied with these choices of \(E_1\). Hence the main step is achieved.

To conclude the proof of Lemma \ref{l: surjection symplectic modulo 2}, we need to construct the columns \( C_5,\ldots, C_{2g}\). We will inductively find the pair \( C_5, C_6\), then the pair \( C_7, C_8\), etc.. Let us construct the first one. We need to find \(E_5, E_6\) so that 
\[ C_5, C_6 \in \left(C_1, C_2, C_3, C_4 \right) ^\perp \text{ and } C_5\cdot C_6 = 1.\]
This means that the equations 
\begin{equation}\label{eq: 1} (E_k\cdot C_1, \ldots , E_k \cdot C_4) =-\frac{1}{2} (C_k'\cdot C_1,\ldots ,C_k'\cdot C_4) \text{ for } k=5,6 \end{equation}
and 
\begin{equation}\label{eq: 2} C_5\cdot (C_6'+ 2 E_6) = 1\end{equation}
hold. We can find \(E_5\) and \(E_6\) so that \eqref{eq: 1} is satisfied, since \(C_1, C_2, C_3, C_4\) is a symplectic family (i.e. \(C_1\cdot C_2=C_3\cdot C_4=1\) and other products are zero). So we can assume that \(C_5'\) and \(C_6'\) belong to the orthogonal of \( C_1, C_2, C_3, C_4\), which is a symplectic submodule of \(\mathbb Z^{2g}\) isomorphic to \(\mathbb Z^{2g-4}\) with the canonical symplectic product. In these coordinates, one of the coordinates of \( C_5'\) is odd (since \(C_5[2]\), the reduction of \(C_5'\) modulo 2, is non zero) so by adding to \(C_5\) an even vector of \(\mathbb Z^{2g-4}\) one can assume that one of the coordinates of \( C_5\) is equal to \(1\). Hence, equation \eqref{eq: 2} being equivalent to \( C_5\cdot E_6= \frac{-C_5\cdot C_6'}{2}\) and the product \( C_5\cdot C_6\) being even, equation \eqref{eq: 2} can be solved for some suitable choice of \(E_6\). We get \(C_5\) and \(C_6\) in this way. For the construction of the other pairs by induction, the argument is similar. 

The first part of the lemma follows, namely the surjectivity of the map \(\text{Stab}_{\text{Aut}(H_1(\Sigma_g, \mathbb Z))}(p) \rightarrow \text{Stab}_{\text{Aut}(H_1(\Sigma_g, \mathbb Z / 2\mathbb Z))}(p[2])\).

For the second part, it suffices to remark that the stabilizer of \(p[2]\) in \( \text{Aut} (H_1(\Sigma_g, \mathbb Z / 2\mathbb Z))\) is the stabilizer of a pair of non colinear elements of \(H^1(\Sigma_g, \mathbb Z / 2\mathbb Z)\) that do not intersect each other. The number of elements in this group is given by the announced formula, by elementary considerations. 

\cbstart
\section{Appendix III: a result in Picard-Lefschetz theory} \label{a: Picard-Lefschetz}

We prove here a result which is a well-known consequence of Picard-Lefschetz theory, but which we cannot find as such in the literature. 

Let \(h : S\rightarrow C\) be a holomorphic map from a compact complex surface \(S\) to a compact complex curve \(C\). We assume both \(S\) and \(C\) are non singular, that the critical points of \(h\) are non degenerate, and that the fibers of \(h\) are connected. 

\begin{lemma} \label{l:picard-lefschetz}
For every \(c\in C\), the sequence 
\begin{equation} \label{eq: exact sequence 1} \pi_1 (h^{-1} (c) ) \rightarrow \pi _1 (S) \rightarrow \pi_1 (C) \end{equation} where the map on the left is induced by inclusion and the one on the right by \(h_*\),  is exact. 
\end{lemma}

\begin{proof} 
Let us first prove the claim for \(c\) a regular value of \(h\). To do so, given a lift \(w\in h^{-1} (c)\), we will use some useful loops \(\tilde{\mu_i}: [0,1] \rightarrow S\) starting and ending at \(w\) that do not intersect the critical fibers of \(h\). 

Denote by \( c_i \), \(i= 1,\ldots, r\) the critical values of \(h\). At any regular point \(p_i\) of the curve \( h^{-1} (c_i )\), the map \(h\) induces a biholomorphism between a compact disc \(D_i\) transversal  to \( h^{-1} (c_i \) at \(p_i\) and a neighborhood \(\delta_i\) of \(c_i \) in \(C\). We can assume that the discs \(\delta_i := h(D_i)\) are disjoint. Given a point \( w_i \in \partial D_i\), let \(\gamma_i:[0,1] \rightarrow E\setminus \cup_i \text{Int} (\delta_i)  \) be a family of paths with origin \(c\) and end point \( h(w_i) \), having the property of being disjoint appart from their origin. Since \(h\) induces a \(C^\infty\)-fibration with connected fibers from \( S \setminus \cup _i h^{-1} (c_i ) \) to \( C \setminus \{c_i \} \), given a lift \(w\in h^{-1} (c)\), we can lift the path \( \gamma_i\) to a path \(\tilde{\gamma_i} \)  starting at \( w \) and ending  at \( w_i\).  We denote by \(\mu_i\) the concatenation \(\gamma_i
 \star \partial \delta_i\star \gamma_i^{-1} \) and by \(\tilde{\mu_i}\) its lift \( \tilde{\gamma_i} \star \partial D_i \star \tilde{\gamma_i}^{-1}\). Notice the following two facts: 
\begin{enumerate} 
\item \(\tilde{\mu_i}\) is homotopically trivial in \(S \), 
\item the representatives of the loops \(\mu_i\)  in \(\pi_1 (C\setminus \{c_i \}, c)\) generate the kernel of the map 
\begin{equation} \label{eq: map fundamental group}  \pi_1 (C\setminus \{z_i \}, c)\rightarrow \pi_1 (C, c)\end{equation}  
induced by inclusion.
\end{enumerate} 
(The first is clear, and the second comes from the fact that the paths \(\gamma_i\) are disjoint appart from their origin.) 

We are now in a position to prove the claim in the case of a regular fiber. Consider an element in the kernel of the map \(h_\star: \pi_1 (S,w) \rightarrow \pi_1(C, c)\) represented by a loop \( \varepsilon : [0,1] \rightarrow S \) starting and ending at \(w\). Up to homotopy, we can assume that \(\varepsilon\) intersects none of the critical fibers of \(h\). The image \( h \circ \varepsilon\) is then a loop starting and ending at \(c\) which defines an element of the fundamental group of \( \pi_1 (C \setminus \{ z_i \} , c ) \) belonging to the kernel of the map \eqref{eq: map fundamental group}. Hence, \(h\circ \varepsilon \) is homotopic to a word \(W (\mu_1, \ldots , \mu_r) \).  Consider the path \(\varepsilon ' = \varepsilon \star W (\tilde{\mu_1}, \ldots, \tilde{\mu_r}) ^{-1} \); by construction, 
\begin{itemize} 
\item[(i)]  \(\varepsilon' \) is homotopic to \(\varepsilon\) is \(S\) 
\item[(ii)] it intersects none of the critical fibers of \(h\), 
\item[(iii)] its image \( h \circ \varepsilon '\) with fixed extremities \(c\) is homotopic in \( C\setminus \{c_i \} \) to the constant path \( c\). 
\end{itemize} 
Since the restriction of \(h\) induces a locally trivial \(C^\infty\) fibration from \(S \setminus \cup_i h^{-1} (c_i ) \) to \(C \setminus \{c_i \}\) --- this is a proper submersion, so this is Ehresmann's fibration theorem --- we can lift the homotopy found in (iii) to a homotopy in \( S \setminus \cup_i h^{-1} (c_i ) \) between \(\varepsilon '\) and a loop contained in the fiber \( h^{-1} ( c)\).  This establishes that \(\varepsilon\) is homotopic with fixed extremities to a loop contained in the fiber \( h^{-1} (c)\), and ends the proof of the claim in the case of a regular fiber.

It remains to prove the claim for the singular fibers of \(h\). For each \(i\), there exists a neighborhood \(U\) of \(h^{-1} (c_i)\) that retracts by deformation on \( h^{-1} (c_i)\) (see \cite[Chapter X, Section 9]{ACG}); in particular, the images of \(\pi_1 (h^{-1} (c_i))\) and \(\pi_1 (U)\) in \(\pi_1 (S)\) coincide. Since \(U\) contains a regular fiber, we know from the prerceeding discussion that the image of \(\pi_1 (U)\) in \(\pi_1 (S)\) contains \( \text{Ker} (h_\star)\). On the other hand, the image of  \( \pi_1 (h^{-1}(c_i))\) in \(\pi_1 (S)\) is contained in \( \text{Ker} (h_\star)\) since \(h\) is constant on \(h^{-1} (c_i)\), so the sequence  \eqref{eq: exact sequence 1} is exact for the critical value \(c=c_i\), and the Lemma follows.
\end{proof} 

\begin{corollary}\label{c: exact sequence}
For every \(c\in C\), the sequence 
\begin{equation} \label{eq: exact sequence 2} H_1 (h^{-1} (c),\mathbb Z ) \rightarrow H _1 (S,\mathbb Z) \rightarrow H_1 (C,\mathbb Z) \end{equation} where the map on the left is induced by inclusion and the one on the right by \(h_*\),  is exact. 
\end{corollary}

\begin{proof}
This is a consequence of Lemma \ref{l:picard-lefschetz} and Hurewicz' theorem.
\end{proof}
\cbend

\end{document}